\documentclass[11pt]{article}%
\usepackage[T2A]{fontenc}%
\usepackage[utf8]{inputenc}%
\usepackage[english]{babel}%
\usepackage{amsmath,amssymb,amsthm,amsfonts}%
\usepackage{hyperref}%
\usepackage{enumerate}%
\usepackage{graphicx}
\usepackage[mathscr]{euscript}
\usepackage[DIV13]{typearea}
\usepackage{color,tikz}
\usepackage[width=.9\textwidth]{caption}
\allowdisplaybreaks%
\numberwithin{equation}{section}%
\usepackage{textgreek}

\newcommand*\mycirc[1]{%
\begin{tikzpicture}[baseline=(C.base)]
\node[draw,circle,inner sep=1pt](C){#1};
\end{tikzpicture}}

\newcommand{\Z}{\mathbb{Z}}
\renewcommand{\C}{\mathbb{C}}
\newcommand{\R}{\mathbb{R}}

\newcommand{\al}{\alpha}
\newcommand{\si}{\sigma}
\newcommand{\la}{\lambda}
\newcommand{\La}{\Lambda}
\newcommand{\be}{\beta}

\newcommand{\de}{\mathrm{e}}
\newcommand{\ga}{\gamma}
\newcommand{\ka}{\varkappa}

\newcommand{\Sym}{\mathrm{Sym}}
\newcommand{\GT}{\mathbb{GT}}

\newcommand{\M}{\mathbf{M}}
\newcommand{\s}{\mathcal{S}}
\newcommand{\bs}{\boldsymbol{\mathcal{S}}}
\newcommand{\bp}{{\mathbf{P}}}
\newcommand{\uni}{\mathrm{P}}

\newcommand{\bq}{{\mathbf{Q}}}
\newcommand{\bm}{{\mathbf{m}}}
\newcommand{\nf}{\xi}
\newcommand{\co}{\theta}
\newcommand{\PBD}{\mathsf{PB}}
\newcommand{\RSD}{\mathsf{RSK}}
\newcommand{\RD}{\mathsf{R}}
\newcommand{\LD}{\mathsf{L}}
\newcommand{\ii}{{\kappa}}
\newcommand{\II}{\mathscr{F}}
\DeclareMathOperator{\Dim}{\mathrm{Dim}}
\DeclareMathOperator{\prob}{\mathrm{Prob}}
\newcommand{\Ptab}{\mathscr{P}}
\newcommand{\Qtab}{\mathscr{Q}}
\newcommand{\ins}{\mathcal{I}}
\newcommand{\sm}{\mathsf{f}}

\newcommand{\alim}{\mathsf{a}}
\newcommand{\tlim}{\mbox{\sf{}\texttau}}
\newcommand{\blim}{\mathsf{b}}
\newcommand{\lalim}{\mathscr{G}}
\newcommand{\lablim}{\mathbf{G}}
\newcommand{\flim}{\mathsf{R}}
\newcommand{\elim}{\mathsf{L}}
\newcommand{\Zf}{\mathsf{Z}}

\newtheorem{proposition}{Proposition}[section]

\newtheorem{theorem}[proposition]{Theorem}
\theoremstyle{definition}

\newtheorem{definition}[proposition]{Definition}
\newtheorem{remark}[proposition]{Remark}
\newtheorem{comment}[proposition]{Comment}

\newtheorem{condition}[proposition]{Condition}
\newtheoremstyle{dynrule}
{}
{}
{\slshape}
{}
{\bf}
{.}
{.5em}
{}
\theoremstyle{dynrule}
\newtheorem*{sprule}{{Short-range Pushing Rule}}
\newtheorem*{drule}{{Donation Rule}}
\newtheorem{dynamics}{Dynamics}

\begin{document}
\title{Nearest neighbor Markov dynamics on Macdonald processes}
\author{Alexei Borodin\thanks{Department of Mathematics, 
Massachusetts Institute of Technology,
77 Massachusetts ave.,
Cambridge, MA 02139, USA.} \thanks{Institute for Information Transmission Problems, Bolshoy Karetny per. 19, Moscow, 127994, Russia.} \footnote{e-mail: \texttt{borodin@math.mit.edu}}
\qquad Leonid Petrov\thanks{Department of Mathematics, Northeastern University, 360 Huntington ave., Boston, MA 02115, USA.} $^{\dagger}$\footnote{e-mail: \texttt{l.petrov@neu.edu}}
}
\date{\vspace{20pt}\hfill\emph{To the memory of Andrei Zelevinsky}}

\maketitle

\begin{abstract}
	Macdonald processes are 
	certain probability measures on 
	two-dimensional
	arrays 
	of interlacing particles
	introduced by Borodin and Corwin in \cite{BorodinCorwin2011Macdonald}. 
	They are defined in terms of 
	nonnegative specializations
	of the Macdonald symmetric 
	functions and depend on two parameters $q,t\in [0; 1)$.
	Our main result is a classification of continuous time, nearest neighbor Markov dynamics on the space of interlacing arrays that act nicely on Macdonald processes.

	The classification unites known 
	examples of such dynamics and also yields many new ones.

	When $t = 0$, one dynamics leads to a new integrable interacting particle system on the one-dimensional lattice, which is a $q$-deformation of the PushTASEP (= long-range TASEP).

	When $q = t$, the Macdonald processes become the Schur processes
	of Okounkov and Reshetikhin \cite{okounkov2003correlation}. 
	In this degeneration, we discover new
	Robinson--Schensted-type 
	correspondences between words and pairs of Young tableaux
	that govern some of our dynamics.
\end{abstract}

\setcounter{tocdepth}{1}
\tableofcontents
\setcounter{tocdepth}{2}

\section{Introduction} 
\label{sec:introduction} 

Since the end of 1990's there has been a 
signficant progress in understanding 
the long time nonequilibrium behavior of certain \emph{integrable} 
(1+1)-dimensional interacting 
particle systems and random growth models in the 
KPZ universality class. 
The miracle of integrability in most cases (with the 
notable exception of the 
partially asymmetric simple exclusion process) can be 
traced to an extension of the Markovian evolution to a 
suitable (2+1)-dimensional 
random growth model whose remarkable properties 
yield the solvability. 

So far there have been 
two sources of such extensions. The first one originated 
from a classical 
combinatorial bijection known as the 
Robinson--Schensted--Knuth 
correspondence (RSK, for short). 
The RSK was first applied in this context by 
Johansson \cite{johansson2000shape} 
and Baik--Deift--Johansson \cite{baik1999distribution}, and the 
dynamical perspective has been 
substantially developed by 
O'Connell \cite{OConnell2003Trans}, \cite{OConnell2003}, \cite{Oconnell2009_Toda}, 
Biane--Bougerol--O'Connell \cite{BBO2004}
(see also Chhaibi \cite{Chhaibi2013}), 
Corwin--O'Connell--Sep\-p\"a\-l\"ainen--Zygou\-ras 
\cite{COSZ2011}, O'Connell--Pei
\cite{OConnellPei2012},
see also 
O'Connell--Sepp\"al\"ainen--Zygo\-uras \cite{OSZ2012}. 

The second approach was 
introduced by Borodin--Ferrari \cite{BorFerr2008DF}, 
and it was 
based on an idea of 
Diaconis--Fill \cite{DiaconisFill1990} of extending 
interwined 
``univariate'' Markov chains to a ``bivariate'' Markov chains that 
projects to 
either of the initial ones. This 
approach was further developed in Borodin--Gorin 
\cite{BorodinGorin2008}, 
Borodin--Gorin--Rains \cite{borodin-gr2009q}, 
Borodin \cite{Borodin2010Schur}, Betea 
\cite{betea2011elliptically}, and Borodin--Corwin 
\cite{BorodinCorwin2011Macdonald}. 
In what follows we use the term 
\emph{push-block dynamics} 
for the Markov chains constructed in this fashion 
(the reason for such a term will become clear later).  

While the two resulting (2+1)-dimensional 
Markov processes that extend the same 
(1+1)-dimensional one share many properties --- same fixed time 
marginals, same 
projections to many (1+1)-dimensional sections --- the relation 
between them 
have 
so far remained poorly understood. The original goal of the 
project whose 
results are presented in this paper was to bridge this gap.  

We start out in a fairly general setting of the ascending 
Macdonald processes 
introduced in Borodin--Corwin \cite{BorodinCorwin2011Macdonald}, 
which can be thought of as a fixed 
time snapshot of a (2+1)-dimensional 
random growth model. Our initial aim was 
to find all possible continuous time 
Markov chains that have the same fixed 
time 
marginals and same trajectory 
measures on certain (1+1)-dimensional 
sections as the push-block dynamics 
of \cite{BorodinCorwin2011Macdonald} 
(these marginals and 
trajectory measures are fairly natural 
in their own right, and they are, in a 
way, more basic than the push-block dynamics). 
Let us note that in the 
Macdonald setting no analog of the RSK was known, 
so we really only had 
one dynamics to start with. 

It quickly became obvious that 
these assumptions are not restrictive enough to 
lead to a meaningful answer, 
and we imposed an additional one --- the 
interaction between the particles 
has to be only via nearest neighbors 
(understood in a certain precise sense described below). Another 
obstacle that 
we faced was that the problem is essentially algebraic, and 
imposing positivity 
on transition probabilities is in a way unnatural; it is much 
easier to deal 
with formal Markov chains (formality in the 
sense of absence of the positivity 
assumption) and \emph{a posteriori} filter out 
those that are not positive. To 
indicate the omission of this 
assumption we write probabilistic terms in 
quotation marks below. 

The main result of the present paper 
is a complete classification of the 
continuous time nearest neighbor 
`Markov dynamics' that have prescribed fixed 
time marginals and prescribed 
evolution along certain one-dimensional sections. 

To our surprize, in the resulting classification 
we find finitely many 
`dynamics' of RSK-type, two of which turn into 
those coming from the RSK in the 
specialization that turns the Macdonald processes 
into the Schur processes 
(this 
corresponds to taking $q=t$, where 
$q$ and $t$ are two parameters of the 
Macdonald polynomials and processes). If we 
denote by $N$ the depth of the 
ascending Macdonald process (this means 
that two-dimensional particle arrays 
live in strip of height $N$), then 
we observe $N!-2$ new dynamics of
RSK-type which in the Schur case give 
rise to the same number of combinatorial 
bijections that are quite similar to the RSK. 
They appear to be new, and their 
investigation is a promising new direction. 
Moreover, in the Schur case it turns out 
that all the dynamics from 
our classification are positive, and thus define 
honest Markov processes. 

In the so-called $q$-Whittaker specialization, 
when we set the Macdonald 
parameter $t$ to 0, one of the RSK-type dynamics 
gives rise to a new integrable 
(1+1)-dimensional interacting particle system 
in the KPZ universality class 
that we call \emph{$q$-PushTASEP}. 
As $q\to0$, it degenerates to the one-sided 
version of the PushASEP of 
Borodin--Ferrari \cite{BorFerr08push}. A detailed analysis 
of this new particle system is a 
subject of a forthcoming publication of 
Corwin--Petrov \cite{CorwinPetrov2013}. 
Let us also note that another $q$-deformation 
of the RSK dynamics was previously found by 
O'Connell--Pei \cite{OConnellPei2012}, 
and we explain below how it relates to our work 
(that dynamics does not have nearest neighbor interactions but its slight modification does).

In a certain $q\to 1$ limit, 
cf. Borodin--Corwin \cite{BorodinCorwin2011Macdonald}, 
the $q$-Whittaker 
processes turn into the 
so-called Whittaker processes that are 
closely related to random directed 
polymers in random enviroment, 
see O'Connell--Yor \cite{OConnellYor2001}, O'Connell 
\cite{Oconnell2009_Toda}, 
Corwin--O'Connell--Sepp\"al\"ainen--Zygouras \cite{COSZ2011}. 
We observe that 
in 
this limit two of our RSK-type dynamics degenerate to those of
\cite{Oconnell2009_Toda} and (a continuous time limit of) 
\cite{COSZ2011}, the push-block dynamics 
turns into the 
so-called symmetric dynamics of \cite{Oconnell2009_Toda}, 
and (many) 
remaining dynamics from our 
classification are positive and new. We hope to 
return to them in a future work. 

Let us now describe the content of the paper in more detail. 

\paragraph{Ascending Macdonald processes.}

The ascending Macdonald processes introduced
in \cite{BorodinCorwin2011Macdonald} are certain probability
measures on triangular arrays of nonnegative integers
$\boldsymbol\la=
\{\la^{(k)}_{j}\}_{1\le j\le k\le N}$ (of depth $N$)
which satisfy \emph{interlacing constraints} 
$\la^{(k)}_{j+1}\le\la^{(k-1)}_{j}\le\la^{(k)}_{j}$ 
(for all meaningful $k$ and $j$), 
see Fig.~\ref{fig:GT_scheme}.\begin{figure}[htbp]
	\begin{center}
		\includegraphics[width=155pt]{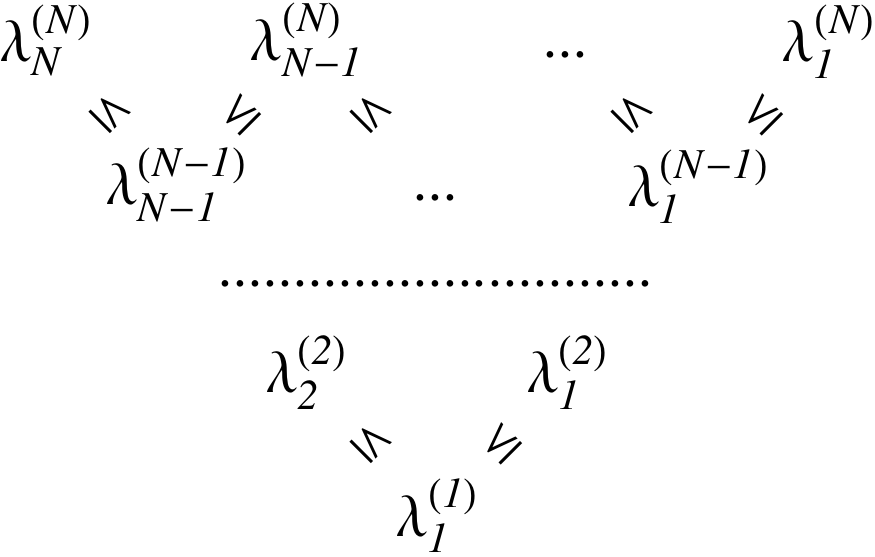}
	\end{center}  
  	\caption{An interlacing integer array of depth $N$.}
  	\label{fig:GT_scheme}
\end{figure} 
We will represent such arrays as 
particle configurations, see Fig.~\ref{fig:la_interlacing}.
These arrays are in bijection with 
semistandard Young tableaux and also with 
certain stepped surfaces in three dimensions
(see \S \ref{sub:semistandard_young_tableaux} below
for the former and, e.g., 
\cite{BorFerr2008DF} for the latter).

The probability weight assigned to each array $\boldsymbol\la$
by the Macdonald process has the form
\begin{align}\label{M_asc_intro}
	\prob(\boldsymbol\la)=
	\frac{P_{\la^{(1)}}(a_1)
	P_{\la^{(2)}/\la^{(1)}}(a_2)\cdots
	P_{\la^{(N)}/\la^{(N-1)}}(a_N)
	Q_{\la^{(N)}}(\rho)}{\Pi(a_1,\dots,a_N;\rho)}
	=:
	\M_{asc}(a_1,\dots,a_N;\rho)
	(\boldsymbol\la), 
\end{align}
where $\la^{(k)}=(\la^{(k)}_{1}\ge
\ldots\ge\la^{(k)}_{k})$, $k=1,\ldots,N$, are rows of the array
(they can be identified with Young diagrams with $\le k$ rows), 
$\Pi(a_1,\dots,a_N;\rho)$ is the normalizing constant, 
and
$P_{\bullet}$ and $Q_{\bullet}$ 
are the Macdonald symmetric functions.
The above probability measure depends on 
arbitrary positive parameters $a_1,\ldots,a_N$ 
and on a positive specialization $\rho$
of the algebra of symmetric functions. 
Moreover, all constructions implicitly
depend on two parameters $q,t\in[0,1)$. 
We review remarkable 
properties of the Macdonald symmetric functions and 
Macdonald processes in Appendix~\ref{sec:macdonald_processes}.

In the present paper we use the 
so-called \emph{Plancherel specializations}
$\rho_\tau$ indexed by one nonnegative parameter $\tau$
which plays the role of \emph{time}. 
These specializations are completely 
determined by the generating series for the 
symmetric functions $Q_{(n)}$ indexed by 
the one-row Young diagrams:
\begin{align*}
	\sum_{n\ge0}Q_{(n)}(\rho_\tau)\cdot u^n=e^{\tau u}.
\end{align*}
When $\tau=0$, 
the corresponding measure $\M_{asc}(a_1,\ldots,a_N;\rho_0)$ 
is concentrated 
on the zero configuration $\la^{(k)}_j=0$, $1\le j\le k\le N$.

\begin{figure}[htbp]
	\begin{center}
		\begin{tabular}{cc}
			\includegraphics[width=191pt]{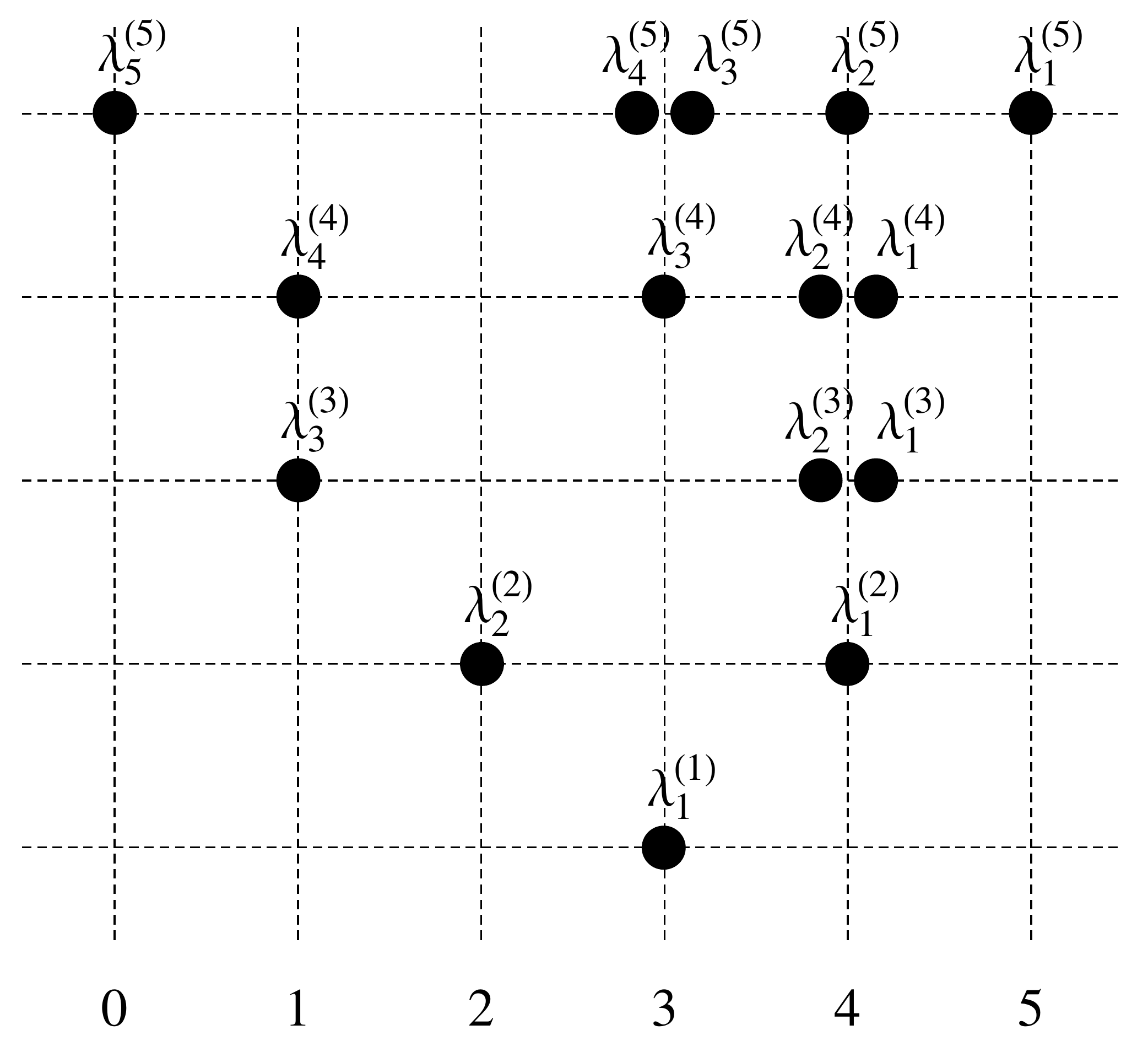}
			&
			\includegraphics[width=191pt]{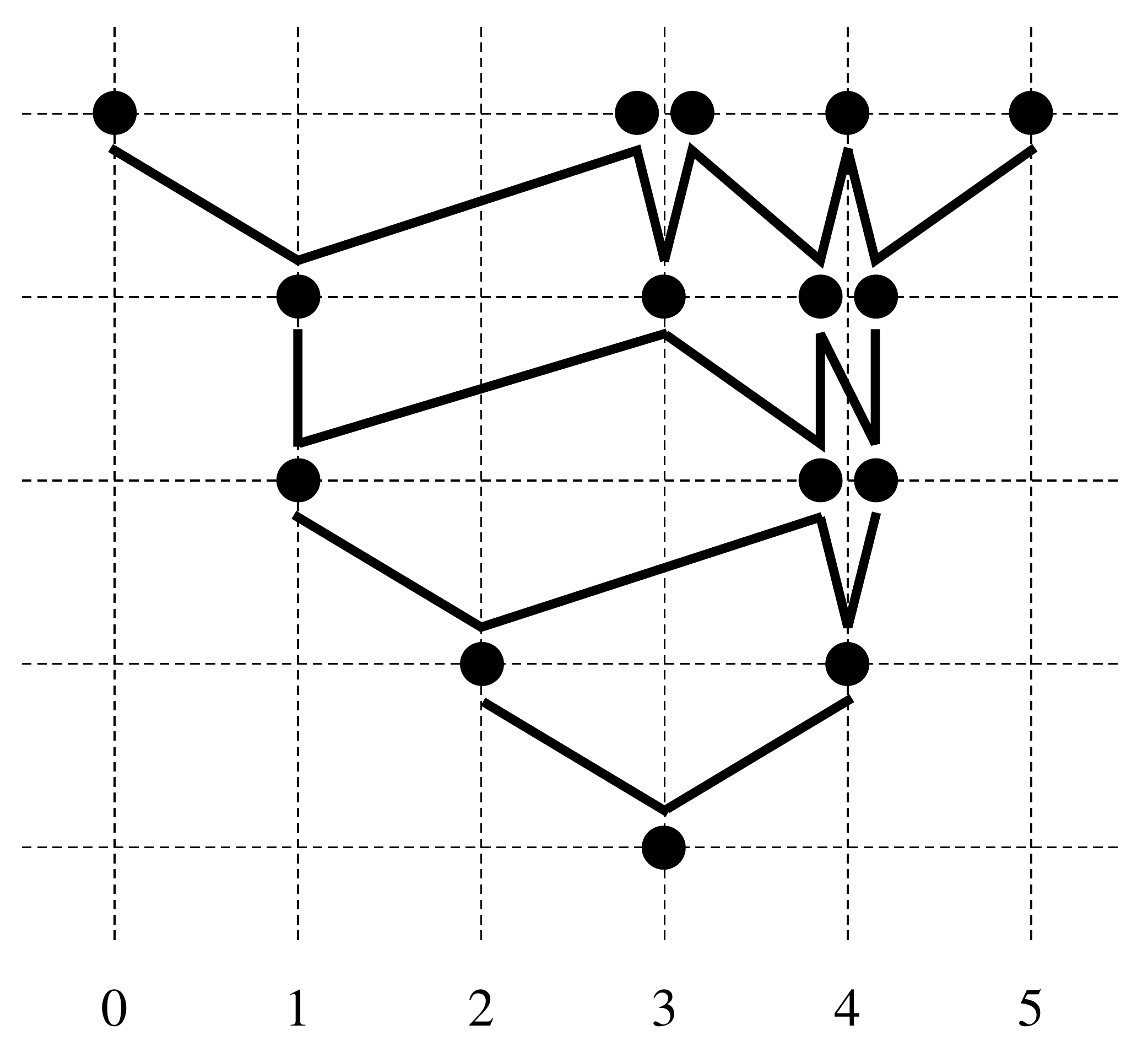}
		\end{tabular}
	\end{center}
  	\caption{Particle configuration $\boldsymbol\la$
  	and a visualization of the interlacing property.}
  	\label{fig:la_interlacing}
\end{figure}

When $q=t$,
random interlacing arrays $\boldsymbol\la$
(with distribution 
\eqref{M_asc_intro}
corresponding
to specialization $\rho_\tau$) may be interpreted as 
images of random words 
with letters appended according to independent
Poisson processes of rates $\{a_j\}$,
under the Robinson--Schensted--Knuth
correspondence (RSK, for short), 
see \cite{OConnell2003Trans}, \cite{OConnell2003}.
Moreover, for $q=t$, the
distribution of each row $\la^{(k)}$ of the array
is deeply related to 
the classical Schur--Weyl duality
\cite{biane2001approximate}, 
\cite{Meliot20091/2},
\cite[\S2.3]{BorodinBufetov2013}.
In general, for $q=t$ the
Macdonald processes 
become the Schur processes 
introduced in \cite{okounkov2003correlation}.

For $t=0$, in a scaling limit as $q\nearrow 1$, the distribution
of the array $\boldsymbol\la$ converges to the image of 
the semi-discrete Brownian polymer 
under the geometric (tropical)
RSK correspondence
\cite{Oconnell2009_Toda}, \cite{BorodinCorwin2011Macdonald}.
In another scaling limit, namely, as $t=q^{\theta}\to1$, 
Macdonald processes lead to
multilevel general $\beta$ Jacobi 
ensemble of random matrix theory
\cite{BorodinGorin2013beta}.

\paragraph{Univariate dynamics.}

For each $k$, the distribution of
the $k$th row $\la^{(k)}=
(\la^{(k)}_{1}\ge \ldots\ge\la^{(k)}_{k})$
of the ascending Macdonald process 
\eqref{M_asc_intro}
is given by the Macdonald measure 
\begin{align*}
	\prob(\la^{(k)})= 
	\frac{P_{\lambda^{(k)}}(a_1,\ldots,a_k) Q_{\lambda^{(k)}}
	(\rho_\tau)} {\Pi(a_1,\ldots,a_k;\rho)}
	=:\M\M(a_1,\ldots,a_k;\rho_\tau)(\la^{(k)}).
\end{align*}

There exists a distinguished 
continuous-time Markov dynamics 
which provides a coupling of 
the measures $\M\M(a_1,\ldots,a_k;\rho_\tau)$
for all $\tau\ge0$. It is defined 
in terms of jump rates as follows:
\begin{align}\label{univariate_intro}
	&
	\mbox{rate}\,
	(\la^{(k)}\to\la^{(k)}+\de_j)=
	\frac{P_{\la^{(k)}+\de_j}
	(a_1,\ldots,a_k)}
	{P_{\la^{(k)}}
	(a_1,\ldots,a_k)}
	\psi'_{\la^{(k)}+\de_j/\la^{(k)}}
	d\tau,
\end{align}
where 
$\de_j=(0,\ldots,0,1,0,\ldots,0)$
(with ``$1$'' at the $j$th place). 
Here $\psi'_{\lambda+e_j/\lambda}$ are 
the ``Pieri coefficients'': 
$(x_1+x_2+...)P_\lambda(x)=
\sum_j \psi'_{\lambda+\de_j/\lambda} P_{\lambda+\de_j}(x)$. 

Let us denote 
by $\uni_k(\tau;\la^{(k)},\mu^{(k)})$
the transition probability
from $\la^{(k)}$
to $\mu^{(k)}$ during time $\tau$
corresponding to the jump rates
\eqref{univariate_intro}.
The coupling mentioned above
is given by
\begin{align}\label{MM_coupling}
	\sum\nolimits _{\la^{(k)}}
	\M\M(a_1,\ldots,a_k;\rho_\si)(\la^{(k)})
	\cdot \uni_k(\tau;\la^{(k)},\mu^{(k)})=
	\M\M(a_1,\ldots,a_k;\rho_{\si+\tau})(\mu^{(k)}),
\end{align}
where $\si\ge0$. In matrix form, 
$\M\M(a_1,\ldots,a_k;\rho_\si)\uni_k(\tau)=
\M\M(a_1,\ldots,a_k;\rho_{\si+\tau})$.

The dynamics $\uni_k$ can be viewed as a
$(q,t)$-analogue of the $k$-particle Dyson Brownian
motion \cite{dyson1962brownian}.
Moreover, it is possible to recover the 
latter process from $\uni_k$
by setting $q=t$ and taking 
a diffusion limit as $\tau\to+\infty$ (in the same limit, 
the Macdonald processes \eqref{M_asc_intro} 
turn into the GUE eigenvalue corners
distributions).
In another scaling regime, namely, for $t=0$ and as
$q\nearrow 1$, the dynamics $\uni_k$ 
becomes closely related to the quantum Toda lattice,
see
\cite{Oconnell2009_Toda}.

\paragraph{Push-block multivariate dynamics.}

We will refer to the above dynamics $\uni_k$
as to the \emph{univariate dynamics}. Each $\uni_k$
lives on the $k$th floor of the interlacing array
$\boldsymbol\la$ (cf. Fig.~\ref{fig:la_interlacing}).
We want to \emph{stitch} the
$\uni_k$'s into a \emph{multivariate} continuous-time 
Markov dynamics
living on interlacing arrays $\boldsymbol\la$.
One such construction
(inspired by an idea of Diaconis--Fill \cite{DiaconisFill1990})
was introduced in \cite[\S2.3.3]{BorodinCorwin2011Macdonald}.
In the present paper 
we call that multivariate dynamics 
the \emph{push-block dynamics}, and 
denote by $\bp^{(N)}_{\PBD}(\tau;\boldsymbol\la,\boldsymbol\nu)$
the corresponding transition probabilities.

The evolution $\bp^{(N)}_{\PBD}$
is 
fairly simple and can be
described as follows. Each particle
$\la^{(k)}_{j}$, $1\le j\le k\le N$,
has an independent exponential clock 
with certain rate $a_k\cdot S_j(\la^{(k-1)},\la^{(k)})$
depending on the configuration of particles
at levels $k-1$ and $k$ (see \eqref{S_j} for 
an explicit formula for $S_j$; for $q=t$, one simply has $S_j\equiv1$). 
When the 
clock of $\la^{(k)}_{j}$ rings, this particle
jumps to the right by one if it is not \emph{blocked} 
by a lower particle (i.e., if $\la^{(k)}_{j}<\la^{(k-1)}_{j-1}$).
Otherwise, the jump does not happen. 
Next, if a jump of $\la^{(k)}_{j}$
violates the interlacing with particles
above 
(i.e., if $\la^{(k)}_j=\la^{(k+1)}_j$), 
then the jumping particle $\la^{(k)}_{j}$
\emph{pushes} $\la^{(k+1)}_j$ 
(as well as all other particles
with $\la^{(m)}_{j}=\la^{(k)}_j$, 
$m\ge k+2$)
to the right by one.

By the very construction, we see that 
during an infinitesimally small time interval, 
in $\bp^{(N)}_{\PBD}$
the transition $\la^{(k)}\to\nu^{(k)}$
at each level $k$ depends only on the
previous and new states ($\la^{(k-1)}$
and $\nu^{(k-1)}$, respectively) at level $k-1$.
We will refer to this as to the 
\emph{sequential update} property. 

Before explaining how $\bp^{(N)}_{\PBD}$
acts on Macdonald processes 
$\M_{asc}$ \eqref{M_asc_intro}, we 
need the fact that $\M_{asc}$ possess a certain
\emph{Gibbs property}. 
It means that for each $k$, given fixed $\la^{(k)}$, the 
conditional distribution 
\begin{align}\label{Gibbs_property_intro}
	\prob(\la^{(1)},\ldots,\la^{(k-1)}\mid\la^{(k)})
	=
	\frac{P_{\la^{(1)}}(a_1)
	P_{\la^{(2)}/\la^{(1)}}(a_2)
	\ldots P_{\la^{(k)}/\la^{(k-1)}}(a_{k})}
	{P_{\la^{(k)}}(a_1,\ldots,a_k)}
\end{align}
is completely determined by $\la^{(k)}$ and does not
depend on the specialization $\rho$. 
(This conditional distribution also does not
depend on $\la^{(k+1)},\ldots,\la^{(N)}$, 
which is a manifestation of the sequential structure of 
Macdonald processes.) Thus, Macdonald processes
are included in a larger class of 
\emph{Gibbs measures} on interlacing arrays,
which, by definition, satisfy \eqref{Gibbs_property_intro}.

\begin{proposition}[\cite{BorodinCorwin2011Macdonald}]
\label{prop:pb_intro}
	{\bf1.\/} The push-block dynamics $\bp^{(N)}_{\PBD}$
	preserves the class of Gibbs measures on interlacing arrays.
	In particular, 
	the evolution
	of each sub-array $\la^{(1)},\ldots,\la^{(k)}$
	is completely determined by the dynamics
	of $\la^{(k)}$. 

	{\bf2.\/} For a Gibbs initial condition, 
	the evolution of $\la^{(k)}$
	under $\bp^{(N)}_{\PBD}$
	coincides with the univariate dynamics 
	$\uni_k$.
\end{proposition}

Together with 
\eqref{MM_coupling},
this implies that the push-block 
dynamics provides the following coupling
of the Macdonald processes 
\eqref{M_asc_intro} corresponding to 
Plancherel specializations $\rho_\tau$:
\begin{align}\label{action_on_Masc_intro}
	\M_{asc}(a_1,\ldots,a_N;\rho_\si)
	\bp^{(N)}_{\PBD}(\tau)=
	\M_{asc}(a_1,\ldots,a_N;\rho_{\si+\tau}).	
\end{align}

For special Macdonald parameters, 
there are other known sequential update 
dynamics $\bp^{(N)}$ 
on interlacing arrays
satisfying the conditions of 
Proposition \ref{prop:pb_intro}
(we will refer to them as to 
\emph{multivariate dynamics}, for short).
Namely, for $q=t$, 
one can get another multivariate dynamics
(different from the $q=t$ degeneration of 
$\bp^{(N)}_{\PBD}$) by
applying the classical RSK insertion to 
growing
random words, 
see \cite{OConnell2003Trans}, \cite{OConnell2003},
and also \cite{BorodinGorinSPB12}. 
In the case $t=0$, 
in a recent paper by O'Connell--Pei \cite{OConnellPei2012},
a certain $q$-deformation of the RSK-driven dynamics
was considered, which required a randomization of the
insertion algorithm. 
Applying this random insertion to the same random
input as for $q=t$, one gets a multivariate
dynamics on interlacing arrays.

The present paper is devoted to a 
systematic classification of
multivariate Markov dynamics.

\paragraph{Nearest neighbor dynamics.}

Let $\bp^{(N)}$ be a multivariate
dynamics. It evolves as a Markov jump process: 
during an infinitesimally small time interval,
several particles in the array 
move to the right by one (at most one particle at 
each level can move at most by one 
because of the nature of the univariate
dynamics). 
Due to the sequential structure of
$\bp^{(N)}$, these moves can be described 
inductively (level by level) in the following way:
\begin{enumerate}[$\bullet$]
	\item Each particle $\la^{(m)}_{i}$,
	$1\le i\le m\le N$, has an
	independent exponential clock with 
	(possibly zero) rate depending on 
	$\la^{(m-1)}$ and $\la^{(m)}$. When the 
	clock rings, $\la^{(m)}_{i}$ jumps
	to the right by one.
	\item When any particle $\la^{(k-1)}_{j}$ moves
	to the right by one (independently or 
	due to a triggered move), it has a chance 
	(with some probabilities depending
	on $\la^{(k-1)}$ and $\la^{(k)}$) to force
	one of the particles 
	$\la^{(k)}_1,\ldots,\la^{(k)}_{k}$
	to instantaneously move to the right by one as well.
	It is also possible that 
	(with some probability)
	the 
	move of 
	$\la^{(k-1)}_{j}$ does not propagate to higher levels.
\end{enumerate}

To obtain a meaningful classification, we 
restrict our 
attention to the following 
subclass of multivariate dynamics:
\begin{figure}[htbp]
	\begin{center}
		\includegraphics[width=345pt]{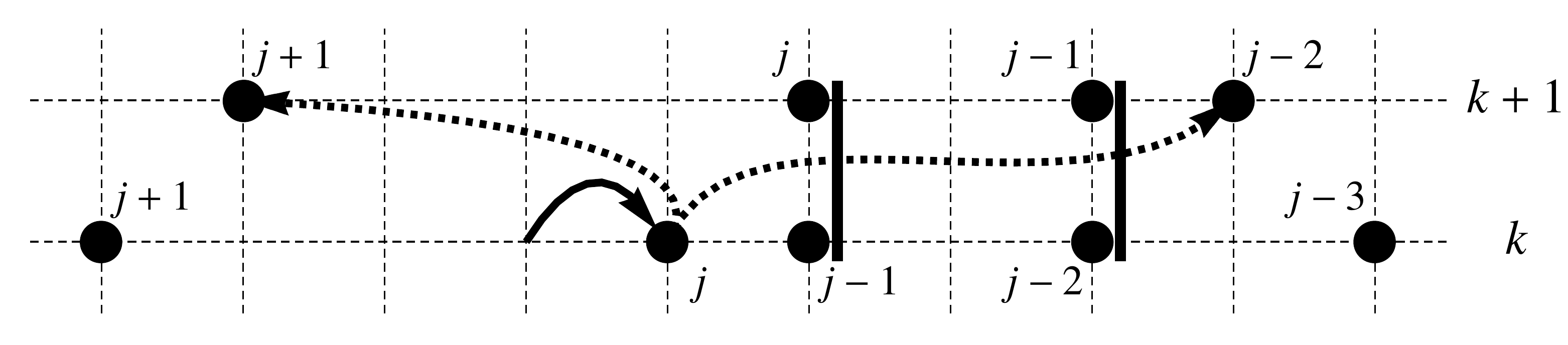}
	\end{center}  
  	\caption{Nearest neighbor interactions: 
  	a moved particle (solid arrow) 
  	can long-range \emph{pull} 
  	its immediate left neighbor,
  	or long-range \emph{push} 
  	its leftmost free right neighbor 
  	(dashed arrows).}
  	\label{fig:NN_intro}
\end{figure} 
\begin{definition}\label{def:NN_intro}
	A multivariate dynamics
	is called 
	\emph{nearest neighbor}
	if the move of any particle $\la^{(k-1)}_j$
	can affect
	only its closest upper
	neighbors in the triangular array. 
	By the \emph{upper right neighbor}
	of $\la^{(k-1)}_j$
	we mean the first particle
	$\la^{(k)}_{i}$
	at level 
	$k$ strictly
	to the right of $\la^{(k)}_{j+1}$
	which is not blocked, i.e., 
	for which $\la^{(k)}_{i}<\la^{(k-1)}_{i-1}$.
	The notion of the \emph{upper left neighbor} 
	is more straightforward,
	this is always the particle
	$\la^{(k)}_{j+1}$ (the moved particle 
	$\la^{(k-1)}_j$ cannot block
	$\la^{(k)}_{j+1}$).
\end{definition}
Note that these interactions 
(see Fig.~\ref{fig:NN_intro})
are
\emph{long-range}, i.e., they 
may happen regardless of the distance
between particles (but the probability
of pushing or pulling can 
depend on this distance, as well as 
on positions of other particles). 
These long-range interactions are to be compared with the
\emph{short-range pushing} 
of the push-block dynamics, 
which happens when interlacing is violated. 
The short-range pushing mechanism must in fact 
be present in any multivariate dynamics.

The push-block dynamics is nearest
neighbor, as well as the RSK-driven
dynamics for $q=t$. O'Connell--Pei's randomized 
($q$-weighted) column
insertion algorithm
\cite{OConnellPei2012}
is not nearest neighbor
(but 
can be modified to become one, see 
\S \ref{ssub:remark_a_nearest_neighbor_markov_process_inspired_by_dynamics_dyn:oconnell} below).
Note also that in the latter two dynamics,
any move propagates to all higher levels with
probability one. We call dynamics with 
such obligatory move 
propagation the
\emph{RSK-type dynamics}. 

We write down (Proposition \ref{prop:rlw_system}) 
a system of linear equations for the 
rates of independent jumps and 
probabilities of triggered moves
in a nearest neighbor dynamics
which is equivalent to the conditions of 
Proposition \ref{prop:pb_intro}.
In this way, the problem of classification
of nearest neighbor dynamics
becomes essentially algebraic. 
It is thus convenient
not to impose positivity 
on transition probabilities,
and consider formal Markov `dynamics' 
(quotation marks indicate the absence 
of this positivity
assumption).

The operation of taking linear combinations
of solutions of the linear system
translates into a certain
\emph{mixing} of nearest-neighbor `dynamics'.
Let us explain how this procedure works
when applied to two `dynamics' $\bp^{(N)}$
and $\bp'^{(N)}$ yielding new `dynamics'
$\tilde\bp^{(N)}$. The result depends on
functions $\co^{(k)}(\la^{(k-1)},\la^{(k)})$, 
$k=2,\ldots,N$, defined on consecutive ``slices''
of our triangular array~$\boldsymbol\la$. 
The mixed `dynamics' $\tilde\bp^{(N)}$ is described in 
the following way:
\begin{enumerate}[$\bullet$]
	\item 
	The rate of independent jump 
	of each particle $\la^{(m)}_{i}$
	is the linear combination 
	of the corresponding rates 
	in the `dynamics'
	$\bp^{(N)}$
	and $\bp'^{(N)}$ with coefficients
	$\co^{(m)}(\la^{(m-1)},\la^{(m)})$
	and $1-\co^{(m)}(\la^{(m-1)},\la^{(m)})$,
	respectively.
	\item Any moving particle $\la^{(k-1)}_{j}$ 
	with some probabilities 
	affects one of its upper neighbors.
	For the `dynamics', $\tilde\bp^{(N)}$,
	these probabilities are obtained 
	as linear combinations of the 
	corresponding probabilities from 
	$\bp^{(N)}$
	and $\bp'^{(N)}$ with coefficients
	$\co^{(k)}(\la^{(k-1)}+\de_j,\la^{(k)})$
	and $1-\co^{(k)}(\la^{(k-1)}+\de_j,\la^{(k)})$,
	respectively.
\end{enumerate}
Mixing of any finite number 
of `dynamics' can be defined in a similar
manner.

There is a lot of freedom 
in choosing arbitrary
coefficients $\co^{(k)}$; 
hence one should expect that 
any nearest neighbor `dynamics' 
can be represented as the mixing of 
a finite number of certain \emph{fundamental} 
nearest neighbor
`dynamics'; we describe them next.
The choice of 
the fundamental `dynamics'
is not canonical, 
and we use the ones that seem the most 
natural to us.

\paragraph{Fundamental `dynamics' and main result.}

There are $N!+2(N-1)!+1$ pairwise distinct
fundamental `dynamics'. They are
divided into three families,
plus the push-block dynamics already described above:
\begin{enumerate}[$\bullet$]
	\item \emph{RSK-type} fundamental `dynamics'  
	are completely characterized
	by the requirement that at each level $k$, $1\le k\le N$, 
	only one
	particle can independently jump
	at rate $a_k$. There are $N!$ such `dynamics'.
	\item \emph{Right-pushing} fundamental `dynamics'
	are characterized by the 
	requirement that at each level $k\le N-1$, only one particle 
	long-range
	pushes (with probability one)
	its upper right neighbor, and there are no other 
	long-range interactions.
	There are $(N-1)!$ such `dynamics'.
	\item In \emph{left-pulling} fundamental `dynamics', 
	at each level $k\le N-1$, only one particle pulls
	its upper left neighbor (with probability one), 
	and there are no 
	other long-range interactions. The number of 
	left-pulling `dynamics' is also $(N-1)!$.
\end{enumerate}
\begin{theorem}\label{thm:main_intro}
	Any nearest neighbor `dynamics' can be obtained
	as a mixing of the fundamental nearest neighbor `dynamics'.
\end{theorem}
We prove this theorem 
in \S\S \ref{sec:multivariate_continuous_time_dynamics_on_interlacing_arrays_}--\ref{sec:nearest_neighbor_multivariate_continuous_time_dynamics_on_interlacing_arrays} 
(in particular, see Theorem \ref{thm:characterization_NN}).
We also discuss possibilities of making the mixing representation
unique in \S \ref{sub:characterization_of_nearest_neighbor_dynamics_} below.

\paragraph{RSK-type dynamics in the Schur case.}
In the case $q=t$, when Macdonald polynomials turn into the Schur polynomials,
all our fundamental nearest neighbor `dynamics' 
have nonnegative transition probabilities, and so become 
honest Markov processes. 

Of particular combinatorial interest
are the $N!$ RSK-type fundamental dynamics 
$\bp^{(N)}_{\RSD[\boldsymbol h]}$
indexed by 
$N$-tuples $\boldsymbol h=(h^{(1)},\ldots,h^{(N)})\in
\{1\}\times \{1,2\}\times \ldots\times\{1,2,\ldots,N\}$.
When $q=t$, the evolution of $\bp^{(N)}_{\RSD[\boldsymbol h]}$
can be interpreted as application of a \emph{deterministic} 
insertion algorithm (we call it \emph{$\boldsymbol h$-insertion})
to growing random words. 

In terms of interlacing arrays
(they are in a classical bijection with  
semistandard Young tableaux, see \S \ref{sub:semistandard_young_tableaux}
and also \S \ref{sub:from_interlacing_arrays_to_semistandard_tableaux}
for a ``dictionary'' between two ``languages''),
$\boldsymbol h$-insertion means that several particles in the array 
move to the right by one. These particles are selected
in a certain deterministic way 
which depends on
$\boldsymbol h$, the current state of the array, and 
the new letter which appears in the growing random word.
If this letter is $k$, then exactly one particle at each of the levels
$k,k+1,\ldots,N$ moves.
See \S \ref{sub:rs_type_fundamental_dynamics_in_the_schur_case_and_deterministic_insertion_algorithms} for a full description of 
$\boldsymbol h$-insertions.

In particular cases, when $\boldsymbol h=(1,1,\ldots,1)$ or 
$\boldsymbol h=(1,2,\ldots,N)$, the $\boldsymbol h$-insertion
becomes the classical RSK row or column insertion, 
respectively. The other $N!-2$ RSK-type 
insertion algorithms seem to be new.

\paragraph{$q$-Whittaker processes and $q$-PushTASEP.}

When $t=0$ and $0<q<1$, 
Macdonald processes turn into
$q$-Whittaker processes,
see \cite[\S3.1]{BorodinCorwin2011Macdonald} 
for properties specific to 
this case.
In contrast with the $q=t$ case, 
for $t=0$
almost none of the fundamental
nearest neighbor `dynamics' 
have nonnegative transition probabilities
(see Proposition \ref{prop:q_honest_Markov}
below for a detailed statement).
A notable exception (along with the push-block dynamics)
is the RSK-type fundamental dynamics in which only the 
rightmost particles $\la^{(1)}_{1},\ldots,\la^{(N)}_{1}$
can jump independently (with rates $a_1,\ldots,a_N$).
See Dynamics \ref{dyn:q_row} 
in \S \ref{sub:multivariate_dynamics_in_the_q_whittaker_case}
for its complete description.

This new dynamics 
(denote it by $\bp^{(N)}_{\text{\textit{q-row}}}$)
may be regarded as a 
$q$-deformation of the Schur case dynamics
driven by the classical row insertion RSK algorithm, 
in the same way as
O'Connell--Pei's insertion algorithm
\cite{OConnellPei2012} 
(and it nearest neighbor modification)
serves
as a $q$-deformation of 
the column insertion RSK algorithm.

The $q$-row insertion dynamics
has a remarkable property 
that the evolution of
the rightmost particles 
$\la^{(1)}_{1}\le\ldots\le\la^{(N)}_{1}$
of the interlacing array
under 
$\bp^{(N)}_{\text{\textit{q-row}}}$ is \emph{Markovian}:
each particle $\la^{(m)}_{1}$
jumps to the right independently of others 
at rate $a_m$; and any moved particle $\la^{(k-1)}_{1}$
long-range pushes $\la^{(k)}_{1}$ with probability 
$q^{\la^{(k)}_{1}-\la^{(k-1)}_{1}}$ 
(where $\la^{(k-1)}_{1}$
is the coordinate of the $(k-1)$th particle before the move). 
We call this 
(1+1)-dimensional interacting particle system
the \mbox{\emph{$q$-PushTASEP}}: when $q=0$, it becomes the 
PushTASEP considered in 
\cite{BorFerr08push}, \cite{BorFerr2008DF}.

Markov evolution of the rightmost
particles that we observe complements
the similar phenomenon for the leftmost
particles known earlier. 
Namely, 
under the push-block dynamics, as
well as under O'Connell--Pei's insertion algorithm, 
the leftmost particles 
of the interlacing array
evolve according to 
$q$-TASEP 
\cite{BorodinCorwin2011Macdonald}, 
\cite{BorodinCorwinSasamoto2012},
\cite{OConnellPei2012}.
The $q$-TASEP and $q$-PushTASEP
seem to be the only Markovian evolutions
which can arise
as restrictions of
$q$-Whittaker
nearest neighbor
dynamics to leftmost (resp. rightmost)
particles of the interlacing array; see
Propositions \ref{prop:qTASEP} and 
\ref{prop:qpushTASEP} in \S 
\ref{sub:taseps}.

\paragraph{Whittaker limit and directed polymers.}

In a suitable scaling limit 
as $q\nearrow1$, see
\eqref{q_to_1_scaling} below,
the $q$-Whittaker processes 
turn into the Whittaker processes.
The latter are certain probability measures
on $\R^{\frac{N(N+1)}{2}}$
(the interlacing constraints disappear in the limit).
See \cite[\S4]{BorodinCorwin2011Macdonald}.
Each of our fundamental nearest neighbor
`dynamics' 
suggests
a system of stochastic differential equations;
the corresponding diffusion in
$\R^{\frac{N(N+1)}{2}}$ should couple
the Whittaker processes in a way similar to 
\eqref{action_on_Masc_intro}. 
Two of such systems of SDEs appeared in 
\cite{Oconnell2009_Toda} (see also
\cite[\S4.1 and \S5.2]{BorodinCorwin2011Macdonald})
in connection with the O'Connell--Yor semi-discrete
directed polymer \cite{OConnellYor2001}.

For each $k=1,\ldots,N$, 
consider the partition function 
of the semi-discrete directed polymer \cite{OConnellYor2001},
\cite{Oconnell2009_Toda}
\begin{align*}
	\Zf^{(k)}(\tlim):=
	\int_{0<s_1<\ldots<s_{k-1}<\tlim} 
	e^{B_1(s_1)+
	\big(B_2(s_2)-B_2(s_1)\big)
	+\ldots+
	\big(B_k(\tlim)-B_k(s_{k-1})\big)}ds_1 \ldots ds_{k-1}.
\end{align*}
Here $B_1,\ldots,B_N$ are independent 
one-dimensional Brownian motions (possibly with drifts)
which start from zero.
The $\Zf^{(k)}$ satisfy the following system of 
SDEs:
\begin{align*}
	d\Zf^{(k)}=Z^{(k-1)}d\tlim + \Zf^{(k)}dB_k,\qquad
	k=1,\ldots,N
\end{align*}
(by agreement, $\Zf^{(0)}\equiv0$); and the free energies
$\mathsf{F}^{(k)}(\tlim):=\log(\Zf^{(k)}(\tlim))$ satisfy
\begin{align}\label{F_equation_intro}
	d\mathsf{F}^{(k)}=
	dB_k+e^{\mathsf{F}^{(k-1)}-\mathsf{F}^{(k)}}d\tlim,
	\qquad k=1,\ldots,N
\end{align}
(with $\mathsf{F}^{(0)}\equiv-\infty$).
In \S \ref{ssub:informal_argument_for_convergence}
below we present an empiric argument
why the position of the $k$-th
particle under the $q$-PushTASEP converges 
(as $q\nearrow1$ under the scaling \eqref{q_to_1_scaling})
to the free energy $\mathsf{F}^{(k)}(\tlim)$.
The Brownian part in the right-hand side of \eqref{F_equation_intro}
corresponds to independent jumps of particles in the 
$q$-PushTASEP, and the coefficient of $d\tlim$
represents the pushing. 
It is worth noting that under the scaling \eqref{q_to_1_scaling}, 
the evolution of $q$-TASEP is described by essentially the same
system of SDEs, see \S \ref{ssub:symmetry} below.

\paragraph{Acknowledgments.}

The authors would like to thank 
Ivan Corwin, Vadim Gorin, and Sergey Fomin
for helpful discussions. We are also grateful to 
Neil O'Connell for useful remarks.
AB~was partially supported by the NSF
grant DMS-1056390.
LP~was partially supported by 
the RFBR-CNRS grants 10-01-93114 and 11-01-93105.


\section
{Markov dynamics preserving Gibbs measures. General formalism} 
\label{sec:markov_dynamics_preserving_gibbs_measures_general_formalism}

\subsection{Gibbs measures} 
\label{sub:gibbs_measures}

Let $\s_1,\ldots,\s_N$ be discrete countable sets, and assume that we have \emph{stochastic links} $\La^{N}_{N-1},\ldots,\La^{2}_{1}$ between them:
\begin{align*}
	\La^{k}_{k-1}\colon\s_k\times\s_{k-1}\to[0,1],
	\qquad
	\sum_{x_{k-1}\in\s_{k-1}}
	\La^{k}_{k-1}(x_k,x_{k-1})=1
\end{align*}
for any $x_k\in\s_k$, where $k=2,\ldots,N$.

Define the \emph{state space}
\begin{align}\label{state_space}
	\bs^{(N)}:=
	\bigg\{
	\mathbf{X}_N=(x_1,\ldots,x_N)\in
	\s_1\times \ldots\times\s_N\colon
	\prod_{k=2}^{N}\La^{k}_{k-1}(x_k,x_{k-1})\ne0
	\bigg\}.
\end{align}
We will also use the notation $\mathbf{X}_k:=(x_1,\ldots,x_k)$, and, more generally, $\mathbf{X}_{a;b}:=(x_a,\ldots,x_b)$ for any $1\le a\le b\le N$.

\begin{definition}\label{def:Gibbs}
	We say that a (nonnegative) probability measure $\bm^{(N)}$ on $\bs^{(N)}$ is \emph{Gibbs} if it can be written in the form
	\begin{align}
		\bm^{(N)}(\mathbf{X}_N)
		=
		m_N(x_N)
		\La^{N}_{N-1}(x_N,x_{N-1})
		\ldots
		\La^{2}_{1}(x_2,x_1),\quad 
		\mathbf{X}_N\in\bs^{(N)},
		\label{m_Gibbs}
	\end{align}
	where $m_N$ is some probability measure on the last set $\s_N$. 
\end{definition}

We emphasize that the definition of the Gibbs property relies on the stochastic links $\La^{k}_{k-1}$. A Gibbs measure is completely determined by its projection $m_N$ onto the last space $\s_N$: 
\begin{align}\label{sum_over_x1_N-1}
	\sum_{\mathbf{X}_{N-1}\in\bs^{(N-1)}}\bm^{(N)}(\mathbf{X}_N)=m_N(x_N).
\end{align}
That is, according to \eqref{m_Gibbs}, to obtain the measure $\bm^{(N)}$, one considers the stochastic evolution the measure $m_N$ on $\s_N$ under the sequence of stochastic links $\La^{N}_{N-1},\ldots,\La^{2}_{1}$.

\begin{remark}\label{rmk:zero_outside_state_space}
	For convenience, we will always assume that $\bm^{(N)}(\mathbf{X}_N)=0$ if $\mathbf{X}_N=(x_1,\ldots,x_N)\in\s_1\times \ldots\times\s_N\setminus\bs^{(N)}$.
\end{remark}

The conditional distribution of $x_1,\ldots,x_k$ given that $x_{k+1},\ldots,x_{N}$ are fixed, is readily seen to be
\begin{align}&
	\bm^{(N)}
	(\mathbf{X}_k\,|\, \mathbf{X}_{k+1;N})
	\nonumber
	=\frac{\bm^{(N)}(\mathbf{X}_N)}
	{\sum_{\mathbf{Y}_{k}\in\bs^{(k)}}
	\bm^{(N)}(\mathbf{Y}_{k},
	\mathbf{X}_{k+1;N})}\\&
	\label{conditional_distribution_Gibbs}
	\hspace{90pt}=
	\frac{m_N(x_N)
	\La^{N}_{N-1}(x_N,x_{N-1})
	\ldots
	\La^{2}_{1}(x_2,x_1)}
	{m_N(x_N)
	\La^{N}_{N-1}(x_N,x_{N-1})
	\ldots
	\La^{k+2}_{k+1}(x_{k+2},x_{k+1})}
	\\&
	\hspace{90pt}=
	\La^{k+1}_{k}(x_{k+1},x_k)
	\ldots 
	\La^{2}_{1}(x_{2},x_1).
	\nonumber
\end{align}
In particular, the conditional distribution of $x_1,\ldots,x_{N-1}$ given that $x_N$ is fixed, is simply $\La^{N}_{N-1}(x_{N},x_{N-1})\ldots \La^{2}_{1}(x_{2},x_1)$. That is, under any Gibbs measure $\bm^{(N)}$, the conditional distribution of several first components of $\mathbf{X}_N\in\bs^{(N)}$ given that the remaining components are fixed, \emph{does not depend} on the measure $\bm^{(N)}$. 

In a certain group-theoretic context, Gibbs measures have
also been 
called \emph{central}, cf. \cite[\S8]{BorodinOlshanski2010GTs}.


\subsection{Sequential update dynamics in discrete time} 
\label{sub:sequential_update_dynamics_in_discrete_time}

Our aim now is to define a certain class of (discrete-time) Markov chains on the state space $\bs^{(N)}$ which act `naturally' on Gibbs measures on this space. 

Let $\bp^{(N)}$ be a $\bs^{(N)}\times\bs^{(N)}$ stochastic matrix:
\begin{align}\label{Markov_negative_boldP}
	\bp^{(N)}\colon\bs^{(N)}\times\bs^{(N)}\to[0,1],\qquad 
	\sum_{\mathbf{Y}_{N}\in\bs^{(N)}}\bp^{(N)}(\mathbf{X}_N,\mathbf{Y}_{N})=1 
\end{align}
for any $\mathbf{X}_N\in\bs^{(N)}$. We will regard $\bp^{(N)}$ as a one-step transition matrix for a discrete-time Markov chain on $\bs^{(N)}$. That is, $\bp^{(N)}(\mathbf{X}_N,\mathbf{Y}_N)$ is the probability that the next state of the chain is $\mathbf{Y}_N$ if its current state is $\mathbf{X}_N$.

\begin{remark}\label{rmk:zero_outside_state_space_P}
	It is convenient to set $\bp^{(N)}(\mathbf{X}_N,\mathbf{Y}_{N})=0$ if either $\mathbf{X}_N$ or $\mathbf{Y}_N$ belongs to $\s_1\times \ldots\s_N\setminus \bs^{(N)}$ (cf. Remark \ref{rmk:zero_outside_state_space}).
\end{remark}
Let us give the main definition of the present subsection:

\begin{definition}\label{def:coherent_discrete}
	A Markov chain $\bp^{(N)}$ is called a \emph{sequential update dynamics} (in discrete time) if it satisfies the following conditions:

	\textbf{1.} 
	$\bp^{(N)}(\mathbf{X}_N,\mathbf{Y}_N)$ can be factorized in the following way:\footnote{Here and below, by agreement, $x_0$ and $y_0$ will be empty arguments, i.e., we will sometimes write $U_1(x_1,y_1\,|\,x_0,y_0)=U_1(x_1,y_1)$. We will also assume $\La^{1}_{0}(x_1,x_0):=1$.}
	\begin{align}\label{sequential_update}
		\bp^{(N)}(\mathbf{X}_N,
		\mathbf{Y}_N)=
		U_{1}(x_1,y_1)
		U_{2}(x_2,y_2\,|\, x_1,y_1)
		\ldots
		U_{N}(x_N,y_N\,|\, x_{N-1},y_{N-1}),
	\end{align}
	for any $\mathbf{X}_N, \mathbf{Y}_N\in\bs^{(N)}$. The functions $U_k$ are assumed to be nonnegative and satisfy
	\begin{align}\label{U_k_sum_to_one}
		\sum_{y_k\in\s_k} 
		U_k(x_k,y_k\,|\, x_{k-1},y_{k-1})=1,
	\end{align}
	for all $x_k\in \s_k$, $x_{k-1},y_{k-1}\in\s_{k-1}$, where $k=1,\ldots,N$.
	
	\textbf{2.} Define for every $k=1,\ldots,N$ the following $\s_k\times\s_k$ matrix:
	\begin{align}\label{Projection_P_k}
		P_k(x_k,y_k):=\sum_{\mathbf{X}_{k-1},\mathbf{Y}_{k-1}\in\bs^{(k-1)}}\bigg(
		\prod_{i=1}^{k}
		U_{i}(x_i,y_i\,|\, x_{i-1},y_{i-1})
		\La^{i}_{i-1}(x_i,x_{i-1})\bigg).
	\end{align}
	It can be readily verified that $P_k$ is a stochastic matrix, and thus defines a Markov chain on $\s_k$. We will refer to $P_k$ as to the (\emph{$k$th}) \emph{projection} of $\bp^{(N)}$: it shows how dynamics $\bp^{(N)}$ looks in restriction to $\s_k$ 
	(when acting on Gibbs measures).

	We require that these projections are compatible with the stochastic links $\La^{k}_{k-1}$ in the following sense:
	\begin{align}&
		\sum_{x_{k-1}\in\s_{k-1}}
		U_{k}(x_{k},y_{k}\,|\,x_{k-1},y_{k-1})
		\La^{k}_{k-1}(x_{k},x_{k-1})
		P_{k-1}(x_{k-1},y_{k-1})
		=
		P_k(x_k,y_k)\La^{k}_{k-1}(y_k,y_{k-1})
		\label{coherency_dynamics}
	\end{align}
	for every $k=1,\ldots,N$, $y_{k-1}\in\s_{k-1}$, 
	and $x_k,y_k\in\s_k$ subject 
	to the condition
	$\La^{k}_{k-1}(y_k,y_{k-1})\ne0$.
\end{definition}

This definition of course relies on our stochastic links $\La^{k}_{k-1}$, $k=2,\ldots,N$ (which are assumed to be fixed). Identity \eqref{coherency_dynamics} is in fact a refinement of a natural commutation relation between the projections $P_k$ and the stochastic links, see \eqref{commutation_relations_P_k} below.

\begin{comment}\label{comm:sequential_update} 
	Property 1 in Definition \ref{def:coherent_discrete} can be interpreted in the following way. Starting from $\mathbf{X}_N=(x_1,\ldots,x_N)$, the chain $\bp^{(N)}$ first chooses $y_1\in\s_1$ at random according to the distribution $U_1(x_1,y_1)$. Then, having determined $x_1$ and $y_1$, it samples random $y_2\in\s_2$ according to the distribution $U_2(x_2,y_2\,|\, x_1,y_1)$ (note that the choice of $y_2$ depends on both the previous state $x_1$ and the new state $y_1$ on the first level $\s_1$). And so on, up to the new $N$th state $y_N\in\s_N$. The new state of the chain $\bp^{(N)}$ at the next discrete-time moment is $\mathbf{Y}_N=(y_1,\ldots,y_N)$.

	We see that the function $U_k(x_k,y_k\,|\, x_{k-1},y_{k-1})$ may be interpreted as the conditional probability distribution of the new state $y_k\in\s_k$ given the old state $x_k$ and conditioned on the event that on $\s_{k-1}$ the chain $\bp^{(N)}$ took $x_{k-1}$ to $y_{k-1}$. Hence the name ``sequential update'' for $\bp^{(N)}$.
\end{comment}

We see that under a sequential update dynamics $\bp^{(N)}$, the evolution at every level $\s_k$ is independent of what happens at higher levels $\s_{k+1},\ldots,\s_N$. Thus, it is possible to define \emph{truncations} $\bp^{(k)}$ on $\bs^{(k)}\times\bs^{(k)}$ by
\begin{align*}
	\bp^{(k)}(\mathbf{X}_k;
	\mathbf{Y}_{k})
	:=
	\prod_{i=1}^{k}
	U_i(x_i,y_i\,|\,x_{i-1},y_{i-1}),
\end{align*}
where $k=1,\ldots,N$. Each $\bp^{(k)}$ is itself a sequential update Markov dynamics on the truncated space $\bs^{(k)}$.

\begin{proposition}\label{prop:sequential_properties}
	Let $\bp^{(N)}$ be a sequential update dynamics.
	
	{\rm{}\bf{}1.\/} Projections $P_k$ of $\bp^{(N)}$ commute with the stochastic links in the following sense (written in matrix product notation):
	\begin{align}\label{commutation_relations_P_k}
		\La^{k}_{k-1}P_{k-1}=P_k\La^{k}_{k-1},\qquad
		k=2,\ldots,N.
	\end{align}

	{\rm{}\bf{}2.\/} Each truncation $\bp^{(k)}$, $k=1,\ldots,N$, preserves the class of Gibbs measures on $\bs^{(k)}$. In more detail, let $m_k$ be any probability measure on $\s_k$, and $\bm^{(k)}$ be the corresponding Gibbs measure on $\bs^{(k)}$ (see Definition~\ref{def:Gibbs}). Let $m_k'=m_kP_k$ be the evolution of $m_k$ under one step of the $k$th projection of~$\bp^{(N)}$:
	\begin{align*}
		m_k'(y_k):=\sum_{x_k\in\s_k}m_k(x_k)P_k(x_k,y_k),\qquad
		y_k\in\s_k.
	\end{align*}
	Let $\bm'^{(k)}$ be the Gibbs measure on $\bs^{(k)}$ corresponding to $m_k'$. Then $\bm'^{(k)}$ coincides with the evolution $\bm^{(k)}\bp^{(k)}$ of $\bm^{(k)}$ under one step of the chain $\bp^{(k)}$.

	{\rm{}\bf{}3.\/} Let $\bp^{(N)}$ be
	of the form \eqref{sequential_update}.
	Then, modulo \eqref{commutation_relations_P_k}, 
	the above condition \eqref{coherency_dynamics}
	built into Definition \ref{def:coherent_discrete}
	is equivalent to the second claim.
\end{proposition}
\begin{proof}
	1. Sum \eqref{coherency_dynamics} over $y_k\in\s_k$ using \eqref{U_k_sum_to_one}.

	2. We argue by induction on $k$. It suffices to let $m_k$ be the delta-measure on $\s_k$ supported by some point $x_k\in\s_k$; the corresponding Gibbs measure on $\bs^{(k)}$ is given by the product
	$\La^{k}_{k-1}(x_{k},x_{k-1})\ldots \La^{2}_{1}(x_{2},x_{1})$. We then need to show that
	\begin{align*}
		&
		\sum_{\mathbf{X}_{k-1}\in\bs^{(k-1)}}
		\bigg(\prod_{i=1}^{k}\La^{i}_{i-1}(x_i,x_{i-1})
		U_i(x_i,y_i\,|\,x_{i-1},y_{i-1})\bigg)
		=P_{k}(x_k,y_k)
		\La^{k}_{k-1}(y_{k},y_{k-1})\ldots \La^{2}_{1}(y_{2},y_{1}).
	\end{align*}
	By using the same property for $k-1$, we can sum over $x_1,\ldots,x_{k-2}$ in the left-hand side. The fact that the result will be equal to the right-hand side is equivalent to condition
	\eqref{coherency_dynamics}.

	3. The third claim readily follows from the previous
	computation.
\end{proof}

It is often convenient to specify projections $P_k$, $k=1,\ldots,N$, first (we will sometimes call them \emph{univariate dynamics}), and then construct \emph{multivariate dynamics} $\bp^{(N)}$ with these given projections. Proposition \ref{prop:sequential_properties}.1 shows that these pre-specified projections $P_k$ must commute with the stochastic links $\La^{k}_{k-1}$ in the sense of \eqref{commutation_relations_P_k} in order for the problem of constructing $\bp^{(N)}$ to have a solution.

Next, recall that \eqref{coherency_dynamics} may be viewed as a refinement of commutation relations \eqref{commutation_relations_P_k} satisfied by these fixed projections $P_k$. We note that different refinements of \eqref{commutation_relations_P_k} correspond to different dynamics $\bp^{(N)}$. All such matrices $\bp^{(N)}$ act on Gibbs measures on $\bs^{(N)}$ in the same way. In the next subsection we consider a particular example of a sequential update dynamics which can be defined on our abstract level.



\subsection{Example: Diaconis--Fill type dynamics} 
\label{sub:example_diaconis_fill_dynamics}

Let Markov chains $P_1,\ldots,P_N$ on $\s_1,\ldots,\s_N$, respectively, be given. Suppose that commutation relations \eqref{commutation_relations_P_k} hold. The simplest example of a sequential update dynamics $\bp^{(N)}$ having projections $P_1,\ldots,P_N$ can be constructed by requiring that the conditional distributions $U_{k}(x_{k},y_{k}\,|\,x_{k-1},y_{k-1})$ do not depend on the previous states $x_{k-1}$. 
Then \eqref{coherency_dynamics} immediately implies that 
\begin{align}
	U_{k}(x_{k},y_{k}\,|\,y_{k-1})&=
	\frac{P_k(x_k,y_k)\La^{k}_{k-1}(y_k,y_{k-1})}
	{\sum_{x_{k-1}\in\s_{k-1}}
	\La^{k}_{k-1}(x_{k},x_{k-1})
	P_{k-1}(x_{k-1},y_{k-1})}
	\label{DiaconisFill_dynamics}
	\\&=
	\frac{P_k(x_k,y_k)\La^{k}_{k-1}(y_k,y_{k-1})}
	{\sum_{y'_{k}\in\s_{k}}
	P_{k}(x_{k},y'_{k})
	\La^{k}_{k-1}(y'_{k},y_{k-1})},
	\nonumber
\end{align}
if the denominator $\sum_{x_{k-1}}
\La^{k}_{k-1}(x_{k},x_{k-1})P_{k-1}(x_{k-1},y_{k-1})$ is nonzero, and~0 otherwise (the last equality is due to \eqref{commutation_relations_P_k}).

The definition and a few examples for $N=2$ were given 
by Diaconis and Fill \cite{DiaconisFill1990}. For general $N$ such chains in various settings were studied in, e.g., 
\cite{BorFerr2008DF}
(in particular, see \S 2), 
\cite{Borodin2010Schur}, and
\cite{BorodinCorwin2011Macdonald}.


\subsection{Sequential update dynamics in continuous time} 
\label{sub:sequential_update_dynamics_in_continuous_time}

Let us present continuous-time analogues of the previous constructions. 

We will consider Markov semigroups $(\bp^{(N)}(\tau))_{\tau\ge0}$ made of stochastic matrices $\bp^{(N)}(\tau)$ satisfying \eqref{Markov_negative_boldP} (with conventions of Remark \ref{rmk:zero_outside_state_space_P}) for every fixed $\tau$.\footnote{To avoid confusion with the Macdonald parameter $t$, throughout the paper we will denote time variable by $\tau$.} The term ``semigroup'' means that these matrices satisfy the Chap\-man-Kolmogorov equations $\bp^{(N)}(\tau_1+\tau_2)=\bp^{(N)}(\tau_1)\bp^{(N)}(\tau_2)$. Each matrix element $\bp^{(N)}(\tau;\mathbf{X}_N,\mathbf{Y}_N)$ represents the probability that the Markov process is at state $\mathbf{Y}_N$ after time $\tau$ if its current state is $\mathbf{X}_N$.

Assume that there exists a matrix of 
\emph{jump rates}\footnote{In
other words, a \emph{Markov generator} of 
the semigroup $(\bp^{(N)}(\tau))_{\tau\ge0}$.}
$\bq^{(N)}$, so that for any 
$\mathbf{X}_N,\mathbf{Y}_N\in\bs^{(N)}$ one 
has\footnote{Here and below $1_{\{\cdot\cdot\cdot\}}$ 
denotes the indicator function of a set.}
\begin{align}\label{bp_bq_relation}
	\bp^{(N)}(\tau;\mathbf{X}_N,\mathbf{Y}_N)
	=1_{\mathbf{X}_N=\mathbf{Y}_N}+
	\bq^{(N)}(\mathbf{X}_N,\mathbf{Y}_N)\cdot \tau+o(\tau), \quad \tau\to0.
\end{align}
We will, moreover, assume that the diagonal elements of $\bq^{(N)}$ satisfy
\begin{align}\label{bq_diagonal_elements}
	\bq^{(N)}(\mathbf{X}_N,
	\mathbf{X}_N)=
	-\sum_{\mathbf{Y}_N\ne \mathbf{X}_N}
	\bq^{(N)}(\mathbf{X}_N,
	\mathbf{Y}_N), \qquad \mathbf{X}_N\in\bs^{(N)}.
\end{align}
Since $\bp^{(N)}(\tau)$ is a Markov semigroup, the off-diagonal elements of $\bq^{(N)}$ are nonnegative.

\begin{condition}\label{cond:finiteness_bq}
	We suppose that for every $\mathbf{X}_N$, only finitely many of the numbers $\bq^{(N)}(\mathbf{X}_N,\mathbf{Y}_N)$ are nonzero, and that they are all uniformly bounded over $\bs^{(N)}$. This implies that the operator $\bq^{(N)}$ is bounded in the Banach space of bounded functions on $\bs^{(N)}$ equipped with the supremum norm. Thus, the matrix of jump rates $\bq^{(N)}$ uniquely defines a Feller Markov process $(\bp^{(N)}(\tau))_{\tau\ge0}$ on $\bs^{(N)}$ that has $\bq^{(N)}$ as its generator, cf. \cite{Liggett2010}: $\bp^{(N)}(\tau)=\exp(\tau \cdot \bq^{(N)})$
    (convergence in operator norm).
\end{condition}

Now we are able to give the definition of a sequential update dynamics in the continuous-time setting (cf. Definition \ref{def:coherent_discrete}):

\begin{definition}
	\label{def:coherent_continuous}
	A Markov semigroup $(\bp^{(N)}(\tau))_{\tau\ge0}$ as above is called a \emph{conti\-nuous-time sequential update dynamics} if it satisfies:

	\textbf{1.} 
	The off-diagonal matrix elements of $\bq^{(N)}$ have the form
	(for some functions $V_k$ and $W_k$, $k=1,\ldots,N$):
	\begin{align}\label{bq_N_product_form}
		\bq^{(N)}(\mathbf{X}_N,\mathbf{Y}_N)=
		W_k(x_k,y_k\,|\,x_{k-1})
		\prod_{i=k+1}^{N}V_{i}(x_i,y_i\,|\,x_{i-1},y_{i-1})
	\end{align}
	if $x_j=y_j$ for all $j\le k-1$, and $x_k\ne y_k$ (the diagonal elements are given by \eqref{bq_diagonal_elements}). The functions $V_k$ and $W_k$ are assumed to satisfy (for all $k$) the following:
	\begin{eqnarray}\label{V_conditions_P1'_1}
		\sum_{y_{k}\in\s_k}V_k(x_k,y_k\,|\,x_{k-1},y_{k-1})&=&1,
		\\
		\label{V_conditions_P1'_2}
		V_k(x_k,y_k\,|\,x_{k-1},x_{k-1})
		&=&
		1_{x_k=y_k},
		\\
		\label{W_conditions_P1'}
		\sum_{y_k\ne x_k}W_k(x_k,y_k\,|\, x_{k-1})
		&=&-W_k(x_k,x_k\,|\, x_{k-1}).
	\end{eqnarray}
	We assume that $V_k(x_k,y_k\,|\,x_{k-1},y_{k-1})$ and $W_k(x_k,y_k\,|\, x_{k-1})$ are uniformly boun\-ded, and that in both sums in \eqref{V_conditions_P1'_1} and \eqref{W_conditions_P1'} only finitely many summands are nonzero. This ensures that $\bq^{(N)}$ satisfies the finiteness Condition~\ref{cond:finiteness_bq}. Moreover, all $V_k$'s and off-diagonal $W_k$'s (i.e., $W_k(x_k,y_k\,|\, x_{k-1})$ with $x_k\ne y_k$) must be nonnegative.

	\textbf{2.} 
	Define for every $m=1,\ldots,N$ the \emph{truncations} (we list only the off-diagonal elements)
	\begin{align}\label{Truncation_Q_k}
		\bq^{(m)}(\mathbf{X}_{m};
		\mathbf{Y}_{m}):=
		W_k(x_k,y_k\,|\,x_{k-1})
		\prod_{i=k+1}^{m}V_{i}(x_i,y_i\,|\,x_{i-1},y_{i-1})
	\end{align}
	if $x_j=y_j$ for all $j\le k-1$, and $x_k\ne y_k$, and the \emph{projections} to $\s_m$ (when acting on Gibbs measures):
	\begin{align}\label{Projection_Q_k}
		Q_m(x_m,y_m):=\sum_{\mathbf{X}_{m-1},
		\mathbf{Y}_{m-1}\in\bs^{(m-1)}}
		\bq^{(m)}(\mathbf{X}_{m};
		\mathbf{Y}_{m})
		\prod_{i=1}^{m}
		\La^{i}_{i-1}(x_i,x_{i-1}).
	\end{align}
	We require that for any $k=1,\ldots,N$, $x_{k},y_{k}\in\s_k$, 
	and $y_{k-1}\in\s_{k-1}$ subject to the condition
	$\La^{k}_{k-1}(y_k,y_{k-1})\ne0$, 
	the following identity must hold:
	\begin{align}
		&
		\sum_{x_{k-1}\in\s_{k-1}}
		V_k(x_k,y_k\,|\, x_{k-1},y_{k-1})
		\La^{k}_{k-1}(x_k,x_{k-1})
		Q_{k-1}(x_{k-1},y_{k-1})
		\label{P2'}
		\\&
		\hspace{85pt}
		+
		W_k(x_k,y_k\,|\,y_{k-1})\La^{k}_{k-1}(x_k,y_{k-1})
		=
		Q_k(x_k,y_k)\La^{k}_{k-1}(y_k,y_{k-1}).
		\nonumber
	\end{align}
\end{definition}
Similarly to \eqref{coherency_dynamics}, identity \eqref{P2'} is also a refinement of certain natural commutation relations. Namely, in our continuous-time case, summing \eqref{P2'} over $y_k$, one gets commutation relations between projections $Q_k$ \eqref{Projection_Q_k} and our stochastic links, see \eqref{commutation_relations_Q_k} below.

\begin{comment}[cf. Comment \ref{comm:sequential_update}]\label{comm:sequential_update_continuous}
	The above definition can be interpreted as follows. Assume that the sequential update dynamics $\bp^{(N)}$ is at state $\mathbf{X}_N=(x_1,\ldots,x_N)\in\bs^{(N)}$. At every level $\s_k$, we have an independent exponential clock with (nonnegative) rate $-W_k(x_k,x_k\,|\, x_{k-1})$. When the $k$th clock rings (during an
	infinitesimally small time interval, only one of the $N$ clocks can ring), the $k$th state $x_k$ changes to a random $y_k$ according to the 
	probability distribution $\big(-{W_k(x_k,y_k\,|\, x_{k-1})}/{W_k(x_k,x_k\,|\, x_{k-1})}\big)$ (cf.~\eqref{W_conditions_P1'}). In short, the process \emph{jumps from $x_k$ to $y_k$ at rate} $W_k(x_k,y_k\,|\, x_{k-1})$.

	When an independent jump at some level $k$ happens, it may trigger instantaneous further moves $x_j\to y_j$ at all levels $j=k+1,\ldots,N$ (while at levels $1,\ldots,k-1$ the state of the process remains unchanged). The probabilities of these further moves are determined by the product of $V_i$, $i\ge k+1$, in \eqref{bq_N_product_form}. 

	These triggered moves may be described as follows. If at any level $\s_j$, $j=1,\ldots,N$ there is a \emph{move} $x_j\to y_j$, $y_j\ne x_j$, then at the same moment there is a triggered 
	move $x_{j+1}\to y_{j+1}$ at the level $\s_{j+1}$ with probability $V_{j+1}(x_{j+1},y_{j+1}\,|\, x_{j},y_j)$. This includes the possibility of the trivial move $x_{j+1}\to x_{j+1}$, in which case (due to \eqref{V_conditions_P1'_2}) the triggered moves do not propagate to higher levels $j+2,j+3,\ldots,N$.
\end{comment}

We will call $W_k$'s the \emph{rates of independent jumps}, 
and $V_k$'s will be referred to as the 
\emph{probabilities of triggered moves}.
We will refer to independent jumps simply as to \emph{jumps}.
By a \emph{move} we will mean either a jump, or a triggered move.

Let us now discuss properties of projections
$Q_k$ and truncations $\bq^{(k)}$.

\begin{condition}\label{cond:finiteness}
	We suppose that for every $x_k$, only finitely many of the numbers $Q_k(x_k,y_k)$ are nonzero, and that they are all uniformly bounded over $\s_k$ (see also Remark~\ref{rmk:finiteness_Q_k} below).
\end{condition}
Condition \ref{cond:finiteness} implies that the matrix of jump rates $Q_k$ uniquely defines a Feller Markov process $(P_k(\tau))_{\tau\ge0}$ on $\s_k$ that has $Q_k$ as its generator \cite{Liggett2010}: $P_k(\tau)=\exp(\tau \cdot Q_k)$.
\begin{remark}\label{rmk:finiteness_Q_k}
	If for every $x_k\in\s_k$ the number of tuples $(x_1,\ldots,x_{k-1})$ such that $(x_1,\ldots,x_k)\in\bs^{(k)}$ is finite, then Condition \ref{cond:finiteness} follows from Condition \ref{cond:finiteness_bq} for $\bq^{(N)}$.
\end{remark}

\begin{proposition}
	\label{prop:sequential_properties_continuous}
	Let $(\bp^{(N)}(\tau))_{\tau\ge0}$ be a sequential update continuous-time dynamics.
	
	{\rm{}\bf{}1.\/} Projections $Q_k$ commute with the stochastic links:
	\begin{align}\label{commutation_relations_Q_k}
		\La^{k}_{k-1}Q_{k-1}=Q_k\La^{k}_{k-1},\qquad
		k=2,\ldots,N.
	\end{align}
	For every $\tau\ge0$, the transition matrices $P_k(\tau)$ commute with the links as well, i.e., they satisfy \eqref{commutation_relations_P_k}.
	
	{\rm{}\bf{}2.\/} Consider the $k$th truncation $\bp^{(k)}(\tau)=\exp(\tau\cdot \bq^{(k)})$, $k=1,\ldots,N$.\footnote{Here $\bq^{(k)}$ is defined in \eqref{Truncation_Q_k}. Its exponent $\bp^{(k)}(\tau)$ exists due to Condition \ref{cond:finiteness_bq}.} For any $\tau\ge0$, $\bp^{(k)}(\tau)$ preserves the class of (nonnegative) Gibbs measures on $\bs^{(k)}$ in the same sense as in Proposition \ref{prop:sequential_properties} (part 2).
\end{proposition}

\begin{proof}
	1. One can get \eqref{commutation_relations_Q_k} by summing \eqref{P2'} over $y_k\in\s_k$. Hence, $\La^{k}_{k-1}(Q_{k-1})^{m}=(Q_k)^{m}\La^{k}_{k-1}$ for every $m=0,1,2,\ldots$. Due to Condition~\ref{cond:finiteness}, we have $P_k(\tau)=\exp(\tau\cdot Q_k)=1+\tau Q_k+\tau^{2}Q_k^{2}/2+\ldots$, and the series converges in the operator norm corresponding to the Banach space of bounded functions on $\s_k$. This implies \eqref{commutation_relations_P_k} for $P_k(\tau)$'s.

	2. We argue similarly to the proof of Proposition \ref{prop:sequential_properties}. We need to establish
	\begin{align}\label{P_big_sum_prop_Q_proof}
		&
		\sum_{\mathbf{X}_{k-1}}
		\La^{k}_{k-1}(x_{k},x_{k-1})\ldots \La^{2}_{1}(x_{2},x_{1})
		\bp^{(k)}(\tau;\mathbf{X}_k;\mathbf{Y}_{k})
		=P_{k}(\tau;x_k,y_k)
		\La^{k}_{k-1}(y_{k},y_{k-1})\ldots \La^{2}_{1}(y_{2},y_{1}).
	\end{align}

	Let us first show (by induction on $k$) that an infinitesimal version of the above identity holds:
	\begin{align}
		&\label{Q_big_sum_prop_Q_proof}
		\sum_{\mathbf{X}_{k-1}}
		\La^{k}_{k-1}(x_{k},x_{k-1})\ldots \La^{2}_{1}(x_{2},x_{1})
		\bq^{(k)}(\mathbf{X}_k;\mathbf{Y}_{k})
		=Q_{k}(x_k,y_k)
		\La^{k}_{k-1}(y_{k},y_{k-1})\ldots \La^{2}_{1}(y_{2},y_{1}).
	\end{align} 
	Using the definition of $\bq^{(k)}$ \eqref{Truncation_Q_k}, we can express it through $\bq^{(k-1)}$ as
	\begin{align}&
		\label{Q_k_through_Q_k-1}
		\bq^{(k)}(\mathbf{X}_k;\mathbf{Y}_{k})
		=
		W_k(x_k,y_k\,|\,y_{k-1})
		1_{\mathbf{X}_{k-1}=\mathbf{Y}_{k-1}}
		+
		\bq^{(k-1)}(\mathbf{X}_{k-1};
		\mathbf{Y}_{k-1})
		V_k(x_k,y_k\,|\,x_{k-1},y_{k-1}).
	\end{align}
	Plugging this into \eqref{Q_big_sum_prop_Q_proof} and summing over $x_{1},\ldots,x_{k-2}$ (using the statement for $k-1$), one can see that the resulting identity is equivalent to \eqref{P2'}.

	Now let us show that \eqref{Q_big_sum_prop_Q_proof} holds also if we replace $\bq^{(k)}$ and $Q_k$ by their powers, $(\bq^{(k)})^{m}$ and $(Q_k)^{m}$, respectively, where $m=0,1,2,\ldots$. 

	Consider the case $m=2$. Let $\mathbf{Z}_{k}\in\bs^{(k)}$, and let us multiply both sides of \eqref{Q_big_sum_prop_Q_proof} by $Q_k(z_k,x_k)$ and sum over $x_k\in\s_k$. In the right-hand side we would have the desired $(Q_k)^{2}$, and in the left-hand side, using \eqref{Q_big_sum_prop_Q_proof}, we obtain:
	\begin{align*}
		&
		\sum_{\mathbf{X}_k\in\bs^{(k)}}
		Q_k(z_k,x_k)
		\prod_{i=1}^{k}\La^{i}_{i-1}(x_i,x_{i-1})
		\bq^{(k)}(\mathbf{X}_k;\mathbf{Y}_{k})
		\\&\hspace{80pt}=
		\sum_{\mathbf{Z}_{k-1}\in\bs^{(k-1)}}\bigg(
		\prod_{i=1}^{k}\La^{i}_{i-1}(z_i,z_{i-1})
		\sum_{\mathbf{X}_k\in\bs^{(k)}}
		\bq^{(k)}(\mathbf{Z}_{k};\mathbf{X}_k)
		\bq^{(k)}(\mathbf{X}_k;
		\mathbf{Y}_{k})\bigg),
	\end{align*}
	which is the left-hand side of \eqref{Q_big_sum_prop_Q_proof} with $(\bq^{(k)})^{2}$. Similarly one shows that \eqref{Q_big_sum_prop_Q_proof} holds for every power $m=0,1,2,\ldots$. Now, using the fact that the exponential series for $\bp^{(k)}(\tau)$ and $P_k(\tau)$ converge, we may organize these identities for all powers $m$ into appropriate exponential series, and get \eqref{P_big_sum_prop_Q_proof}. This concludes the proof.
\end{proof}
\begin{remark}
	Identity \eqref{P_big_sum_prop_Q_proof} above can be alternatively proved in the following way. Using the backward Kolmogorov equation $dP_k(\tau)/dt=Q_k P_k(\tau)$, and similarly for $\bp^{(k)}$, with the help of \eqref{Q_big_sum_prop_Q_proof} it is possible to show that both sides of \eqref{P_big_sum_prop_Q_proof} solve the same Cauchy problem of the form 
	\begin{align*}
		\frac{d}{d\tau}F(\tau)=AF(\tau), \quad \tau>0,
		\qquad F(0){}={}\mbox{fixed vector}.
	\end{align*}
	Then it can be shown that such a Cauchy problem has a unique solution (e.g., see \cite[IX.1.3]{Kato1980}). For such an argument in a similar (but more involved) situation see \cite[\S6.3 and proof of Prop.~8.3]{BorodinOlshanski2010GTs}.
\end{remark}

Let us now emphasize several aspects of the continuous-time formalism just presented. We will employ this formalism by specifying univariate Markov jump processes $Q_k$ on $\s_k$ which commute with the stochastic links in the sense of \eqref{commutation_relations_Q_k}; then we will consider the problem of constructing multivariate dynamics $\bq^{(N)}$ on the state space $\bs^{(N)}$. Various multivariate dynamics correspond to various refinements of commutation relations \eqref{commutation_relations_Q_k}.

\begin{remark}\label{rmk:condition_xk_ne_yk}
	Assume that projections $Q_k$ commuting with the stochastic links in the sense of \eqref{commutation_relations_Q_k} are given. If condition \eqref{P2'} holds for every $y_{k-1}$ and $x_k\ne y_k$, then it also holds for $x_k=y_k$. Indeed, this follows from the fact that summing \eqref{P2'} over $y_k$ gives \eqref{commutation_relations_Q_k}. On the other hand, due to \eqref{V_conditions_P1'_2}, in identity \eqref{P2'} for $x_k\ne y_k$ the summation over $x_{k-1}$ may be carried over $x_{k-1}\ne y_{k-1}$, because the summand corresponding to $x_{k-1}=y_{k-1}$ vanishes. Thus, under \eqref{commutation_relations_Q_k} it is possible to rewrite \eqref{P2'} involving only the off-diagonal elements of $Q_{k}, Q_{k-1}$, as well as of $W_k(\cdot,\cdot\,|\, y_{k-1})$.
\end{remark}

\begin{remark}\label{rmk:Diaconis_Fill_later}
	Starting from given projections $Q_k$,
	it is also possible
	to introduce distinguished
	continuous-time dynamics which is similar
	to the discrete-time Diaconis--Fill type process 
	of \S \ref{sub:example_diaconis_fill_dynamics}.
	For a general construction, which is more involved than in 
	discrete time, we 
	refer to \cite[\S 8]{BorodinOlshanski2010GTs}. We will 
	consider 
	such dynamics in a 
	more concrete situation (when $\bs^{(N)}$ 
	consists of interlacing integer arrays) 
	later in \S 
	\ref{sub:example_push_block_process_on_interlacing_arrays}.
\end{remark}

In fact, all our considerations of \S \ref{sub:sequential_update_dynamics_in_discrete_time} and \S \ref{sub:sequential_update_dynamics_in_continuous_time} work in a slightly more general algebraic setting:

\begin{remark}\label{rmk:algebraic_sense}
	Observe that all definitions involving discrete-time objects $\bp^{(k)}$, $U_k$,  $P_k$ (\S \ref{sub:sequential_update_dynamics_in_discrete_time}), as well as their continuous-time counterparts $\bp^{(k)}(\tau)$, $\bq^{(k)}$, $V_k$, $W_k$, $P_k(\tau)$, $Q_k$ (\S \ref{sub:sequential_update_dynamics_in_continuous_time}) have purely linear-algebraic nature. Thus, for example, we may relax the condition that the matrix elements of $\bp^{(N)}$ \eqref{Markov_negative_boldP} are nonnegative, and formally speak about discrete-time multivariate `dynamics', meaning simply the corresponding matrix $\bp^{(N)}$. This is true in the continuous-time setting as well, as all properties and results discussed in \S \ref{sub:sequential_update_dynamics_in_continuous_time} do not require nonnegativity of the corresponding matrix elements.
	However, one does need to assume that the 
	denominators in the above definitions do not vanish.

	Though all univariate dynamics considered in the present paper are honest probabilistic objects, we will deal with more general multivariate `dynamics' which do not satisfy the 
	positivity assumption.
\end{remark}


\subsection{From discrete to continuous time} 
\label{sub:from_discrete_to_continuous_time}

Our Definition \ref{def:coherent_continuous} of a continuous-time sequential update dynamics may be read off the discrete-time case (Definition \ref{def:coherent_discrete}). This is similar to the classical fact that simple random walks in discrete time converge (under a suitable scaling) to Markov jump processes.

Let $\bp^{(N)}_{\varepsilon}$ be a sequential update discrete-time Markov dynamics as in Definition \ref{def:coherent_discrete}, and assume that it depends on a small parameter $\varepsilon$ in the following way (here $\mathbf{X}_N,\mathbf{Y}_N\in\bs^{(N)}$):
\begin{align}\label{P_N_eps}
	\bp^{(N)}_{\varepsilon}(\mathbf{X}_N, \mathbf{Y}_N)=
	1_{\mathbf{X}_N=\mathbf{Y}_N}+\varepsilon
	\cdot\bq^{(N)}(\mathbf{X}_N,\mathbf{Y}_N)+o(\varepsilon), \quad \varepsilon\to0.
\end{align}
Here $\bq^{(N)}$ is a jump rate matrix; we assume that is satisfies \eqref{bq_diagonal_elements} and Condition \ref{cond:finiteness_bq}. Then it is possible to define Markov transition matrices corresponding to the continuous-time setting:
\begin{align*}
	\bp^{(N)}(\tau):=\lim_{\varepsilon\to0}
	\left(\bp^{(N)}_{\varepsilon}\right)^{[\tau/\varepsilon]}.
\end{align*}
Clearly, $\bp^{(N)}(\tau)=\exp(\tau\cdot \bq^{(N)})=\mathbf{1}+\tau\cdot \bq^{(N)}+o(\tau)$, where $\mathbf{1}$ means the identity operator.

Moreover, assume that each $U_i=U_i^{(\varepsilon)}$ in \eqref{sequential_update} is differentiable  in $\varepsilon$:
\begin{align}\label{U_i_eps}
	U_i^{(\varepsilon)}
	(x_i,y_i\,|\,x_{i-1},y_{i-1})=
	V_i
	(x_i,y_i\,|\,x_{i-1},y_{i-1})+
	\varepsilon\cdot
	W_i
	(x_i,y_i\,|\,x_{i-1},y_{i-1})+o(\varepsilon),
\end{align}
as $\varepsilon\to0$.

\begin{proposition}
	The object $(\bp^{(N)}(\tau))_{\tau\ge0}$ defines a sequential update conti\-nuous-time Markov dynamics (Definition \ref{def:coherent_continuous}) corresponding to the functions $V_k$ and $W_k$ as in \eqref{U_i_eps}.
\end{proposition}
\begin{proof}
	Let us check the requirements of Definition \ref{def:coherent_continuous}.

	Denote $W_{k}(x_k,y_k\,|\,x_{k-1}):=W_{k}(x_k,y_k\,|\,x_{k-1},x_{k-1})$. From \eqref{U_i_eps} we see that the off-diagonal elements of $\bq^{(N)}$ look as \eqref{bq_N_product_form}, and the diagonal elements are given by \eqref{bq_diagonal_elements}. Indeed, identities \eqref{V_conditions_P1'_1} and \eqref{V_conditions_P1'_2} are obtained from \eqref{U_k_sum_to_one} and \eqref{coherency_dynamics}, respectively, by setting $\varepsilon=0$. Identity \eqref{W_conditions_P1'} may be obtained by induction on $k$ from \eqref{bq_diagonal_elements} and \eqref{bq_N_product_form} by using \eqref{Q_k_through_Q_k-1}. Equivalently, \eqref{W_conditions_P1'} follows by considering the coefficient by $\varepsilon$ in \eqref{U_k_sum_to_one}.

	Finally, \eqref{P2'} is obtained by taking the coefficient by $\varepsilon$ in \eqref{coherency_dynamics}. This concludes the proof.

	Note that \eqref{V_conditions_P1'_2} implies that the functions $W_k(x_k,y_k\,|\,x_{k-1},y_{k-1})$, 
	where $x_{k-1}\ne y_{k-1}$ do not affect the continuous-time dynamics $\bp^{(N)}(\tau)$ constructed from $\bp^{(N)}_{\varepsilon}$.
\end{proof}



\section{Combinatorics of interlacing arrays and related objects} 
\label{sec:combinatorics_of_interlacing_arrays_and_related_objects}

Let us start specializing the general constructions 
of 
\S \ref{sec:markov_dynamics_preserving_gibbs_measures_general_formalism} 
to our concrete situation. 
In this section we briefly describe the 
state space of interlacing arrays, 
and give the necessary 
combinatorial background.

\subsection{Signatures, Young diagrams, and interlacing arrays} 
\label{sub:signatures_young_diagrams_and_interlacing_arrays}

By a \emph{signature} of length $N$ we will mean a nonincreasing $N$-tuple of integers $\la=(\la_1\ge \ldots\ge\la_N)\in\Z^{N}$. Let $\GT_N$ denote the set of all signatures of length $N$ (`$\GT$' stands for `Gelfand--Tsetlin', see \eqref{GT_Scheme}).\footnote{The set $\GT_N$ parametrizes irreducible representations of the unitary group $U(N)$ \cite{Weyl1946}, and in the literature signatures are also referred to as \emph{highest weights}.} By agreement, $\GT_0$ consists of a single empty signature $\varnothing$. We denote $|\la|:=\la_1+\la_2+\ldots+\la_N$.

By $\GT_N^{+}\subset \GT_N$ we mean the subset of \emph{nonnegative signatures}, i.e., signatures for which $\la_N\ge0$ (we assume that $\GT_0^+=\GT_0=\{\varnothing\}$). Nonnegative signatures are also called \emph{partitions} and are identified with Young diagrams \cite[I.1]{Macdonald1995}. 

Every nonnegative signature $\la=(\la_1,\ldots,\la_N)\in\GT_N^{+}$ may be also viewed as an element $(\la_1,\ldots,\la_N,0)\in\GT_{N+1}^{+}$. Thus, one may speak about nonnegative signatures without referring explicitly to their length, and appending them by zeroes if necessary. Let $\GT^+:=\bigcup_{N\ge0}\GT_N^{+}$ be the set of all nonnegative signatures (= partitions). By $\ell(\la)$ denote the number of strictly positive parts in $\la\in\GT^+$.

Let $\mu\in\GT_{N-1}$ and $\la\in\GT_N$. By $\mu\prec\la$ we mean that $\mu$ and $\la$ \emph{interlace}:
\begin{align}\label{interlacing_two_signatures}
  \la_N\le\mu_{N-1}\le\la_{N-1}\le \ldots\le \la_2\le\mu_1\le\la_1.
\end{align}

The main combinatorial objects in the present paper are sequences of interlacing signatures:
\begin{align}\label{GT_Scheme}
	\boldsymbol\la=(\varnothing\prec\la^{(1)}\prec\la^{(2)}\prec 
	\ldots\prec\la^{(N-1)}\prec\la^{(N)}),
	\qquad\la^{(k)}\in\GT_k.
\end{align}
Such sequences are called \emph{Gelfand--Tsetlin schemes}, they are 
conveniently visualized as arrays of interlacing integers, 
see Fig.~\ref{fig:GT_scheme} in the Introduction. 
We call $N$ the \emph{depth} of a Gelfand--Tsetlin scheme. We will mostly consider Gelfand--Tsetlin schemes 
of fixed finite depth $N$. 
Let $\GT(N)$ denote the set of 
Gelfand--Tsetlin schemes of depth $N$, and $\GT^{+}(N)$ be the set of 
Gelfand--Tsetlin schemes made of nonnegative signatures: 
in \eqref{GT_Scheme}, $\la^{(k)}\in\GT_k^{+}$, $k=1,\ldots,N$.


\subsection{Semistandard Young tableaux} 
\label{sub:semistandard_young_tableaux}

There is another classical point of 
view on Gelfand--Tsetlin schemes 
$\boldsymbol\la\in\GT^{+}(N)$ made of 
nonnegative signatures. Namely, such Gelfand--Tsetlin schemes
correspond to semistandard Young tableaux.

\begin{definition}\label{def:SSYT}
	Let $\la$ be a Young diagram. A \emph{semistandard} (\emph{Young}) 
	\emph{tableau} of shape $\la$ over the alphabet $\{1,\ldots,N\}$ is a 
	filling of boxes of the diagram $\la$ by letters from this alphabet 
	(each letter may be used several times)
	such 
	that letters in a tableau \emph{weakly} increase along rows and \emph{strictly} increase 
	down columns.
\end{definition}
Clearly, to consider semistandard 
tableaux of some shape $\la$ over the 
alphabet $\{1,\ldots,N\}$, the number of rows in $\la$ 
(of positive length) must be~$\le N$; that is, $\la$ must belong
to $\GT^{+}_N$. Below 
is an example of a 
semistandard Young tableau of shape $\la=(4,3,1)$ over the 
alphabet $\{1,2,3,4,5\}$:
\begin{equation}\label{SSYT example}
	\left.
	\begin{array}{|c|c|c|c|c|c|c|}
	\hline
	1&1&1&2&5\\
	\hline
	2&2&3&3\\
	\cline{1-4}
	3&4&4\\
	\cline{1-3}
	4&5&5\\
	\cline{1-3}
	\end{array}
	\right.
\end{equation}

\begin{proposition}\label{prop:SSYT}
	Semistandard Young tableaux of shape 
	$\la\in\GT^+_N$ over the alphabet 
	$\{1,\ldots,N\}$ are in one-to-one correspondence 
	with Gelfand--Tsetlin sche\-mes 
	$\boldsymbol\la\in\GT^+(N)$ 
	of depth $N$ with top row $\la^{(N)}=\la$. 
\end{proposition}
\begin{proof}
	Indeed, each row $\la^{(k)}$, $k=1,\ldots,N$, 
	in a Gelfand--Tsetlin scheme is identified 
	with the shape formed by boxes 
	of a semistandard 
	Young tableau consisting letters $1,\ldots,k$. 
	For example, semistandard tableau \eqref{SSYT example} 
	corresponds to the Gelfand--Tsetlin scheme
	\begin{align}\label{SSYT_GT_scheme_example}
		\varnothing\prec(3)\prec(4,2)\prec(4,4,1)\prec(4,4,3,1)
		\prec(5,4,3,3,0),
	\end{align}
	which can be drawn as the following 
	interlacing array (see also Fig.~\ref{fig:GT_scheme}):
	\begin{align}\label{SSYT_GT_scheme_array_example}
		\begin{matrix}
			0&&3&&3&&4&&5\\
			&1&&3&&4&&4\\
			&&1&&4&&4\\
			&&&2&&4\\
			&&&&3
		\end{matrix}
	\end{align}
	This concludes the proof.
\end{proof}

Let us denote by $\Dim_N\la$
the number of semistandard Young
tableaux of shape $\la$ over the alphabet
$\{1,\ldots,N\}$. In fact, $\Dim_N\la$
equals the dimension of an irreducible
representation of the unitary group $U(N)$
corresponding to~$\la$ 
(e.g., see \cite{Weyl1946}).

\begin{definition}
[cf. Definition \ref{def:SSYT}]
\label{def:SYT}
	Let $\la$ be a Young diagram. 
	A \emph{standard} (\emph{Young}) 
	\emph{tableau} of shape $\la$ 
	is a filling
	of boxes of the diagram $\la$ by letters 
	$1,2,\ldots,|\la|$
	(each letter is used only once)
	such 
	that letters in a tableau \emph{strictly} increase 
	both along rows and down columns.
\end{definition}

Let $\dim\la$ denote the number of standard
tableaux of shape $\la$. In fact, $\dim\la$
equals the dimension of an irreducible
representation of the symmetric group
$\mathfrak{S}(|\la|)$ corresponding to $\la$
(e.g., see \cite{sagan2001symmetric}).

\smallskip

We employ the Young tableaux perspective 
in \S \ref{sec:schur_degeneration_and_robinson_schensted_correspondences} 
where we will discuss Robinson--Schensted correspondences.


\subsection{Visualizing interlacing arrays by particle configurations} 
\label{sub:visualizing_interlacing_arrays}

The main topic of the present paper is 
stochastic dynamics on Gelfand--Tsetlin schemes. 
The latter can be represented by 
interlacing arrays of integers as on Fig.~\ref{fig:GT_scheme} in the Introduction. 
An elementary move in such a dynamics consists of 
increasing several coordinates $\la^{(k)}_{j}$ in 
the Gelfand--Tsetlin scheme by one. 

It thus would be very convenient for us to employ 
the intuition of \emph{particle configurations}.
For a Gelfand--Tsetlin scheme 
$\boldsymbol\la\in\GT(N)$, place $N(N+1)/2$ particles at 
points\footnote{Since the coordinates of each signature 
$\la^{(k)}$ are weakly decreasing, 
at some positions there could be more than one particle.} 
\begin{align*}
	\Big\{(\la^{(k)}_{j}, k)\colon k=1,\ldots,N,\ 
	j=1,\ldots,k\Big\}\subset\Z^{2},
\end{align*}
see Fig.~\ref{fig:la_interlacing} in the Introduction.
Then an elementary move in dynamics means 
that several particles in a configuration jump to the right by one.






\section{Ascending Macdonald processes and univariate dynamics on signatures} 
\label{sec:ascending_macdonald_processes}

Here we briefly discuss a special class of measures
on interlacing arrays,
namely, (ascending) Macdonald processes, 
introduced in \cite{BorodinCorwin2011Macdonald}. 
Detailed definitions and properties 
related to Macdonald processes are given 
in Appendix \ref{sec:macdonald_processes}. 
We also describe univariate dynamics preserving 
Macdonald measures (which are marginal 
distributions of Macdonald processes, see \S \ref{sub:ascending_macdonald_processes}).

\subsection{Stochastic links with Macdonald parameters} 
\label{sub:stochastic_links_with_macdonald_parameters}

Let $q,t\in[0,1)$ be the Macdonald parameters
(see \S \ref{sub:macdonald_symmetric_functions}), 
and $a_1,\ldots,a_N$ be some positive variables 
(we assume that they are fixed throughout the paper). 
Let $\GT_k$ play the role of the set 
$\s_k$ in \S \ref{sub:gibbs_measures}, 
and let us define stochastic links $\La^{k}_{k-1}\colon 
\GT_k\times\GT_{k-1}\to[0,1]$ 
as follows (cf. \eqref{Links_Macdonald_appendix})
\begin{align}\label{Links_Macdonald}
	\La^{k}_{k-1}(\la^{(k)},\la^{(k-1)})
	:=\frac
	{P_{\la^{(k-1)}}(a_1,\ldots,a_{k-1})}
	{P_{\la^{(k)}}(a_1,\ldots,a_{k})}
	P_{\la^{(k)}/\la^{(k-1)}}(a_k)
\end{align}
if $\la^{(k-1)}\prec\la^{(k)}$, 
and $0$ otherwise 
($\la^{(j)}\in\GT_j$). 
Here $P_\la=P_\la(\cdot;q,t)$ and 
$P_{\la/\mu}=P_{\la/\mu}(\cdot;q,t)$ are the ordinary 
and skew Macdonald polynomials, 
respectively (see \S \ref{sub:macdonald_symmetric_functions} 
and \S \ref{sub:skew_functions}).

\begin{remark}\label{rmk:translation_invariance_links}
	To define $P_{\la^{(k)}}$, $P_{\la^{(k-1)}}$, 
	and $P_{\la^{(k)}/\la^{(k-1)}}$ 
	for not necessarily nonnegative signatures, 
	we use Remarks \ref{rmk:Macdonald_negative_sign} 
	and \ref{rmk:negative_skew_functions}, 
	and, in particular, formula \eqref{P_one_variable} 
	for the skew (possibly Laurent) 
	polynomial $P_{\la^{(k)}/\la^{(k-1)}}(a_k)$. 

	One can readily deduce the 
	translation invariance property of the links: 
	\begin{align*}
		\La^{k}_{k-1}(\la^{(k)},\la^{(k-1)})=
		\La^{k}_{k-1}(\la^{(k)}+1,\la^{(k-1)}+1) 	
	\end{align*}
	(in the right-hand side we 
	add $1$ to every part of each signature).
\end{remark}

The space $\GT(N)$ of interlacing arrays 
$\boldsymbol \la$ (see 
\S \ref{sec:combinatorics_of_interlacing_arrays_and_related_objects}) is 
readily identified with the state 
space $\bs^{(N)}$ as 
in \eqref{state_space}. 
Thus, in this setting we can consider 
Gibbs measures on $\GT(N)$; 
one can refer to them as 
\emph{Macdonald--Gibbs measures}. 
In the next subsection we will discuss a 
useful subclass of Macdonald--Gibbs measures, 
namely, the (ascending) Macdonald processes.


\subsection{Ascending Macdonald processes} 
\label{sub:ascending_macdonald_processes}

Let $a_1,\ldots,a_N$ be fixed positive parameters, and $\rho$ be a Macdonald-nonnega\-tive specialization of the algebra of symmetric functions $\Sym$ (\S \ref{sub:nonnegative_specializations}) corresponding to parameters $(\al,\be;\ga)$ as in \eqref{Pi_nonneg_spec}. We will always assume that $a_i\alpha_j<1$ for all possible $i,j$ to ensure finiteness of the normalizing constant in \eqref{M_asc} below.

\begin{definition}
	The \emph{ascending Macdonald process} $\M_{asc}(a_1,\dots,a_N;\rho)$ is a probability measure on the set $\GT(N)$ of interlacing arrays $\boldsymbol\la$ \eqref{GT_Scheme} of depth~$N$ (supported on the subset $\GT^{+}(N)\subset\GT(N)$ of nonnegative arrays) defined~as
	\begin{align}\label{M_asc}
		\M_{asc}(a_1,\dots,a_N;\rho)
		(\boldsymbol\la)=
		\frac{P_{\la^{(1)}}(a_1)
		P_{\la^{(2)}/\la^{(1)}}(a_2)\cdots
		P_{\la^{(N)}/\la^{(N-1)}}(a_N)
		Q_{\la^{(N)}}(\rho)}{\Pi(a_1,\dots,a_N;\rho)},
	\end{align}
	where $\la^{(j)}\in\GT_j^{+}$. Here $Q_{\la^{(N)}}$ is the Macdonald symmetric function (\S \ref{sub:macdonald_symmetric_functions}); it is a certain scalar multiple of $P_{\la^{(N)}}$. The normalizing constant $\Pi(a_1,\dots,a_N;\rho)$ is defined in \S \ref{ssub:cauchy_identity}, it is finite due to our assumptions on $a_i$ and $\al_j$.

	In fact, there exist more general Macdonald processes, see \cite[\S2.2.2]{BorodinCorwin2011Macdonald}, 
	\cite{BCGS2013}.
\end{definition}

Projections of Macdonald processes to every fixed row $\la^{(k)}$ 
of a Gelfand--Tsetlin scheme
(in other words, their marginal distributions)
have an explicit form. 
Namely, under the 
Macdonald process $\M_{asc}(a_1,\ldots,a_N;\rho)$ \eqref{M_asc}, 
the distribution of the row $\la^{(k)}\in\GT_k$ 
(cf. Fig.~\ref{fig:GT_scheme}) is given by 
the \emph{Macdonald measure} (cf. \eqref{MMeasure_appendix})
\begin{align}\label{MMeasure}
	\M\M(a_1,\ldots,a_k;\rho)(\la^{(k)})= \frac{P_{\lambda^{(k)}}(a_1,\ldots,a_k) Q_{\lambda^{(k)}}(\rho)} {\Pi(a_1,\ldots,a_k;\rho)}.
\end{align}
This fact can be readily deduced from identities of \S \ref{sub:identities}. Note that this Macdonald measure is supported by partitions $\la^{(k)}\in\GT_k^+$, as it should be.

\begin{proposition}
	Ascending Macdonald processes belong 
	to the class of Mac\-donald--Gibbs 
	measures on $\GT(N)$ 
	(see \S \ref{sub:stochastic_links_with_macdonald_parameters} 
	for the definition).
\end{proposition}
\begin{proof}
	Clearly, one can write
	\begin{align*}
		\M_{asc}(a_1,\dots,a_N;\rho)(\boldsymbol \la)=
		\M\M(a_1,\ldots,a_N;\rho)(\la^{(N)})\cdot
		\prod_{i=1}^{N}\La^{i}_{i-1}
		(\la^{(i)},\la^{(i-1)}),
	\end{align*}
	with $\La^{i}_{i-1}$ given by \eqref{Links_Macdonald}. This is exactly a specialization of the general Definition \ref{def:Gibbs}.
\end{proof}
Note also that if one applies the stochastic link $\La^{k}_{k-1}$ to a Macdonald measure on $\GT_k$, one will get a corresponding Macdonald measure on $\GT_{k-1}$, cf. \eqref{MM_commutes_La}.


\subsection[Univariate continuous-time dynamics]
{Univariate continuous-time dynamics preserving the class of Macdonald measures} 
\label{sub:univariate_continuous_time_dynamics_preserving_the_class_of_macdonald_measures}

Let us now discuss univariate continuous-time dynamics $Q_k$ living 
on each $k$th row, $k=1,\ldots,N$, of a Gelfand--Tsetlin scheme.
One of the main goals of the present paper is to describe
multivariate dynamics $\bq^{(N)}$ 
on Gelfand--Tsetlin schemes having 
these given univariate 
projections (cf. \S \ref{sub:sequential_update_dynamics_in_continuous_time}). 

The univariate dynamics $Q_k$ we are about to 
describe was introduced in 
\cite[\S2.3.1]{BorodinCorwin2011Macdonald}. 
It is defined in terms of jump rates as follows. 
For $\la,\nu\in\GT_k$, $\la\ne\nu$, set 
\begin{align}\label{univariate_Q_k_Macdonald}
	Q_k(\la,\nu):=
	\begin{cases}
		\dfrac{P_\nu(a_1,\ldots,a_k)}
		{P_\la(a_1,\ldots,a_k)}
		\psi'_{\nu/\la},&\mbox{if $\nu=\la+\de_j$ for some $j=1,\ldots,k$};\\
		0,&\mbox{otherwise}.
	\end{cases}
\end{align}
Here the notation $\nu=\la+\de_j$ is explained in \eqref{add_one_box}, and the quantity $\psi'_{\nu/\la}$ is given by \eqref{psi_prime_one_box}. The diagonal elements of $Q_k$ are defined as 
\begin{align}\label{univariate_Q_k_Macdonald_diagonal}
	Q_k(\la,\la):=-\sum_{\nu\in\GT_k\colon\nu\ne\la}Q_k(\la,\nu)=-(a_1+\ldots+a_k).
\end{align}
The last equality follows from \eqref{skew_Cauchy_particular}: 
one should take $\rho_2=\hat\varepsilon$ to be the 
specialization into one dual variable $\varepsilon$ 
(cf. \S \ref{sub:nonnegative_specializations}), 
and then consider the coefficient 
by $\varepsilon$ in \eqref{skew_Cauchy_particular}.

\begin{remark}\label{rmk:particle_jumping}
	Representing signatures $\la\in\GT_k$ as 
	particle configurations $\la_1\ge \ldots\ge\la_k$ on 
	$\Z$ according to 
	\S \ref{sub:visualizing_interlacing_arrays} (see especially 
	Fig.~\ref{fig:la_interlacing}), 
	we see that the jump $\la\to\nu=\la+\de_j$, $j=1,\ldots,k$, 
	of the univariate dynamics 
	\eqref{univariate_Q_k_Macdonald} means that the particle 
	$\la_j$ jumps to the right by one. 
	If this jump is not possible (i.e., if $\la_{j}=\la_{j-1}$),
	then we say that the $j$th 
	particle $\la_j$ is \emph{blocked} by $\la_{j-1}$.
\end{remark}

\begin{proposition}
	\label{prop:univariate_Q_k_Macdonald_finiteness}
	Jump rates $Q_k$ \eqref{univariate_Q_k_Macdonald}, 
	\eqref{univariate_Q_k_Macdonald_diagonal} satisfy 
	finiteness Condition \ref{cond:finiteness}, and thus 
	define a Feller Markov jump process with 
	semigroup $(P_k(\tau))_{\tau\ge0}$, 
	where $P_k(\tau)=\exp(\tau\cdot Q_k)$.
\end{proposition}
\begin{proof}
	Clearly, for every fixed $\la\in\GT_k$, only finitely many of the numbers $Q_k(\la,\nu)$, $\nu\in\GT_k$, are nonzero: they correspond either to $\nu=\la$, or to $\nu=\la+\de_j$, $j=1,\ldots,k$.

	All off-diagonal elements of $Q_k$ are nonnegative. To show that they are uniformly bounded, note first that the jump rates $Q_k$ are translation-invariant, i.e., $Q_k(\la,\nu)=Q_k(\la+1,\nu+1)$ (see Remarks \ref{rmk:Macdonald_negative_sign} and \ref{rmk:negative_skew_functions}). Thus, it suffices to assume that $\la\in\GT_k^+$. Consider the sum
	\begin{align*}
		\sum_{\nu\in\GT_k\colon\nu\ne\la}
		Q_k(\la,\nu)
		=\sum_{\nu\in\GT_k\colon 
		\text{$\nu=\la+\de_j$ for some $j$}}
		\dfrac{P_\nu(a_1,\ldots,a_k)}
		{P_\la(a_1,\ldots,a_k)}
		\psi'_{\nu/\la}.
	\end{align*}
	Let us add more (nonnegative) summands to the above sum: 
	namely, the one with $\nu=\la$ (we have $\psi'_{\la/\la}=1$), 
	and also all other summands for which $\psi'_{\nu/\la}\ne0$. 
	The latter requirement implies that $\nu/\la$ is a vertical 
	strip, which in particular means that $\nu\in\GT^+$ 
	(see \S \ref{sub:skew_shapes_and_semistandard_tableaux} 
	and \eqref{Q_one_dual_variable}). 
	Thus, we see that the above sum is not greater than
	\begin{align*}
		\sum_{\nu\in\GT^{+}}
		\dfrac{P_\nu(a_1,\ldots,a_k)}
		{P_\la(a_1,\ldots,a_k)}
		\psi'_{\nu/\la}
		&=
		\frac{1}{P_\la(a_1,\ldots,a_k)}
		\sum_{\nu\in\GT^{+}}
		P_\nu(a_1,\ldots,a_k)
		Q_{\nu/\la}(\hat 1)
		\\&\hspace{110pt}
		=\Pi(a_1,\ldots,a_k;\hat 1)
		=
		(1+a_1)\ldots (1+a_k).
	\end{align*}
	Here $\hat 1$ means the specialization into one dual variable $\be_1=1$ (\S\ref{sub:nonnegative_specializations}), and we have also used identity \eqref{skew_Cauchy_particular}. Thus, we get the desired uniform bound. This concludes the proof.
\end{proof}

Thus, we can start the Markov jump process with generator $Q_k$ 
from any point and any probability distribution on $\GT_k$. 
A particularly nice class of initial conditions 
is formed by the Macdonald measures \eqref{MMeasure} 
which are supported on $\GT_k^+$:
\begin{proposition}
	\label{prop:univariate_Q_k_Macdonald_MMeasures}
	Let $\rho$ be a Macdonald-nonnegative specialization 
	(see \S \ref{sub:nonnegative_specializations}), $\M\M(a_1,\ldots,a_k;\rho)$ be the corresponding Macdonald measure on $\GT_k^+$ \eqref{MMeasure}, and $P_k(\tau)=\exp(\tau\cdot Q_k)$ be the $\GT_k\times\GT_k$ transition matrix (during time interval $\tau\ge0$) of the univariate dynamics, see \eqref{univariate_Q_k_Macdonald}. Then
	\begin{align*}
		\M\M(a_1,\ldots,a_k;\rho)P_k(\tau)=
		\M\M(a_1,\ldots,a_k;\rho,\rho_{\tau}),
	\end{align*}
	where $\rho_{\tau}$ is the Plancherel specialization with $\ga=\tau\ge0$ (\S \ref{sub:nonnegative_specializations}), and $(\rho,\rho_{\tau})$ means the union of specializations (\S \ref{sub:symmetric_functions}).
\end{proposition}
In other words, if the dynamics 
$P_k(\tau)=\exp(\tau \cdot Q_k)$ starts from a 
Macdonald measure, then it is possible to 
explicitly write down the distribution of the 
Gelfand--Tsetlin scheme at any given moment $\tau\ge0$.
\begin{proof}
	The argument is similar to \cite[Prop. 2.3.6]{BorodinCorwin2011Macdonald}, and is in the spirit of~\S \ref{sub:from_discrete_to_continuous_time}. From \eqref{Q_one_dual_variable} we see that every off-diagonal matrix element $Q_k(\la,\nu)$, $\la\ne\nu$, coincides with the coefficient by $\varepsilon$ in $p^{\uparrow}_{\la\nu}(a_1,\ldots,a_k;\hat\varepsilon)$, where $p^{\uparrow}$ is defined in \eqref{p_Markov_Macdonald_operator} and $\hat\varepsilon$ is the specialization into one dual variable $\varepsilon$. Thus, 
	\begin{align*}
		p^{\uparrow}_{\la\nu}(a_1,\ldots,a_k;\hat\varepsilon)=
		1_{\la=\nu}+\varepsilon\cdot Q_k(\la,\nu)+o(\varepsilon),\qquad
		\varepsilon\to0.
	\end{align*}
	It follows that $\lim_{\varepsilon\to0}\big(p^{\uparrow}(a_1,\ldots,a_k;\hat\varepsilon)\big)^{[\tau/\varepsilon]}=P_k(\tau)$. Moreover, from Proposition \ref{prop:commuting_Markov}.2 we know that 
	\begin{align*}
		\M\M(a_1,\ldots,a_k;\rho)
		\big(p^{\uparrow}(a_1,\ldots,a_k;\hat\varepsilon)\big)^{[\tau/\varepsilon]}
		=
		\M\M(a_1,\ldots,a_k;\rho,\si_{\tau,\varepsilon}),
	\end{align*}
	where $\si_{\tau,\varepsilon}$ is the specialization into $[\tau/\varepsilon]$ dual variables equal to $\varepsilon$. From \S \ref{sub:nonnegative_specializations} we see that for any symmetric function $f\in\Sym$, $\lim\limits_{\varepsilon\to0}f(\si_{\tau,\varepsilon})=f(\rho_\tau)$, and also $\lim\limits_{\varepsilon\to0}\Pi(a_1,\ldots,a_k;\rho,\si_{\tau,\varepsilon})=\Pi(a_1,\ldots,a_k;\rho,\rho_\tau)$. Thus, taking the limit $\varepsilon\to0$ in the above identity for the action of a power of $p^{\uparrow}$ on Macdonald measures, we arrive at the claim of the proposition.
\end{proof}

\begin{proposition}\label{prop:jump_rate_Q_k_Macdonald_commute}
	Jump rate matrices \eqref{univariate_Q_k_Macdonald} commute with the stochastic links \eqref{Links_Macdonald} in the sense of \eqref{commutation_relations_Q_k}, i.e., $\La^{k}_{k-1}Q_{k-1}=Q_k\La^{k}_{k-1}$.
\end{proposition}
\begin{proof}
	Matrices $p^{\uparrow}_{\la\nu}(a_1,\ldots,a_k;\hat\varepsilon)$
	commute with the stochastic links 
	(Proposition \ref{prop:commuting_Markov}.4); considering the 
	coefficient by $\varepsilon$ in the latter commutation 
	relation \eqref{p_up_commutes_links} and using 
	\eqref{univariate_Q_k_Macdonald_diagonal}, 
	we arrive at the desired relation 
	\eqref{commutation_relations_Q_k} for~$Q_k$.
\end{proof}

Let us emphasize once again that the jump rates $Q_k$ 
\eqref{univariate_Q_k_Macdonald} define \emph{univariate} 
dynamics, i.e., processes 
on each row $k$ of the particle 
array (see Fig.~\ref{fig:la_interlacing}). 
The process $Q_k$ does not see what is happening on
other rows of the array. 
Our aim in the next section is to 
stitch all univariate processes $Q_k$ 
into a multivariate continuous-time process on 
interlacing particle arrays (as on Fig.~\ref{fig:la_interlacing}).
Such stitching is possible because 
the generators $Q_k$ commute with the stochastic links 
(Proposition \ref{prop:jump_rate_Q_k_Macdonald_commute}); 
but it is not unique, 
cf. \S \ref{sub:sequential_update_dynamics_in_continuous_time}.



\section{Multivariate continuous-time dynamics on interlacing arrays} 
\label{sec:multivariate_continuous_time_dynamics_on_interlacing_arrays_}

\subsection{Definition} 
\label{sub:definition}

Let us specialize general definitions 
of \S \ref{sub:sequential_update_dynamics_in_continuous_time}
to our concrete situation involving interlacing particle 
arrays. In particular, the notation
will be $\mathbf{X}_N=\boldsymbol\la=(\la^{(1)}\prec
\ldots\prec\la^{(N)})$, and $x_k=\la^{(k)}$. Moreover, 
$\La^{k}_{k-1}$ will denote stochastic
links with Macdonald parameters defined in \S \ref{sub:stochastic_links_with_macdonald_parameters}.

\begin{definition}\label{def:multivariate_GT}
	Let $(\bp^{(N)}(\tau))_{\tau\ge0}$ be a 
	Markov jump process on the space $\GT(N)$ 
	of Gelfand--Tsetlin schemes
	of depth $N$ (equivalently, on the space
	of interlacing particle arrays of 
	depth $N$ as on Fig.~\ref{fig:la_interlacing}). 
	We will call $\bp^{(N)}$ a \emph{multivariate 
	dynamics on interlacing arrays} 
	(\emph{with Macdonald parameters}) if:
	\begin{enumerate}[\bf1.]
		\item $\bp^{(N)}$ satisfies Definition 
		\ref{def:coherent_continuous} 
		(of sequential update continuous-time
		dynamics)
		and all conventions of 
		\S \ref{sub:sequential_update_dynamics_in_continuous_time}. 
		\item Projections $Q_k$ of the dynamics $\bp^{(N)}$ defined by 
		\eqref{Projection_Q_k} coincide with
		the univariate dynamics of \S \ref{sub:univariate_continuous_time_dynamics_preserving_the_class_of_macdonald_measures}.
	\end{enumerate}
\end{definition}
\begin{remark}\label{rmk:multivariate_action_on_Gibbs}
	We emphasize that a multivariate dynamics is not 
	uniquely determined by Definition \ref{def:multivariate_GT}. 
	However, the action
	of \emph{any} multivariate dynamics $\bp^{(N)}$
	on Macdonald processes is the same:
	\begin{align}\label{action_on_Masc}
		\M_{asc}(a_1,\dots,a_N;\rho)
		\bp^{(N)}(\tau)
		=\M_{asc}(a_1,\dots,a_N;\rho,\rho_\tau),
	\end{align}
	where $\rho_{\tau}$ is the Plancherel 
	specialization corresponding to $\tau\ge0$ 
	(cf. Proposition \ref{prop:univariate_Q_k_Macdonald_MMeasures}).
\end{remark}
Observe that 
Definition \ref{def:multivariate_GT} implies a natural restriction
on possible jumps of a multivariate dynamics.
Namely, at each level $k=1,\ldots,N$, no more than one particle 
can jump (to the right by one) at any given moment.
Indeed, this follows
(with the help of \eqref{Projection_Q_k}) from the
corresponding property
of 
the univariate
dynamics $Q_k$ (\S \ref{sub:univariate_continuous_time_dynamics_preserving_the_class_of_macdonald_measures}): a jump $\la^{(k)}\to\nu^{(k)}$, 
$\la^{(k)},\nu^{(k)}\in\GT_k$, can occur under $Q_k$ only if 
$\nu^{(k)}=\la^{(k)}+\de_j$ for some $j=1,\ldots,k$
(see \eqref{add_one_box} for this notation).


\subsection{Specialization of formulas from \S \ref{sub:sequential_update_dynamics_in_continuous_time}} 
\label{sub:specialization_of_formulas_from_s_sub:sequential_update_dynamics_in_continuous_time}

As follows from Definition \ref{def:multivariate_GT},
a multivariate dynamics $\bp^{(N)}(\tau)$ is completely 
determined (via its jump rates expressed
as \eqref{bq_N_product_form}) 
by the rates of independent jumps 
$W_k(\la^{(k)},\nu^{(k)}\,|\,\la^{(k-1)})$ 
together with the
probabilities of triggered moves 
$V_k(\la^{(k)},\nu^{(k)}\,|\,\la^{(k-1)},\nu^{(k-1)})$,
which must satisfy \eqref{V_conditions_P1'_1}, 
\eqref{V_conditions_P1'_2}, and \eqref{W_conditions_P1'}. 
In \S \S \ref{sub:specialization_of_formulas_from_s_sub:sequential_update_dynamics_in_continuous_time},
\ref{sub:when_a_jumping_particle_has_to_push_its_immediate_upper_right_neighbor}, and 
\ref{sub:characterization_of_multivariate_dynamics}
we will write down certain \emph{necessary and 
sufficient} conditions on the $W_k$'s and the $V_k$'s
under which they give rise to a multivariate dynamics
with Macdonald parameters.

Observe that the main identity
\eqref{P2'} that we need to 
specialize to our situation 
involves only two consecutive
levels, $k-1$ and $k$,
of the interlacing array
$\boldsymbol\la$ (see Fig.~\ref{fig:GT_scheme}
and \ref{fig:la_interlacing}). Thus, 
let us fix $k=1,\ldots,N$, 
and restrict our
attention to
$\GT_{k-1}$ and $\GT_{k}$.
Moreover, to shorten formulas below, 
we will introduce some additional notation. 
We will denote signatures $\la^{(k)},\nu^{(k)},\ldots$ from
$\GT_k$ simply by $\la,\nu,\ldots$; 
and signatures $\la^{(k-1)},\nu^{(k-1)},\ldots\in\GT_{k-1}$
will be denoted with with a bar: $\bar\la,\bar\nu,\ldots$. 
Also, on $\GT_{k-1}$ we will use the notation 
$\bar\nu=\bar\la+\bar\de_i$ (equivalently, 
$\bar\la=\bar\nu-\bar\de_i$)
if $\bar\nu$ is obtained from $\bar\la$
by adding one to the coordinate
$\bar\la_i$.
At level $k$, we will 
write as before, $\nu=\la+\de_j$,
for a similar relation.

Now, assume that functions 
$W_k(\la,\nu\,|\,\bar\la)$ and $V_k(\la,\nu\,|\,\bar\la,\bar\nu)$ 
satisfy 
\begin{eqnarray}\label{spec_V_conditions_P1'}
	\sum_{\nu\in\GT_k}V_k(\la,\nu\,|\,\bar\la,\bar\nu)=1,
	\qquad&&\qquad
	V_k(\la,\nu\,|\,\bar\la,\bar\la)
	=
	1_{\la=\nu},
	\\
	\label{spec_W_conditions_P1'}
	\sum_{\nu\in\GT_k\colon\nu\ne\la}
	W_k(\la,\nu\,|\,\bar\la)
	&=&-
	W_k(\la,\la\,|\,\bar\la).
\end{eqnarray}
These identities are just specializations of the general conditions 
\eqref{V_conditions_P1'_1}, 
\eqref{V_conditions_P1'_2}, and \eqref{W_conditions_P1'}.

\begin{proposition}
\label{prop:main_general_identity}
	For a fixed $k=1,\ldots,N$, consider functions 
	$W_k(\la,\nu\,|\,\bar\la)$ and $V_k(\la,\nu\,|\,\bar\la,\bar\nu)$.
	Under 
	\eqref{spec_V_conditions_P1'}--\eqref{spec_W_conditions_P1'}, 
	the general condition \eqref{P2'} on $W_k,V_k$ is equivalent to 
	the following family of identities
	(quantities $\psi$ and $\psi'$ are 
	given in \S \ref{sub:skew_functions} and
	\S \ref{sub:endomorphism_q,t_}):
	\begin{align}\label{P2'_psi}
		&\sum_{i=1}^{k-1}
		V_k(\la,\la+\de_j\,|\,
		\bar\nu-\bar\de_i,\bar\nu)
		\psi_{\la/\bar\nu-\bar\de_i}
		\psi'_{\bar\nu/\bar\nu-\bar\de_i}
		+
		a_k^{-1}
		W_k(\la,\la+\de_j\,|\,\bar\nu)
		\psi_{\la/\bar\nu}
		=\psi'_{\la+\de_j/\la}\psi_{\la+\de_j/\bar\nu}.
	\end{align}
	These identities are written out 
	for all $\la\in\GT_k$, all $j=1,\ldots,k$ 
	such that $\la_{j}<\la_{j-1}$ (so that $\la+\de_j$ 
	is also a signature $\in\GT_k$), and 
	all $\bar\nu\in\GT_{k-1}$. We also 
	need to impose the 
	condition that $\bar\nu\prec\la+\de_j$.

	In the summation over $i$, we agree that if 
	$\bar\nu-\bar\de_i$ is not a signature
	(i.e., if $\bar\nu_i=\bar\nu_{i+1}$), 
	then $\psi'_{\bar\nu/\bar\nu-\bar\de_i}=0$.
\end{proposition}
The summation in \eqref{P2'_psi} can be informally
understood as follows. Having any move
$\la\to\nu=\la+\de_j$ at the upper level $\GT_k$, and a fixed new
state $\bar\nu\in\GT_{k-1}$ at the lower level, 
the sum is taken over 
all possible ``histories'' $\bar\la=\bar\nu-\bar\de_i$, 
i.e., over all moves $\bar\la\to\bar\nu$ 
which could have 
happened\footnote{In our continuous-time 
jump dynamics, only one independent jump 
(cf. Comment~\ref{comm:sequential_update_continuous}) 
can happen during an infinitesimally 
small time interval.} at $\GT_{k-1}$. 
Of course, it must be $\bar\la\prec\la$ and $\bar\nu\prec\nu$.
Note that we explicitly impose the latter condition in 
the formulation, while the former condition is 
ensured by the presence
of the coefficient $\psi_{\la/\bar\nu-\bar\de_i}$
which vanishes unless $\bar\la=\bar\nu-\bar\de_i\prec\la$.

\bigskip
\noindent
\emph{Proof of Proposition \ref{prop:main_general_identity}.}
Fix $\la,\nu\in\GT_k$ and $\bar\nu\in\GT_{k-1}$, 
and write identity \eqref{P2'} 
for $x_k=\la$, $y_k=\nu$, and $y_{k-1}=\bar\nu$.
Due to Remark \ref{rmk:condition_xk_ne_yk} and Proposition 
\ref{prop:jump_rate_Q_k_Macdonald_commute},
we may assume that $\nu\ne\la$; and we also may 
run the summation over
$\bar\la\ne\bar\nu$. We obtain:
\begin{align}
	\label{P2'_la_nu_proof}
	&
	\sum_{\bar\la\in\GT_{k-1}
	\colon\bar\la\ne\bar\nu}
	V_k(\la,\nu\,|\,\bar\la,\bar\nu)
	\La^{k}_{k-1}(\la,\bar\la)
	Q_{k-1}(\bar\la,\bar\nu)
	+
	W_k(\la,\nu\,|\,\bar\nu)
	\La^{k}_{k-1}(\la,\bar\nu)
	=
	Q_k(\la,\nu)\La^{k}_{k-1}(\nu,\bar\nu).
\end{align}
Then,
plugging in the definitions of 
stochastic links 
(\S \ref{sub:stochastic_links_with_macdonald_parameters}; 
we also use \eqref{P_one_variable}) 
and univariate dynamics 
(\S \ref{sub:univariate_continuous_time_dynamics_preserving_the_class_of_macdonald_measures}) 
and crossing out (nonzero) common factors involving 
Macdonald polynomials,
we obtain
\begin{align*}
	&
	\sum_{\bar\la\in\GT_{k-1}
	\colon\bar\la\ne\bar\nu}
	V_k(\la,\nu\,|\,\bar\la,\bar\nu)
	\cdot
	\psi_{\la/\bar\la}a_k^{|\la|-|\bar\la|}\cdot
	\psi'_{\bar\nu/\bar\la}
	1_{\text{$\bar\nu=\bar\la+\bar\de_i$ for some $i$}}
	\\&
	\hspace{125pt}
	+
	W_k(\la,\nu\,|\,\bar\nu)
	\psi_{\la/\bar\nu}a_k^{|\la|-|\bar\nu|}
	=
	\psi'_{\nu/\la}
	1_{\text{$\nu=\la+\de_j$ for some $j$}}
	\cdot
	\psi_{\nu/\bar\nu}a_k^{|\nu|-|\bar\nu|}.
\end{align*}
This clearly coincides with the desired claim of the proposition.
\qed
\bigskip

In fact, for $k=1$ identity \eqref{P2'_psi} (which then does not 
contain the summation over $i$) means that $W_1=a_1 Q_1$, where $Q_1$
is the univariate jump rate matrix on the first level $\GT_1\cong\Z$ 
(see \S \ref{sub:univariate_continuous_time_dynamics_preserving_the_class_of_macdonald_measures}).
That is, the bottommost particle in the interlacing array 
(as on Fig.~\ref{fig:la_interlacing}) performs the univariate 
dynamics with speed scaled by $a_1$. 

\begin{remark}\label{rmk:la=nu}
	As pointed out in Remark \ref{rmk:condition_xk_ne_yk},
	identity similar to \eqref{P2'_la_nu_proof} but with
	$\nu=\la$ follows automatically from 
	\eqref{P2'_psi} and commutation relations 
	(Proposition \ref{prop:jump_rate_Q_k_Macdonald_commute}). Using \eqref{univariate_Q_k_Macdonald_diagonal}, we see that this identity has the form:
	\begin{align}\label{P2'_la=nu}
		\sum_{i=1}^{k-1}
		V_k(\la,\la\,|\,
		\bar\nu-\bar\de_i,\bar\nu)
		\psi_{\la/\bar\nu-\bar\de_i}
		\psi'_{\bar\nu/\bar\nu-\bar\de_i}
		+
		(a_k^{-1}
		W_k(\la,\la\,|\,\bar\nu)+1)
		\psi_{\la/\bar\nu}
		=0.
	\end{align}
\end{remark}
Let us make one more natural remark
about the functions $V_k,W_k$:
\begin{remark}\label{rmk:zero_outside_GT}
	It will be sometimes convenient to 
	extend the definition of 
	$V_k(\la,\nu\,|\,\bar\la,\bar\nu)$ and 
	$W_k(\la,\nu\,|\,\bar\nu)$ beyond our usual 
	assumptions 
	by setting them equal to zero
	if $\bar\la\not\prec\la$ or $\bar\nu\not\prec\nu$
	(cf. Remarks \ref{rmk:zero_outside_state_space} and
	\ref{rmk:zero_outside_state_space_P}). 
	Probabilistically this means that, for example, 
	a jump at level $k$ 
	from $\la$ to $\nu$ with 
	$\bar\nu\not\prec\nu$
	is impossible.
\end{remark}


\subsection[When a moving particle has to push its upper right neighbor]
{When a moving particle has to short-range 
push its immediate upper right neighbor} 
\label{sub:when_a_jumping_particle_has_to_push_its_immediate_upper_right_neighbor}

It is convenient now to separate a special case of 
identities \eqref{P2'_psi}, namely, when $\bar\nu\not\prec\la$. 
Since $\bar\nu$ is the new state of the dynamics 
at level $\GT_{k-1}$, 
it must differ as $\bar\nu=\bar\la+\bar\de_i$ 
from a previous state $\bar\la$ for 
which $\bar\la\prec\la$. 
This condition defines the signature $\bar\la$
uniquely.
Next, the interlacing 
constraints (cf.~Fig.~\ref{fig:la_interlacing}) 
clearly imply that there exists a unique $\nu=\la+\de_j$ such
that $\bar\nu\prec\nu$, and, moreover, $i=j$ (see also 
Fig.~\ref{fig:nu_not_prec_la} below).
Thus, we see that identity
\eqref{P2'_psi} for $\bar\nu\not\prec\la$
takes the form:
\begin{align}\label{Vpsipsi=psipsi}
	V_k(\la,\la+\de_j\,|\,
	\bar\nu-\bar\de_j,\bar\nu)
	\psi_{\la/\bar\nu-\bar\de_j}
	\psi'_{\bar\nu/\bar\nu-\bar\de_j}=
	\psi'_{\la+\de_j/\la}
	\psi_{\la+\de_j/\bar\nu}.
\end{align}
Therefore, the value of 
$V_k(\la,\la+\de_j\,|\,\bar\nu-\bar\de_j,\bar\nu)$
is completely determined, and we in fact 
can compute it:
\begin{proposition}\label{proposition:nu_not_prec_la}
	For $\bar\nu\not\prec\la$,
	$\bar\nu\prec\la+\de_j$, 
	and $\bar\nu-\bar\de_j\prec\la$,
	one has
	\begin{align}\label{psipsi=psipsi}
		\psi_{\la/\bar\nu-\bar\de_j}
		\psi'_{\bar\nu/\bar\nu-\bar\de_j}=
		\psi'_{\la+\de_j/\la}
		\psi_{\la+\de_j/\bar\nu},
	\end{align}
	and so in this case
	\begin{align}\label{V_k_blocking}
		V_k(\la,\la+\de_j\,|\,\bar\nu-\bar\de_j,\bar\nu)
		=1.
	\end{align}
\end{proposition}
\begin{proof}
	This can be checked directly using explicit formulas 
	for $\psi,\psi'$ (\S\ref{sub:skew_functions}, 
	\S\ref{sub:endomorphism_q,t_}). 
	A more structured way to see this is to consider the skew Cauchy
	identity \eqref{skew_Cauchy} 
	with $\rho_1=a$ (one usual variable), 
	and $\rho_2=\hat\varepsilon$ (one dual variable), and 
	extract from
	this identity the coefficient 
	by the first power of~$\varepsilon$. 
	
	This skew Cauchy identity reads (we rewrite it in 
	terms of $\psi,\psi'$):
	\begin{align*}
		(1+a \varepsilon)
		\sum_{\mu\in\GT^+}
		\varepsilon^{|\bar\nu|-|\mu|}
		\psi'_{\bar\nu/\mu}
		a^{|\la|-|\mu|}\psi_{\la/\mu}
		=
		\sum_{\ka\in\GT^+}a^{|\ka|-|\bar\nu|}
		\psi_{\ka/\bar\nu}
		\varepsilon^{|\ka|-|\la|}
		\psi'_{\ka/\la},
	\end{align*}
	and considering the coefficient by $\varepsilon$ 
	leads to
	\begin{align}\label{inf_skew_cauchy}
		\sum_{i=1}^{k-1}
		\psi'_{\bar\nu/\bar\nu-\bar\de_i}
		\psi_{\la/\bar\nu-\bar\de_i}
		+\psi_{\la/\bar\nu}=
		\sum_{j=1}^{k}
		\psi_{\la+\de_j/\bar\nu}
		\psi'_{\la+\de_j/\la}.
	\end{align}
	Now one can readily see that under the 
	assumptions of the proposition, 
	identity \eqref{psipsi=psipsi} holds. 
	To obtain the second claim (from \eqref{Vpsipsi=psipsi}), 
	we note that under
	our assumptions both sides of \eqref{psipsi=psipsi} are 
	nonzero.
\end{proof}

\begin{remark}
	In fact, identity \eqref{inf_skew_cauchy}
	is equivalent to the 
	commutation relations between 
	the jump rate matrices $Q_k$'s 
	and the stochastic links 
	(see Proposition \ref{prop:jump_rate_Q_k_Macdonald_commute}).
	
	Note also that if 
	we sum \eqref{P2'_psi} over all possible $j=1,\ldots,k$, 
	and add \eqref{P2'_la=nu}, we get \eqref{inf_skew_cauchy}.
	In other words, various 
	multivariate dynamics correspond to various
	refinements of 
	\eqref{inf_skew_cauchy} 
	(cf. discussions 
	after \eqref{coherency_dynamics}~and~\eqref{P2'}). 
\end{remark}

\begin{remark}\label{rmk:W_k_nu_always_prec_la}
	Note that the values $W_k(\la,\nu\,|\,\bar\nu)$ with 
	$\bar\nu\not\prec\la$ do not formally enter
	identities \eqref{P2'_psi}, see \eqref{Vpsipsi=psipsi},
	and thus cannot be determined by them.
	However, 
	these values are also
	not employed in the definition of 
	multivariate dynamics (see \eqref{bq_N_product_form}
	or \eqref{spec_bq_N_product_form} below).
	Thus, when speaking
	about $W_k(\la,\nu\,|\,\bar\nu)$,
	we will always assume that $\bar\nu\prec\la$. 
\end{remark}

Proposition \ref{proposition:nu_not_prec_la} 
is equivalent to 
the following rule which
holds for \emph{any} 
multivariate continuous-time dynamics
on interlacing particle 
arrays (see Fig.~\ref{fig:la_interlacing}):
\begin{sprule}
	\label{rule:block-pushing}
	If any particle $\la_{j}^{(k-1)}$ moves to
	the right by one
	and $\la^{(k)}_j=\la^{(k)}_{j-1}$, then the particle 
	$\la^{(k)}_j$ 
	is forced to
	instantly move to the right by one,
	and there are
	no other triggered
	moves at the $k$th row of the array.

	In other
	words, the triggered move of the 
	upper particle $\la^{(k)}_{j}$ is necessary
	to immediately restore the interlacing
	condition $\la_{j}^{(k)}\ge\la^{(k-1)}_{j}$
	that was broken by the move of $\la^{(k-1)}_{j}$.
\end{sprule}
Here and below 
by a \emph{jump}
we mean
an independent jump, and 
\emph{move} means an independent jump
or a triggered move 
(cf. Comment \ref{comm:sequential_update_continuous}).

The above rule
in fact coincides with one of the 
rules for the ``push-block'' dynamics introduced in 
\cite{BorFerr2008DF} (in the Schur
case), namely, with the ``push'' rule.
See \S 
\ref{sub:example_push_block_process_on_interlacing_arrays}
for more discussion.

\begin{figure}[htbp]
	\begin{center}
		\includegraphics[width=300pt]{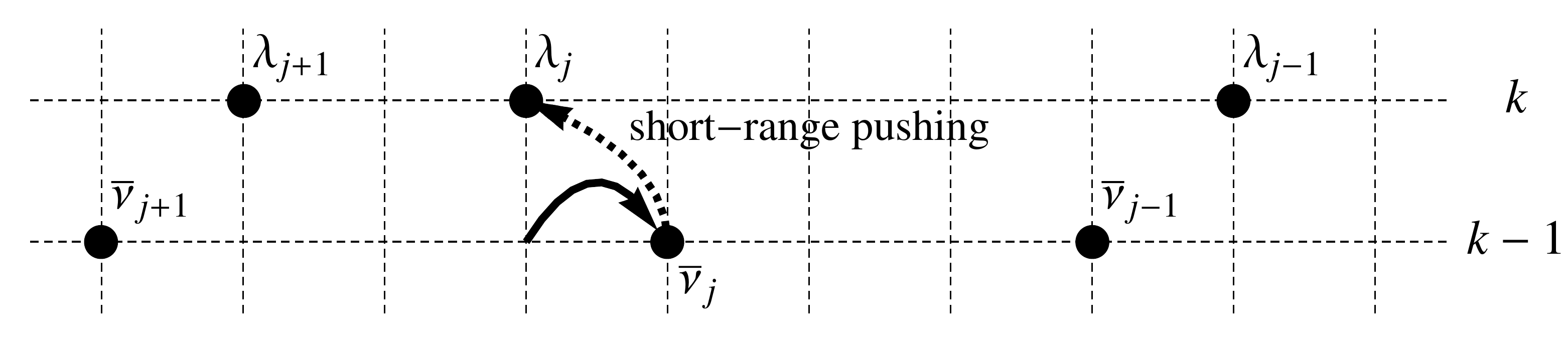}
	\end{center}  
  	\caption{The moved particle $\bar\nu_j$
  	violates the interlacing constraints and must 
  	instantly
  	short-range push $\la_j$. Dashed arrow 
  	represents this pushing interaction.}
  	\label{fig:nu_not_prec_la}
\end{figure}

Observe that 
the interaction 
between particles
described by the above rule
happens only at \emph{short distance} 
(namely, when the particle $\la^{(k-1)}_{j}$ is 
right under $\la^{(k)}_{j}$, cf.~Fig.~\ref{fig:la_interlacing}).
We will call this interaction
the \emph{short-range pushing}.


\subsection{Characterization of multivariate dynamics} 
\label{sub:characterization_of_multivariate_dynamics}

In the general situation, i.e., 
when $\bar\nu\prec\la$ (in contrast
with the development of 
\S \ref{sub:when_a_jumping_particle_has_to_push_its_immediate_upper_right_neighbor}), 
it turns out to be convenient to divide \eqref{P2'_psi} by 
$\psi_{\la/\bar\nu}$, which is now nonzero. 
Indeed, then 
the resulting quantities in the left-hand 
side of \eqref{P2'_psi} are given explicitly by
\begin{align}
	\nonumber
	&
	T_i(\bar\nu,\la):=\frac{\psi_{\la/\bar\nu-\bar\de_i}}
	{\psi_{\la/\bar\nu}}
	\psi'_{\bar\nu/\bar\nu-\bar\de_i}=
	\frac{(1-q^{\la_i - \bar\nu_i}t)(1-q^{\bar\nu_i-\la_{i+1}})}
	{(1-q^{\la_i - \bar\nu_i+1})
	(1-q^{\bar\nu_i-1-\la_{i+1}}t)}
	\\&
	\label{T_i}\hspace{100pt}\times
	\prod_{r=1}^{i-1}
	\frac{
	(1-q^{\la_r - \bar \nu_i} t^{i - r+1})(1-q^{\bar\nu_r-\bar\nu_i+1}t^{i-r-1})}
	{
	(1-q^{\la_r - \bar \nu_i + 1} t^{i - r})
	(1-q^{\bar\nu_r-\bar\nu_i}t^{i-r})}
	\\&
	\nonumber\hspace{100pt}\times
	\prod_{s=i+1}^{k-1}
	\frac{(1-q^{\bar\nu_i - \bar\nu_s-1} t^{s - i+1})
	(1-q^{\bar\nu_i - \la_{s + 1}} t^{s -i})}
	{(1-q^{\bar\nu_i - \bar\nu_s} t^{s - i})
	(1-q^{\bar\nu_i - \la_{s + 1}-1} t^{s -i+1})},
\end{align}
and in the right-hand side we get
\begin{align}&
	S_j(\bar\nu,\la):=\nonumber
	\frac{\psi_{\la+\de_j/\bar\nu}}
	{\psi_{\la/\bar\nu}}
	\psi'_{\la+\de_j/\la}=
	\prod_{r=1}^{j-1}
	\frac{(1-q^{\bar\nu_r-\la_j}t^{j-r-1})
	(1-q^{\la_r-\la_j-1}t^{j-r+1})
	}
	{(1-q^{\bar\nu_r-\la_j-1}t^{j-r})
	(1-q^{\la_r-\la_j}t^{j-r})
	}
	\\&
	\label{S_j}
	\hspace{110pt}\times
	\prod_{s=j}^{k-1}
	\frac{(1-q^{\la_j-\la_{s+1}+1}t^{s-j})
	(1-q^{\la_j-\bar\nu_s}t^{s-j+1})
	}
	{(1-q^{\la_j-\la_{s+1}}t^{s-j+1})
	(1-q^{\la_j-\bar\nu_s+1}t^{s-j})}.
\end{align}
The above explicit formulas for $T_i$ and $S_j$ follow
from the formulas
for the quantities $\psi,\psi'$, see 
\S\ref{sub:skew_functions} and \S\ref{sub:endomorphism_q,t_}.
The quantities
$T_i$ and $S_j$ of course depend on $k$, 
as they are defined for $\bar\nu\in\GT_{k-1}$ 
and $\la\in\GT_k$. However, we will not indicate this
dependence on $k$ explicitly.

Now we are in a position to summarize the 
development of \S\S \ref{sub:definition},
\ref{sub:specialization_of_formulas_from_s_sub:sequential_update_dynamics_in_continuous_time}, and 
\ref{sub:when_a_jumping_particle_has_to_push_its_immediate_upper_right_neighbor} by giving a 
general characterization of 
multivariate continuous-time dynamics
(with Macdonald parameters)
living on the space of interlacing particle
arrays (as on Fig.~\ref{fig:la_interlacing}). 

Assume that we are given 
functions 
\begin{align*}&
	W_k(\cdot,\cdot\,|\,\cdot)
	\colon\GT_k\times \GT_k\times\GT_{k-1}\to\R
	,& k=1,\ldots,N;
	\\&
	V_k(\cdot,\cdot\,|\,\cdot,\cdot)
	\colon
	\GT_k\times\GT_k\times\GT_{k-1}\times\GT_{k-1}
	\to[0,1],& k=2,\ldots,N.
\end{align*}
Set
\begin{align}\label{spec_bq_N_product_form}
	\bq^{(N)}(\boldsymbol\la,
	\boldsymbol\nu):=
	W_k(\la^{(k)},\nu^{(k)}\,|\,
	\la^{(k-1)})
	\prod_{i=k+1}^{N}
	V_i(\la^{(i)},\nu^{(i)}\,|\,
	\la^{(i-1)},\nu^{(i-1)}),
\end{align}
for each pair $\boldsymbol\la\ne\boldsymbol\nu$
of Gelfand--Tsetlin schemes of depth $N$ (see \S \ref{sub:signatures_young_diagrams_and_interlacing_arrays}) such that $\la^{(j)}=\nu^{(j)}$
for all $j\le k-1$, and $\la^{(k)}\ne\nu^{(k)}$. 
Define the diagonal elements of $\bq^{(N)}$ by 
$\bq^{(N)}(\boldsymbol\la,\boldsymbol\la):=
-\sum_{\boldsymbol\nu\in\GT(N)\colon\boldsymbol\nu\ne\boldsymbol\la}
\bq^{(N)}(\boldsymbol\la,\boldsymbol\nu)$. 

We further assume that the functions $W_k,V_k$, $k=1,\ldots,N$, 
satisfy \eqref{spec_V_conditions_P1'}--\eqref{spec_W_conditions_P1'}, 
and that $W_k(\la^{(k)},\nu^{(k)}\,|\,\la^{(k-1)})\ge0$ 
for $\la^{(k)}\ne\nu^{(k)}$.

\begin{theorem}\label{thm:general_multivariate}
	The functions $W_k,V_k$, $k=1,\ldots,N$, described
	right before the theorem define a multivariate
	continuous-time Markov dynamics
	on interlacing arrays 
	(cf. Definition \ref{def:multivariate_GT})
	with jump
	rates (= Markov generator) 
	\eqref{spec_bq_N_product_form} 
	if and only~if: 
	\begin{enumerate}[\bf1.]
		\item For any $k=2,\ldots,N$, any $\la\in\GT_k$,
		$\bar\nu\in\GT_{k-1}$
		with $\bar\nu\not\prec\la$, for which
		there
		exists (in fact, unique) $j=1,\ldots,k-1$ such that 
		$\bar\nu\prec\la+\de_j$ and $\bar\nu-\bar\de_j\prec\la$,
		we have 
		$V_k(\la,\la+\de_j
		\,|\,\bar\nu-\bar\de_j,\bar\nu)=1$.
		\item For any $k=1,\ldots,N$, any
		$\la\in\GT_k$, $\bar\nu\in\GT_{k-1}$
		with $\bar\nu\prec\la$, and any $j=1,\ldots,k$ such that
		$\la_{j}<\bar\nu_{j-1}$,\footnote{That is, 
		$\la+\de_j$ is a signature, and 
		$\bar\nu\prec\la+\de_j$. In other words, this means that 
		the particle $\la_j$ is not blocked and can move to the 
		right.} we have
		\begin{align}\label{P2'_T_S}
			\sum_{i=1}^{k-1}
			V_k(\la,\la+\de_j\,|\,
			\bar\nu-\bar\de_i,\bar\nu)
			T_i(\bar\nu,\la)+
			a_k^{-1}
			W_k(\la,\la+\de_j\,|\,\bar\nu)
			=S_j(\bar\nu,\la),
		\end{align}
		where $T_i$ and $S_j$ are given by \eqref{T_i} and 
		\eqref{S_j}, respectively. By agreement, if
		in the summation $\bar\nu-\bar\de_i$ is not a signature, 
		we set $T_i(\bar\nu,\la)=0$.
	\end{enumerate}
\end{theorem}
\begin{proof}
	The first claim is due to 
	Proposition \ref{proposition:nu_not_prec_la} 
	(see also the short-range pushing rule in \S
	\ref{sub:when_a_jumping_particle_has_to_push_its_immediate_upper_right_neighbor}), 
	and the 
	second claim follows from 
	Proposition \ref{prop:main_general_identity}.
\end{proof}

\begin{remark}\label{rmk:stacking}
	As we mentioned earlier, 
	it is possible to solve 
	our equations on the functions $W_m,V_m$, $m=1,\ldots,N$ 
	(of Theorem \ref{thm:general_multivariate}) 
	consecutively level by level.
	Then, ``stacking'' these solutions as in 
	\eqref{spec_bq_N_product_form} for $m=1,\ldots,N$, 
	we obtain a multivariate dynamics
	on the whole space $\GT(N)$.

	Therefore, it is possible
	(and often convenient)
	to restrict attention to a \emph{slice}
	formed by rows $k-1$ and $k$
	of the 
	interlacing array 
	(see Fig.~\ref{fig:la_interlacing}),
	where $k=2,\ldots,N$ is fixed. 
	Let us denote this slice by 
	$\GT_{(k-1;k)}:=
	\{(\bar\nu,\la)\in\GT_{k-1}\times\GT_{k}
	\colon\bar\nu\prec\la\}$.
	Let also, by agreement, 
	$\GT_{(0;1)}:=\GT_1$.
\end{remark}

On each fixed slice $\GT_{(k-1;k)}$
one can perform the following operation.
Assume we are given a
multivariate dynamics
defined by $\{(W_m,V_m)\}_{m=1}^{N}$. 
Let also
$(W_k',V_k')$ for some fixed $k$ be 
any other solution
of equations 
of Theorem \ref{thm:general_multivariate}
(this solution must also satisfy
the assumptions before the 
theorem).\footnote{It must be $k=2,\ldots,N$
because $V_1$ makes no sense, and
$W_1=a_1Q_1$ must be the same
for any multivariate dynamics.}
Then it is possible to 
\emph{mix}
the old quantities $(W_k,V_k)$ 
with the solution $(W_k',V_k')$,
i.e., replace $(W_k,V_k)$ 
by the following convex combination:
\begin{align}\label{mixing_WV}
	\begin{array}{rcl}
		\tilde W_k(\la,\nu\,|\,\bar\nu)&:=&
		\co(\bar\nu,\la)W_k(\la,\nu\,|\,\bar\nu)
		+\big(1-\co(\bar\nu,\la)\big)
		W_k'(\la,\nu\,|\,\bar\nu),\\
		\tilde V_k(\la,\nu\,|\,\bar\la,\bar\nu)&:=&
		\co(\bar\nu,\la)V_k(\la,\nu\,|\,\bar\la,\bar\nu)
		+\big(1-\co(\bar\nu,\la)\big)
		V_k'(\la,\nu\,|\,\bar\la,\bar\nu),
	\end{array}
\end{align}
where, as usual, $\la,\nu\in\GT_k$, 
$\bar\la,\bar\nu\in\GT_{k-1}$. Here $\co(\bar\nu,\la)$
is any function on $\GT_{(k-1;k)}$
(indeed, due to claim 1 of Theorem \ref{thm:general_multivariate},
it suffices to define it for $\bar\nu\prec\la$)
such that $0\le \co(\bar\nu,\la)\le1$ for all $\bar\nu,\la$.
In particular, for $\co(\bar\nu,\la)\equiv 1$
one simply replaces $(W_k,V_k)$
by $(W_k',V_k')$
in the definition of multivariate dynamics. 
All other quantities $(W_m,V_m)$, $m\ne k$, remain unchanged.

One can take these 
rather general coefficients
(i.e., depending on $\bar\nu$ and~$\la$ 
in an arbitrary way) because
equations \eqref{P2'_T_S} are written down for each 
fixed pair $\bar\nu,\la$, and so $(\tilde W_k,\tilde V_k)$
again satisfies Theorem \ref{thm:general_multivariate}.
One can also readily check that 
conditions 
\eqref{spec_V_conditions_P1'}--\eqref{spec_W_conditions_P1'}
hold for
$(\tilde W_k,\tilde V_k)$ 
as well.
Therefore, the result of mixing \eqref{mixing_WV}
is again a multivariate dynamics.

\begin{remark}
	One can go even further and make the coefficient
	$\co(\bar\nu,\la)$ 
	in \eqref{mixing_WV}
	depend also 
	on $\nu=\la+\de_j$ ($j=1,\ldots,k$), 
	because for each
	$\bar\nu,\la$, and each $j$
	we have a separate equation
	\eqref{P2'_T_S}.
	Taking such coefficients
	$\co(\bar\nu,\la,\la+\de_j)$
	in \eqref{mixing_WV} also
	produces multivariate dynamics.
	However, in our future 
	treatment of 
	nearest neighbor multivariate
	dynamics (\S \ref{sec:nearest_neighbor_multivariate_continuous_time_dynamics_on_interlacing_arrays})
	we do not need this generality
	to establish the characterization
	(Theorem \ref{thm:characterization_NN}). 
	This happens because of our parametrization
	and linear equations 
	for this type of dynamics
	(see 
	\S \ref{sub:parametrization_and_linear_equations}).
\end{remark}


\subsection{Example: push-block dynamics on interlacing arrays} 
\label{sub:example_push_block_process_on_interlacing_arrays}

In this and the next subsection
we discuss two previously known
examples of multivariate
dynamics on interlacing arrays
which fall under the setting of our Theorem
\ref{thm:general_multivariate}. The construction
of our first example 
works for general Macdonald parameters $(q,t)$ and was introduced
in
\cite[\S2.3.3]{BorodinCorwin2011Macdonald}.
The second model 
(\S \ref{sub:o_connell_pei_s_process_on_interlacing_arrays}) was constructed
in \cite{OConnellPei2012} in the $q$-Whittaker (i.e., $t=0$)
case.

A general construction of a 
continuous-time 
Diaconis--Fill type
dynamics 
is explained in \cite[\S8]{BorodinOlshanski2010GTs} 
(see \cite[\S2]{BorFerr2008DF} and also \S \ref{sub:example_diaconis_fill_dynamics} for a 
discrete-time version). 

We will call the distinguished
Diaconis--Fill type dynamics in continuous time 
the \emph{push-block dynamics}.
In the setting with general Macdonald parameters, the 
push-block dynamics was considered 
in \cite[\S2.3.3]{BorodinCorwin2011Macdonald}. 
Let us present an independent characterization of this 
dynamics using our formalism:
\begin{definition}\label{def:DF_continuous}
	Let us call a multivariate dynamics 
	on interlacing arrays (Definition \ref{def:multivariate_GT}) 
	\emph{push-block dynamics} if 
	$V_k(\la,\nu\,|\,\bar\la,\bar\nu)=
	1_{\la=\nu}$
	for any $k=2,\ldots,N$, $\la,\nu\in\GT_k$ 
	and $\bar\la,\bar\nu\in\GT_{k-1}$ such that 
	$\bar\la\prec\la$, $\bar\nu\prec\nu$,
	and, moreover, $\bar\nu\prec\la$.
\end{definition}
In other words, we distinguish the push-block dynamics by requiring 
that it has minimal possible triggered moves $V_k$
(cf. Comment \ref{comm:sequential_update_continuous}). 
Namely, if the move $\bar\la\to\bar\nu$ 
at $\GT_{k-1}$
does not break the interlacing between 
rows $k-1$ and $k$ of the array, 
then this move does not propagate to higher levels $k,k+1,\ldots$. 
However, one cannot completely eliminate the triggered moves,
cf. the obligatory short-range pushing rule in \S
\ref{sub:when_a_jumping_particle_has_to_push_its_immediate_upper_right_neighbor}.

One easily checks that Definition \ref{def:DF_continuous}
\emph{uniquely defines} a multivariate dynamics.
In particular, this definition implies that 
the 
nonzero off-diagonal
rates of independent jumps $W_k$, $k=1,\ldots,N$, 
are (here $\bar\nu\prec\la$, 
cf. Remark \ref{rmk:W_k_nu_always_prec_la}):
\begin{align*}
	W_k(\la,\la+\de_j\,|\,\bar\nu)=
	a_k\cdot S_j(\bar\nu,\la),
\end{align*}
where $S_j$ is given by \eqref{S_j}. 
(The diagonal values of $W_k$ are of course
defined by \eqref{spec_W_conditions_P1'}.) 

Moreover, one can also check that 
dynamics of Definition \ref{def:DF_continuous}
\emph{coincides} with the continuous-time 
dynamics on 
interlacing arrays introduced in 
\cite[\S2.3.3]{BorodinCorwin2011Macdonald}.
In probabilistic terms, the process on interlacing arrays
is described as follows:
\begin{dynamics}[push-block dynamics]
\label{dyn:DF_Mac}
	{\ }\begin{enumerate}[(1)]
		\item (independent jumps)
		Each particle $\la^{(k)}_{j}$ at each level $k=1,\ldots,N$ 
		has its own independent exponential clock with rate
		$a_k S_j(\la^{(k-1)},\la^{(k)})$. 
		When the $\la^{(k)}_j$th clock rings, 
		the particle $\la^{(k)}_j$ 
		jumps to the right by one
		if it is not blocked by a lower particle, 
		i.e., if $\la^{(k)}_j<\la^{(k-1)}_{j-1}$.
		If the particle $\la^{(k)}_{j}$ 
		is blocked, no jump occurs 
		because then
		$S_j(\la^{(k-1)},\la^{(k)})=0$.
		\item (triggered moves)
		When any particle 
		$\la_{j}^{(k-1)}$, ($k=2,\ldots,N$, $j=1,\ldots,k$), 
		moves to
		the right by one,
		it short-range pushes $\la^{(k)}_{j}$ according to the
		rule of \S
		\ref{sub:when_a_jumping_particle_has_to_push_its_immediate_upper_right_neighbor}. 
		That is,
		if $\la^{(k-1)}_{j}=\la^{(k)}_{j}$,
		then $\la^{(k)}_{j}$ also jumps to the right by one.
		No other triggered moves occur.
	\end{enumerate}
\end{dynamics}

One of the features of this dynamics is that
in the $q$-Whittaker case (i.e.,~when $t=0$), 
the stochastic evolution
of the leftmost particles 
$\la^{(k)}_{k}$, $k=1,\ldots,N$,
of the interlacing array
(cf. Fig.~\ref{fig:la_interlacing})
is Markovian.
This dynamics is called $q$-TASEP.
We briefly discuss it in \S \ref{sub:taseps} below.

In the Schur degeneration ($q=t$), one has 
$S_j(\bar\nu,\la)=
1_{\bar\nu\prec\la+\de_j}1_{\la\prec\la+\de_j}$, so each 
particle on level $k$ has an independent exponential 
clock with rate $a_k$ if it is not blocked. 
In this case the dynamics
of the rightmost particles 
$\la^{(k)}_{1}$, $k=1,\ldots,N$,
is also 
Markovian, it may be 
called PushTASEP (or long-range TASEP). See \S\ref{sub:taseps}
below and also \cite{BorFerr08push}, \cite{BorFerr2008DF}.


\subsection{Example: O'Connell--Pei's randomized insertion algorithm} 
\label{sub:o_connell_pei_s_process_on_interlacing_arrays}

Another multivariate dynamics 
on interlacing arrays
was introduced recently in \cite{OConnellPei2012}
in the $q$-Whittaker ($t=0$) case (we also
discuss its modification in \S \ref{ssub:remark_a_nearest_neighbor_markov_process_inspired_by_dynamics_dyn:oconnell} below). 
Let us rewrite the original definition 
(given in the language of semistandard
Young tableaux, cf. \S \ref{sub:semistandard_young_tableaux})
in terms of interlacing arrays. 
See also 
\S \ref{sub:from_interlacing_arrays_to_semistandard_tableaux}
below
for a ``dictionary''
between the two ``languages''.

Denote for
all $\bar\nu\in\GT_{k-1}$ and
$\la\in\GT_{k}$
\begin{align}\label{F_big}
	F_i(\bar\nu,\la):=&\begin{cases}
		0,&i=1;\\
		q^{\bar\nu_{i-1}-\la_i}
		,&2\le i\le k,
	\end{cases}
	\\
	\label{f_small}
	f_i(\bar\nu,\la)
	:=&
	\begin{cases}
		1,&i=1\\
		({1-q^{\bar\nu_{i-1}-\la_i}})/
		({1-q^{\bar\nu_{i-1}-\bar\nu_i+1}}),
		&
		2\le i\le k-1,\\
		1-q^{\bar\nu_{k-1}-\la_k}
		,&i=k.
	\end{cases}
\end{align}

Next, for all 
$k=1,\ldots,N$ 
and all 
$(\bar\nu,\la)\in\GT_{(k-1;k)}$, 
let the 
rates of independent jumps be
($j=1,\ldots,k$)
\begin{align}\label{W_k_OConnell}
	W_k(\la,\la+\de_j\,|\,\bar\nu):=
	a_k
	\big(1-F_j(\bar\nu,\la)\big)
	\prod_{r=j+1}^{k}
	F_r(\bar\nu,\la),
\end{align}
and
the probabilities
of triggered moves be
($j=1,\ldots,k$, $i=1,\ldots,k-1$):
\begin{align}&
	\label{V_k_OConnell}
	V_k(\la,\la+\de_j\,|\,\bar\nu-\bar\de_i,\bar\nu)
	:=\begin{cases}
		f_i(\bar\nu,\la),& j=i;\\
		\big(1-F_j(\bar\nu,\la)\big)
		\big(1-f_i(\bar\nu,\la)\big)
		\prod_{r=j+1}^{i-1}
		F_r(\bar\nu,\la)
		,& j<i;\\
		0,& j>i,
	\end{cases}
\end{align}
Using
the first relation in 
\eqref{spec_V_conditions_P1'}, 
one can readily see
from \eqref{V_k_OConnell}
that 
the probability
that a move does not propagate 
to a higher level is zero, i.e.,
$V_k(\la,\la\,|\,
\bar\nu-\bar\de_i,\bar\nu)=0$.
We also impose the second relation
in \eqref{spec_V_conditions_P1'}
on the~$V_k$'s.
Moreover, we define
the diagonal 
jump rates $W_k(\la,\la\,|\,\bar\nu)$ by 
\eqref{spec_W_conditions_P1'} (in fact, 
they are equal to $(-a_k)$). 
Finally, observe that 
if $\bar\nu\not\prec\la$, formula
\eqref{V_k_OConnell} is equivalent
to 
\eqref{V_k_blocking} 
because in this case $\bar\nu_i=\la_i+1$ for a 
suitable $i$, and it must be $j=i$ (cf. 
\S \ref{sub:when_a_jumping_particle_has_to_push_its_immediate_upper_right_neighbor}).

\begin{proposition}\label{prop:OConnell}
	Jump rates \eqref{W_k_OConnell} and probabilities
	of triggered moves \eqref{V_k_OConnell} define a 
	multivariate continuous-time dynamics
	on interlacing arrays in the sense of 
	Definition \ref{def:multivariate_GT}.
\end{proposition}

\begin{proof}
	This proposition was proved (in another form) 
	in \cite{OConnellPei2012}. 
	However, 
	let us redo the 
	necessary computations
	here in order to
	demonstrate 
	the application of 
	our general Theorem \ref{thm:general_multivariate}.

	In the $t=0$ case, the quantities
	$T_i$ and $S_j$ are given by 
	\eqref{T_S_Whittaker}--\eqref{T_1S_1} below.
	It suffices to check
	\eqref{P2'_T_S}.
	The fact that all
	other identities 
	\eqref{spec_V_conditions_P1'}, 
	\eqref{spec_W_conditions_P1'},
	and \eqref{V_k_blocking}
	hold follows from the discussion after 
	\eqref{W_k_OConnell}--\eqref{V_k_OConnell}.
	
	Fix any $(\bar\nu,\la)\in\GT_{(k-1;k)}$ 
	any $j=1,\ldots,k$.
	Identity \eqref{P2'_T_S} for $j=k$ reads
	\begin{align*}
		a_k^{-1}W_k(\la,\la+\de_k\,|\,\bar\nu)
		=S_k(\bar\nu,\la),
	\end{align*}
	which is obvious due to \eqref{W_k_OConnell} 
	and \eqref{T_1S_1}. For $1\le j\le k-1$, 
	we must show that (omitting the notation
	$(\bar\nu,\la)$):
	\begin{align}\label{Oconnell_514_identity}
		f_jT_j+
		(1-F_j)F_{j+1}\ldots F_{k}
		+
		\sum_{i=j+1}^{k-1}
		(1-F_j)F_{j+1}\ldots F_{i-1}(1-f_i)T_i
		=S_j.
	\end{align}
	It is straightforward to check that 
	$(1-f_{m-1})T_{m-1}=F_{m-1}(1-F_{m})$ for $m=2,\ldots,k-1$.
	This means in particular that the last summand
	in the above sum (corresponding to $i=k-1$)
	turns into 
	$(1-F_j)F_{j+1}\ldots F_{k-1}(1-F_k)$, and we get
	\begin{align*}
		(1-F_j)F_{j+1}\ldots F_{k-1}(1-F_k)+
		(1-F_j)F_{j+1}\ldots F_{k}=
		(1-F_j)F_{j+1}\ldots F_{k-1}.
	\end{align*}
	Using a similar argument for summands
	corresponding to $i=k-2,\ldots,j+1$, we 
	reduce
	the desired identity to
	\begin{align}\label{OConnell_proof_collapsing}
		f_jT_j+(1-F_{j})F_{j+1}=S_j.
	\end{align}
	To check that this holds for $j=1,\ldots,k-1$ 
	is also straightforward.
\end{proof}

The stochastic evolution 
introduced in \cite{OConnellPei2012} (see \S5 of that paper)
is defined in discrete time, 
but at every discrete time moment only one independent 
jump is allowed. 
Equivalently, 
we may define
this
model
in continuous time
by requiring that
time moments of jumps 
follow a constant-rate
Poisson process.
In this interpretation
one can check that the 
dynamics of \cite{OConnellPei2012}
coincides with the multivariate
dynamics of
Proposition \ref{prop:OConnell}.
Probabilistically this 
multivariate dynamics is described as follows:
\begin{dynamics}[Dynamics driven by O'Connell--Pei's insertion algorithm]
	\label{dyn:OConnell}
	{\ }\begin{enumerate}[(1)]
		\item (independent jumps)
		Each particle $\la^{(k)}_j$ has an independent exponential clock with rate \eqref{W_k_OConnell}. When this clock rings, the particle jumps to the right by one if $\la^{(k)}_{j}<\la^{(k)}_{j-1}$. Note that the blocking by $\la^{(k-1)}_{j-1}$ is implicitly present in \eqref{W_k_OConnell}, i.e., the jump rate is zero if $\la^{(k)}_{j}=\la^{(k-1)}_{j-1}$.
		\item (triggered moves)
		When any particle 
		$\la_{j}^{(k-1)}$, ($k=2,\ldots,N$, $j=1,\ldots,k$), 
		moves to
		the right by one,
		it instantly forces exactly 
		one of its upper right neighbors
		$\la^{(k)}_{r}$, $r\le j$, to 
		move to the right as well.
		Which of these neighbors
		jumps is determined by probabilities \eqref{V_k_OConnell},
		where $\la=\la^{(k)}$ and 
		$\bar\nu-\bar\de_j=\la^{(k-1)}$.

		Note that when $\la^{(k)}_{r}=\la^{(k-1)}_{r-1}$ (that is, 
		when the pushed particle $\la^{(k)}_{r}$ is blocked and cannot move 
		to the right), the corresponding probability $V_k$ 
		vanishes, as it should be. Note also that the 
		short-range pushing
		mechanism 
		(\S \ref{sub:when_a_jumping_particle_has_to_push_its_immediate_upper_right_neighbor})
		is
		built into the definition of $V_k$ 
		(see the discussion after 
		\eqref{V_k_OConnell}).
	\end{enumerate}
\end{dynamics}

Let us point out a few
differences between 
the two examples above,
Dynamics \ref{dyn:DF_Mac} and 
\ref{dyn:OConnell}.

In Dynamics \ref{dyn:OConnell} 
(in contrast with Dynamics \ref{dyn:DF_Mac}) 
a moving
particle $\la^{(k)}_j$ \emph{always} pushes 
some particle at the level $k$.
Thus, an independent jump 
happening at any level $k=1,\ldots,N$
in the interlacing
array (see Fig.~\ref{fig:la_interlacing}) 
always propagates to all higher levels $k+1,\ldots,N$.

As a consequence, 
the 
pushing mechanism in Dynamics \ref{dyn:OConnell}
works at long distances (in contrast with
the short-range pushing of Dynamics \ref{dyn:DF_Mac}). 
That is, in Dynamics \ref{dyn:OConnell}
a moving
particle $\la^{(k-1)}_{j}$
can push some particle 
$\la^{(k)}_{r}$ on the upper level
regardless of the distance $|\la^{(k)}_{r}-\la^{(k-1)}_{j}|$
between them. We will call this type of pushing
the \emph{long-range pushing}.



\section{Nearest neighbor dynamics} 
\label{sec:nearest_neighbor_multivariate_continuous_time_dynamics_on_interlacing_arrays}

From now on we will focus on 
a subclass of multivariate dynamics 
which we call \emph{nearest neighbor dynamics}.
Generally speaking, in such dynamics
a particle $\la^{(k-1)}_{j}$
moving by one to the right 
can affect
only its \emph{immediate} left or right neighbor
at level $k$.

In \S\S 
\ref{sub:notation_and_definition}--\ref{sub:fundamental_solutions} we fix $k=2,\ldots,N$, and restrict our attention
to the slice $\GT_{(k-1;k)}$ formed by rows $k-1$ and $k$
of the interlacing array 
(see~Fig.~\ref{fig:la_interlacing}), 
cf.~Remark~\ref{rmk:stacking}. 
We will discuss rates of
independent jumps $W_k$
and probabilities of triggered moves $V_k$ 
corresponding to 
the $k$-th slice
of a
nearest neighbor dynamics. 
Then in \S \ref{sub:fundamental_dynamics_}
and especially in 
\S \ref{sub:characterization_of_nearest_neighbor_dynamics_}
we will ``stack'' the functions $(W_n,V_n)$
for different $n$, i.e., consider nearest neighbor 
multivariate dynamics on interlacing arrays
whose Markov generators are expressed through the
$(W_n,V_n)$'s by 
\eqref{spec_bq_N_product_form}.

\subsection{Notation and definition} 
\label{sub:notation_and_definition}

We will use our usual notation explained
in the beginning of 
\S \ref{sub:specialization_of_formulas_from_s_sub:sequential_update_dynamics_in_continuous_time}. 
In particular,
$\bar\nu\in\GT_{k-1}$ will always denote the new state 
at level $k-1$, and $\la\in\GT_{k}$
will mean the old state at level $k$.

\begin{figure}[htbp]
	\begin{center}
		\includegraphics[width=300pt]{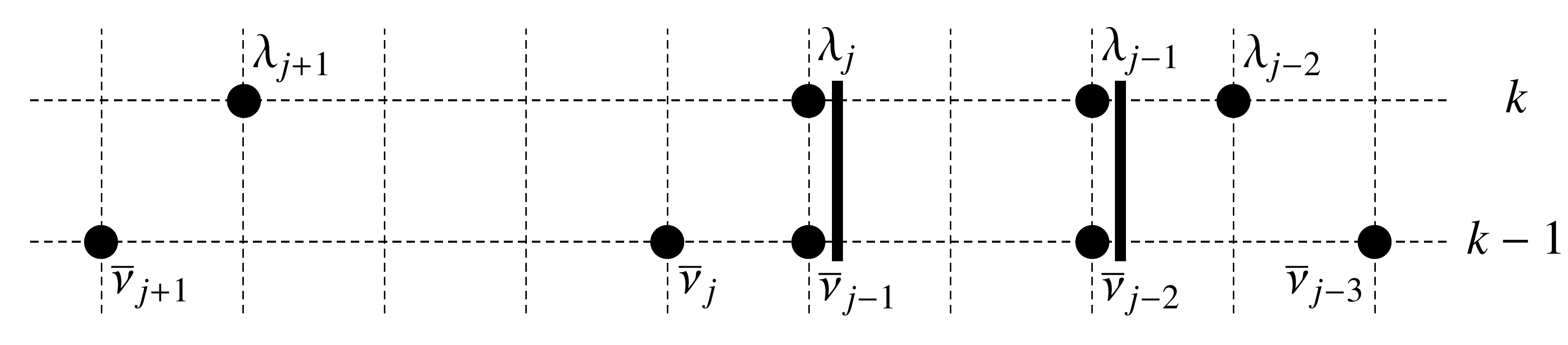}
	\end{center}  
  	\caption{Particles $\la_{j}$ and $\la_{j-1}$
  	are blocked and cannot move to the right.}
  	\label{fig:i_notation}
\end{figure}

Consider a set of indices 
\begin{align}\label{II}
	\II=\II(\bar\nu,\la):=
	\{j\colon \mbox{$1\le j\le k$, and 
	$\la_{j}<\bar\nu_{j-1}$}\}.
\end{align}
In words, $\II(\bar\nu,\la)$
represents indices of particles
at level $k$ that are \emph{free} (not blocked by 
a particle at level $k-1$)
and can move to the right by one.
For example, on Fig.~\ref{fig:i_notation} 
we have $j,j-1\notin\II$, and $j+1,j-2\in\II$. 
Equivalently, $\bar\nu\prec\la+\de_j
\Leftrightarrow \bar\nu-\bar\de_{j-1}\prec\la$ 
iff $j\in\II$.
Let us also denote by $\ii=\ii(\bar\nu,\la)\le k$
the cardinality of $\II(\bar\nu,\la)$.

For every $i=1,\ldots,k$, let 
\begin{align}\label{nf_next_free_particle}
	\nf(i):=\max\{j\colon 
	\mbox{$j\le i$ and $j\in\II(\bar\nu,\la)$}\}
\end{align}
be the index of the 
first right neighbor of $\la_i$ (including $\la_i$)
at level $k$ that is free.
On Fig.~\ref{fig:i_notation}  we have 
$\nf(j+1)=j+1$, and $\nf(j)=\nf(j-1)=\nf(j-2)=j-2$.

\begin{figure}[htbp]
	\begin{center}
		\includegraphics[width=300pt]
		{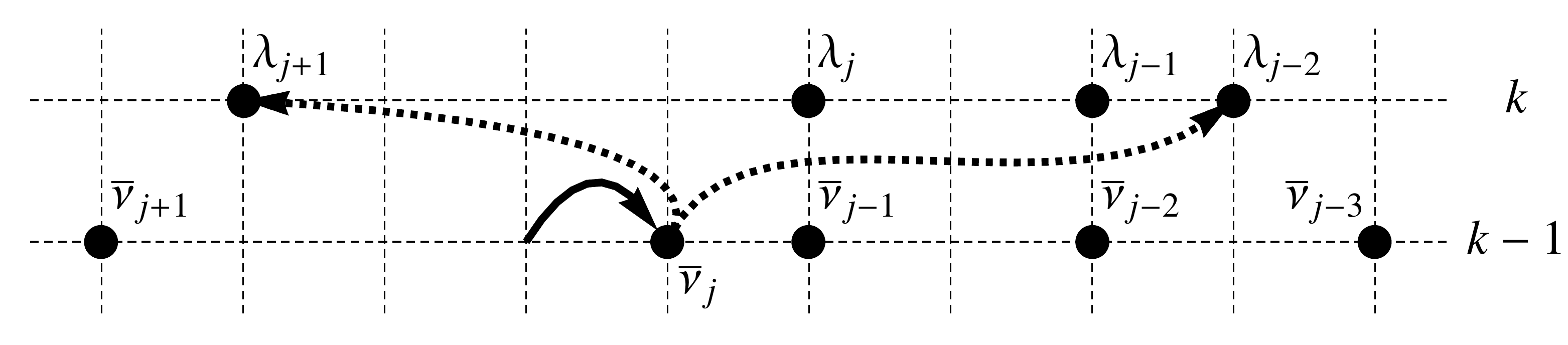}
	\end{center}  
  	\caption{The moved particle $\bar\nu_{j}$
  	can (long-range) \emph{pull} its 
  	immediate left neighbor $\la_{j+1}$
  	or long-range \emph{push} its first free right neighbor $\la_{\nf(j)}$.}
  	\label{fig:nn_dynamics_def}
\end{figure}

\begin{definition}\label{def:nn_dynamics}
	A multivariate dynamics with Markov generator
	$\bq^{(N)}$ which is
	determined 
	(as in \S \ref{sec:multivariate_continuous_time_dynamics_on_interlacing_arrays_}) 
	by functions 
	$(W_n,V_n)$, $n=1,\ldots,N$,
	will be called \emph{nearest neighbor} if 
	for
	all $k=2,\ldots,N$ and 
	any fixed
	$\bar\nu\in\GT_{k-1}$ and $\la\in\GT_{k}$
	with $\bar\nu\prec\la$
	the following condition holds:
	\begin{quote}
		For every $j=1,\ldots,k-1$ such that
		$\la_{j+1}<\bar\nu_{j}$
		(equivalently,
		$j+1\in\II(\bar\nu,\la)$), 
		and for all $m=1,\ldots,k$, 
		we have
		$V_k(\la,\la+\de_m\,|\,
		\bar\nu-\bar\de_j,\bar\nu)
		=0$
		unless $m=j+1$ or 
		$m=\nf(j)$.
	\end{quote}
	In words (see also
	Fig.~\ref{fig:nn_dynamics_def}), 
	let us assume that $\bar\nu_j$
	represents the particle that 
	has just moved\footnote{We need 
	the condition $\la_{j+1}<\bar\nu_{j}$
	in Definition \ref{def:nn_dynamics} to ensure
	that $\bar\nu_j$ indeed can be the coordinate
	of a particle that has just moved.} at level $k-1$
	(independently or due to a triggered move).
	Then this particle has a possibility 
	(i.e., with some probabilities)
	to either (long-range) \emph{pull} its immediate left neighbor
	$\la_{j+1}$ or, alternatively, (long-range) \emph{push} its first
	right neighbor $\la_{\nf(j)}$
	that is not blocked.
	The pulled or pushed particle 
	will instantly move to the right by one.

	Typically, when all particles at levels 
	$k-1$ and $k$ are ``apart'' (i.e., when 
	$\ii(\bar\nu,\la)=k$), this 
	means that a moved
	particle can affect 
	only its immediate left or right neighbor, and this interaction
	may happen at long range.
\end{definition}


\subsection{Parametrization and linear equations} 
\label{sub:parametrization_and_linear_equations}
According to Definition \ref{def:nn_dynamics},
let us denote 
\begin{align}\label{l_r_V}
	\begin{array}{rclcl}
		l_j&=&l_j(\bar\nu,\la)&:=&
		V_k(\la,\la+\de_{j+1}\,|\,\bar\nu-\bar\de_j,\bar\nu),
		\\
		r_j&=&r_j(\bar\nu,\la)&:=&
		V_k(\la,\la+\de_{\nf(j)}\,|\,\bar\nu-\bar\de_j,\bar\nu)
	\end{array}
\end{align}
for all $j=1,\ldots,k-1$ such that  
$j+1\in\II(\bar\nu,\la)$.

That is, 
$r_j$ and $l_j$ are the probabilities
that a moved particle $\bar\nu_j$
will long-range push (resp. pull) its right (resp. left) neighbor. 
We also leave open the possibility
that (with probability $1-r_j-l_j$)
this move at level $k-1$
does not propagate to the next level $k$.
We will call 
\begin{align}\label{c_propagation}
	c_j=c_j(\bar\nu,\la):=r_j(\bar\nu,\la)+l_j(\bar\nu,\la),
	\qquad j+1\in\II,
\end{align}
the \emph{probability of propagation} of the 
corresponding move 
$\bar\nu-\bar\de_j\to\bar\nu$
on level $k-1$
to the next level $k$.

Let us also introduce a shorthand for the rates 
of independent jumps (at level $k$)
of the particles that actually can jump:
\begin{align}\label{w_m_shorthand}
	w_m=w_m(\bar\nu,\la):=a_k^{-1}W_k(\la,\la+\de_m\,|\,\bar\nu),
	\qquad
	m\in\II(\bar\nu,\la).
\end{align}

With the help of 
the short-range pushing rule
of \S \ref{sub:when_a_jumping_particle_has_to_push_its_immediate_upper_right_neighbor}
(see also Remark \ref{rmk:W_k_nu_always_prec_la})
which dictates the values 
of $V_k(\la,\cdot\,|\,\cdot,\bar\nu)$
for $\bar\nu\not\prec\la$, 
we see that a multivariate nearest neighbor dynamics
is completely determined 
(on the slice $\GT_{(k-1;k)}$)
by the following parameters:
\begin{align}
	\label{rlw_parameters}
	\big\{r_j(\bar\nu,\la),c_j(\bar\nu,\la),
	w_m(\bar\nu,\la)\colon
	\mbox{$(\bar\nu,\la)\in\GT_{(k-1;k)}$ and
	$m,j+1\in\II(\bar\nu,\la)$}
	\big\}.
\end{align}
For future convenience,
let us enumerate 
the set $\II(\bar\nu,\la)$ as follows:
\begin{align}\label{i_enumeration}
	\II(\bar\nu,\la)=:
	\{j_1+1<j_2+1<\ldots<j_{\ii(\bar\nu,\la)}+1\}.
\end{align}
Note that always $j_1=0$ because the first particle at 
level $k$
cannot be blocked.

One can write down a system of linear equations
for parameters \eqref{rlw_parameters}:

\begin{proposition}\label{prop:rlw_system}
	Parameters \eqref{rlw_parameters}
	on the slice $\GT_{(k-1;k)}$
	correspond to a multivariate dynamics
	if and only if 
	for any 
	$(\bar\nu,\la)\in\GT_{(k-1;k)}$
	they satisfy the 
	following system of linear
	equations 
	(we 
	use enumeration~\eqref{i_enumeration},
	and
	omit the dependence 
	on~$(\bar\nu,\la)$ in the notation):
	\begin{align}\label{rlw_system_of_equations}
		\left\{
		\begin{array}{rclclc}
			r_{j_2}T_{j_2}&+&w_{j_1+1}&=&S_{j_1+1};
			\\
			r_{j_{m+1}}T_{j_{m+1}}
			+
			(c_{j_m}-r_{j_m})T_{j_{m}}&+&w_{j_m+1}&=&S_{j_m+1},&
			\ 
			m=2,\ldots,\ii-1;\\
			(c_{j_{\ii}}-r_{j_{\ii}})
			T_{j_{\ii}}&+&w_{j_{\ii}+1}
			&=&S_{j_{\ii}+1}.
		\end{array}
		\right.
	\end{align}
	Here the coefficients $T_{j_m}(\bar\nu,\la)$ and
	$S_{j_m+1}(\bar\nu,\la)$ are defined in 
	\eqref{T_i}--\eqref{S_j} 
	(see also \eqref{T_S_Whittaker}--\eqref{T_1S_1}
	for the $t=0$ 
	specialization).\footnote{Observe 
	that all 
	$T_{j_m}$'s and $S_{j_m+1}$'s are strictly
	positive. See also \eqref{T_S_zero} below.}
\end{proposition}
\begin{proof}
	Immediately follows
	from claim~2 of Theorem \ref{thm:general_multivariate}.
	The $m$th equation above (where $m=1,\ldots,\ii$)
	corresponds to 
	writing \eqref{P2'_T_S}
	for 
	one of the $\ii$ allowed
	moves $\la\to\la+\de_{j_{m}+1}$
	at level $k$.
\end{proof}

\begin{remark}\label{rmk:algebraic_rlw}
	Due to the nature of 
	Proposition \ref{prop:rlw_system},
	it is natural to
	work in a linear-algebraic setting without
	assuming nonnegativity
	of probabilities or jump rates
	(cf. Remark \ref{rmk:algebraic_sense}).
	Namely, we will say that arbitrary
	real parameters~\eqref{rlw_parameters}
	\emph{define a multivariate `dynamics'}
	(in quotation marks)
	on the slice $\GT_{(k-1;k)}$
	if they satisfy
	Proposition \ref{prop:rlw_system}.
	If, in addition, 
	one has 
	$w_m\ge0$
	and $0\le r_j\le c_j\le 1$ for all possible 
	$j$ and $m$, then
	these parameters
	clearly correspond
	to an honest multivariate 
	continuous-time Markov dynamics on interlacing arrays.

	However, to make the discussion more
	understandable, we will use probabilistic 
	language
	even when speaking
	about multivariate `dynamics' 
	which may have negative jump rates or
	probabilities of triggered moves.
	In particular, we will 
	speak about 
	the jump 
	rate matrix
	(= generator)
	$\bq^{(N)}$
	of a multivariate `dynamics' in the same sense as in 
	Remark \ref{rmk:algebraic_sense}. 
\end{remark}

Thus, for each pair of signatures $\bar\nu,\la$
with $\bar\nu\prec\la$, we have 
\begin{align*}
	2\big(\ii(\bar\nu,\la)-1\big)+
	\ii(\bar\nu,\la)=
	3\ii(\bar\nu,\la)-2
\end{align*}
real parameters \eqref{rlw_parameters}
of our `dynamics'. These parameters
must satisfy 
$\ii(\bar\nu,\la)$ linear equations 
corresponding to each free particle
at level $k$
(Proposition~\ref{prop:rlw_system}). 
Therefore, there should be $2\ii(\bar\nu,\la)-2$ 
independent parameters involved in the description of 
the solution of each such system.

In the 
rest of the section
we describe building blocks
of solutions of 
\eqref{rlw_system_of_equations}, namely,
the so-called \emph{fundamental solutions} 
(\S \ref{sub:fundamental_solutions}).
An arbitrary solution of 
the linear system~\eqref{rlw_system_of_equations}
will be a certain linear combination of 
the fundamental solutions 
(\S
\ref{ssub:arbitrary_solutions_as_linear_combinations_of_fundamental_ones}).

After that we will show how to organize fundamental
solutions corresponding
to systems \eqref{rlw_system_of_equations} with
various $\bar\nu,\la$ 
into certain special 
multivariate `dynamics' (\emph{fundamental `dynamics'},
\S \ref{sub:fundamental_dynamics_})
which in turn serve as building blocks
for arbitrary nearest neighbor 
multivariate `dynamics' on interlacing arrays
(\S \ref{sub:characterization_of_nearest_neighbor_dynamics_}).


\subsection{General solution of the system} 
\label{sub:general_solution_of_the_system}

Observe that
the sum of all equations
\eqref{rlw_system_of_equations} reads
\begin{align}\label{sum_of_all_equations}
	\sum_{m=2}^{\ii}
	T_{j_m}c_{j_m}
	+
	\sum_{m=1}^{\ii}w_{j_m+1}=
	\sum_{m=1}^{\ii}S_{j_m+1}.
\end{align}
Next, note that by \eqref{T_i}--\eqref{S_j} one has
\begin{align}\label{T_S_zero}
	\begin{array}{cc}
		S_j(\bar\nu,\la)=T_{j-1}(\bar\nu,\la)=0
		&\quad
		\mbox{if $j\notin\II(\bar\nu,\la)$};\\
		S_j(\bar\nu,\la),\;T_{j-1}(\bar\nu,\la)>0
		&\quad
		\mbox{if $j\in\II(\bar\nu,\la)$}.
	\end{array}
\end{align}
Thus,
keeping only nonzero terms, we rewrite
the 
commutation relation of Proposition 
\ref{prop:jump_rate_Q_k_Macdonald_commute}
as 
\begin{align}\label{1_T_S}
	1+\sum_{m=2}^{\ii}
	T_{j_m}=
	\sum_{m=1}^{\ii}S_{j_m+1}.
\end{align}
Combining this with \eqref{sum_of_all_equations}, 
we get
the following relation between 
the $w_{j_m+1}$'s and the $c_{j_m}$'s:
\begin{align}\label{la=nu_NN_identity}
	\sum_{m=1}^{\ii}w_{j_m+1}
	=1+
	\sum_{m=2}^{\ii}
	T_{j_m}(1-c_{j_m}).
\end{align}
This is equivalent to the identity 
in Remark \ref{rmk:la=nu}, as it 
should be. Thus, we see that \eqref{la=nu_NN_identity}
follows from our system of equations
\eqref{rlw_system_of_equations}. 

Let us treat $w_{j_1+1},\ldots,w_{j_{\ii}+1}$
and $c_{j_2},\ldots,c_{j_{\ii}}$ as parameters of a 
solution of the system \eqref{rlw_system_of_equations}. 
There is only one dependence 
between these parameters, namely, the above 
linear identity
\eqref{la=nu_NN_identity}. 
Then we have:
\begin{proposition}\label{prop:general_solution_r}
	Assume that $(\bar\nu,\la)\in\GT_{(k-1;k)}$ are fixed.
	For any given collection of $2\ii-1$ parameters
	$w_{j_1+1},\ldots,w_{j_{\ii}+1}$
	and $c_{j_2},\ldots,c_{j_{\ii}}$
	satisfying one linear relation \eqref{la=nu_NN_identity}, 
	there exists a unique solution 
	$r_{j_2},\ldots,r_{j_\ii}$
	of the system \eqref{rlw_system_of_equations} 
	given by
	\begin{align}
		\label{r_general_solution}
		r_{j_{m+1}}&=
		\frac{1}{T_{j_{m+1}}}
		\Big(
		S_{j_1+1}+\ldots+S_{j_m+1}
		-w_{j_1+1}-\ldots-w_{j_m+1}
		-c_{j_1}T_{j_1}
		-c_{j_2}T_{j_2}-
		\ldots
		-c_{j_{m}}T_{j_m}
		\Big),
	\end{align}
	$m=1,\ldots,\ii-1$,
	with the agreement $T_{j_1}=T_0\equiv0$.
\end{proposition}
\begin{proof}
	This can be readily verified
	by solving 
	equations \eqref{rlw_system_of_equations}
	one by one,
	because the matrix of the system 
	\eqref{rlw_system_of_equations}
	is two-diagonal.
	By \eqref{T_S_zero}, one can divide by 
	the quantities
	$T_{j_m}$
	because they are strictly positive.
	The fact that $\ii$
	equations in $\ii-1$ variables
	are consistent follows
	from~\eqref{la=nu_NN_identity}.
	The dimension of the 
	linear space of solutions of 
	\eqref{rlw_system_of_equations} 
	(in all
	$3\ii-2$ variables) is $2\ii-2$, 
	as it should be.
\end{proof}


\subsection{Fundamental solutions} 
\label{sub:fundamental_solutions}

In this subsection we will focus on one 
system \eqref{rlw_system_of_equations} corresponding
to some fixed pair of signatures $\bar\nu\prec\la$,
and 
will present certain distinguished \emph{fundamental
solutions} of this system.

\subsubsection{The push-block solution} 
\label{ssub:push_block_solution}

The first of our
fundamental solutions has
zero propagation
probabilities, i.e., 
$c_{j_2}^{\PBD}=\ldots=c_{j_\ii}^{\PBD}=0$.
In this case the only probabilistically 
meaningful solution must 
have $r^{\PBD}_{j_2}=\ldots=r^{\PBD}_{j_{\ii}}=0$
as well (otherwise there will be negative 
probabilities of triggered moves). Then 
from \eqref{rlw_system_of_equations}
we see that it must be 
$w_{j}^{\PBD}=S_{j}$ for all $j\in\II$.

This corresponds to the observation
made in 
\S \ref{sub:example_push_block_process_on_interlacing_arrays}
that there is a unique honest Markov multivariate dynamics
with zero propagation probability,\footnote{Of course, if $\bar\nu\not\prec\la$ with 
$\bar\nu-\bar\de_j\prec\la$,
then (by the 
short-range pushing rule of \S
\ref{sub:when_a_jumping_particle_has_to_push_its_immediate_upper_right_neighbor})
the move $\bar\nu-\bar\de_j\to\bar\nu$ 
at level $k-1$ must propagate
to the next level $k$ with probability one
(and not zero).}
namely, the push-block process (Dynamics~\ref{dyn:DF_Mac}) introduced
in \cite[\S2.3.3]{BorodinCorwin2011Macdonald}
(hence the letters ``$\PBD$'' in the notation).


\subsubsection{RSK-type fundamental solutions} 
\label{ssub:rs_type_solutions}

In the next family of $\ii$ fundamental solutions
all the
propagation
probabilities are equal to one.
Let us give the corresponding definition:
\begin{definition}\label{def:RS_type}
	A multivariate `dynamics'
	in which a
	move
	at any level $n-1$
	always 
	(i.e., with probability 1)
	propagates
	to the next level $n$ (where $n=2,\ldots,N$), 
	will be called a \emph{Robinson--Schensted--Knuth--type}
	(\emph{RSK-type})
	\emph{multivariate 
	`dynamics'}.\footnote{See also the
	beginning of \S \ref{sub:_boldsymbol_h_robinson_schensted_correspondences} for a discussion of 
	the ``RSK'' terminology.}
\end{definition}

This term is suggested by 
considering the Robinson--Schensted row
insertion algorithm restated in terms of interlacing arrays
(cf.~\S \ref{sub:semistandard_young_tableaux}). 
This algorithm (see \S \ref{sub:row_and_column_insertions} below 
for more detail) 
starts 
with an initial jump
at some level of the interlacing
array (as on Fig.~\ref{fig:la_interlacing}),
and then 
triggered moves
always propagate to all upper levels of the 
array (thus changing the shape of the Young diagram
$\la^{(N)}$). 
We discuss the row insertion,
as well as other
insertion
algorithms in detail
in \S \ref{sec:schur_degeneration_and_robinson_schensted_correspondences} below.

Note also that
Dynamics~\ref{dyn:OConnell}
discussed in 
\S \ref{sub:o_connell_pei_s_process_on_interlacing_arrays} is an RSK-type dynamics. However, that dynamics
is not nearest neighbor.

By the above definition, RSK-type solutions of 
\eqref{rlw_system_of_equations} must have
\begin{align*}
	c_{j_2}^{\RSD}=\ldots=c_{j_\ii}^{\RSD}=1.
\end{align*}
Then from \eqref{la=nu_NN_identity} we conclude that
\begin{align}\label{sum_of_w_is_one}
	w_{j_1+1}^{\RSD}+\ldots+w_{j_\ii+1}^{\RSD}=1.
\end{align}
The \emph{fundamental RSK-type solutions}, by definition, 
correspond to setting one of the $w_{j_m+1}$'s above to 
one, and all others to zero. That is, these fundamental solutions
are indexed by $h\in\II$, and are defined by
\begin{align}\label{w_j_RSh}
	w_{j}^{\RSD(h)}:=1_{j=h}
	\qquad\mbox{for all $j\in\II$}.
\end{align}
Then we have from 
Proposition \ref{prop:general_solution_r}:
\begin{align}\label{r_j_RSh}
	r_j^{\RSD(h)}=T_{j}^{-1}
	\big(
	S_1+\ldots+S_j-T_1- \ldots-T_{j-1}
	-1_{h\le j}
	\big),\quad j+1\in\II.
\end{align}
Here we used the fact that 
$j_m+1\in\II$, and $j_m+2,\ldots,j_{m+1}\notin\II$,
while $j_{m+1}+1\in\II$ (for any $m=1,\ldots,\ii-1$), 
together with \eqref{T_S_zero}, to write 
$r_j^{\RSD(h)}$ in a nicer form.


\subsubsection{Right-pushing fundamental solutions} 
\label{ssub:right_pushing_fundamental_solutions}

In \S \ref{ssub:push_block_solution} and 
\S \ref{ssub:rs_type_solutions} 
we considered solutions which have constant 
propagation probability. 
Here and in \S \ref{ssub:left_pushing_fundamental_solutions}
we assume, on the contrary,
that one of the propagation probabilities
$c_j$, $j+1\in\II$,
is equal to one, and all the other $c_j$'s are zero.

Let us first consider
such solutions which do not have 
the pulling mechanism 
(such as when $\bar\nu_j$ pulls $\la_{j+1}$ on 
Fig.~\ref{fig:nn_dynamics_def}).
We thus arrive at 
$\ii-1$ solutions which we call
the \emph{right-pushing fundamental solutions}.
They
are indexed by $h\in\{1,\ldots,k-1\}$ with $h+1\in\II$, and are defined by setting
\begin{align*}
	c_j^{\RD(h)}=r_j^{\RD(h)}:=1_{j=h}\qquad 
	\mbox{for all $j$ such that $j+1\in\II$}.
\end{align*}
System \eqref{rlw_system_of_equations} then implies
that for the right-pushing fundamental solutions we have
\begin{align*}
	w_{j_m+1}^{\RD(h)}=
	S_{j_m+1}-1_{h=j_{m+1}}T_{j_{m+1}}
	\qquad\mbox{for all $m=1,\ldots,\ii$}.
\end{align*}
In other words, 
we have
\begin{align*}
	w_{j}^{\RD(h)}=S_{j}\mbox{\; if $j\in\II$ and $j\ne\nf(h)$}, 
	\qquad
	w_{\nf(h)}^{\RD(h)}=
	S_{\nf(h)}-T_{h}.
\end{align*}
Here $\nf(h)$ is defined by \eqref{nf_next_free_particle}.
We can also write equivalently
\begin{align*}
	w_{j}^{\RD(h)}=S_{j}-1_{h=\nf^{-1}(j)}T_{h},
	\qquad j\in\II,
\end{align*}
with the understanding that $\nf^{-1}(j)$ 
denotes the unique index of a particle such that 
$\nf^{-1}(j)+1\in\II$, and $\nf(\nf^{-1}(j))=j$.


\subsubsection{Left-pulling fundamental solutions} 
\label{ssub:left_pushing_fundamental_solutions}

Let us now define
\emph{left-pulling fundamental solutions}
which are similar to the right-pushing
ones (\S \ref{ssub:right_pushing_fundamental_solutions}),
but they do not allow (long-range) pushes. That is,
the left-pulling fundamental solutions
are indexed by $h\in\{1,\ldots,k-1\}$
such that $h+1\in\II$, and are determined by
\begin{align*}
	c_j^{\LD(h)}:=1_{j=h}
	\mbox{\; and\; }
	r_j^{\LD(h)}:=0
	\qquad 
	\mbox{for all $j$ such that $j+1\in\II$}.
\end{align*}
Then system \eqref{rlw_system_of_equations} implies that
the jump rate variables have the following form:
\begin{align*}
	w_{j}^{\LD(h)}=S_j-1_{j=h+1}T_{h},\qquad
	j\in\II.
\end{align*}


\subsubsection{Arbitrary solutions as linear combinations of fundamental
solutions} 
\label{ssub:arbitrary_solutions_as_linear_combinations_of_fundamental_ones}

Above we have described $3\ii-1=1+\ii+(\ii-1)+(\ii-1)$ 
so-called fundamental solutions of our linear system
\eqref{rlw_system_of_equations}.
The dimension of the space of 
solutions of this system is $2\ii-2$ 
(see Proposition \ref{prop:general_solution_r}).
Our goal now is to explain
how an arbitrary solution of \eqref{rlw_system_of_equations} 
decomposes 
as a linear combination of the fundamental ones.
To make this decomposition unique 
(see also Remark \ref{rmk:dimension} below), 
it suffices to take 
two of the three families of fundamental solutions 
(RSK-type, right-pushing and left-pulling
solutions). When taking the last two families (i.e., 
without RSK-type solutions), 
one should also add the push-block solution 
to them.

\begin{proposition}[RSK-type and 
right-pushing]
\label{prop:RS_R}
	For level number $k\ge3$,
	any solution $(w,c,r)$ of \eqref{rlw_system_of_equations}
	can be uniquely decomposed in the following way:
	\begin{align}
		\begin{pmatrix}
			w\\c\\r
		\end{pmatrix}
		=\sum_{i=1}^{k-1}
		\co_{\RD(i)}
		\begin{pmatrix}
			w^{\RD(i)}\\c^{\RD(i)}\\r^{\RD(i)}
		\end{pmatrix}
		+
		\sum_{h=1}^{k}
		\co_{\RSD(h)}
		\begin{pmatrix}
			w^{\RSD(h)}\\c^{\RSD(h)}\\r^{\RSD(h)}
		\end{pmatrix},
		\label{RS_R_decomp}
	\end{align}
	where 
	\begin{align}\label{sum_RS_R_one}
		\sum_{i=1}^{k-1}\co_{\RD(i)}+
		\sum_{h=1}^{k}\co_{\RSD(h)}=1,
	\end{align}
	and, by 
	agreement\footnote{This agreement
	is used to simplify notation
	because fundamental
	solutions
	$(w^{\RSD(h)},c^{\RSD(h)},r^{\RSD(h)})$ 
	do not make sense if $h\notin\II$, and 
	similarly for $(w^{\RD(i)},c^{\RD(i)},r^{\RD(i)})$.}
	\begin{align*}
		\co_{\RD(i)}=0
		\mbox{\; if $i+1\notin\II$},
		\qquad
		\co_{\RSD(h)}=0
		\mbox{\; if $h\notin\II$}.
	\end{align*}
\end{proposition}
Here and below
we write identities like
\eqref{RS_R_decomp}
as shorthands
for three separate similar identities
for the $c_j$'s, $r_j$'s (where $j+1\in\II$), 
and $w_m$'s (where $m\in\II$), respectively.
These identities must hold for all such $j$ and~$m$.

\begin{proof}
	We need to find quantities $\co_{\RD(i)}$, $i+1\in\II$, and 
	$\co_{\RSD(h)}$, $h\in\II$, from 
	\begin{align}\label{RS_R_eqn}
		\begin{array}{rcll}
			c_j&=&\displaystyle
			\co_{\RD(j)}+\sum_{h=1}^{k}\co_{\RSD(h)},&\qquad
			j+1\in\II;\\
			w_m&=&\displaystyle
			\co_{\RSD(m)}-
			\co_{\RD(\nf^{-1}(m))}
			T_{\nf^{-1}(m)}
			+S_m\sum_{i=1}^{k-1}\co_{\RD(i)},&\qquad
			m\in\II.
		\end{array}
	\end{align}
	Then due to Proposition \ref{prop:general_solution_r}, the 
	$r$'s will also satisfy the linear 
	relations~\eqref{RS_R_decomp}.
	
	To simplify the notation, let us consider 
	the case $\ii=k$ (i.e., all particles
	at level $k$ can move to the right). The
	general case is similar.

	First, summing the $w_m$'s above and 
	using
	\eqref{1_T_S}--\eqref{la=nu_NN_identity}, 
	one can get \eqref{sum_RS_R_one}. Moreover, 
	\eqref{sum_RS_R_one} is in fact equivalent to
	\eqref{la=nu_NN_identity}. 
	
	In view of this fact, it is clear that the above equations 
	\eqref{RS_R_eqn}
	on the $\co$'s have a unique solution.
	Namely, we have 
	$c_j=\co_{\RD(j)}+1-\sum_{i=1}^{k-1}\co_{\RD(i)}$
	from \eqref{sum_RS_R_one} and 
	from the first equation in \eqref{RS_R_eqn}.
	This yields
	\begin{align*}
		\co_{\RD(i)}=c_i+\frac{1}{k-2}
		(1-c_1- \ldots-c_{k-1}),\qquad i=1,\ldots,k-1.
	\end{align*}
	Then one can find the $\co_{\RSD(h)}$'s from the 
	second equation in \eqref{RS_R_eqn}.
	This concludes the proof.
\end{proof}

The next two propositions are proved similarly to 
Proposition~\ref{prop:RS_R}:

\begin{proposition}[RSK-type and 
left-pulling]
\label{prop:RS_L}
	For level number $k\ge3$, 
	any solution $(w,c,r)$ 
	of \eqref{rlw_system_of_equations}
	can be uniquely decomposed in the following way:
	\begin{align}
		\begin{pmatrix}
			w\\c\\r
		\end{pmatrix}
		=\sum_{i=1}^{k-1}
		\co_{\LD(i)}
		\begin{pmatrix}
			w^{\LD(i)}\\c^{\LD(i)}\\r^{\LD(i)}
		\end{pmatrix}
		+
		\sum_{h=1}^{k}
		\co_{\RSD(h)}
		\begin{pmatrix}
			w^{\RSD(h)}\\c^{\RSD(h)}\\r^{\RSD(h)}
		\end{pmatrix},
		\label{RS_L_decomp}
	\end{align}
	where 
	$\sum_{i=1}^{k-1}\co_{\LD(i)}+
	\sum_{h=1}^{k}\co_{\RSD(h)}=1,$
	and, by agreement,
	\begin{align*}
		\co_{\LD(i)}=0
		\mbox{\; if $i+1\notin\II$},
		\qquad
		\co_{\RSD(h)}=0
		\mbox{\; if $h\notin\II$}.
	\end{align*}
\end{proposition}

\begin{proposition}[right-pushing and 
left-pulling]
\label{prop:R_L}
	For a level
	$k=2,\ldots,N$,
	any 
	solution $(w,c,r)$ of \eqref{rlw_system_of_equations}
	can be uniquely decomposed as follows:
	\begin{align}
		\begin{pmatrix}
			w\\c\\r
		\end{pmatrix}
		=
		\co_{\PBD}
		\begin{pmatrix}
			w^{\PBD}\\c^{\PBD}\\r^{\PBD}
		\end{pmatrix}
		+
		\sum_{i=1}^{k-1}
		\co_{\LD(i)}
		\begin{pmatrix}
			w^{\LD(i)}\\c^{\LD(i)}\\r^{\LD(i)}
		\end{pmatrix}
		+
		\sum_{i=1}^{k-1}
		\co_{\RD(i)}
		\begin{pmatrix}
			w^{\RD(i)}\\c^{\RD(i)}\\r^{\RD(i)}
		\end{pmatrix},
		\label{R_L_decomp}
	\end{align}
	where 
	$\co_{\PBD}+\sum_{i=1}^{k-1}\co_{\LD(i)}+
	\sum_{i=1}^{k-1}\co_{\RD(i)}=1,$
	and, by agreement,
	$\co_{\LD(i)}=\co_{\RD(i)}=0$
	if $i+1\notin\II$.
\end{proposition}

\begin{remark}\label{rmk:k=2_bad}
	In Propositions \ref{prop:RS_R}
	and \ref{prop:RS_L}
	we must assume that $k\ge3$ because
	if $k=2$, then the solutions
	$\RD(1)$ and $\LD(1)$
	belong to 
	the linear span of
	RSK-type fundamental solutions
	(and have propagation probability $c_1=1$).
	Thus, for $k=2$
	the 
	linear combinations 
	of Propositions
	\ref{prop:RS_R}
	and \ref{prop:RS_L}
	do not exhaust 
	all possible solutions
	of \eqref{rlw_system_of_equations}.
\end{remark}

Let us also consider a special case of solutions, 
namely, when the probability of propagation
$c_{j}$ (where $j+1\in\II$) does not depend on the particle that
has moved:
\begin{align}\label{constant_pp_C}
	c_{j_2}=c_{j_3}=\ldots=c_{j_{\ii}}=C.
\end{align}
In this case the natural fundamental solutions
to choose are the push-block solution and the 
RSK-type solutions which also have constant
propagation probabilities (0 and 1, respectively).

\begin{proposition}[constant propagation probability]
\label{prop:const_pp}
	Any solution $(w,r)$ of
	\eqref{rlw_system_of_equations},
	where the $c_j$'s satisfy
	\eqref{constant_pp_C}, can be 
	uniquely expressed as a linear combination
	of RSK-type and push-block fundamental solutions 
	in the following way:
	\begin{align}
	 	\begin{pmatrix}
			w\\r
		\end{pmatrix}
		=
		(1-C)
		\begin{pmatrix}
			w^{\PBD}\\r^{\PBD}
		\end{pmatrix}
		+
		\sum_{h=1}^{k}
		\co_{\RSD(h)}
		\begin{pmatrix}
			w^{\RSD(h)}\\r^{\RSD(h)}
		\end{pmatrix},
		\label{const_pp_decomp}
 	\end{align} 
	where $\sum_{h=1}^{k}\co_{\RSD(h)}=C$, 
	and, by agreement, $\co_{\RSD(h)}=0$
	if $h\notin\II$.
\end{proposition}
\begin{proof}
	Similar to the proof of Proposition
	\ref{prop:RS_R}.
	In fact, $\co_{\RSD(h)}$ in \eqref{const_pp_decomp}
	is equal to
	$\co_{\RSD(h)}=w_h-(1-C)S_h$
	for every $h\in\II$.
\end{proof}

\begin{remark}\label{rmk:dimension}
	In each of Propositions \ref{prop:RS_R}, \ref{prop:RS_L},
	\ref{prop:R_L},
	and \ref{prop:const_pp}
	we have taken one more fundamental solution than 
	the dimension $2\ii-2$ of the space of solutions.
	This is simply because our linear system
	\eqref{rlw_system_of_equations} is
	non-homogeneous, and its solutions form
	an affine subspace.
	This also results in the constraint on
	the $\co$'s in each of the
	propositions (i.e., the 
	sum of the corresponding
	coefficients $\co$ must be one).
\end{remark}



\subsection{Fundamental `dynamics'} 
\label{sub:fundamental_dynamics_}

\subsubsection{Motivation: 
fundamental solutions\\ and arbitrary 
nearest neighbor `dynamics'} 
\label{ssub:fundamental_solutions_and_arbitrary_nearest_neighbor_dynamics_}

Before going into details, let us present a brief 
overview of what we did 
in~\S\S \ref{sub:parametrization_and_linear_equations}--\ref{sub:fundamental_solutions},
and explain our further steps.

Let $(W_k,V_k)$ define a multivariate
`dynamics' on the slice $\GT_{(k-1;k)}$ (we assume 
that $k=2,\ldots,N$ is fixed). 
By the discussion in \S \ref{sub:parametrization_and_linear_equations}, 
these functions $(W_k,V_k)$ are 
in one-to-one correspondence
with 
the parameters $(w,c,r)$ as 
in~\eqref{l_r_V}--\eqref{rlw_parameters}.
These parameters
depend on a pair of signatures
$(\bar\nu,\la)\in\GT_{(k-1;k)}$,
and
satisfy the linear system~\eqref{rlw_system_of_equations}.
As we saw in 
\S \ref{ssub:arbitrary_solutions_as_linear_combinations_of_fundamental_ones}, 
for every $(\bar\nu,\la)$ the 
parameters $(w,c,r)$ can be expressed
as a linear combination of fundamental solutions
(in three ways, see
Propositions \ref{prop:RS_R}, 
\ref{prop:RS_L}, and \ref{prop:R_L}). 
In this sense, one can formulate the following property:
\begin{quote}
	An arbitrary nearest neighbor 
	`dynamics' on the slice $\GT_{(k-1;k)}$ 
	can be viewed as 
	a linear combination
	of fundamental solutions 
	(in one of the three ways,
	see Propositions \ref{prop:RS_R}, 
	\ref{prop:RS_L}, and \ref{prop:R_L})
	with coefficients $\co_{\cdots}(\bar\nu,\la)$ 
	of this linear combination
	which sum to one and 
	depend on 
	a pair of signatures $(\bar\nu,\la)\in\GT_{(k-1;k)}$.
\end{quote}

In this and the next subsection
our aim is to restate this 
property in a more convenient form
(i.e., in terms of
multivariate `dynamics'
on interlacing arrays).

First, we organize
the fundamental solutions
into certain natural fundamental
`dynamics'. These `dynamics'
have nice `probabilistic' descriptions,
but conceptually they are nothing more than 
unions of fundamental solutions described 
in \S \ref{sub:fundamental_solutions}.
Then in \S \ref{sub:characterization_of_nearest_neighbor_dynamics_} 
we show that the generator
of any nearest neighbor multivariate `dynamics'
can be expressed as a certain linear combination
of generators of these fundamental `dynamics'.
See Theorem \ref{thm:characterization_NN} for a final formulation.


\subsubsection{Setup and push-block dynamics} 
\label{ssub:setup}

Now let us start constructing
the fundamental `dynamics'.
In \S \ref{sub:fundamental_solutions}
we have discussed 
the push-block fundamental solution 
$(w^{\PBD},c^{\PBD},r^{\PBD})$,
and also defined
three distinguished 
families of fundamental 
solutions:
\begin{align*}&
	\begin{pmatrix}
		w^{\RSD(h)}\\c^{\RSD(h)}\\r^{\RSD(h)}
	\end{pmatrix}
	,\quad 
	h\in\II(\bar\nu,\la);
	\qquad
	\begin{pmatrix}
		w^{\RD(i)}\\c^{\RD(i)}\\r^{\RD(i)}
	\end{pmatrix}
	,\ 
	\begin{pmatrix}
		w^{\LD(i)}\\c^{\LD(i)}\\r^{\LD(i)}
	\end{pmatrix}
	,\quad 
	i+1\in\II(\bar\nu,\la),
\end{align*}
which are constructed 
for any fixed $(\bar\nu,\la)\in\GT_{(k-1;k)}$.

\begin{figure}[htbp]
	\begin{center}
		\includegraphics[width=300pt]{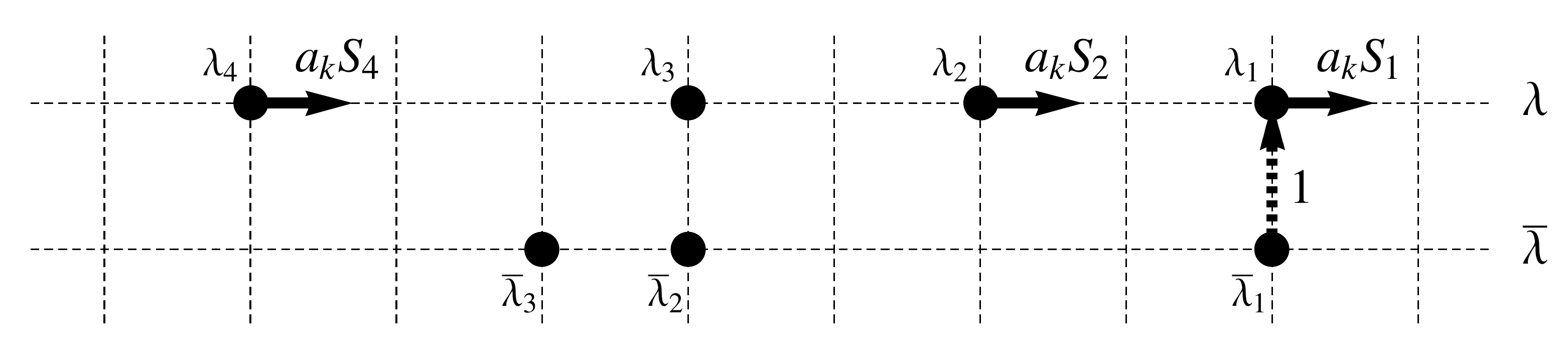}
	\end{center}  
  	\caption{Example of the behavior of the
	push-block fundamental dynamics 
  	$(W_k^{\PBD},V_k^{\PBD})$ for $k=4$ (see
  	Dynamics \ref{dyn:DF_Mac} for a
  	detailed description). 
  	The particle $\la_3$ is blocked
  	and has jump rate $a_kS_3(\bar\la,\la)=0$.
  	The jump rates of $\la_1,\la_2$, and $\la_4$
  	are shown on the picture.
  	The dashed arrow from $\bar\la_1$ to $\la_1$ denotes
  	short-range pushing: if $\bar\la_1$ moves,
  	then it pushes $\la_1$ with probability~1
  	(cf. Fig.~\ref{fig:nu_not_prec_la}
  	where the lower row is the new state~$\bar\nu$ 
  	instead of the old state~$\bar\la$ here). 
  	No
  	pushing occurs
  	if $\bar\la_2$ or $\bar\la_3$ move.}
  	\label{fig:DF_pic}
\end{figure}
Note that the push-block solutions already correspond to an 
existing multivariate Markov dynamics discussed in 
\S \ref{sub:example_push_block_process_on_interlacing_arrays}
(see Fig.~\ref{fig:DF_pic}). This dynamics
always has nonnegative 
jump rates $V_k^{\PBD}$ 
and probabilities of triggered moves $W_k^{\PBD}$.
Let us 
denote by $\bq^{(N)}_{\PBD}$ the 
generator of the push-block dynamics.

We will now focus on the three 
remaining families of fundamental solutions.
We want to naturally define 
for each pair $\bar\nu\prec\la$ the
rates of independent jumps
and the probabilities of triggered moves
\begin{align*}
	W_k^{\RSD(h)}(\la,\la+\de_m\,|\,\bar\nu),&
	\qquad
	m\in\II(\bar\nu,\la),\\
	V_k^{\RSD(h)}(\la,\la+\de_m\,|\,
	\bar\nu-\bar\de_j,\bar\nu),
	&\qquad
	m,j+1\in\II(\bar\nu,\la)
\end{align*}
(indexed by \emph{all} $h\in\{1,\ldots,k\}$),
and similarly for $\RD(i)$ and $\LD(i)$
(where $i=1,\ldots,k-1$), 
corresponding to fundamental 
solutions.\footnote{The quantities 
$W_k^{\RSD(h)}(\la,\cdot\,|\,\bar\nu)$ and
$V_k^{\RSD(h)}(\la,\cdot\,|\,\cdot,\bar\nu)$
for $\bar\nu\not\prec\la$ are always
dictated by the short-range pushing rule (see
\S \ref{sub:when_a_jumping_particle_has_to_push_its_immediate_upper_right_neighbor}), and similarly
for $\RD(i)$ and $\LD(i)$.}
We have to make sure that 
these $W_k$'s and $V_k$'s 
satisfy Theorem~\ref{thm:general_multivariate},
and thus they will produce the desired
\emph{fundamental multivariate `dynamics'}
on the slice $\GT_{(k-1;k)}$. Unifying 
all the slices, one can obtain 
(via \eqref{spec_bq_N_product_form})
the corresponding fundamental
`dynamics'
on interlacing arrays 
which serve as building blocks for 
all nearest neighbor `dynamics'
(\S \ref{sub:characterization_of_nearest_neighbor_dynamics_}).


\subsubsection{RSK-type fundamental `dynamics'} 
\label{ssub:rs_type_fundamental_dynamics_}

The RSK-type fundamental solutions 
(\S \ref{ssub:rs_type_solutions})
are indexed by $h\in\II(\bar\nu,\la)$
for each pair $\bar\nu\prec\la$. Thus, 
we can define
$W_k^{\RSD(h)}(\la,\cdot\,|\,\bar\nu)$ 
and $V_k^{\RSD(h)}(\la,\cdot\,|\,\cdot,\bar\nu)$ for
$h\in\II(\bar\nu,\la)$ 
simply by \eqref{l_r_V}--\eqref{w_m_shorthand}.
It remains to \emph{extend} the definition of
$W_k^{\RSD(h)}, V_k^{\RSD(h)}$ to the case 
$h\notin\II(\bar\nu,\la)$.

\begin{remark}\label{rmk:local_extension}
	We would like to perform 
	this extension (and a similar extension in 
	\S \ref{ssub:right_and_left_pushing_fundamental_dynamics_}) 
	in a ``local'' way. That is, for now 
	the action of `dynamics'
	$\RSD(h)$ on 
	$\GT_{(k-1;k)}$
	is defined on configurations
	in which the particle $\la_h$ is not 
	blocked (i.e., $h\in\II(\bar\nu,\la)$). 
	If this particle $\la_h$ becomes blocked,
	then this situation requires a special 
	treatment in the ``vicinity'' of $\la_h$. 
	But we would like the particles that are 
	far from $\la_h$ to not ``feel'' 
	that $\la_h$ became blocked.
\end{remark}

According to Remark \ref{rmk:local_extension},
observe that the quantities $r_j^{\RSD(h)}$, 
$j+1\in\II(\bar\nu,\la)$, 
are still well-defined by \eqref{r_j_RSh}
in the case $h\notin\II$ (in particular, $T_j$
is nonzero by \eqref{T_S_zero}). Thus, 
in this case 
we can put,
as before,
\begin{align}\label{Vk_RS_h}
	V_k^{\RSD(h)}(\la,\la+\de_m\,|\,
	\bar\nu-\bar\de_j,\bar\nu)=
	r_j^{\RSD(h)}(\bar\nu,\la)
	1_{m=\nf(j)}+
	\big(1-r_j^{\RSD(h)}(\bar\nu,\la)\big)
	1_{m=j+1}
\end{align}
for all meaningful $m$ and $j$
(the index $\nf^{-1}(m)$ is defined in 
\S \ref{ssub:right_pushing_fundamental_solutions}).

Then one can readily check that  
condition \eqref{P2'_T_S}
(or, equivalently, system~\eqref{rlw_system_of_equations}),
yields for $m\in\II(\bar\nu,\la)$:
\begin{align*}&
	W_k^{\RSD(h)}(\la,\la+\de_m\,|\,\bar\nu)
	=
	a_k\big(
	1_{h=m}+1_{h=m+1}+\ldots+1_{h=\nf^{-1}(m)}
	\big).
\end{align*}

Clearly, the above definitions
of $V_k^{\RSD(h)}$ and $W_k^{\RSD(h)}$
now work for any $h\in\{1,\ldots,k\}$.
Thus, for each $h$ we have constructed 
functions $(W_k^{\RSD(h)},V_k^{\RSD(h)})$
satisfying Theorem \ref{thm:general_multivariate},
and so we get $k$ different `dynamics' on the slice 
$\GT_{(k-1;k)}$.

The corresponding
`dynamics' on interlacing arrays
which are obtained by ``stacking'' the 
$(W_k^{\RSD(h))},V_k^{\RSD(h))})$'s
for all $k=2,\ldots,N$ 
using \eqref{spec_bq_N_product_form},
are para\-metrized by integer sequences 
\begin{align}\label{RS_h_index}
	\boldsymbol h:=
	(h^{(1)},h^{(2)}\ldots,h^{(N)}),\qquad
	1\le h^{(j)}\le j.
\end{align}
Let us denote 
by $\bq^{(N)}_{\RSD[\boldsymbol h]}$
the corresponding
generator defined by \eqref{spec_bq_N_product_form}
with $W_k=W_k^{\RSD(h^{(k)})}$
and $V_k=V_k^{\RSD(h^{(k)})}$, $k=2,\ldots,N$.
We will refer to these~$N!$ `dynamics' 
$\bq^{(N)}_{\RSD[\boldsymbol h]}$
on 
interlacing arrays
as to the \emph{fundamental RSK-type `dynamics'}. 
They can be `probabilistically' described
as follows:
\begin{dynamics}[RSK-type fundamental `dynamics'
$\bq^{(N)}_{\RSD[\boldsymbol h]}$]
\label{dyn:RS_type_fund}
	Let $k=1,\ldots,N$.
	To shorten the notation, denote $h=h^{(k)}$.
	{\ }\begin{enumerate}[(1)]
		\item (independent jumps)
		The only particle that can try to
		jump at level $k$ is $\la_{h}^{(k)}$. 
		It has the independent jump rate $a_k$.
		If this jump
		is blocked,
		i.e., if 
		$\la_{h}^{(k)}=\la_{h-1}^{(k-1)}$,
		then the first free
		particle $\la_{\nf(h)}^{(k)}$ 
		to the right of $\la_h^{(k)}$
		jumps instead.

		\item (triggered moves)
		If any particle $\la^{(k-1)}_{j}$
		moves to the right by one
        and $\la^{(k-1)}_{j}<\la^{(k)}_{j}$, 
        then this particle
        long-range
        pushes
        $\la_{\nf(j)}^{(k)}$
        or pulls $\la_{j+1}^{(k)}$
        with probabilities
        $1-r_j^{\RSD(h)}(\nu^{(k-1)},\la^{(k)})$ and 
        $r_j^{\RSD(h)}(\nu^{(k-1)},\la^{(k)})$,
        respectively (see \eqref{r_j_RSh}), 
        where $\nu^{(k-1)}$ differs from 
        $\la^{(k-1)}$ only by 
        $\nu^{(k-1)}_j=\la^{(k-1)}_j+1$.
        If $\la^{(k-1)}_{j}=\la^{(k)}_{j}$,
        then a short-range push happens 
        according to  
     	\S \ref{sub:when_a_jumping_particle_has_to_push_its_immediate_upper_right_neighbor}.
	\end{enumerate}
\end{dynamics}
\begin{figure}[htbp]
	\begin{center}
		\includegraphics[width=345pt]{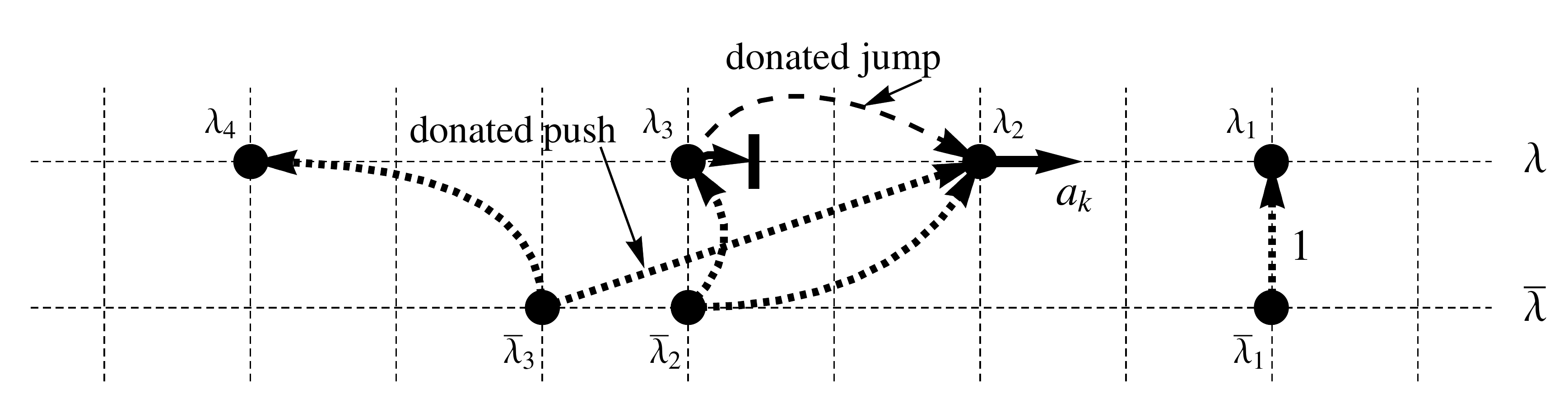}
	\end{center}  
  	\caption{Example of the behavior of 
  	an RSK-type fundamental `dynamics'
  	on the slice $\GT_{(3;4)}$
  	with $h=h^{(4)}=3$.
  	(see Dynamics \ref{dyn:RS_type_fund}).
  	The only particle that can independently
  	jump on the 
  	upper level is $\la_2$: because $\la_3$
  	is blocked, it donates its jump rate to~$\la_2$. 
  	If $\bar\la_1$ moves, it short-range
  	pushes $\la_1$ (the rule of 
  	\S \ref{sub:when_a_jumping_particle_has_to_push_its_immediate_upper_right_neighbor}). 
  	If $\bar\la_2$ moves, it pushes 
  	$\la_2$ with probability $r_{2}^{\RSD(3)}
  	(\bar\la+\bar\de_2,\la)$ (which is 
  	well-defined), and pulls $\la_3$ with the 
  	complementary probability $1-r_{2}^{\RSD(3)}
  	(\bar\la+\bar\de_2,\la)$.
  	If $\bar\la_3$ moves, it pushes $\la_2$
  	with probability $r_{3}^{\RSD(3)}
  	(\bar\la+\bar\de_3,\la)$ (this push would 
  	have gone to $\la_3$ if $\la_3$ were not blocked),
  	and pulls $\la_4$ with the complementary probability
  	$1-r_{3}^{\RSD(3)}
  	(\bar\la+\bar\de_3,\la)$.}
  	\label{fig:RS_h}
\end{figure}

One can say that
in both cases above 
the blocking of a particle
leads to \emph{donation} of a jump 
or a push.
Let us formulate a 
corresponding
rule for future reference:
\begin{drule}
	Assume that a particle
	$\la_m^{(k)}$ at level $k$
	must move to the right by one
	(due to an independent jump or 
	a triggered move) but is blocked, 
	i.e., 
	$\la^{(k)}_m=\la^{(k-1)}_{m-1}$.
	Then this particle
	donates the move
	to the first free particle $\la^{(k)}_{\nf(m)}$
	to the right of itself (see Fig.~\ref{fig:RS_h} 
	for an example).
\end{drule}\label{donation}
Using this rule, one can simplify
descriptions of nearest neighbor
multivariate `dynamics'
by not explaining 
what happens when some of the particles
of the interlacing array
become blocked (and the above donation
rule will take care of
blocking situations automatically). 
In fact, 
by the general Theorem \ref{thm:characterization_NN}
below (and using Dynamics \ref{dyn:DF_Mac},
\ref{dyn:RS_type_fund}, \ref{dyn:Ri}, and \ref{dyn:Li}), 
one can \emph{always} employ such a simplified description of 
\emph{any}
nearest neighbor multivariate `dynamics' (and,
in particular, assign nonzero jump rates 
to blocked particles).
Below we use these simplified descriptions
when it is convenient.


\subsubsection{Right-pushing and left-pulling 
fundamental `dynamics'} 
\label{ssub:right_and_left_pushing_fundamental_dynamics_}

Similarly to \S \ref{ssub:rs_type_fundamental_dynamics_},
we will now define fundamental
`dynamics'
corresponding to the 
right-pushing and left-pulling fundamental solutions, 
respectively (see 
\S \ref{ssub:right_pushing_fundamental_solutions}
and \S \ref{ssub:left_pushing_fundamental_solutions}).

First, observe that we can define 
the jump rates $W_k^{\RD(h)}(\la,\cdot\,|\,\bar\nu)$
and the probabilities of triggered moves
$V_k^{\RD(h)}(\la,\cdot\,|\,\cdot,\bar\nu)$,
and similarly for $\LD(h)$,
in the case when $h+1$ belongs to
$\II(\bar\nu,\la)$. That is, by 
\eqref{l_r_V}--\eqref{w_m_shorthand}
with the help of the corresponding fundamental
solutions we set:
\begin{align}\label{Ri_Li_VW}
	\begin{array}{rcl}
		V_k^{\RD(h)}
		(\la,\la+\de_m\,|\,\bar\nu-\bar\de_j,\bar\nu)
		&=&1_{j=h}1_{m=\nf(h)},
		\\
		W_k^{\RD(h)}
		(\la,\la+\de_m\,|\,\bar\nu)
		&=&a_k(S_m-1_{\nf^{-1}(m)=h}T_h);
		\\
		V_k^{\LD(h)}
		(\la,\la+\de_m\,|\,\bar\nu-\bar\de_j,\bar\nu)
		&=&
		1_{j=h}1_{m=h+1},
		\\
		W_k^{\LD(h)}
		(\la,\la+\de_m\,|\,\bar\nu)
		&=&a_k(S_m-1_{m=h+1}T_h)
	\end{array}
\end{align}
for all $m\in\II(\bar\nu,\la)$ 
and all $j$ with $j+1\in\II(\bar\nu,\la)$,
where $\nf^{-1}(m)$ is defined 
in~\S \ref{ssub:right_pushing_fundamental_solutions}.

Next, note that if $h+1\notin\II(\bar\nu,\la)$, then
$\bar\nu_h$ cannot be
the particle that has just moved at level $k-1$, 
and it also cannot happen that $j=h$ in
the definitions of $V_k^{\RD(h)}$ and
$V_k^{\LD(h)}$ in 
\eqref{Ri_Li_VW}.
Still, in both `dynamics' $\RD(h)$ and 
$\LD(h)$ we would like to let \emph{only} the particle 
$\bar\nu_h$ at level $k-1$ 
to push (resp. pull) particles at level $k$
(here we use the ``locality'' idea, cf.
Remark \ref{rmk:local_extension}). 
Thus, for $h+1\notin\II(\bar\nu,\la)$ 
it is natural not to allow
\emph{any} pushing (resp. pulling) at all. 
Note that in this case by \eqref{T_S_zero}
we have $T_h=0$ in \eqref{Ri_Li_VW}.

Thus, 
on the slice $\GT_{(k-1;k)}$
we have completely defined 
the rates of independent jumps 
$W_k^{\RD(h)}$ and $W_k^{\LD(h)}$
and the probabilities of triggered
moves $V_k^{\RD(h)}$ and $V_k^{\LD(h)}$
for each $h=1,\ldots,k-1$.

To define the corresponding `dynamics'
on interlacing arrays, let us choose an integer
sequence
\begin{align}\label{RL_h_index}
	\boldsymbol
	h=(h^{(1)},h^{(2)},\ldots,h^{(N-1)}),
	\qquad 1\le h^{(j)}\le j,
\end{align}
and denote by $\bq^{(N)}_{\RD[\boldsymbol h]}$
the generator given by \eqref{spec_bq_N_product_form}
with $W_k=W_k^{\RD(h^{(k-1)})}$
and $V_k=V_k^{\RD(h^{(k-1)})}$
(note the difference with the RSK-type `dynamics'
in \S \ref{ssub:rs_type_fundamental_dynamics_}),
$k=2,\ldots,N$,
and similarly for 
$\bq^{(N)}_{\LD[\boldsymbol h]}$. We will
call these `dynamics' the 
\emph{right-pushing and left-pulling fundamental `dynamics'},
respectively. Both families
consist of $(N-1)!$ `dynamics'.
Let us provide their equivalent 
`probabilistic'
description.

\begin{dynamics}
[Right-pushing fundamental `dynamics'
$\bq^{(N)}_{\RD[\boldsymbol h]}$]
\label{dyn:Ri}
	Let us take any $k=1,\ldots,N$, and 
	put $h=h^{(k-1)}$.
	{\ }\begin{enumerate}[(1)]
		\item (independent jumps)
		Each particle $\la_m^{(k)}$,
		$m\ne h$,
		has an independent
		exponential clock with
		rate $a_kS_m(\la^{(k-1)},\la^{(k)})$ 
		(or simply $a_1$ for $k=1$).
		The clock of the particle 
		$\la_h^{(k)}$ has a different
		rate
		\begin{align*}
			a_k\big(S_h(\la^{(k-1)},\la^{(k)})
			-
			T_h(\la^{(k-1)},\la^{(k)})
			\big)
		\end{align*}
		(there is no such special particle 
		if $k=1$).
		When the clock 
		of any $\la_{m}^{(k)}$, $m=1,\ldots,k$,
		rings, this
		particle tries to jump to the right by one,
		using donation rule of \S \ref{donation}
		if it is blocked.

		\item (triggered moves)
		If the particle $\la^{(k-1)}_{h}$ 
		moves to the right by one
        and $\la^{(k-1)}_{h}<\la^{(k)}_{h}$, 
        then $\la^{(k-1)}_{h}$ 
        long-range pushes
        its first unblocked 
        upper right neighbor
        $\la_{\nf(h)}^{(k)}$ with probability 1.
        No other long-range pushing or pulling is present.
		
		When $\la^{(k-1)}_{h}=\la^{(k)}_{h}$,
		or a move of any other particle at level $k-1$
		breaks the interlacing, then
		short-range pushing 
        takes place
        according to~\S \ref{sub:when_a_jumping_particle_has_to_push_its_immediate_upper_right_neighbor}.
	\end{enumerate}
\end{dynamics}
In the right-pushing
fundamental `dynamics' one sees the same 
mechanism of donation of moves
as in Dynamics \ref{dyn:RS_type_fund}
above. See also Fig.~\ref{fig:Ri} for an example.

\begin{figure}[htbp]
	\begin{center}
		\includegraphics[width=345pt]{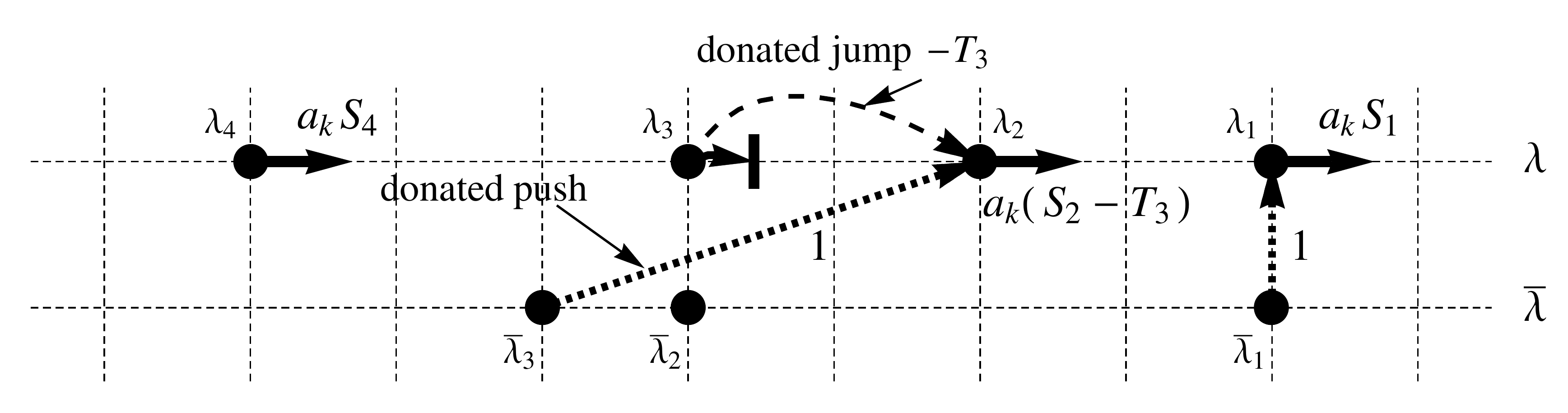}
	\end{center}  
  	\caption{Example of the behavior of 
  	a right-pushing fundamental `dynamics'
  	on the slice $\GT_{(3;4)}$
  	with $h=h^{(3)}=3$
  	(see Dynamics \ref{dyn:Ri}).
  	Rates of independent
  	jumps of the upper level
  	particles are given. 
  	Since $\la_3$ is blocked and 
  	$S_3(\bar\la,\la)=0$,
  	$\la_3$ donates 
  	the remaining nonzero
  	jump rate
  	$-T_3(\bar\la,\la)$
  	to the first unblocked particle
  	$\la_{\nf(3)}=\la_2$.
  	On the lower level,
  	if $\bar\la_1$ moves, it short-range
  	pushes $\la_1$ (the rule of 
  	\S \ref{sub:when_a_jumping_particle_has_to_push_its_immediate_upper_right_neighbor}). 
  	If $\bar\la_3$ moves, it 
  	long-range
  	pushes $\la_2$
  	with probability~1
  	(this push would 
  	have gone to $\la_3$ if $\la_3$ were not blocked).
  	No other pushes or pulls are possible on this picture. 
  	In particular,
  	if $\bar\la_2$ moves, it affects
  	no one at the upper level.}
  	\label{fig:Ri}
\end{figure}
\begin{dynamics}
[Left-pulling fundamental `dynamics'
$\bq^{(N)}_{\LD[\boldsymbol h]}$]
Take any $k=1,\ldots,N$, and 
put $h=h^{(k-1)}$.
\label{dyn:Li}
	{\ }\begin{enumerate}[(1)]
		\item (independent jumps)
		Each particle $\la_m^{(k)}$,
		$m\ne h+1$,
		has an independent
		exponential clock with
		rate $a_kS_m(\la^{(k-1)},\la^{(k)})$
		(or simply $a_1$ for $k=1$).
		The clock of the particle 
		$\la_{h+1}^{(k)}$ has a different
		rate
		\begin{align*}
			a_k\big(S_{h+1}(\la^{(k-1)},\la^{(k)})
			-
			T_h(\la^{(k-1)},\la^{(k)})
			\big)
		\end{align*}
		(there is no such special particle 
		if $k=1$).
		When the clock 
		of any $\la_{m}^{(k)}$, $m=1,\ldots,k$,
		rings, this
		particle jumps to the right by one;
		the rate automatically vanishes
		if $\la^{(k)}_{m}$ is blocked.
		
		\item (triggered moves)
		If the particle $\la^{(k-1)}_{h}$ 
		moves to the right by one
		and $\la^{(k-1)}_{h}<\la^{(k)}_{h}$, 
		then $\la^{(k-1)}_{h}$ 
		pulls
		its immediate
		upper left neighbor
		$\la_{h+1}^{(k)}$ with probability 1.
		No other long-range pushes or pulls are present.
		
		When $\la^{(k-1)}_{h}=\la^{(k)}_{h}$,
		or a move of any other particle at level $k-1$
		breaks the interlacing, then
		short-range pushing 
		takes place
		according to~\S \ref{sub:when_a_jumping_particle_has_to_push_its_immediate_upper_right_neighbor}.
	\end{enumerate}
\end{dynamics}
In contrast with Dynamics \ref{dyn:RS_type_fund} and
\ref{dyn:Ri}, we see that in the 
left-pulling fundamental `dynamics'
no independent jumps or pulls
need to be donated.
See also Fig.~\ref{fig:Li} 
for an example 
of a left-pulling fundamental `dynamics'.

\begin{figure}[htbp]
	\begin{center}
		\includegraphics[width=345pt]{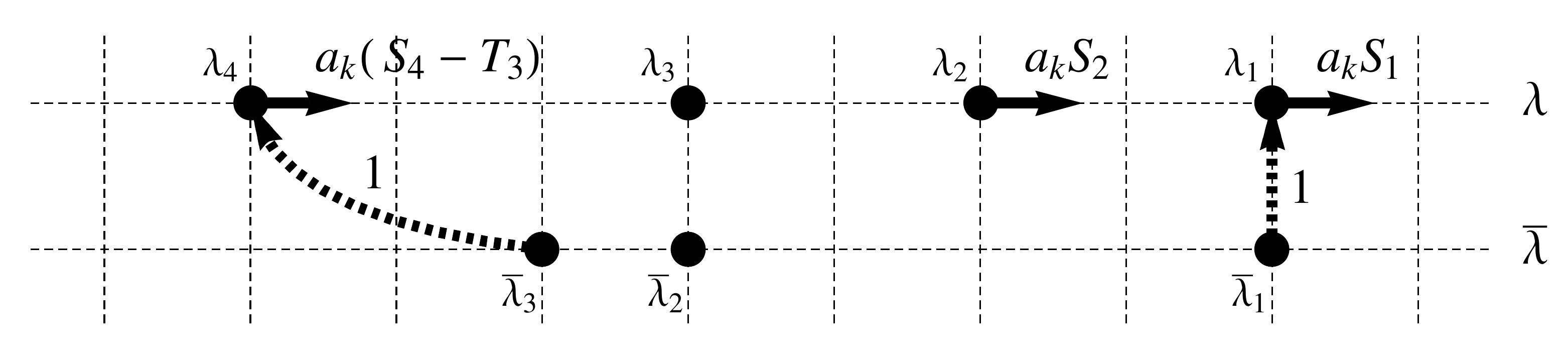}
	\end{center}  
  	\caption{Example of the behavior of 
  	a left-pulling fundamental `dynamics'
  	on the slice $\GT_{(3;4)}$
  	with $h=h^{(3)}=3$
  	(see Dynamics \ref{dyn:Ri}).
  	Rates of independent
  	jumps of the upper level
  	particles are given. 
  	The particle $\la_3$ is blocked and 
  	has jump rate
  	$a_kS_3(\bar\la,\la)=0$.
  	On the lower level,
  	if $\bar\la_1$ moves, it short-range
  	pushes $\la_1$ (the rule of 
  	\S \ref{sub:when_a_jumping_particle_has_to_push_its_immediate_upper_right_neighbor}). 
  	If $\bar\la_3$ moves, it
  	(long-range)
  	pulls $\la_4$
  	with probability~1.
  	No other pushes or pulls are possible on this picture. 
  	In particular,
  	if $\bar\la_2$ moves, it affects
  	no one.}
  	\label{fig:Li}
\end{figure}


\subsubsection{Conclusion} 
\label{ssub:conclusion}

In this subsection we have 
introduced
$1+N!+2(N-1)!$
fundamental `dynamics'
on interlacing arrays 
(as on Fig.~\ref{fig:la_interlacing})
by unifying 
fundamental solutions of the 
linear system \eqref{rlw_system_of_equations}
in a certain way.
The first fundamental 
`dynamics', namely,
the push-block process,
already appeared 
in \cite[\S2.3.3]{BorodinCorwin2011Macdonald}.
This push-block dynamics 
can be characterized as 
having no long-range 
pushes or pulls~(\S \ref{sub:example_push_block_process_on_interlacing_arrays}).

The right-pushing and left-pulling
`dynamics' may be viewed as 
having only one 
long-range push (resp. pull) 
at each slice
$\GT_{(k-1;k)}$.
These `dynamics' are parametrized
by sequences $(h^{(1)},\ldots,h^{(N-1)})$,
$1\le h^{(j)}\le j$,
where $h^{(j)}$ is the index of the only particle 
at level $j$
which pushes (resp. pulls) someone at level $j+1$.

Finally, the RSK-type fundamental `dynamics'
can be characterized as having the 
least possible number of independent
jumps. They are parametrized
by sequences $(h^{(1)},\ldots,h^{(N)})$,
$1\le h^{(j)}\le j$, 
and only the particle $h^{(k)}$
at each level $k$ is 
allowed to jump.
This of course leads to 
the presence of many 
long-range pushes and pulls in an 
RSK-type `dynamics'.

Observe that all the fundamental
`dynamics' 
on interlacing arrays
are pairwise distinct if particles are 
``apart''. That is, for every $k=2,\ldots,N$
the functions
$W_k(\la,\cdot\,|\,\bar\nu),
V_k(\la,\cdot\,|\,\cdot,\bar\nu)$
corresponding to different `dynamics'
are different when 
no particles
at level $k$ are blocked. When particles
at level $k-1$ and $k$ get closer to each other
in the sense that $\ii(\bar\nu,\la)<k$, some of these
functions coincide. 
One can readily describe these coincidences
in detail, but we will not need pursue a description.
Note that at the first level $k=1$ of the 
interlacing array 
each of the fundamental 
`dynamics' behaves in the same way dictated
by the univariate dynamics $Q_1$ 
(see the discussion before Remark \ref{rmk:la=nu}).

Let us also emphasize that 
the fundamental `dynamics'
do not necessarily have nonnegative
jump rates or probabilities
of triggered moves. We are guaranteed, however,
that at least the push-block dynamics 
(Dynamics \ref{dyn:DF_Mac})
is always an honest
Markov process because of \eqref{T_S_zero}.
Moreover, we will see in 
\S \ref{sec:schur_degeneration_and_robinson_schensted_correspondences} that \emph{all} the fundamental
`dynamics' have nonnegative jump rates and 
probabilities of triggered moves 
in the Schur ($q=t$) case. 
In \S \ref{sec:multivariate_dynamics_in_the_q_whittaker_case_and_q_pushasep_} we will also 
consider this problem of nonnegativity
in the $q$-Whittaker (i.e., $t=0$) 
case. It turns out that 
in this case
some of the
fundamental `dynamics' are also 
honest Markov processes.



\subsection[Classification of nearest neighbor 
`dynamics']{Classification of nearest neighbor 
`dynamics'\\ on interlacing arrays} 
\label{sub:characterization_of_nearest_neighbor_dynamics_}

Now we are in a position to present the main result
of the present section.
Informally, we would like to say that an 
arbitrary nearest neighbor 
multivariate `dynamics'
$\bq^{(N)}$ on interlacing arrays
can be expressed as a linear combination
of the fundamental `dynamics'
(\S \ref{sub:fundamental_dynamics_})
with coefficients which sum to one (cf.
the discussion in 
\S \ref{ssub:fundamental_solutions_and_arbitrary_nearest_neighbor_dynamics_}).

To be more precise, let us 
take any two nearest neighbor multivariate
`dynamics' with generators
$\bq^{(N)}$ and $\bq'^{(N)}$, respectively.
Recall that a generator $\bq^{(N)}$ 
(or $\bq'^{(N)}$)
as in \eqref{spec_bq_N_product_form}
is a matrix with row and columns indexed by 
Gelfand--Tsetlin schemes
of depth $N$.
Let $\bq^{(N)}$ and $\bq'^{(N)}$
correspond to 
functions $\{W_k,V_k\}$
and $\{W_k',V_k'\}$, respectively.
Let us define
\emph{mixing} 
$\tilde\bq^{(N)}$ 
(this generator corresponds to 
$\{\tilde W_k,\tilde V_k\}$)
of the `dynamics' 
$\bq^{(N)}$ and $\bq'^{(N)}$
as follows:
\begin{align}\label{mixing_generators}
	\begin{array}{l}
		\tilde W_k(\la^{(k)},\nu^{(k)}\,|\,\nu^{(k-1)})
		:=
		\co^{(k)}(\nu^{(k-1)},\la^{(k)})
		W_k(\la^{(k)},\nu^{(k)}\,|\,\nu^{(k-1)})
		\\
		\rule{0pt}{15pt}
		\hspace{155pt}
		+\big(1-\co^{(k)}(\nu^{(k-1)},\la^{(k)})
		\big)
		W_k'(\la^{(k)},\nu^{(k)}\,|\,\nu^{(k-1)}),\\
		\rule{0pt}{20pt}
		\tilde V_k(\la^{(k)},\nu^{(k)}
		\,|\,\la^{(k-1)},\nu^{(k-1)}):=
		\co^{(k)}(\nu^{(k-1)},\la^{(k)})
		V_k(\la^{(k)},\nu^{(k)}
		\,|\,\la^{(k-1)},\nu^{(k-1)})
		\\
		\rule{0pt}{15pt}
		\hspace{155pt}
		+\big(1-\co^{(k)}(\nu^{(k-1)},\la^{(k)})\big)
		V_k'(\la^{(k)},\nu^{(k)}
		\,|\,\la^{(k-1)},\nu^{(k-1)}),
	\end{array}
\end{align}
for each $k=2,\ldots,N$,
where 
$\boldsymbol\la=(\la^{(1)}\prec
\ldots
\la^{(N)})$
and 
$\boldsymbol\nu=(\nu^{(1)}\prec
\ldots
\nu^{(N)})$
are our Gelfand--Tsetlin 
schemes of depth $N$, and 
$\co^{(k)}(\cdot,\cdot)$ 
is any function on the slice $\GT_{(k-1;k)}$. 
Mixing \eqref{mixing_generators} of 
multivariate `dynamics'
is simply a composition of 
operations \eqref{mixing_WV}
performed on each slice.
Thus, we conclude that $\tilde\bq^{(N)}$
is again a multivariate `dynamics'
(in the sense of 
\S \ref{sub:characterization_of_multivariate_dynamics}).
Moreover, the `dynamics' $\tilde\bq^{(N)}$
is clearly nearest neighbor.

Operation \eqref{mixing_generators}
is more general than a 
linear combination of two `dynamics'
$\bq^{(N)}$ and $\bq'^{(N)}$
with coefficients which sum to one, as it
allows \emph{a lot} of freedom
in choosing the coefficients
$\co^{(k)}$ for each $k=2,\ldots,N$.
Moreover, 
mixing \eqref{mixing_generators}
includes
linear combinations 
of the generators with constant
coefficients (summing to one)
as a particular case.
Of course, one can define a similar mixing of
any finite number of multivariate `dynamics'.

The next theorem
(already formulated as 
Theorem 
\ref{thm:main_intro}
in the Introduction)
summarizes the development of the 
present section, and follows
from all the above definitions
together with Propositions
\ref{prop:RS_R}, \ref{prop:RS_L}, 
and~\ref{prop:R_L}:

\begin{theorem}\label{thm:characterization_NN}
	Any multivariate nearest neighbor
	`dynamics' 
	$\bq^{(N)}$
	on interlacing
	arrays of depth	$N$
	can be obtained as a mixing
	\eqref{mixing_generators}
	of fundamental `dynamics',
	which are:

	$\bullet$ the push-block dynamics $\bq^{(N)}_{\PBD}$;

	$\bullet$ $N!$ RSK-type fundamental 
	`dynamics' 
	\begin{align*}
		\bq^{(N)}_{\RSD[\boldsymbol h]},\qquad
		\boldsymbol h\in\{1\}\times\{1,2\}
		\times \ldots\times\{1,\ldots,N\};
	\end{align*}

	$\bullet$ $(N-1)!$ right-pushing 
	fundamental `dynamics'
	\begin{align*}
		\bq^{(N)}_{\RD[\boldsymbol h]},
		\qquad
		\boldsymbol h\in\{1\}\times\{1,2\}
		\times \ldots\times\{1,\ldots,N-1\};
	\end{align*}

	$\bullet$ $(N-1)!$ left-pulling fundamental 
	`dynamics'
	\begin{align*}
		\bq^{(N)}_{\LD[\boldsymbol h]},
		\qquad
		\boldsymbol h\in\{1\}\times\{1,2\}
		\times \ldots\times\{1,\ldots,N-1\}.
	\end{align*}
\end{theorem}

Moreover, to get \emph{all} possible nearest neighbor 
`dynamics', in the above theorem
it is enough to take either of the three combinations:

\begin{enumerate}[($\RSD$--$\RD$--$\PBD$)]
	\item[($\RSD$--$\RD$--$\PBD$)]
	all RSK-type and all right-pushing `dynamics'
	together with the push-block dynamics;
	\item[($\RSD$--$\LD$--$\PBD$)]
	all RSK-type and all left-pulling `dynamics'
	together with the push-block dynamics;
	\item[($\RD$--$\LD$--$\PBD$)]
	all right-pushing and left-pulling `dynamics'
	together with the push-block dynamics.
\end{enumerate}
Each combination 
directly corresponds to one of 
Propositions
\ref{prop:RS_R}, \ref{prop:RS_L}, 
and~\ref{prop:R_L}.
We have added the push-block
dynamics 
to the first two families
to resolve the 
issues for $k=2$ described in 
Remark \ref{rmk:k=2_bad}.

The expression of a 
nearest neighbor `dynamics' $\bq^{(N)}$
as a mixing
of the fundamental ones is far from 
being unique.
Let us briefly indicate how one can
formulate a precise 
uniqueness statement (without 
going into further detail).
It would require
a separate treatment
when particles
of interlacing arrays
$\boldsymbol\la$ and $\boldsymbol\nu$
are close to each other
(cf. 
the discussion in
\S \ref{ssub:conclusion}).
Moreover, to get uniqueness
with mixing \eqref{mixing_generators}, 
one would have to
reduce the number of allowed fundamental
`dynamics' 
in Theorem \ref{thm:characterization_NN}
to linear in $N$.
See also Remark \ref{rmk:RS_Schur_unique} below
for a related discussion
under $q=t$ (Schur) degeneration.

Finally, we note that 
taking certain (not completely arbitrary)
mixings (similarly to \eqref{mixing_generators})
of the push-block dynamics
with all RSK-type fundamental `dynamics', 
one gets a subclass
of nearest neighbor `dynamics'
which have constant
probability $C$ of move propagation
on every slice
(as in \eqref{constant_pp_C}).
Namely, to get the propagation probability
$C$, one must impose 
(for all $k$)
the constraint 
$\co^{(k)}_{\PBD}(\nu^{(k-1)},\la^{(k)})\equiv 1-C$
on the push-block coefficients 
in the corresponding mixing.
This follows from Proposition \ref{prop:const_pp}.
See also 
Proposition~\ref{prop:RS_Schur} below.



\section
{Schur degeneration and Robinson--Schensted correspondences} 
\label{sec:schur_degeneration_and_robinson_schensted_correspondences}

In this section we discuss a $q=t$
degenerate 
version of the formalism of multivariate dynamics
on interlacing arrays
developed in \S \ref{sec:multivariate_continuous_time_dynamics_on_interlacing_arrays_} 
and \S \ref{sec:nearest_neighbor_multivariate_continuous_time_dynamics_on_interlacing_arrays}.
This degeneration turns the 
Macdonald polynomials into 
much simpler 
Schur polynomials, 
see~\S \ref{sub:schur_appendix}.
In this case certain \emph{Schur-multivariate stochastic
dynamics} (namely, the fundamental RSK-type dynamics,
see \S \ref{ssub:rs_type_fundamental_dynamics_})
lead to many new \emph{deterministic} insertion algorithms 
extending the usual
row and column insertions. 
These considerations 
lead to new Robinson--Schensted-type correspondences.

\subsection{Row and column insertions} 
\label{sub:row_and_column_insertions}

Here we recall the classical row and column insertion
algorithms, and translate them into the
language of interlacing arrays.

\subsubsection{Row insertion} 
\label{ssub:row_insertion}

The input of the row insertion algorithm
is $(\Ptab,m)$, where $\Ptab$ is a semistandard
Young tableau 
over the alphabet $\{1,\ldots,N\}$
(see Definition \ref{def:SSYT}),
and $m$ is a letter from $1$ to $N$.
The output is a semistandard Young
tableau which we will denote by 
$\ins^{\mathit{row}}_{m}\Ptab$.	
The shape of $\ins^{\mathit{row}}_{m}\Ptab$
is obtained from the shape of $\Ptab$
by adding exactly one box.

The row insertion algorithm proceeds 
according to the following
rules:
{\sl\begin{enumerate}[(R-start)]
	\item[(R-start)]
	Start by inserting
	the letter $m$ into the
	first row. (That is, initialize $x:=m$, 
	and $R:=\mbox{the}$ first row of the tableau.)
	\item[(R-step)]
	Suppose a letter $x$ is to be 
	inserted into some row $R$
	of the tableau.
	\begin{enumerate}
		\item 
		Let $y_0$ (if it exists)
		be the smallest letter in 
		$R$ which is strictly
		greater than~$x$.
		In this case, 
		$x$ replaces $y_0$ in $R$.
		The letter $y_0$ is \textit{bumped} out of 
		the row $R$, and 
		we insert $y_0$ into 
		the next 
		row of the tableau.
		That is, set $x:=y_0$ and 
		$R:=\mbox{the}$ next row, and repeat~(R-step).
		\item 
		Otherwise, 
		if $x\ge y$ for every letter $y$
		in $R$, then 
		$x$ is appended to the end of the row $R$, and 
		the insertion algorithm \textbf{ends}.
	\end{enumerate}
\end{enumerate}}

An example of 
the row insertion of the letter 2
is given below:
\begin{align}\label{row_insertion_example}
	\Ptab=
	\begin{array}{|c|c|c|c|c|c|c|}
	\hline
	1&1&2&3&5\\
	\hline
	2&3&5&5\\
	\cline{1-4}
	3\\
	\cline{1-1}
	4\\
	\cline{1-1}
	\end{array}
	\qquad
	\qquad
	\ins^{\mathit{row}}_{2}\Ptab=
	\begin{array}{|c|c|c|c|c|c|c|}
	\hline
	1&1&2&\mathbf{2}&5\\
	\hline
	2&3&\mathbf{3}&5\\
	\cline{1-4}
	3&\mathbf{5}\\
	\cline{1-2}
	4\\
	\cline{1-1}
	\end{array}
\end{align}
The bold letters indicate 
how the insertion $\ins^{\mathit{row}}_{2}$ proceeded: 
2 bumps 3 from the first row, then
3 bumps 5 from the second row, and
finally 5 settles in the third row.


\subsubsection{Column insertion} 
\label{ssub:column_insertion}

The column insertion is another
algorithm which
transforms a pair $(\Ptab,m)$ 
($\Ptab$~is a semistandard Young tableau, 
$m$ is a letter)
into a new semistandard Young tableau
which we denote by
$\ins^{\mathit{col}}_{m}\Ptab$.

The rules are the following:
{\sl\begin{enumerate}[(C-start)]
	\item[(C-start)]
	Start by inserting
	the letter $m$ into the
	first column. (That is, initialize $x:=m$, 
	and $C:=\mbox{the}$ first column of the tableau.)
	\item[(C-step)]
	Suppose a letter $x$ is to be 
	inserted into some column $C$
	of the tableau.
	\begin{enumerate}
		\item 
		Let $y_0$ (if it exists)
		be the smallest letter in 
		$C$ which is greater than or equal to~$x$.
		In this case, 
		$x$ replaces $y_0$ in $C$.
		The letter $y_0$ is bumped out of 
		the column $C$, and 
		we insert $y_0$ into 
		the next 
		column of the tableau.
		That is, we set $x:=y_0$ and 
		$C:=\mbox{the}$ next column, and repeat (C-step).
		\item 
		Otherwise, 
		if $x> y$ for every letter $y$
		in $C$, then 
		$x$ is appended to the bottom of the 
		column $C$, and 
		the insertion algorithm \textbf{ends}.
	\end{enumerate}
\end{enumerate}}
The following example illustrates the column
insertion of the letter $2$ into a 
tableau $\Ptab$ (same as in 
\eqref{row_insertion_example}):
\begin{align}\label{column_insertion_example}
	\Ptab=
	\begin{array}{|c|c|c|c|c|c|c|}
	\hline
	1&1&2&3&5\\
	\hline
	2&3&5&5\\
	\cline{1-4}
	3\\
	\cline{1-1}
	4\\
	\cline{1-1}
	\end{array}
	\qquad
	\qquad
	\ins^{\mathit{col}}_{2}\Ptab=
	\begin{array}{|c|c|c|c|c|c|c|}
	\hline
	1&1&2&3&\mycirc{{\bf5}}&\mathbf{5}\\
	\hline
	{\mycirc{{\bf2}}}&
	\mathbf{2}&\mathbf{3}&\mathbf{5}\\
	\cline{1-4}
	3\\
	\cline{1-1}
	4\\
	\cline{1-1}
	\end{array}
\end{align}
Again, the letters in bold
(including the circled ones)
show how the column insertion proceeded.

Column insertion also
admits a more ``row-oriented'' description.
Let us discuss it in our 
concrete example \eqref{column_insertion_example}.
One can say that the letter~2 
is inserted into the second row of $\Ptab$ (the 
circled ``2''
in $\ins^{\mathit{col}}_{2}\Ptab$ 
in \eqref{column_insertion_example}), 
and then
\emph{shifts} all the possible
letters that were to the right of it.
Namely, these 
are the letters 2, 3, and 5, and they remain
in the second row in \eqref{column_insertion_example}. 
However, the rightmost
``5'' in the second row of $\Ptab$ cannot be shifted 
in a similar way,
as this would 
violate the definition of a semistandard
tableau. This letter 5 then goes to the
first row (and becomes 
the circled ``5'' in 
$\ins^{\mathit{col}}_{2}\Ptab$),
and shifts ``5'' that was already present in the 
first row of $\Ptab$.


\subsubsection{Insertions in terms of interlacing arrays} 
\label{ssub:insertions_in_terms_of_interlacing_arrays}

Now we aim to translate the row and column insertions
into the language of interlacing arrays
which are in a natural bijection
with semistandard Young
tableaux (Proposition \ref{prop:SSYT}).

Assume that one wants to \emph{row insert}
a letter $m$
into a semistandard Young tableau 
$\Ptab$ which is represented 
by an interlacing array
$\boldsymbol\la=
(\la^{(1)}\prec \ldots\prec\la^{(N)})$. 
The row insertion works according to the following
rules (formulated in terms of particles of
$\boldsymbol\la$ moving to the right by one):
{\sl\begin{enumerate}[(R-start)]
	\item[(R-start)] 
	The process starts with the
	rightmost
	particle $\la^{(m)}_{1}$ 
	at level $m$
	jumping to the right by one.
	\item[(R-step)]  
	Suppose a particle $\la_{j}^{(k-1)}$
	(for some $k=m+1,\ldots,N$ and $j=1,\ldots,k$)
	has moved (to the right by one). 
	If $\la_{j}^{(k)}=\la_{j}^{(k-1)}$, then
	the upper right neighbor
	$\la_{j}^{(k)}$ of $\la_{j}^{(k-1)}$ moves. Otherwise, 
	the upper left neighbor
	$\la_{j+1}^{(k)}$ moves.
	When $k=N$, the insertion \textbf{ends}.
\end{enumerate}}
For example, 
insertion 
\eqref{row_insertion_example}
of the letter $2$ 
in terms of interlacing
arrays
looks as follows:
\begin{align*}
	\boldsymbol\la=
	\begin{matrix}
		0&&1&&1&&4&&5\\
		&1&&1&&2&&4\\
		&&1&&2&&4\\
		&&&1&&3\\
		&&&&2
	\end{matrix}
	\qquad
	\ins^{\mathit{row}}_{2}\boldsymbol\la=
	\begin{matrix}
		0&&1&&\mycirc{2}&&4&&5\\
		&1&&1&&\mycirc{3}&&4\\
		&&1&&\mycirc{3}&&4\\
		&&&1&&\mycirc{4}\\
		&&&&2
	\end{matrix}
\end{align*}
Circled are the particle
positions that changed during the row insertion.

The \emph{column insertion}
of a letter $m$ into $\boldsymbol\la$
works as follows:
{\sl\begin{enumerate}[(C-donate)]
	\item[(C-start)] 
	The process starts with the
	leftmost
	particle $\la^{(m)}_{m}$ 
	at level $m$
	trying to jump to the right by one.
	\item[(C-donate)] 
	If a particle $\la_{j}^{(k)}$
	(for some $k=m,\ldots,N$ and $j=1,\ldots,k$)
	tries to move
	but is blocked (i.e., 
	if $\la_{j}^{(k)}=\la_{j-1}^{(k-1)}$), 
	then it donates the move
	to 
	its first unblocked right neighbor
	$\la^{(k)}_{\nf(j)}$ at level $k$ 
	(cf. the rule of \S \ref{donation}).
	If $\la_{j}^{(k)}$ is not blocked, 
	then
	$\nf(j)=j$ by agreement
	(cf. \eqref{nf_next_free_particle}).
	\item[(C-step)] 
	Suppose a particle $\la_{j}^{(k-1)}$
	(for some $k=m+1,\ldots,N$ and $j=1,\ldots,k$)
	has moved (to the right by one). 
	Then its upper right
	neighbor
	$\la_{j}^{(k)}$
	also moves to the right,
	or, if it is blocked,
	donates this move
	with the help of \mbox{(C-donate)} rule.
	When $k=N$, the insertion \textbf{ends}.
\end{enumerate}}
Let us illustrate
the column insertion
\eqref{column_insertion_example}
of the letter 2
using interlacing arrays.
We have:
\begin{align*}
	\boldsymbol\la=
	\begin{matrix}
		0&&1&&1&&4&&5\\
		&1&&1&&2&&4\\
		&&1&&2&&4\\
		&&&1&&3\\
		&&&&2
	\end{matrix}
	\quad
	\ins^{\mathit{col}}_{2}\boldsymbol\la=
	\begin{matrix}
		0&&1&&1&&4&&\mycirc{6}\\
		&1&&1&&\mycirc{3}&&4\\
		&&1&&\mycirc{3}&&4\\
		&&&\mycirc{2}&&3\\
		&&&&2
	\end{matrix}
\end{align*}
Here we have
also circled
the particle positions
that changed during the column
insertion.
Note that
when the move propagates 
from level 4 to level 5, the particle
$\la^{(5)}_{2}=\la^{(4)}_{1}=4$
is blocked and thus the move is 
donated to~$\la^{(5)}_{1}$.

It is straightforward to check the 
equivalence of the 
two descriptions of row and column insertions
(in terms of interlacing arrays 
and semistandard tableaux).
This equivalence also 
follows from a more detailed discussion 
in \S \ref{sub:from_interlacing_arrays_to_semistandard_tableaux} below.



\subsection{From interlacing arrays to 
semistandard tableaux and back} 
\label{sub:from_interlacing_arrays_to_semistandard_tableaux}

Insertion algorithms which we introduce 
later in this section
will be described in terms of 
interlacing arrays.
Here let us present a ``dictionary'' which would help
one
to restate these algorithms
in the language of semistandard Young tableaux.

Let $\Ptab$ be a semistandard Young
tableau corresponding to an interlacing 
array $\boldsymbol\la=(\la^{(1)}\prec
\ldots\prec\la^{(N)})\in\GT^+(N)$ 
of depth $N$ (as on Fig.~\ref{fig:la_interlacing}):
\begin{enumerate}[${\rm{}(iii)}$]
	\item[(i)] ($\Ptab\leftrightarrow\boldsymbol\la$) 
	By the very construction of
	\S \ref{sub:semistandard_young_tableaux}, 
	the $k$th level $\la^{(k)}_1\ge \ldots
	\ge \la^{(k)}_k$ of the array $\boldsymbol\la$
	means the shape 
	occupied by the letters $1,\ldots,k$ in the tableau $\Ptab$.
	In particular, 
	the coordinate of
	each rightmost particle 
	$\la^{(k)}_{1}$ is the total number of letters $1,\ldots,k$
	in the first row of $\Ptab$; and the coordinate 
	of each leftmost particle
	$\la^{(k)}_{k}$ is the number of letters 
	$k$ in the $k$th row of $\Ptab$.
\end{enumerate}
Let us describe the start of an insertion algorithm:
\begin{enumerate}[${\rm{}(iii)}$]
	\item[(ii)] (initial insert $\leftrightarrow$ ``independent jump'')
	Every insertion algorithm starts 
	with some \emph{new} letter $m$, $m=1,\ldots,N$, which must be inserted
	into the tableau $\Ptab$. This means that the shape of 
	$\Ptab$ occupied by the letters $1,\ldots,m-1$
	does not change. 
	Thus, the alteration of the array
	$\boldsymbol\la$ begins with 
	level $\la^{(m)}$, and after that affects all the higher levels
	$\la^{(m+1)},\ldots,\la^{(N)}$.

	If the new letter $m$ is to be inserted into some row $j=1,\ldots,m$ of the 
	tableau $\Ptab$,
	this means that the particle $\la^{(m)}_{j}$ must jump to the right by one.
	
	(ii$^{*}$)
		(donation of initial jump)
		The jump of $\la^{(m)}_{j}$ 
		may be not possible, this happens 
		if $\la^{(m)}_{j}=\la^{(m-1)}_{j-1}$. 
		Equivalently, the number of letters 
		$1,\ldots,m-1$ in the $(j-1)$th row of the tableau 
		$\Ptab$ is the same as
		the number of letters $1,\ldots,m$ in the $j$th row
		(in fact, this number of letters can be zero).

		In this case, the letter $m$ tries to be inserted into 
		the row $j-1$. Equivalently, the particle
		$\la^{(m)}_{j}$ donates its jump 
		with the help of the rule of \S \ref{donation}.
\end{enumerate}
Now let us discuss inductive rules with which 
an insertion algorithm can proceed after the initial
insertion of a letter. 
The
row and column insertion algorithms
(\S\S\ref{ssub:row_insertion}--\ref{ssub:insertions_in_terms_of_interlacing_arrays})
may be viewed as examples of combinations of these rules.

Fix $k=2,\ldots,N$ and $j=1,\ldots,k-1$, 
and assume that 
the letter
$k-1$ was just inserted into the row $j$ of the tableau $\Ptab$.
In terms of the interlacing array $\boldsymbol\la$,
this means that 
the particle 
$\la^{(k-1)}_{j}$
(call it a \emph{trigger})
at level $k-1$ of the array $\boldsymbol\la$
has just moved to the right by one
(due to an initial insert or 
following a move at level $k-2$).
This move at level $k-1$ in $\boldsymbol\la$
must result in exactly one of the following
three moves at higher levels:
\begin{enumerate}[${\rm{}(iii)}$]
	\item[(iii)] (``gap'' in letters in $\Ptab$ 
	$\leftrightarrow$ short-range pushing)
	Assume that there are no letters 
	$k,k+1,\ldots,k'-1$
	in the $j$th row of $\Ptab$, 
	but this row contains at least one letter~$k'>k$. 
	This means that the insertion
	of $k-1$ into the $j$th row of~$\Ptab$
	would next affect (i.e., bump or shift, see below)
	the letter $k'$, skipping all the letters 
	$k,k+1,\ldots,k'-1$. In terms of the interlacing
	array $\boldsymbol\la$, this means that 
	$\la^{(k-1)}_{j}=\la^{(k)}_{j}=\ldots=\la^{(k'-1)}_{j}$. 
	Thus, the particle $\la^{(k-1)}_{j}$
	moving 
	to the right by one will
	short-range push (\S \ref{sub:when_a_jumping_particle_has_to_push_its_immediate_upper_right_neighbor})
	all the particles $\la^{(k)}_{j},\ldots,\la^{(k'-1)}_{j}$
	to the right by one. 
	
	The insertion algorithm
	then continues inductively
	at levels $k',k'+1,\ldots,N$,
	triggered
	by the move of the last particle
	$\la^{(k'-1)}_{j}$.
	For the purposes of this inductive continuation
	of moves, one can assume that the inserted letter was $k'-1$
	and not $k-1$, i.e., that the 
	trigger now is the particle $\la^{(k'-1)}_{j}$.

	\item[(iv)] (bumping = row insertion $\leftrightarrow$ 
	pulling) 
	Assume now that the $j$th row of the tableau $\Ptab$
	contained at least one letter $k$, and that the 
	letter $k-1$ inserted into the $j$th row 
	\emph{bumps} this letter $k$ out of this row.
	The bumped letter $k$ will be inserted into the
	row $j+1$. 
	Observe that such an insertion is always possible.

	In terms of interlacing arrays, 
	the bumping precisely means that the moved
	particle $\la^{(k-1)}_{j}$ \emph{pulls}
	its immediate upper \emph{left} neighbor $\la^{(k)}_{j+1}$ 
	(see also \S \ref{ssub:insertions_in_terms_of_interlacing_arrays}).
	The trigger now is the particle
	$\la^{(k)}_{j+1}$.

	\item[(v)] (shifting~=~column insertion $\leftrightarrow$ 
	long-range pushing)
	Suppose
	that the $j$th row of the tableau $\Ptab$
	contained at least one letter $k$, and that
	the 
	letter $k-1$ inserted into the $j$th row 
	must now \emph{shift} all the letters $k$ (belonging to this row) 
	to the right, cf. the end of \S \ref{ssub:column_insertion}.

	In terms of insertions into columns, this means that the 
	inserted letter $k-1$ (arrived in the $j$th row of $\Ptab$) 
	bumps the letter $k$, and this bumped $k$ must go 
	to the next column (=~the column to the right).
	This column-bumping process continues 
	until the bumped letter becomes strictly greater than $k$; 
	we understand this sequence of column-bumpings as one step.

	In the language of 
	interlacing arrays $\boldsymbol\la$, 
	the shifting means that the moved particle
	$\la^{(k-1)}_{j}$
	\emph{long-range pushes}
	its upper \emph{right} neighbor $\la^{(k)}_{j}$ 
	(see also \S \ref{ssub:insertions_in_terms_of_interlacing_arrays}).\footnote{In principle, this long-range 
	push
	includes the short-range push described in 
	(iii). However,
	we would like to think that the short-range
	interaction is \emph{stronger} than 
	(it takes preference over)
	long-range pushes and pulls.}

	(v$^*$)
		(donation of moves)
		It can happen that the pushed particle 
		$\la^{(k)}_{j}$ is blocked and thus cannot be pushed
		(when $\la^{(k)}_{j}=\la^{(k-1)}_{j-1}$).
		Then this push is donated to 
		$\la^{(k)}_{\nf(j)}$
		using the rule of \S \ref{donation}.
		This situation is similar to 
		the donation of the initial move, see above.

		In terms of a sequence
		of column-bumpings of the letters $k$, 
		this means that some of the bumped
		letters $k$ will be inserted
		not into the row $j$ (because this would violate
		the definition of a semistandard tableau),
		but into other row or rows 
		which are above the $j$th row in the 
		semistandard tableau $\Ptab$.
	
	For the purposes of the inductive continuation
	of moves, the trigger now is the particle 
	$\la^{(k)}_{\nf(j)}$.
\end{enumerate}

Further developments in the present section 
(in particular, see \S \ref{sub:rs_type_fundamental_dynamics_in_the_schur_case_and_deterministic_insertion_algorithms})
show how one can \emph{combine}
inductive bumping and
shifting steps (iv) and (v) in certain ways
to get
many new insertion algorithms with properties
similar to those of the 
row and column insertions of 
\S \ref{sub:row_and_column_insertions}.
We describe a total of $N!$ such insertion algorithms
which can be applied to words 
in the alphabet $\{1,\ldots,N\}$.


\subsection{Nearest neighbor 
RSK-type `dynamics'
as (possibly random) insertion algorithms} 
\label{sub:rs_type_dynamics_and_random_insertions}

In view of insertion algorithms 
(\S\S \ref{sub:row_and_column_insertions}--\ref{sub:from_interlacing_arrays_to_semistandard_tableaux}), 
let us
take a second look at the nearest neighbor 
RSK-type multivariate
`dynamics' 
with Macdonald parameters $q$ and $t$
(see Definitions~\ref{def:nn_dynamics}
and \ref{def:RS_type}; we also follow the conventions of 
Remark~\ref{rmk:algebraic_rlw}). 

Consider the following input of 
random letters from $1$ to $N$:
\begin{align}\label{random_input}
	\parbox{0.85\textwidth}
	{\vspace{4pt}
	Each letter $m$, $1\le m\le N$
	appears independently 
	of other letters
	after exponentially distributed
	time intervals of rate $a_m$.\vspace{4pt}}
\end{align}
Equivalently, one can say that we have $N$ independent
Poisson processes with rates $a_1,\ldots,a_N$,
and a new letter $m$ appears precisely 
when the $m$th Poisson
process makes an increment.

Let $\bq^{(N)}$ be a nearest neighbor RSK-type `dynamics'
on interlacing arrays of depth $N$. Let the current
state of the `dynamics' be 
represented by an interlacing
array $\boldsymbol\la=(\la^{(1)}
\prec \ldots\prec\la^{(N)})$
which corresponds
to a semistandard Young tableau $\Ptab$.
We will freely switch between the
languages of interlacing arrays and semistandard 
tableaux with the help of 
\S \ref{sub:from_interlacing_arrays_to_semistandard_tableaux}.

Due to 
\eqref{w_m_shorthand} and
\eqref{sum_of_w_is_one},
for an RSK-type `dynamics'
we must have for all $k$
(see 
\S \ref{sub:characterization_of_multivariate_dynamics} and in particular
\eqref{spec_bq_N_product_form}
for the notation):
\begin{align}\label{sum_Wk_a_k}
	\sum_{\nu\in\GT_{k}}
	W_k(\la,
	\nu\,|\,\bar\la)=a_k,
	\qquad k=1,\ldots,N,
	\quad \la=\la^{(k)},
	\quad \bar\la=\la^{(k-1)},
\end{align}
which means that an independent
jump at each level $k$ happens
at rate~$a_k$.

This implies that an
instantaneous
transition $\boldsymbol\la\to\boldsymbol\nu$
in the multivariate 
`dynamics' $\bq^{(N)}$ occurs according
to the following three steps:
\begin{enumerate}[{(I)}]
	\item A new random letter $m$ arrives 
	from the random input \eqref{random_input}. 
	This letter must be inserted into 
	the tableau $\Ptab$ (=~array $\boldsymbol\la$).
	\item The new letter $m$
	randomly chooses to be inserted into 
	one of the rows $i=1,\ldots,m$ of $\la^{(m)}$ 
	(equivalently, one of the particles 
	$\la^{(m)}_{i}$
	jumps to the right by one)
	with probabilities
	(cf. \eqref{sum_Wk_a_k})
	\begin{align*}
		\prob(
		\mbox{$m$ is inserted into row $i$})
		=
		a_m^{-1}W_m(\la^{(m)},\nu^{(m)}\,|\,\la^{(m-1)}),
	\end{align*}
	where 
	the Young diagram 
	$\nu^{(m)}$ differs from 
	$\la^{(m)}$
	by adding one box to the $i$th row.
	\item 
	The initial insertion 
	of $m$ into the $i$th row of $\Ptab$
	triggers successive insertion
	steps which ultimately result
	in a change of the shape 
	of $\Ptab$ (i.e., in adding a box to the Young
	diagram $\la^{(N)}$). This is because
	the `dynamics' $\bq^{(N)}$ is RSK-type, and 
	hence moves always propagate to higher rows.

	This insertion process 
	is possibly \emph{random}.
	Namely, if at level $k-1$
	(of the interlacing array $\boldsymbol\la$)
	some particle $\la^{(k-1)}_{j}$
	has moved to the right, then 
	this move leads to one of the following
	consequences at level $k$.
	First, if a \emph{short-range push} 
	is necessary (i.e., if
	$\la^{(k-1)}_{j}=\la^{(k)}_{j}$),
	then it occurs with probability 
	one (\S \ref{sub:from_interlacing_arrays_to_semistandard_tableaux}.(iii)).
	If there is no need for a short-range push,
	then a \emph{bumping} 
	(\S \ref{sub:from_interlacing_arrays_to_semistandard_tableaux}.(iv)) 
	or a \emph{shifting}
	(\S \ref{sub:from_interlacing_arrays_to_semistandard_tableaux}.(v))
	happens with probabilities
	$V_k(\la,\la+\de_{j+1}\,|\,
	\bar\la,\bar\la+\bar\de_j)$
	and $V_k(\la,\la+\de_{\nf(j)}\,|\,
	\bar\la,\bar\la+\bar\de_j)$, 
	respectively (where $\la=\la^{(k)}$,
	and $\bar\la=\la^{(k-1)}$),
	see \S \ref{sub:parametrization_and_linear_equations}
	for notation. The 
	sum of these bumping and shifting
	probabilities is one because 
	our `dynamics' is RSK-type and nearest neighbor.
\end{enumerate}

Thus, we see that any RSK-type nearest neighbor
`dynamics' corresponds to 
a possibly random \emph{insertion algorithm}
(proceeding with bumps and shifts)
which is applied to the 
random input \eqref{random_input}.


\subsection{RSK-type fundamental dynamics in the 
Schur case and deterministic insertion algorithms} 
\label{sub:rs_type_fundamental_dynamics_in_the_schur_case_and_deterministic_insertion_algorithms}

Now assume that we are in the Schur case, i.e.,
we have $q=t$, where $q$ and $t$ are 
the Macdonald parameters. 
Consider the RSK-type fundamental `dynamics'
$\bq^{(N)}_{\RSD[\boldsymbol h]}$
introduced in 
\S \ref{ssub:rs_type_fundamental_dynamics_}.
They are indexed by $N$-tuples of integers
$\boldsymbol h=
(h^{(1)},h^{(2)}\ldots,h^{(N)})$,
where
$1\le h^{(j)}\le j$.
Thus, there are $N!$ such `dynamics'. 

In the Schur case,
the quantities
$T_i(\bar\nu,\la)$ and 
$S_j(\bar\nu,\la)$ 
\eqref{T_i}--\eqref{S_j}
(employed in the definition of the RSK-type
fundamental `dynamics')
clearly become
\begin{align}\label{T_S_Schur}
	T_i(\bar\nu,\la)=
	1_{\bar\nu-\bar\de_i\prec\la}
	1_{\bar\nu-\bar\de_i\prec\bar\nu},
	\qquad
	S_j(\bar\nu,\la)=
	1_{\bar\nu\prec\la+\de_j}
	1_{\la\prec\la+\de_j}.
\end{align}
Here we use the notational conventions
of \S \ref{sub:specialization_of_formulas_from_s_sub:sequential_update_dynamics_in_continuous_time}
(in particular,
$\bar\nu\in\GT_{k-1}$ and $\la\in\GT_{k}$
for some fixed $k$), and we also 
assume that $\bar\nu\prec\la$ in \eqref{T_S_Schur}.
This implies the following:
\begin{proposition}\label{prop:RS_nonnegative}
	In the case $q=t$,
	for every $\boldsymbol h$
	as above,
	the fundamental RSK-type `dynamics' 
	$\bq^{(N)}_{\RSD[\boldsymbol h]}$
	has nonnegative jump rates and probabilities
	of triggered moves (and so it an
	honest Markov process).\footnote{In the 
	general Macdonald (and $q$-Whittaker)
	case this statement fails, see 
	\S\ref{ssub:fundamental_dynamics_}.}
\end{proposition}
\begin{proof}
	Directly follows from 
	\eqref{T_S_Schur} and definitions
	of \S \ref{ssub:rs_type_fundamental_dynamics_}
	(see also \eqref{r_j_RSh}).
\end{proof}

\begin{dynamics}[RSK-type 
fundamental dynamics in the Schur case]
\label{dyn:RS_type_Schur}
{\ }\begin{enumerate}[(1)]
		\item (independent jumps)
		Under $\bq^{(N)}_{\RSD[\boldsymbol h]}$,
		at each level $k$
		only one particle, 
		namely, $\la^{(k)}_{h^{(k)}}$, can
		independently try to jump to the right
		(with jump rate $a_k$).

		\item (triggered moves)
		If a particle $\la^{(k-1)}_{j}$
		moves at level 
		$k-1$, then this move 
		propagates to level $k$
		according to 
		the probabilities of triggered moves 
		(defined via \eqref{l_r_V} and \eqref{Vk_RS_h}):
		\begin{align}\label{r_j_RSD_Schur}
			r_j^{\RSD(h^{(k)})}
			(\nu^{(k-1)},\la^{(k)})=1_{j<h^{(k)}},
			\qquad
			j=1,\ldots,k-1.
		\end{align}
		Here $\nu^{(k-1)}$
		differs from $\la^{(k-1)}$
		only by
		$\nu^{(k-1)}_j=\la^{(k-1)}_j+1$.
		Note that the propagation probabilities
		$c_j=r_j+l_j$ \eqref{c_propagation} 
		in any RSK-type dynamics are all 
		equal to one.
	\end{enumerate}

	Finally, $\bq^{(N)}_{\RSD[\boldsymbol h]}$
	possesses our usual move donation mechanism
	of \S \ref{donation}.
\end{dynamics}

We see from \eqref{r_j_RSD_Schur}
and \eqref{Vk_RS_h}
that the probabilities
of triggered moves 
$V_k$ are all equal to~0 or 1. 
This means that 
the insertion algorithm corresponding
to $\bq^{(N)}_{\RSD[\boldsymbol h]}$
(steps (II)--(III) in
\S \ref{sub:rs_type_dynamics_and_random_insertions})
is \emph{deterministic}, so all the randomness
in the dynamics $\bq^{(N)}_{\RSD[\boldsymbol h]}$
consists in the random choice \eqref{random_input} of 
the new arriving letters 
(see \S \ref{sub:rs_type_dynamics_and_random_insertions}.(I)).

\begin{definition}\label{def:h_insertion}
	We will call the insertion
	of a new letter
	corresponding to the dynamics
	$\bq^{(N)}_{\RSD[\boldsymbol h]}$ 
	the \textit{$\boldsymbol h$-insertion}.
	If the new inserted letter is $m$, we will denote 
	this operation~by~$\ins^{\boldsymbol h}_{m}$.
\end{definition}

Let us describe how an 
$\boldsymbol h$-insertion
works.
Assume that one wants to $\boldsymbol h$-insert
a letter $m$
into a semistandard Young tableau 
$\Ptab$ which is represented 
by an interlacing array
$\boldsymbol\la=
(\la^{(1)}\prec \ldots\prec\la^{(N)})$.
The output is a semistandard
Young tableau denoted 
by $\ins^{\boldsymbol h}_m\Ptab=\ins^{\boldsymbol h}_m
\boldsymbol \la$,
and the insertion proceeds according to
the following rules:
{\sl\begin{enumerate}[($\boldsymbol h$-donate)]
	\item[($\boldsymbol h$-start)] 
	The process starts with the
	particle 
	$\la^{(m)}_{h^{(m)}}$ 
	at level $m$
	trying to jump to the right by one.
	\item[($\boldsymbol h$-donate)] 
	If a particle $\la_{j}^{(k)}$
	(for some $k=m,\ldots,N$ and $j=1,\ldots,k$)
	tries to move
	but is blocked (i.e., 
	if $\la_{j}^{(k)}=\la_{j-1}^{(k-1)}$), 
	then it donates the move
	to 
	its first unblocked right neighbor
	according to the rule of \S \ref{donation}.
	\item[($\boldsymbol h$-step)] 
	Suppose a particle $\la_{j}^{(k-1)}$
	(for some $k=m+1,\ldots,N$ and $j=1,\ldots,k$)
	has moved (to the right by one). 
	If $\la^{(k)}_j=\la^{(k-1)}_j$
	or $j<h^{(k)}$, 
	then the upper right neighbor $\la^{(k)}_j$
	of $\la^{(k-1)}_{j}$ tries to move to the right by one.
	Otherwise, 
	the upper left neighbor
	$\la^{(k)}_{j+1}$
	of $\la^{(k-1)}_{j}$
	moves to the right by one
	(note that $\la^{(k)}_{j+1}$
	cannot be blocked).
	When $k=N$, the insertion \textbf{ends}.
\end{enumerate}}

Let us illustrate this by considering an 
$\boldsymbol h$-insertion
for $N=6$ with $\boldsymbol h=(1,2,1,4,4,1)$,
see Fig.~\ref{fig:big_rs}.\begin{figure}[htbp]
	\begin{center}
		\includegraphics[width=150pt]
		{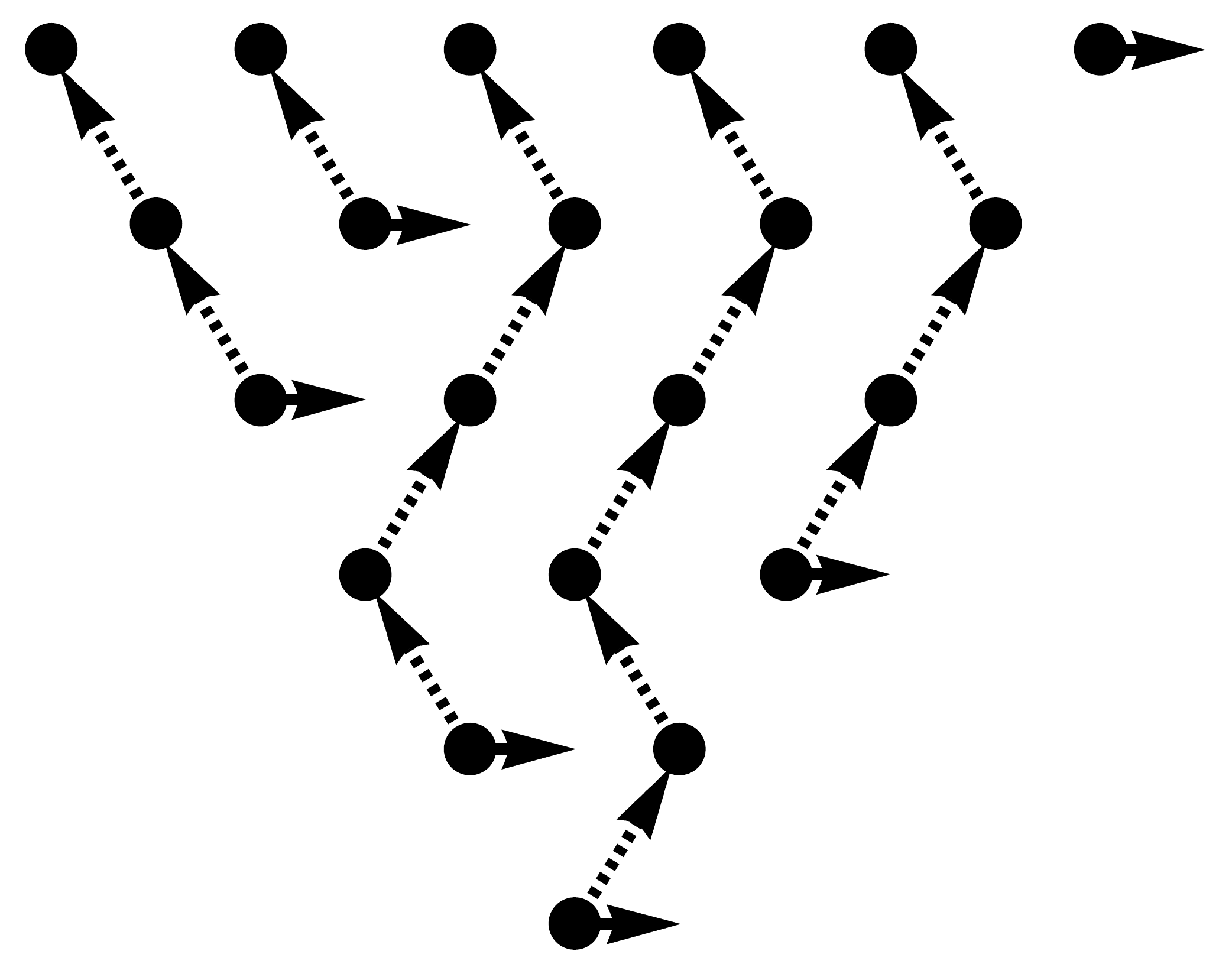}
	\end{center}
  	\caption{
  	Schematic picture of 
  	an $\boldsymbol h$-insertion
  	with $N=6$ and $\boldsymbol h=(1,2,1,4,4,1)$.
	Arrows in front of each particle
	represent nonzero rates of independent jumps
	in $\bq^{(N)}_{\RSD[\boldsymbol h]}$
	where 
	the $\boldsymbol h$-insertion 
	of a new letter starts (arrival of each letter $k$
	starts insertion at level $k$ of the array).
	Dashed arrows 
	pointing to the right mean 
	long-range pushes; 
	when dashed arrows point 
	to the left, this corresponds to
	long-range pulling.
	This picture of jump rates and 
	pushes/pulls is valid
	when the particles 
	are ``apart'';
	otherwise, there would be 
	short-range pushes
	and/or 
	donations of moves (see the rules
	of \S \ref{sub:when_a_jumping_particle_has_to_push_its_immediate_upper_right_neighbor} and
	\S \ref{donation}).}
  	\label{fig:big_rs}
\end{figure}
When particles are ``apart'',
the insertion trajectory
in an interlacing array
follows the dashed arrows
on Fig.~\ref{fig:big_rs}. 
For example, the $\boldsymbol h$-insertion of the letter 2 
into the array
\begin{align}\label{big_rs_lambda}
	\boldsymbol\la=
	\begin{array}{ccccccccccccc}
		0&&0&&1&&4&&4&&5\\
		&0&&1&&3&&4&&4\\
		&&1&&2&&4&&4\\
		&&&1&&3&&4\\
		&&&&1&&3\\
		&&&&&2
	\end{array}
\end{align}
yields
\begin{align}\label{big_rs_I2lambda}
	\ins^{\boldsymbol h}_{2}\boldsymbol\la=
	\begin{array}{ccccccccccccc}
		0&&0&&\mycirc{2}&&4&&4&&5\\
		&0&&1&&\mycirc{4}&&4&&4\\
		&&1&&\mycirc{3}&&4&&4\\
		&&&\mycirc{2}&&3&&4\\
		&&&&\mycirc{2}&&3\\
		&&&&&2
	\end{array}
\end{align}
(circled are the positions of particles that changed).

However, when particles are close to each other,
a short-range push or a donation of move can happen.
After that, the insertion trajectory 
continues along the dashed arrows on Fig.~\ref{fig:big_rs}.
Let us illustrate this effect by inserting another letter 2
into the above array $\ins^{\boldsymbol h}_{2}\boldsymbol\la$
\eqref{big_rs_I2lambda}, where $\boldsymbol\la$
is given in \eqref{big_rs_lambda}. We get
\begin{align}\label{big_rs_I2I2lambda}
	\ins^{\boldsymbol h}_{2}
	\ins^{\boldsymbol h}_{2}
	\boldsymbol\la=
	\begin{array}{ccccccccccccc}
		0&&0&&{2}&&4&&\mycirc{5}&&5\\
		&0&&1&&{4}&&4&&\mycirc{5}\\
		&&1&&{3}&&4&&\mycirc{5}\\
		&&&{2}&&\mycirc{4}&&4\\
		&&&&{2}&&\mycirc{4}\\
		&&&&&2
	\end{array}
\end{align}
Again, we have circled positions
of particles that changed.
Donations of moves happen at the 
second level (during an initial
insertion, cf. \S \ref{sub:from_interlacing_arrays_to_semistandard_tableaux}.(ii$^{*}$)), and also 
during the move propagation
from the third to the fourth level.
In fact, one can say that the
propagation of move from level 4 to level 5
is governed by a short-range push, and this agrees with 
the long-range push (see dashed arrows on Fig.~\ref{fig:big_rs}).

One can readily translate the 
$\boldsymbol h$-insertion algorithm 
into the language of semistandard tableaux
using our ``dictionary'' in 
\S \ref{sub:from_interlacing_arrays_to_semistandard_tableaux}. The example 
\eqref{big_rs_lambda}--\eqref{big_rs_I2I2lambda}
becomes
\begin{align*}
	\Ptab=
	\begin{array}{|c|c|c|c|c|c|c|}
	\hline
	1&1&2&3&6\\
	\hline
	2&3&3&4\\
	\cline{1-4}
	3&4&5&6\\
	\cline{1-4}
	4\\
	\cline{1-1}
	\end{array}
	\quad
	\ins^{\boldsymbol h}_{2}\Ptab
	=
	\begin{array}{|c|c|c|c|c|c|c|}
	\hline
	1&1&2&3&6\\
	\hline
	2&\mathbf{2}&3&4\\
	\cline{1-4}
	3&\mathbf{3}&\mathbf{4}&\mathbf{5}\\
	\cline{1-4}
	4&\mathbf{6}\\
	\cline{1-2}
	\end{array}
	\quad
	\ins^{\boldsymbol h}_{2}
	\ins^{\boldsymbol h}_{2}\Ptab
	=
	\begin{array}{|c|c|c|c|c|c|c|}
	\hline
	1&1&2&\mathbf{2}&\mathbf{4}\\
	\hline
	2&{2}&3&\mathbf{3}&\mathbf{6}\\
	\cline{1-5}
	3&{3}&{4}&{5}\\
	\cline{1-4}
	4&{6}\\
	\cline{1-2}
	\end{array}
\end{align*}
Bold letters indicate differences
before and after each of the two 
$\boldsymbol h$-insertions.

\begin{proposition}\label{prop:row_col}
	The row and column insertions
	(\S \ref{sub:row_and_column_insertions})
	arise as particular cases of 
	$\boldsymbol h$-insertions
	for $\boldsymbol h=(1,\ldots,1)$
	and
	$\boldsymbol h=(1,2,\ldots,N)$,
	respectively (see Fig.~\ref{fig:row_col}).
\end{proposition}
\begin{proof}
	Immediately
	follows from descriptions
	of insertions in terms of
	interlacing arrays.
\end{proof}

\begin{figure}[phtb]
	\begin{center}
		\begin{tabular}{|r|r|}
			\hline
			\rule{0pt}{95pt}
			\includegraphics[width=102pt]{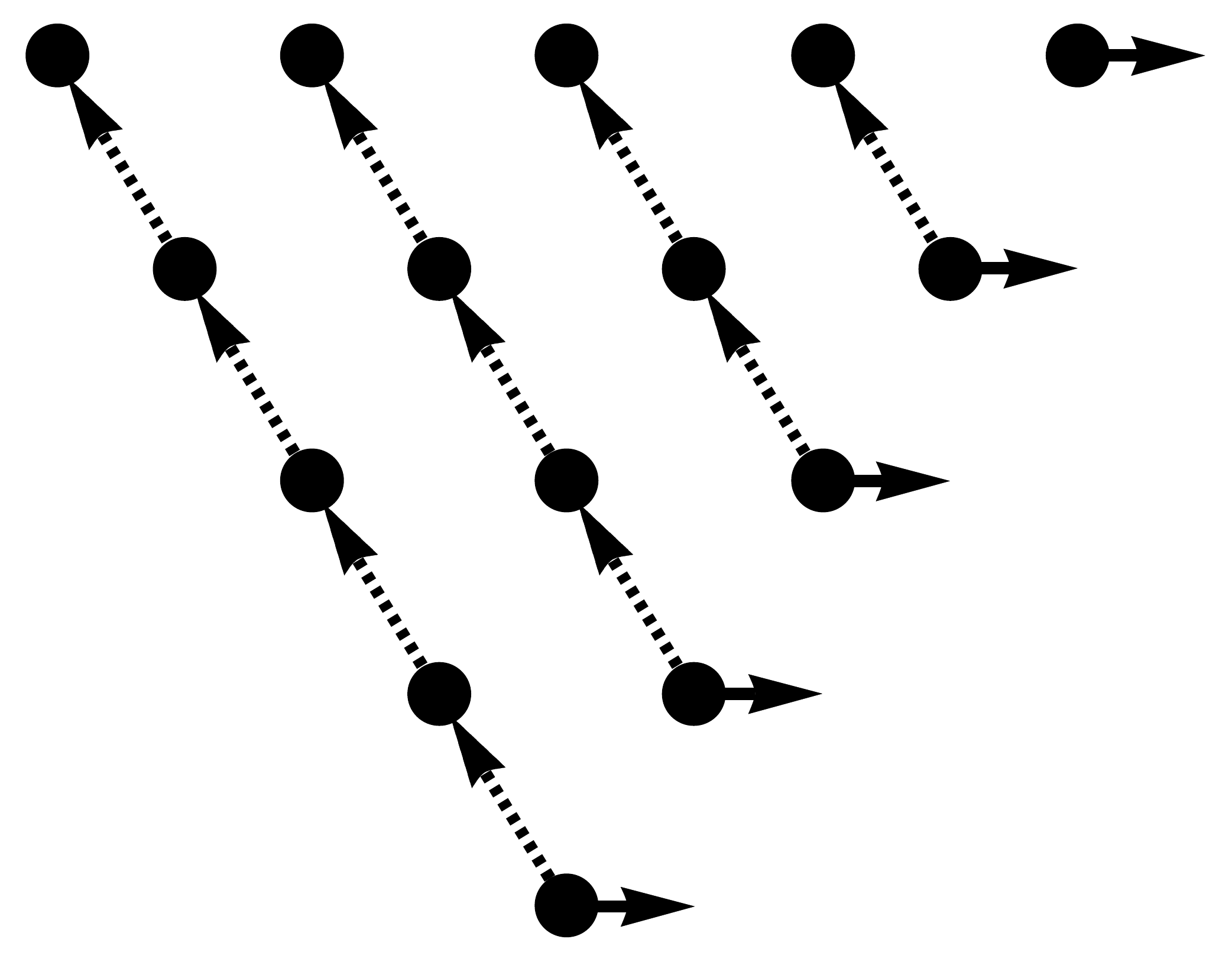}
			&
			\includegraphics[width=94.5pt]{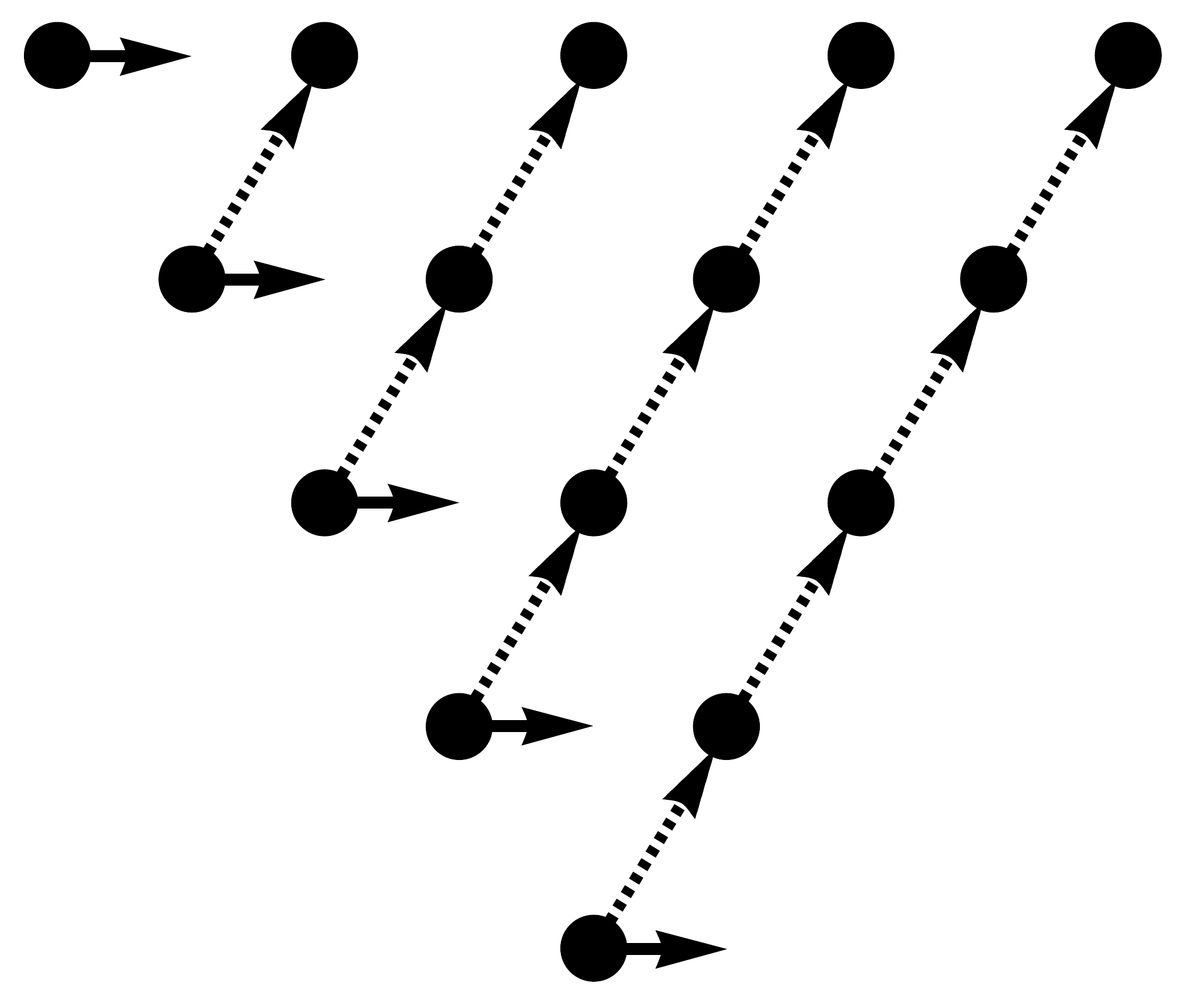}
			\\
			(a)&(b)
			\\
			\hline
		\end{tabular}
	\end{center}
  	\caption{
  	Schematic pictures of 
  	(a) row; (b) column
  	insertion algorithms 
  	and the corresponding multivariate dynamics
  	(see Fig.~\ref{fig:big_rs} for more
  	explaination).}
  	\label{fig:row_col}
\end{figure}

Let us denote 
the
row (resp. column)
insertion Schur-multivariate
dynamics $\bq^{(N)}_{\RSD[\boldsymbol h]}$
with $\boldsymbol h=(1,\ldots,1)$
(resp. $\boldsymbol h=(1,2,\ldots,N)$)
by $\bq^{(N)}_{\mathit{row}}$
(resp. $\bq^{(N)}_{\mathit{col}}$).
These dynamics 
have already appeared in the literature.
Application of Robin\-son--Schensted
insertion algorithms to random words 
can be traced back to \cite{Vershik1986}.
The column insertion dynamics
$\bq^{(N)}_{\mathit{col}}$
was introduced and
studied 
in \cite{OConnell2003Trans}, \cite{OConnell2003}.
Questions related to application of
insertion algorithms (including
the general Robinson--Schensted--Knuth 
correspondence) 
to random input 
were considered in, e.g., 
\cite{baik1999distribution},
\cite{johansson2000shape},
\cite[\S5]{Johansson2005lectures}, and
\cite{ForresterRains2007}.
See 
\cite[\S6.3]{BorodinGorinSPB12}
for a discussion of the row insertion dynamics, 
and also
\cite{COSZ2011} for a ``geometric''
(Whittaker process) analogue (cf.
\S \ref{sub:limit_of_our_formulas_as_qto1_}
below).

\begin{remark}
	A possibly interesting connection is that
	schematic pictures corresponding to
	$\boldsymbol h$-insertions
	(like Fig.~\ref{fig:big_rs} and \ref{fig:row_col})
	also appear in the description of 
	simple vertices of Gelfand--Tsetlin
	polytopes, e.g., see \cite{SmirnovSchubert2012}. 
	We thank Evgeny Smirnov for pointing out 
	this appearance of the same schematic pictures.
\end{remark}


\subsection{Fundamental nearest neighbor dynamics
in the Schur case} 
\label{sub:fundamental_nearest_neighbor_dynamics_in_the_schur_case}

One can completely characterize RSK-type
nearest neighbor Markov dynamics
on interlacing arrays in the Schur ($q=t$)
case:

\begin{proposition}\label{prop:RS_Schur}
	Let $\bq^{(N)}$	be 
	the generator of
	a Schur-multivariate 
	RSK-type 
	nea\-rest-neighbor
	dynamics (Definitions \ref{def:nn_dynamics} and 
	\ref{def:RS_type})
	which has nonnegative jump rates and 
	probabilities of triggered moves.
	Then $\bq^{(N)}$ can be expressed as a convex
	linear combination of 
	generators of
	RSK-type fundamental
	dynamics:
	\begin{align}
		\bq^{(N)}(\boldsymbol\la,
		\boldsymbol\nu)=
		\sum_{\boldsymbol h}
		\co_{\boldsymbol h}(\boldsymbol\la,\boldsymbol\nu)
		\bq^{(N)}_{\RSD[\boldsymbol h]}
		(\boldsymbol\la,\boldsymbol\nu).
		\label{RS_type_decomp}
	\end{align}
	Here
	$\boldsymbol\la,\boldsymbol\nu\in\GT(N)$,
	the sums above are taken over
	$\boldsymbol h=(h^{(1)},\ldots,h^{(N)})$, $1\le 
	h^{(k)}\le k$,
	and the coefficients $\co_{\boldsymbol h}$ have the form
	\begin{align}\label{co_h_product1}
		\co_{\boldsymbol h}(\boldsymbol\la,\boldsymbol\nu)=
		\co^{(2)}_{h^{(2)}}(\nu^{(1)},\la^{(2)})
		\ldots
		\co^{(N)}_{h^{(N)}}(\nu^{(N-1)},\la^{(N)}),
	\end{align}
	where one has (for any $k=2,\ldots,N$)
	\begin{align}\label{co_h_product2}
		\co^{(k)}_{h}
		(\nu^{(k-1)},\la^{(k)})\ge0,\quad
		1\le h\le k,
		\qquad
		\sum_{h=1}^{k}
		\co^{(k)}_{h}
		(\nu^{(k-1)},\la^{(k)})\equiv1.
	\end{align}
\end{proposition}
\begin{proof}
	The existence of decomposition
	essentially follows from 
	Theorem \ref{thm:characterization_NN},
	but one has to turn
	the notion of mixing 
	\eqref{mixing_generators}
	into a
	convex combination
	\eqref{RS_type_decomp}.

	Fix $k$ 
	and signatures
	$(\bar\nu,\la)\in\GT_{(k-1;k)}$.
	The
	RSK-type
	dynamics $\bq^{(N)}$
	corresponds to 
	quantities $w_j(\bar\nu,\la)\ge0$
	and $r_j(\bar\nu,\la)\ge0$ 
	(as explained in 
	\S \ref{sub:parametrization_and_linear_equations})
	which satisfy
	system \eqref{rlw_system_of_equations}
	with $c_j(\bar\nu,\la)=1$ for all $j$.
	By Proposition \ref{prop:const_pp} with $C=1$, 
	this solution $(w,r)$
	can be written as a
	linear
	combination (with coefficients
	summing to one) of $k$ solutions 
	corresponding to the RSK-type fundamental dynamics.
	The coefficients
	of this linear combination are equal to $w_j$
	and hence are nonnegative. 
	Combining these decompositions
	for all $(\bar\nu,\la)$ 
	and at all slices $\GT_{(k-1;k)}$, 
	we see that the coefficients
	in \eqref{RS_type_decomp} may be written 
	as \eqref{co_h_product1}--\eqref{co_h_product2}.
	This concludes the proof.
\end{proof}

\begin{remark}\label{rmk:RS_Schur_unique}
	For any coefficients 
	$\co_{\boldsymbol h}(\boldsymbol\la,\boldsymbol\nu)$
	given by \eqref{co_h_product1}
	and satisfying \eqref{co_h_product2},
	the right-hand side of \eqref{RS_type_decomp}
	is an honest multivariate RSK-type dynamics.

	The coefficients 
	$\co^{(k)}_{h}(\bar\nu,\la)$, 
	$k=2,\ldots,N$,
	$1\le h\le k$,
	(which constitute 
	the $\co_{\boldsymbol h}$'s)
	are determined uniquely by $\bq^{(N)}$ 
	if particles of the 
	array
	$\boldsymbol\la$ 
	are ``apart'' from each other.
	Using
	Proposition~\ref{prop:const_pp},
	it is possible to formulate a precise
	uniqueness statement when particles become close.
	We do not pursue this direction in the present paper.
\end{remark}

\begin{figure}[phtb]
	\begin{center}
		\begin{tabular}{|r|r|r|}
			\hline
			\rule{0pt}{90pt}
			\includegraphics[width=95pt]{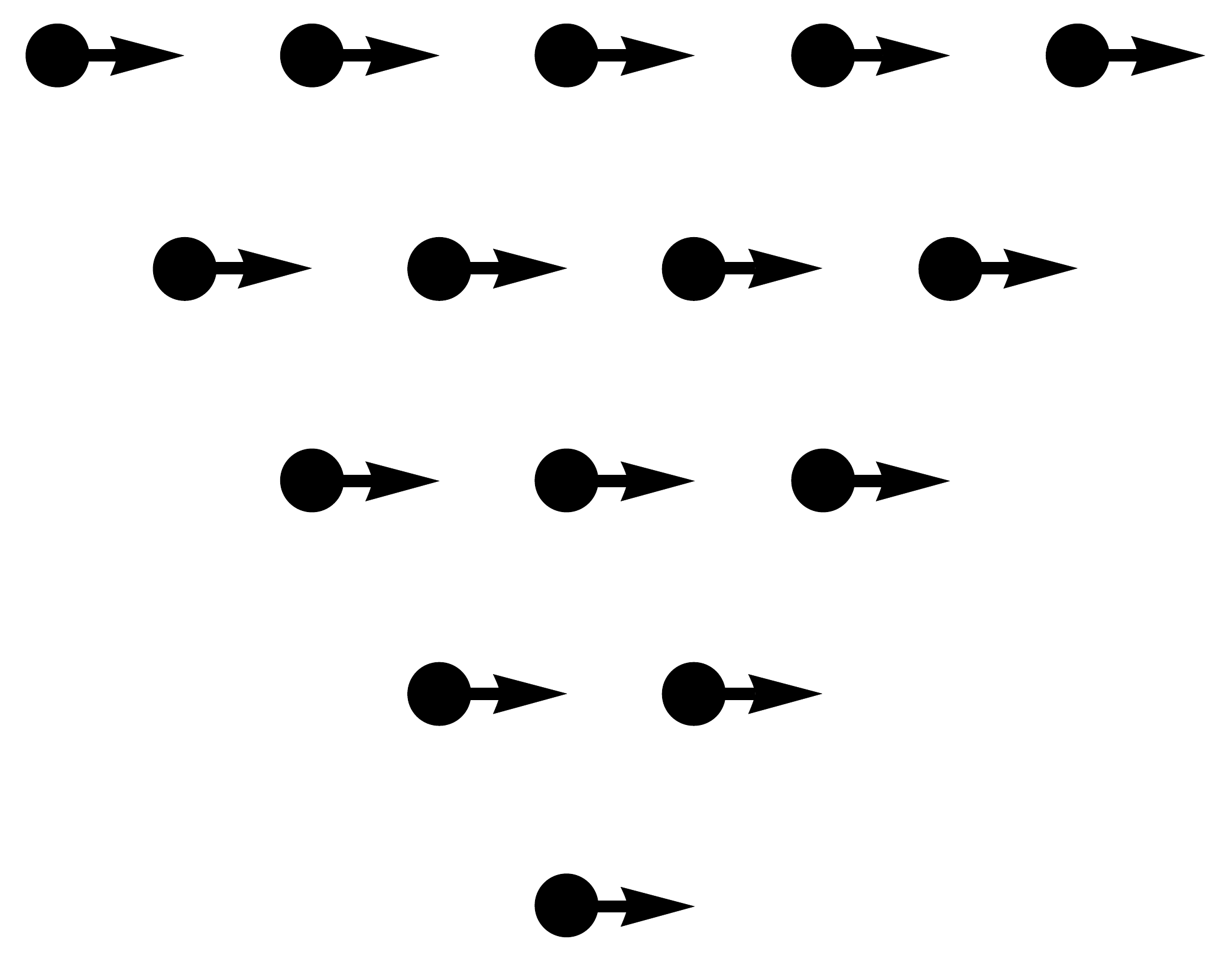}
			&
			\includegraphics[width=95pt]{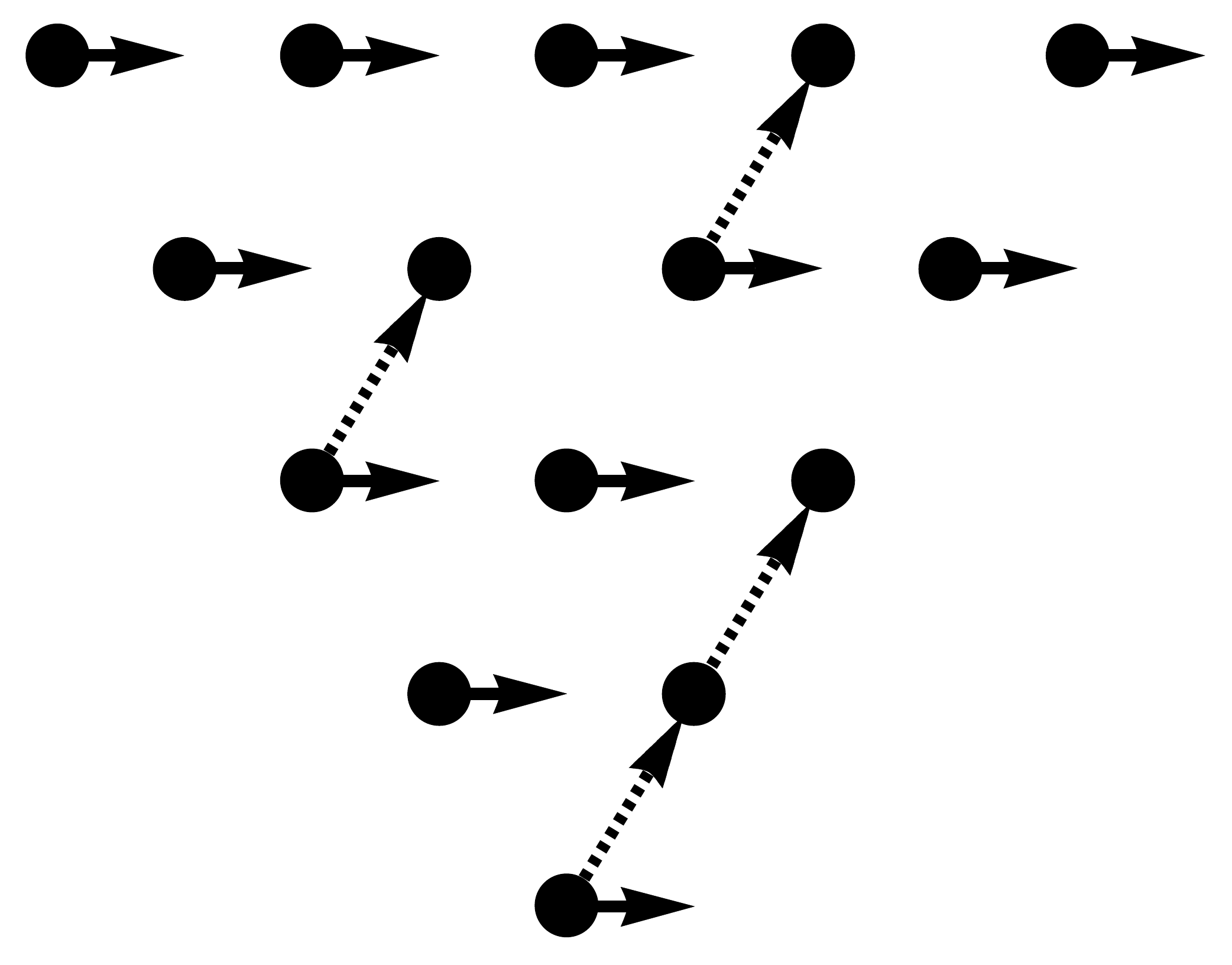}
			&
			\includegraphics[width=95pt]{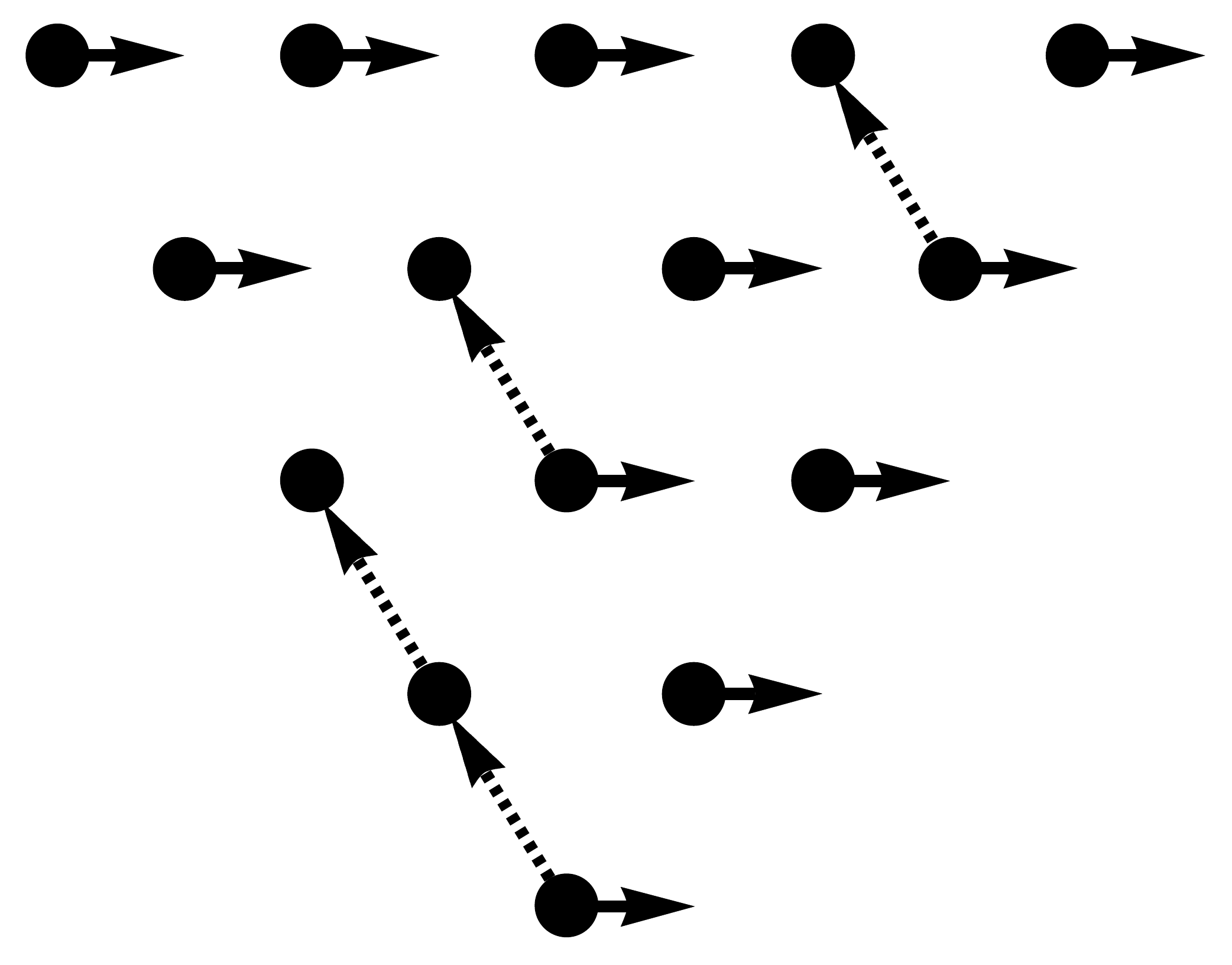}
			\\
			(a)&(b)&(c)
			\\
			\hline
		\end{tabular}
	\end{center}
  	\caption{Schematic pictures of 
  	(a) push-block; 
  	(b) right-pushing 
  	with $\boldsymbol h=(1,1,3,2)$;
  	(c) left-pulling with 
  	$\boldsymbol h=(1,2,2,1)$
  	insertion 
  	algorithms 
  	and the corresponding multivariate dynamics
  	(see Fig.~\ref{fig:big_rs} for more
  	explaination).}
  	\label{fig:pb_left_right}
\end{figure}

Let us now briefly look at all the remaining
(non-RSK-type) fundamental
`dynamics' 
introduced in \S \ref{sub:fundamental_dynamics_}.
At the general Macdonald 
(or $q$-Whittaker) level 
we are not guaranteed that these `dynamics' have 
nonnegative jump rates and probabilities
of triggered moves 
(cf. \S \ref{ssub:fundamental_dynamics_} 
below). 
However, it can be readily checked that
this nonnegativity holds under the
Schur degeneration (i.e.,
one has a statement similar to 
Proposition \ref{prop:RS_nonnegative}). 

The push-block dynamics $\bq^{(N)}_{\PBD}$
in the Schur case was introduced and studied
in \cite{BorFerr2008DF}. 
Under this dynamics, all particles at all levels
jump to the right by one 
if not blocked
(the jump rate at level $k$ is $a_k$),
and the only pushing mechanism
is the short-range one (see 
\S \ref{sub:when_a_jumping_particle_has_to_push_its_immediate_upper_right_neighbor}).

The right-pushing
and left-pulling dynamics
$\bq^{(N)}_{\RD[\boldsymbol h]}$
and 
$\bq^{(N)}_{\LD[\boldsymbol h]}$
indexed by $(N-1)$-tuples
$\boldsymbol h$
\eqref{RL_h_index}
may be viewed as certain 
``minimal perturbations''
of~$\bq^{(N)}_{\PBD}$. 
Under the 
dynamics 
$\bq^{(N)}_{\RD[\boldsymbol h]}$,
on each slice $\GT_{(k-1;k)}$ 
the behavior of all the particles
is the same as under $\bq^{(N)}_{\PBD}$, 
except that the particle
$\la^{(k)}_{\nf(h^{(k-1)})}$ now has zero 
independent
jump rate,
but this is ``compensated'' by
the long-range pushing (with probability one)
of $\la^{(k)}_{\nf(h^{(k-1)})}$ by
$\la^{(k-1)}_{h^{(k-1)}}$.
Typically (when all particles are ``apart''),
one has $\nf(h^{(k-1)})=h^{(k-1)}$ 
(see \eqref{nf_next_free_particle} for a 
general definition).
The left-pulling dynamics 
$\bq^{(N)}_{\LD[\boldsymbol h]}$
can be described
in a similar way, see
also \S \ref{ssub:conclusion}.
See Fig.~\ref{fig:pb_left_right} 
for schematic pictures 
of $\bq^{(N)}_{\PBD}$,
$\bq^{(N)}_{\RD[\boldsymbol h]}$,
and~$\bq^{(N)}_{\LD[\boldsymbol h]}$.


\subsection{$\boldsymbol h$-Robinson--Schensted correspondences} 
\label{sub:_boldsymbol_h_robinson_schensted_correspondences}

This subsection
is devoted to new 
Robinson--Schensted-type
correspondences between words and pairs of
Young tableaux which arise from our
$\boldsymbol h$-insertions introduced in 
\S \ref{sub:rs_type_fundamental_dynamics_in_the_schur_case_and_deterministic_insertion_algorithms}.

Traditionally,
the most general \emph{Robinson--Schensted--Knuth}
(RSK) correspondences associate to
a matrix with nonnegative integer entries
a pair of semistandard Young tableaux
(Definition \ref{def:SSYT})
of the same shape. 

We are working in a less general
setting when this
nonnegative integer 
matrix
is in fact a zero-one
matrix having exactly one ``1''
in each column.
If this matrix has $N$ rows,
then it may be 
equivalently
described as a
word in the alphabet $\{1,\ldots,N\}$: 
each $i$th 
letter of the word corresponds to the 
row number of the entry 1
in the $i$th column of the matrix.
The length of the word is equal to the number
of columns.
A \emph{Robinson--Schensted} 
(RS)
correspondence
associates with such a word a pair
of Young tableaux 
$(\Ptab,\Qtab)$
of the same shape,
where $\Ptab$
is a
\emph{semistandard} tableau 
over the alphabet
$\{1,\ldots,N\}$, and $\Qtab$ is a
\emph{standard} tableau
(\S \ref{sub:semistandard_young_tableaux}).

For more detail on
Robinson--Schensted(--Knuth)
correspondences and 
insertion algorithms
we refer to 
\cite{sagan2001symmetric}, 
\cite{Stanley1999},
\cite{fulton1997young},
and 
\cite{Knuth1973art3}.

\subsubsection{Definition and
invertibility of $\boldsymbol h$-correspondences} 
\label{ssub:definition_of_boldsymbol_h_correspondences}

We assume that an integer $N$
and an $N$-tuple $\boldsymbol h$ 
as in \eqref{RS_h_index}
are fixed.
Let $\boldsymbol w=w_1 \ldots w_n$ be 
a word of length
$n$ in the alphabet $\{1,\ldots,N\}$.
Start with an empty semistandard
Young tableau $\varnothing$, and
$\boldsymbol h$-insert 
(\S \ref{sub:rs_type_fundamental_dynamics_in_the_schur_case_and_deterministic_insertion_algorithms})
into it 
the letters $w_1,\ldots,w_n$
one by one. Thus, we 
get a semistandard Young tableau
\begin{align*}
	\Ptab^{\boldsymbol h}_{\boldsymbol w}
	:=
	\ins^{\boldsymbol h}_{w_n}
	\ldots
	\ins^{\boldsymbol h}_{w_1}
	\varnothing
\end{align*}
having $n$ boxes. 

The insertion of every letter $w_i$
changes the shape of the semistandard
tableau by adding exactly one box. 
Let us \emph{record} this added box 
with the help of another (now standard) 
Young tableau. 
Initially, the recording tableau is also
empty. 
When
each new letter
$w_i$ is $\boldsymbol h$-inserted
into $\ins^{\boldsymbol h}_{w_{i-1}}
\ldots\ins^{\boldsymbol h}_{w_1}\varnothing$,
one adds a box at the same position 
in the recording tableau
$\Qtab^{\boldsymbol h}_{w_1\ldots w_{i-1}}$, and 
puts the letter $i$ into this 
new box of the tableau
$\Qtab^{\boldsymbol h}_{w_1\ldots w_{i}}$.

In this way, to every 
word $\boldsymbol w=w_1\ldots w_n$
over the alphabet 
$\{1,\ldots,N\}$
one associates a pair
$(\Ptab^{\boldsymbol h}_{\boldsymbol w},
\Qtab^{\boldsymbol h}_{\boldsymbol w})$
of Young tableaux of the same shape, 
where $\Ptab^{\boldsymbol h}_{\boldsymbol w}$ is a semistandard,
and $\Qtab^{\boldsymbol h}_{\boldsymbol w}$
is a standard tableau.

\begin{theorem}\label{thm:h_RS_correspondence}
	For each $\boldsymbol h$, 
	we get a bijection
	\begin{align}\label{h_RS_correspondence}
		\boldsymbol w
		=w_1 \ldots w_n
		\xrightarrow{\text{\rm{}\sf{}$\boldsymbol h$-RS}}
		(\Ptab^{\boldsymbol h}_{\boldsymbol w},
		\Qtab^{\boldsymbol h}_{\boldsymbol w})
	\end{align}
	between the set of words of length $n$
	in the alphabet $\{1,\ldots,N\}$
	and the set of pairs of Young tableaux
	$(\Ptab,\Qtab)$
	having the same shape (with $n$ boxes),
	where $\Ptab$ is semistandard
	and $\Qtab$ is standard.
\end{theorem}
We will call
this bijection the 
\textit{$\boldsymbol h$-Robinson--Schensted
correspondence}, and denote it
by \textsf{$\boldsymbol h$-RS}.
\begin{proof}
	It suffices to observe that 
	the correspondence in \eqref{h_RS_correspondence}
	admits an inverse. 
	One can construct such an inverse letter by letter.

	Let us argue in
	terms of the interlacing arrays.
	The recording tableau $\Qtab$ 
	provides information
	where the insertion
	trajectory
	in the array  
	(like the one 
	in \eqref{big_rs_I2lambda}
	or \eqref{big_rs_I2I2lambda}) 
	ends at level $N$.
	Indeed, what happens at level $N$
	is responsible for the change of shape 
	of the tableaux $(\Ptab,\Qtab)$.

	It suffices then to reconstruct the 
	insertion trajectory in the interlacing 
	array~$\boldsymbol\nu$,
	and thus produce (in a unique way) 
	the previous
	array $\boldsymbol\la$ and
	the letter $m$ such that 
	$\boldsymbol\nu=
	\ins^{\boldsymbol h}_m\boldsymbol\la$.
	One can do this consecutively on each slice, 
	from the 
	top level $N$ down to level~1. 
	
	So now we can assume 
	that on some slice $\GT_{(k-1;k)}$,
	$k=2,\ldots,N$,
	the particle $\nu^{(k)}_{j}$
	moved to the right. 
	Set $\la^{(k)}_{j}:=\nu^{(k)}_{j}-1$.
	If $\la^{(k)}_{j}<\nu^{(k-1)}_{j}$,
	then the short-range pushing
	(\S \ref{sub:when_a_jumping_particle_has_to_push_its_immediate_upper_right_neighbor})
	took place, and
	$\nu^{(k-1)}_{j}$ 
	is the particle that moved at level $k-1$.
	Proceed to the slice $\GT_{(k-2;k-1)}$.

	Otherwise, set $\nu=\nu^{(k)}$,
	$\la=\nu-\de_j$, and $\bar\nu=\nu^{(k-1)}$
	(i.e., we use the notational
	conventions of \S \ref{sub:specialization_of_formulas_from_s_sub:sequential_update_dynamics_in_continuous_time}).
	We have $\bar\nu\prec\la$.
	Let us write the equation 
	of the linear system
	\eqref{rlw_system_of_equations}
	that corresponds to the move $\la\to\nu=\la+\de_j$.
	It has the form ($c_{j-1}=1$ because this
	is an RSK-type dynamics)
	\begin{align}\label{h_RS_correspondence_proof1}
		S_j=w_j+(1-r_{j-1})T_{j-1}+
		r_{\nf^{-1}(j)}T_{\nf^{-1}(j)},
	\end{align}
	where $\nf^{-1}(j)$ is defined in the end of 
	\S \ref{ssub:right_pushing_fundamental_solutions}.
	If $j=1$ or $j=k$, then the above equation
	will of course contain fewer terms, cf. 
	\eqref{rlw_system_of_equations}.
	All quantities in \eqref{h_RS_correspondence_proof1}
	depend on $(\bar\nu,\la)$. Using \eqref{T_S_Schur}
	and the fact that in our situation
	$S_j\ne0$,
	we can rewrite \eqref{h_RS_correspondence_proof1} as
	\begin{align}\label{h_RS_correspondence_proof2}
		w_j+(1-r_{j-1})+
		r_{\nf^{-1}(j)}=1.
	\end{align}
	By the properties of 
	the fundamental RSK-type dynamics
	(in particular, by \eqref{r_j_RSD_Schur}),
	we conclude that exactly one
	of the summands 
	$w_j$, $1-r_{j-1}$, and $r_{\nf^{-1}(j)}$
	in \eqref{h_RS_correspondence_proof2}
	is one, and the other two are zero.
	Which of the summands is nonzero
	depends only on 
	our sequence $\boldsymbol h$.

	Using this information, we can reconstruct
	the insertion trajectory
	as follows (cf. \S \ref{sub:rs_type_dynamics_and_random_insertions}):
	\begin{enumerate}[1.]
		\item If $w_j=1$, then 
		the particle
		$\nu_{j}^{(k)}$ 
		performed an 
		independent jump. Thus, 
		the letter that 
		was inserted is $m=k$, and 
		the insertion trajectory 
		\textbf{ends} at level~$k$.
		\item If $1-r_{j-1}=1$, then 
		$\nu_{j}^{(k)}$ was pulled 
		by the particle $\nu_{j-1}^{(k-1)}$
		that moved at level $k-1$.
		Proceed to the lower slice $\GT_{(k-2;k-1)}$.
		\item If $r_{\nf^{-1}(j)}=1$, then 
		$\nu_{j}^{(k)}$ was long-range pushed
		(with a possible move donation)
		by the particle $\nu_{\nf^{-1}(j)}^{(k-1)}$
		that moved at level $k-1$.
		With this 
		information, we 
		proceed to the lower slice $\GT_{(k-2;k-1)}$.
	\end{enumerate}
	Thus, one can reconstruct
	the insertion trajectory
	level by level
	in a unique way.
	Continuing 
	this procedure for each letter, 
	it is possible 
	to reconstruct the 
	word $\boldsymbol w$
	from the tableaux 
	$(\Ptab^{\boldsymbol h}_{\boldsymbol w},
	\Qtab^{\boldsymbol h}_{\boldsymbol w})$.
	Of course, this reconstruction 
	depends on $\boldsymbol h$.
	This concludes the proof.
\end{proof}

Thus, for large enough $n$, 
we get $N!$ \emph{different} bijections 
(indexed by $\boldsymbol h\in\{1\}\times\{1,2\}
\times \ldots\times\{1,\ldots,N\}$)
between words of length
$n$ in the alphabet 
$\{1,\ldots,N\}$,
and pairs of semistandard
Young tableaux. 
Two of these bijections are well-known, these
are the row and column Robinson--Schensted
correspondences 
arising for $\boldsymbol h=(1,\ldots,1)$
and $\boldsymbol h=(1,2,\ldots,N)$, 
respectively (cf. Proposition \ref{prop:row_col}).
For future convenience let us denote these
two distinguished correspondences
by {\rm{}\sf{}row-RS} and {\rm{}\sf{}col-RS}.

All our $\boldsymbol h$-Robinson--Schensted 
correspondences 
fall under the general formalism of Fomin
\cite{fomin1995schur}. 
They seem to form a relatively tractable subclass 
in the variety of general bijections constructed
in \cite{fomin1995schur}.


\subsubsection{Action of {\sf{}$\boldsymbol h$-RS} on permutation words} 
\label{ssub:action_of_boldsymbol_h_rs_on_permutation_words}

Let us now present a few \emph{experimental} observations 
about our $N!$ 
correspondences \textsf{$\boldsymbol h$-RS}.
For simplicity, we reduce the set of allowed
words $\boldsymbol w$
to
\emph{permutation words} 
$w_1 \ldots w_N$
of length $N$
in which each letter $1,\ldots,N$ appears only once.
These words may be identified with permutations
of $\{1,\ldots,N\}$
(the image of $i$ is $w_i$) and, equivalently, 
with permutation
matrices (cf. the beginning of 
\S \ref{sub:_boldsymbol_h_robinson_schensted_correspondences}).
Denote by 
$\mathfrak{W}_{N}$
the set of all permutation
words of length $N$.
Under each of the \textsf{$\boldsymbol h$-RS},
such words correspond to 
pairs of \emph{standard}
Young tableaux of the same shape.

\medskip

\noindent\textbf{1.} Our first observation is triggered by the 
connection of the two distinguished
Robinson--Schensted
correspondences 
{\rm{}\sf{}row-RS} and {\rm{}\sf{}col-RS}
with
increasing and decreasing
subsequences in the word $\boldsymbol w$, 
e.g.,
see \cite[Ch. III]{sagan2001symmetric}.
In fact, this connection implies:
\begin{proposition}[\cite{Schensted1961}]
	Take a word $\boldsymbol w=w_1 \ldots w_N$, and 
	let
	$\boldsymbol w'=w_N \ldots w_1$ be the same word in
	reverse order. Then the application
	of {\rm{}\sf{}row-RS} to $\boldsymbol w$
	gives the same pair of standard tableaux as
	the application
	of {\rm{}\sf{}col-RS} to $\boldsymbol w'$.
\end{proposition}
Moreover \cite{Schensted1961}, the length of the
first row 
of the tableau $\Ptab^{\mathit{row}}_{\boldsymbol w}$
is equal to the length of the longest increasing 
subsequence in $\boldsymbol w$, while
the
first row 
of $\Ptab^{\mathit{col}}_{\boldsymbol w}$
is equal to the 
length of the longest \emph{decreasing} 
(=~increasing in another linear order 
$1\succ 2\succ\ldots\succ N$) 
subsequence
of $\boldsymbol w$. Thus,
one can think that {\rm{}\sf{}col-RS}
is somehow related to {\rm{}\sf{}row-RS}
modulo 
another linear order on the letters:
$1\succ 2\succ\ldots\succ N$ instead of $1<2<\ldots<N$.

However, all our $N!$ correspondences 
\textsf{$\boldsymbol h$-RS}
\emph{do not} admit an interpretation 
in terms of 
different linear orders 
on $\{1,\ldots,N\}$, as one could suspect.
One can see this by taking $N=4$ and 
considering the word $\boldsymbol w=3241$.
All $24=4!$ 
correspondences 
\textsf{$\boldsymbol h$-RS} 
applied
to this 
word produce Young tableaux 
$(\Ptab^{\boldsymbol h}_{\boldsymbol w},
\Qtab^{\boldsymbol h}_{\boldsymbol w})$
whose first row has length at most 3.
On the other hand, the length of the 
longest increasing
subsequence in $\boldsymbol w=3241$
in the linear order $3\prec 2\prec 4\prec 1$ 
is of course equal to~4.

\medskip

\noindent\textbf{2.} For each 
$\boldsymbol h\in
\{1\}\times\{1,2\}
\times \ldots\times\{1,\ldots,N\}$,
let us consider the composition of 
\textsf{$\boldsymbol h$-RS}
with the inverse of 
the usual row insertion
correspondence \textsf{row-RS}. 
Thus, for each $\boldsymbol h$ we get a 
bijection of $\mathfrak{W}_{N}$ with itself defined by
\begin{align*}
	\mathfrak{W}_{N}
	\xrightarrow{\text{\rm{}\sf{}$\boldsymbol h$-RS}}
	\{(\Ptab,\Qtab)\}
	\xrightarrow{(\text{\rm{}\sf{}row-RS})^{-1}}
	\mathfrak{W}_{N},\qquad
	\boldsymbol w\mapsto \sm_{\boldsymbol h}(\boldsymbol w),
\end{align*}
where 
$\{(\Ptab,\Qtab)\}$
is the set
of pairs of standard
Young tableaux of the same shape.

Let us write down the maps 
$\sm_{\boldsymbol h}$
for $N=3$.
We will list images of 
\begin{align*}
	\mathfrak{W}_{3}=
	{\scriptsize\begin{Bmatrix}
		1&2&3\\
		1&3&2\\
		2&1&3\\
		2&3&1\\
		3&1&2\\
		3&2&1
	\end{Bmatrix}}
\end{align*}
under each of the 
maps $\sm_{\boldsymbol h}$. We have
\begin{align*}&
	\sm_{(1,1,1)}\mathfrak{W}_{3}=
	{\scriptsize\begin{Bmatrix}
		1&2&3\\
		1&3&2\\
		2&1&3\\
		2&3&1\\
		3&1&2\\
		3&2&1
	\end{Bmatrix}},\qquad
	\sm_{(1,1,2)}\mathfrak{W}_{3}=
	{\scriptsize\begin{Bmatrix}
		1&3&2\\
		3&1&2\\
		2&1&3\\
		3&2&1\\
		1&2&3\\
		2&3&1
	\end{Bmatrix}},\qquad
	\sm_{(1,1,3)}\mathfrak{W}_{3}=
	{\scriptsize\begin{Bmatrix}
		1&3&2\\
		3&1&2\\
		3&2&1\\
		2&1&3\\
		1&2&3\\
		2&3&1
	\end{Bmatrix}},\\
	\rule{0pt}{32pt}
	&
	\sm_{(1,2,1)}\mathfrak{W}_{3}=
	{\scriptsize\begin{Bmatrix}
		2&1&3\\
		2&3&1\\
		1&2&3\\
		1&3&2\\
		3&2&1\\
		3&1&2
	\end{Bmatrix}},\qquad
	\sm_{(1,2,2)}\mathfrak{W}_{3}=
	{\scriptsize\begin{Bmatrix}
		2&1&3\\
		3&2&1\\
		1&3&2\\
		3&1&2\\
		2&3&1\\
		1&2&3
	\end{Bmatrix}},\qquad
	\sm_{(1,2,3)}\mathfrak{W}_{3}=
	{\scriptsize\begin{Bmatrix}
		3&2&1\\
		2&1&3\\
		1&3&2\\
		3&1&2\\
		2&3&1\\
		1&2&3
	\end{Bmatrix}}.
\end{align*}
We see that the maps 
$\sm_{\boldsymbol h}\colon
\mathfrak{W}_{N}\to\mathfrak{W}_{N}$
are pairwise distinct.
Moreover,
$\sm_{\boldsymbol h}$
preserves the obvious 
structure of the symmetric group
on $\mathfrak{W}_{N}$ (when each permutation
word is identified with the corresponding
permutation) 
only in the case $\boldsymbol h=(1,\ldots,1)$.

\medskip

\noindent\textbf{3.} Instead of looking at the 
individual maps
$\sm_{\boldsymbol h}$,
it could be more natural to consider 
the subgroup $G(N)\subset\mathfrak{S}(N!)$
generated by all the maps $\sm_{\boldsymbol h}$.
Here we understand
$\mathfrak{S}(N!)$
as the group of permutations of the set
$\mathfrak{W}_{N}$.

The orders of the groups $G(N)$ for small $N$
can be computed using a computer algebra system.
They are as follows:
\begin{align*}
	|G(2)|=2
	,\quad
	|G(3)|=2\cdot(3!)^2
	,\quad
	|G(4)|=2\cdot(12!)^2
	,\quad
	|G(5)|=2\cdot(60!)^2.
\end{align*}
We conjecture that in general,
\begin{align*}
	|G(N)|=2\cdot \Big(\big(\tfrac12
	N!\big)!\Big)^{2}.
\end{align*}

\medskip

\noindent\textbf{4.} 
We conjecture that the group $G(N)$ 
may be identified with a subgroup of permutations
of $\mathfrak{W}_N$
which have the following form.
For each $N$, there should exist
a representation of the set 
$\mathfrak{W}_N$
as a disjoint union 
$\mathfrak{W}_N'\sqcup \mathfrak{W}_N''$
such that 
$|\mathfrak{W}_N'|=|\mathfrak{W}_N''|=\frac12 N!$.
Fix any one-to-one map $\chi\colon
\mathfrak{W}_N'\to \mathfrak{W}_N''$.
Then $G(N)$ is generated by $\chi$ and
by all the permutations of 
$\mathfrak{W}_N$ which permute the parts
$\mathfrak{W}_N'$ and $\mathfrak{W}_N''$
separately.

For $N=3$ and $4$ we are able to 
explicitly write down the above splitting 
$\mathfrak{W}_N=\mathfrak{W}_N'\sqcup \mathfrak{W}_N''$.
In the case $N=3$ it looks as 
\begin{align*}
	\mathfrak{W}_3'=
	{\scriptsize\begin{Bmatrix}
		1&2&3\\
		1&3&2\\
		3&1&2
	\end{Bmatrix}},\qquad
	\mathfrak{W}_3''=
	{\scriptsize\begin{Bmatrix}
		3&2&1\\
		2&3&1\\
		2&1&3
	\end{Bmatrix}}.
\end{align*}
For $N=4$, the splitting is
given by
\begin{align*}
	\mathfrak{W}_4'=
	{\scriptsize\begin{Bmatrix}
		1 & 2 & 3 & 4 \\
		1 & 2 & 4 & 3 \\
		1 & 3 & 2 & 4 \\
		1 & 3 & 4 & 2 \\
		1 & 4 & 2 & 3 \\
		1 & 4 & 3 & 2 \\
		3 & 1 & 2 & 4 \\
		3 & 1 & 4 & 2 \\
		3 & 4 & 1 & 2 \\
		4 & 1 & 2 & 3 \\
		4 & 1 & 3 & 2 \\
		4 & 3 & 1 & 2 
	\end{Bmatrix}},\qquad
	\qquad
	\mathfrak{W}_4''=
	{\scriptsize\begin{Bmatrix}
		4 & 3 & 2 & 1 \\
		3 & 4 & 2 & 1 \\
		4 & 2 & 3 & 1 \\
		2 & 4 & 3 & 1 \\
		3 & 2 & 4 & 1 \\
		2 & 3 & 4 & 1 \\
		4 & 2 & 1 & 3 \\
		2 & 4 & 1 & 3 \\
		2 & 1 & 4 & 3 \\
		3 & 2 & 1 & 4 \\
		2 & 3 & 1 & 4 \\
		2 & 1 & 3 & 4
	\end{Bmatrix}}.
\end{align*}
In both cases $N=3$ and $4$,
one can choose
a one-to-one map $\chi\colon
\mathfrak{W}_N'\to \mathfrak{W}_N''$ 
which acts by rewriting each word 
in the reverse order.




\section[Multivariate dynamics in the $q$-Whittaker case and $q$-PushTASEP]
{Multivariate dynamics in the $q$-Whittaker case\\ and $q$-PushTASEP} 
\label{sec:multivariate_dynamics_in_the_q_whittaker_case_and_q_pushasep_}

Here we will discuss 
multivariate dynamics on interlacing
arrays in the $q$-Whit\-taker case
(we will sometimes refer to them as to 
\emph{$q$-Whittaker-multivariate} dynamics).
That is,
throughout the whole section we will assume that 
$t=0$ and $0<q<1$, where $q$ and $t$ are 
the Macdonald parameters.

In \S \ref{sub:limit_of_our_formulas_as_qto1_}
we write down (formal) scaling limits 
as  $q\nearrow1$
of our
$q$-Whittaker-multivariate
dynamics. This leads to diffusions
on $\R^{\frac{N(N+1)}2}$, and we write down systems 
of stochastic differential equations (SDEs) for them.

The multivariate dynamics in the $q$-Whittaker case
($0<q<1$)
lead to discovery 
of a new one-dimensional interacting particle system
which we call \textit{$q$-PushTASEP} (\S \ref{sub:taseps}).
We briefly indicate its relevance to the 
O'Connell--Yor semi-discrete directed polymer
\cite{OConnellYor2001}, \cite{Oconnell2009_Toda}.

\subsection{$q$-quantities} 
\label{sub:_q_quantities}

Before discussing
$q$-Whittaker-multivariate dynamics, 
let us write down the quantities $T_i$ and $S_j$ 
\eqref{T_i}--\eqref{S_j}
specialized for $t=0$. 
These quantities 
enter our linear conditions
on multivariate dynamics 
(Theorem \ref{thm:general_multivariate} and 
system \eqref{rlw_system_of_equations}).
We have
\begin{align}
	T_i(\bar\nu,\la)&=
	\frac{(1-q^{\bar\nu_i-\la_{i+1}})
	(1-q^{\bar\nu_{i-1}-\bar\nu_i+1})}
	{1-q^{\la_i-\bar\nu_i+1}};
	\quad\label{T_S_Whittaker}
	S_j(\bar\nu,\la)&=
	\frac{(1-q^{\bar\nu_{j-1}-\la_j})
	(1-q^{\la_j-\la_{j+1}+1})}
	{1-q^{\la_j-\bar\nu_j+1}}.
\end{align}
Here as usual we assume that $k=2,\ldots,N$ is fixed,
and that $\la\in\GT_k$, $\bar\nu\in\GT_{k-1}$.
Moreover, we must have $\bar\nu\prec\la$, 
see \eqref{T_i}--\eqref{S_j}.
In \eqref{T_S_Whittaker}
one has $i,j=2,\ldots,k-1$.
The remaining cases 
are resolved as follows:
\begin{align}\label{T_1S_1}
	T_1(\bar\nu,\la)=\frac{1-q^{\bar\nu_1-\la_2}}
	{1-q^{\la_1-\bar\nu_1+1}},\qquad
	S_1(\bar\nu,\la)=\frac{1-q^{\la_1-\la_2+1}}
	{1-q^{\la_1-\bar\nu_1+1}},
	\qquad
	S_k(\bar\nu,\la)&=1-q^{\bar\nu_{k-1}-\la_k}.
\end{align}
That is, if a bracket of the from 
$(1-q^{\cdots})$
does not make sense, we set it equal to one.
Note that the quantity $T_k$ also does not make sense;
by agreement, we set it equal to zero.

\begin{remark}
	Observe that $T_i(\bar\nu,\la)$ given 
	by \eqref{T_S_Whittaker}--\eqref{T_1S_1}
	vanishes if 
	the particle $\la_{i+1}$ is blocked and cannot jump
	(i.e., if $\la_{i+1}=\bar\nu_{i}$). Similarly, 
	$S_j(\bar\nu,\la)$ 
	given by \eqref{T_S_Whittaker}--\eqref{T_1S_1}
	vanishes if $\la_j$ is blocked.
	This agrees with the definitions in the general 
	Macdonald case
	\eqref{T_i}--\eqref{S_j}. Thus, for $t=0$, $0<q<1$ 
	one does not need 
	to insert any indicators in the definitions
	of $T_i$ and $S_j$. Recall that 
	for $q=t=0$ (Schur case) the situation
	was different, cf. \eqref{T_S_Schur}.
\end{remark}

We will also use the 
quantities (cf. \eqref{F_big})
\begin{align}\label{big_F_8}
	F_j(\bar\nu,\la)=&\begin{cases}
		0,&j=1;\\
		q^{\bar\nu_{j-1}-\la_j}
		,&2\le j\le k;
		\\
		1,&j=k+1.
	\end{cases}
\end{align}
It can be readily checked that for any $j=1,\ldots,k$
one has
\begin{align}\label{STFF}
	S_j(\bar\nu,\la)-T_j(\bar\nu,\la)
	=
	F_{j+1}(\bar\nu,\la)-F_j(\bar\nu,\la).
\end{align}


\subsection{Nearest neighbor 
$q$-Whittaker-multivariate dynamics} 
\label{sub:multivariate_dynamics_in_the_q_whittaker_case}

Here we discuss various multivariate
dynamics on interlacing arrays
in the \mbox{$q$-Whittaker} case. 
Probability measures on interlacing
arrays on which these dynamics act nicely 
(see Remark \ref{rmk:multivariate_action_on_Gibbs})
are the Macdonald processes 
(\S \ref{sub:ascending_macdonald_processes})
in which the Macdonald parameter $t$ is specialized to zero. 
They are called the 
\textit{$q$-Whittaker processes},
see \cite[Ch. 3]{BorodinCorwin2011Macdonald} for
a detailed discussion.

\subsubsection{Fundamental `dynamics'} 
\label{ssub:fundamental_dynamics_}

Almost none of the fundamental nearest neighbor
`dynamics' introduced in \S \ref{sub:fundamental_dynamics_}
have nonnegative jump rates and probabilities of triggered
moves in the $q$-Whittaker case
(in contrast with the Schur case, cf. 
\S\S \ref{sub:rs_type_fundamental_dynamics_in_the_schur_case_and_deterministic_insertion_algorithms}--\ref{sub:fundamental_nearest_neighbor_dynamics_in_the_schur_case}). 
\begin{proposition}\label{prop:q_honest_Markov}
	Below is the complete list 
	of fundamental
	nearest neighbor `dynamics' 
	which are honest Markov processes
	on interlacing arrays:
	\begin{enumerate}[$\bullet$]
		\item The push-block dynamics $\bq^{(N)}_{\PBD}$;
		\item The RSK-type dynamics
		$\bq^{(N)}_{\RSD[\boldsymbol h]}$
		with $\boldsymbol h=({1,1,\ldots,1})$
		($N$ ones);
		\item 
		The right-pushing dynamics
		$\bq^{(N)}_{\RD[\boldsymbol h]}$
		with $\boldsymbol h=({1,\ldots,1})$
		($N-1$ ones).
	\end{enumerate}
\end{proposition}
\begin{proof}
	For each of the 
	fundamental `dynamics' defined in
	\S \ref{sub:fundamental_dynamics_}, 
	one has to check that 
	$w_m\ge0$ and $0\le r_j\le c_j\le 1$
	for all possible $m$ and $j$ at each slice
	$\GT_{(k-1;k)}$ (see 
	\S \ref{sub:parametrization_and_linear_equations}
	and
	\eqref{rlw_parameters} 
	in particular
	for the 
	definition of the parameters $(w,c,r)$).

	\textbf{1.}
	For the push-block dynamics $\bq^{(N)}_{\PBD}$
	(Dynamics \ref{dyn:DF_Mac}), one has
	$r_j=c_j\equiv0$, and $w_m=S_m\ge0$, 
	see \eqref{T_S_Whittaker}--\eqref{T_1S_1}. 
	Thus, the push-block dynamics enters the desired list,
	which is not surprising, as it is positive 
	in the general $(q,t)$-setting as well, 
	cf. \S \ref{ssub:setup}.

	\textbf{2.}
	Now consider the RSK-type fundamental `dynamics'
	$\bq^{(N)}_{\RSD[\boldsymbol h]}$ 
	(Dynamics \ref{dyn:RS_type_fund}). Let us 
	argue on a fixed slice $\GT_{(k-1;k)}$.
	Denote $h:=h^{(k)}\in\{1,\ldots,k\}$.
	Clearly, 
	for an RSK-type fundamental `dynamics'
	one has $w_m\ge0$ and $c_j=1\ge0$ for all
	possible
	$m$ and $j$. 
	By \eqref{r_j_RSh} and \eqref{STFF}, we have 
	\begin{align}\label{r_j_RSh_q}
		r_j^{\RSD(h)}=
		T_j^{-1}
		\big(
		S_j+F_j-1_{h\le j}
		\big)=
		1+\frac{F_{j+1}-1_{h\le j}}{T_j}.
	\end{align}
	In the second equality we have used \eqref{STFF} again.
	One also has
	\begin{align}\label{l_j_RSh_q}
		1-r_j^{\RSD(h)}=
		\frac{1_{h\le j}-F_{j+1}}{T_j}.
	\end{align}
	From \eqref{big_F_8}
	we see that this quantity can be made negative unless $h=1$.
	Therefore, of all the RSK-type fundamental
	`dynamics', only the one with $\boldsymbol h=(1,1,\ldots,1)$
	is an honest Markov process.

	\textbf{3.} 
	Take a right-pushing fundamental
	`dynamics' 
	$\bq^{(N)}_{\RD[\boldsymbol h]}$
	(Dynamics~\ref{dyn:Ri})
	on the slice $\GT_{(k;k-1)}$. Denote $h:=h^{(k-1)}
	\in\{1,\ldots,k-1\}$.
	By the very definition of the `dynamics',
	all the $r_j$'s and $c_j$'s are nonnegative, 
	and, moreover, $r_j\le c_j$. 
	We also have $w_m=S_m$ for all 
	possible $m$ except $m=h$. In the latter case 
	\begin{align*}
		w_h^{\RD(h)}=S_h-T_h=F_{h+1}-F_{h}.
	\end{align*}
	This quantity can be made negative unless $h=1$.
	In the case $h=1$ we obtain
	$w_1=q^{\bar\nu_1-\la_2}$, which is always nonnegative.
	Thus, the only honest Markov process
	among the right-pushing `dynamics'
	is the one with $\boldsymbol h=(1,\ldots,1)$.

	\textbf{4.} 
	Finally, consider a left-pulling fundamental
	`dynamics' $\bq^{(N)}_{\LD[\boldsymbol h]}$
	(Dynamics~\ref{dyn:Li})
	on the slice $\GT_{(k;k-1)}$, and denote $h:=h^{(k-1)}
	\in\{1,\ldots,k-1\}$.
	We have (using \eqref{STFF})
	\begin{align*}
		w_{h+1}^{\LD(h)}=S_{h+1}-T_h=
		\big(F_{h+2}+T_{h+1}\big)
		-
		\big(
		F_{h+1}+T_{h}
		\big).
	\end{align*}
	It is possible check that all these 
	numbers can be made negative
	by an appropriate choice of
	$\bar\nu$
	and $\la$.
	This concludes the proof.
\end{proof}

Let us briefly describe the three
dynamics of Proposition \ref{prop:q_honest_Markov}
for $0<q<1$. 

\begin{dynamics}[$q$-Whittaker case 
of Dynamics \ref{dyn:DF_Mac}]
\label{dyn:q_DF}
	The push-block $q$-Whittaker-multi\-variate 
	dynamics $\bq^{(N)}_{\PBD}$
	was introduced and studied in 
	\cite[\S 3.3]{BorodinCorwin2011Macdonald}.
	In this dynamics, every particle $\la^{(k)}_{j}$
	jumps to the right at rate 
	$a_kS_j(\la^{(k-1)},\la^{(k)})$
	which is given by \eqref{T_S_Whittaker}--\eqref{T_1S_1}.
	(If $\la^{(k)}_{j}$ is blocked, 
	then
	$S_j=0$, so this particle does not jump.)
	If a moved particle $\la^{(k)}_{j}$
	violates interlacing with 
	$\la^{(k+1)}_{j}$, then a short-range push
	happens according to~\S \ref{sub:when_a_jumping_particle_has_to_push_its_immediate_upper_right_neighbor}.
\end{dynamics}

\begin{dynamics}[$q$-Whittaker 
case of Dynamics \ref{dyn:RS_type_fund}
with $\boldsymbol h=(1,\ldots,1)$]
\label{dyn:q_row}
	{\ }\begin{enumerate}[(1)]
		\item (independent jumps)
		Under the RSK-type dynamics
		of Proposition \ref{prop:q_honest_Markov},
		each rightmost particle 
		$\la^{(k)}_{1}$
		jumps to the right at rate $a_k$.
		There are no other independent jumps in this dynamics.	
		
		\item (triggered moves)	
		When a particle $\la^{(k-1)}_{j}$
		moves, 
		it 
		long-range
		pushes its first unblocked
		upper right neighbor 
		$\la^{(k)}_{\nf(j)}$
		with probability
		(see \eqref{r_j_RSh_q})
		\begin{align}\label{r_j_q_row}
			r_j^{\RSD(1)}=
			\begin{cases}
				q^{\la_1^{(k)}-\la_1^{(k-1)}},& j=1;\\
				q^{\la_j^{(k)}-\la_j^{(k-1)}}
				\dfrac{1-q^{\la_{j-1}^{(k-1)}-\la_{j}^{(k)}}}
				{1-q^{\la^{(k-1)}_{j-1}-\la^{(k-1)}_{j}}},
				& 2\le j\le k-1.
			\end{cases}
		\end{align}
		Here $\la^{(k-1)}_{j}$
		is the coordinate of the $j$th 
		particle at level $k-1$ before its move.
		With the complementary probability $1-r_j^{\RSD(1)}$,
		the moved particle
		$\la^{(k-1)}_{j}$
		(long-range) pulls
		its immediate upper left neighbor
		$\la^{(k)}_{j+1}$.
	\end{enumerate}
\end{dynamics}

Note that
the short-range pushing mechanism
(\S \ref{sub:when_a_jumping_particle_has_to_push_its_immediate_upper_right_neighbor}) is built 
into the probabilities \eqref{r_j_q_row}. Indeed, 
for $\la^{(k-1)}_j=\la^{(k)}_j$, one has
$r_j^{\RSD(1)}=1$, so the
long-range push of $\la^{(k)}_j$ 
happens with probability 1
and coincides with the short-range push.

Moreover, in this dynamics pushes and
jumps are never donated (cf. the rule of 
\S \ref{donation}). 
Indeed, 
if $\la_{j}^{(k)}=\la_{j-1}^{(k-1)}$, then
$r_j^{\RSD(1)}=0$, 
so the moved particle $\la^{(k-1)}_{j}$
will always pull $\la^{(k)}_{j+1}$
in this situation. Jumps cannot
be donated as well because the only
jumping particles are the rightmost ones.

Dynamics \ref{dyn:q_row} 
may be regarded as a natural
$q$-Whittaker deformation of
the row insertion Schur-multivariate
dynamics 
$\bq^{(N)}_{\mathit{row}}$
(e.g., see Fig.~\ref{fig:row_col}(a)),
so let 
us denote the generator
of the former dynamics
by $\bq^{(N)}_{\text{\textit{q-row}}}$.
In particular, the rates
of independent jumps
of $\bq^{(N)}_{\text{\textit{q-row}}}$
coincide with those of 
$\bq^{(N)}_{\mathit{row}}$.
However, in the $q$-Whittaker case,
the dynamics
$\bq^{(N)}_{\text{\textit{q-row}}}$
has \emph{random} long-range interactions
(i.e., pushes and pulls).

\begin{dynamics}
[$q$-Whittaker version of Dynamics 
\ref{dyn:Ri} with
$\boldsymbol h=(1,\ldots,1)$]
	Under the right-pushing dynamics 
	of Proposition \ref{prop:q_honest_Markov},
	all particles of the array
	$\boldsymbol\la$
	behave in the same way as 
	in Dynamics \ref{dyn:q_DF}, 
	except for the rightmost particles
	$\la^{(k)}_{1}$, $k=1,\ldots,N$.
	These particles have jump rates 
	\begin{align*}
		a_k\big(
		S_1(\la^{(k-1)},\la^{(k)})-T_1(\la^{(k-1)},\la^{(k)})
		\big)=
		a_kq^{\la^{(k-1)}_{1}-\la^{(k)}_{2}},\qquad
		k=2,\ldots,N
	\end{align*}
	(the bottommost particle $\la^{(1)}_{1}$ has jump 
	rate $a_1$).
	When a particle $\la^{(k)}_{1}$ moves, 
	it long-range pushes
	all its upper right 
	neighbors
	$\la^{(k+1)}_{1},\ldots,\la^{(N)}_{1}$
	(with probability one).
	No other long-range interactions are present.
	Of course, there are still short-range pushes, cf. 
	\S \ref{sub:when_a_jumping_particle_has_to_push_its_immediate_upper_right_neighbor}.
\end{dynamics}


\subsubsection{$q$-Whittaker-multivariate 
`dynamics' with deterministic long-range interactions} 
\label{ssub:_q_whittaker_multivariate_dynamics_with_deterministic_long_range_interactions}

Let us now consider 
another natural class of
$q$-Whittaker-multivariate
nearest neighbor
`dynamics'
which are distinguished by having 
\emph{deterministic} long-range interactions.
This condition in fact 
forces
each particle
in the interlacing array $\boldsymbol\la$
to have a nonzero jump rate.
However, these jump 
rates are not necessarily nonnegative
(but we are still using probabilistic
language, cf. Remark~\ref{rmk:algebraic_rlw}).

\begin{proposition}
\label{prop:q_h_insertion}
	Fix
	$\boldsymbol h\in\{1\}\times \{1,2\}
	\times \ldots\times \{1,2,\ldots,N\}$, 
	and consider a $q$-Whittaker-multivariate
	`dynamics' 
	$\bq^{(N)}$
	with deterministic
	propagation of moves 
	given by the 
	$\boldsymbol h$-insertion.\footnote{See
	\S\S \ref{sub:rs_type_dynamics_and_random_insertions}--\ref{sub:rs_type_fundamental_dynamics_in_the_schur_case_and_deterministic_insertion_algorithms}, 
	and in particular \eqref{r_j_RSD_Schur} 
	and
	Fig.~\ref{fig:big_rs}.}
	Then on each slice
	$\GT_{(k-1;k)}$
	this `dynamics' has the following
	rates of independent jumps:
	\begin{align}\label{mixed_h_qWhittaker}
		W_k(\la,\la+\de_j\,|\,\bar\nu)=
		a_k\big(
		S_j(\bar\nu,\la)-
		1_{j<h^{(k)}}\cdot T_{j}(\bar\nu,\la)
		-1_{j>h^{(k)}}\cdot T_{j-1}(\bar\nu,\la)
		\big),
	\end{align}
	where $\bar\nu\in\GT_{k-1}$
	and $\la\in\GT_{k}$.
	We assume that the jump donation 
	rule of \S\ref{donation}
	is applied to \eqref{mixed_h_qWhittaker}.
\end{proposition}
Note that some of the jump rates \eqref{mixed_h_qWhittaker}
can be made negative
(by a choice of $\bar\nu$ and $\la$); see also the proof
of Proposition 
\ref{prop:q_honest_Markov}.
\begin{proof}
	This is evident if 
	for each slice $\GT_{(k-1;k)}$
	one puts
	$c_j=1$ 
	and $r_j=1_{j<h^{(k)}}$
	(for all $j$)
	in \eqref{rlw_system_of_equations}.
\end{proof}

The above `dynamics' 
with $\boldsymbol h$-insertions
is RSK-type
(Definition \ref{def:RS_type}).
One can readily check that its generator $\bq^{(N)}$ 
can be represented 
as a linear combination of those of 
the fundamental RSK-type `dynamics' (cf.~Proposition 
\ref{prop:RS_Schur}):
\begin{align}\label{q_push_representation}
	\bq^{(N)}(\boldsymbol\la,
	\boldsymbol\nu)=
	\sum_{\tilde{\boldsymbol h}}
	\co_{\tilde{\boldsymbol h}}(\boldsymbol\la,\boldsymbol\nu)
	\bq^{(N)}_{\RSD[\tilde{\boldsymbol h}]}
	(\boldsymbol\la,\boldsymbol\nu),
\end{align}
where
$\boldsymbol\la,\boldsymbol\nu\in\GT(N)$ and
the sum is taken over all 
$\tilde{\boldsymbol h}=
(\tilde{h}^{(1)},\ldots,\tilde{h}^{(N)})$
such that $1\le \tilde{h}^{(k)}\le k$.
The coefficients 
$\co_{\tilde{\boldsymbol h}}$
have the product form as in 
\eqref{co_h_product1}
with terms
given by (the quantities below
depend on two signatures $\nu^{(k-1)}$ and 
$\la^{(k)}$)
\begin{align*}
	\co^{(k)}_{\tilde{h}^{(k)}}=
	S_{\tilde{h}^{(k)}}-
	1_{\tilde{h}^{(k)}<h^{(k)}}T_{j}
	-1_{\tilde{h}^{(k)}>h^{(k)}}T_{j-1},
\end{align*}
where $k=2,\ldots,N$.
This formula for the 
coefficients of course directly follows
from \eqref{mixed_h_qWhittaker}.

\medskip

It can be seen that for
$q=0$, the
 `dynamics'
$\bq^{(N)}$ of Proposition \ref{prop:q_h_insertion}
corresponding to $\boldsymbol h
\in\{1\}\times \{1,2\}
\times \ldots\times \{1,2,\ldots,N\}$
turns into the 
Schur-multivariate
fundamental RSK-type dynamics
$\bq^{(N)}_{\RSD[\boldsymbol h]}$
(see \S \ref{sub:rs_type_fundamental_dynamics_in_the_schur_case_and_deterministic_insertion_algorithms}
for a detailed discussion of the latter dynamics).

\medskip

For $\boldsymbol h=(1,\ldots,1)$, 
the `dynamics' of Proposition 
\ref{prop:q_h_insertion}
has the property that 
any moving particle (long-range)
pulls its immediate
upper left neighbor with probability
one. 
Let us denote this dynamics by 
$\bq^{(N)}_{\text{\textit{q-pull}}}$.
The jump rates in this `dynamics'
are given by 
$W_k^{\text{\textit{q-pull}}}
(\la,\la+\de_j\,|\,\bar\nu)=
a_k\big(S_j(\bar\nu,\la)-T_{j-1}(\bar\nu,\la)\big)$
(with the understanding that $T_0\equiv 0$).
One can argue that
$\bq^{(N)}_{\text{\textit{q-pull}}}$
and 
$\bq^{(N)}_{\text{\textit{q-row}}}$ 
(see \S \ref{ssub:fundamental_dynamics_})
provide two (very different)
$q$-Whittaker deformations
of the Schur-multivariate row insertion dynamics
$\bq^{(N)}_{\mathit{row}}$.
One of these deformations is an honest Markov process,
and the other is not. 

Denote also by 
$\bq^{(N)}_{\text{\textit{q-push}}}$
the generator of the `dynamics' of 
Proposition \ref{prop:q_h_insertion}
corresponding to 
$\boldsymbol h=(1,2,\ldots,N)$.
In this `dynamics', 
any moving particle long-range
pushes its 
first unblocked 
upper right neighbor with probability
one. The jump rates in 
$\bq^{(N)}_{\text{\textit{q-push}}}$
have a rather simple form:
\begin{align*}
	W_k^{\text{\textit{q-push}}}
	(\la,\la+\de_j\,|\,\bar\nu)=
	\begin{cases}
		a_k
		q^{\bar\nu_1-\la_2}
		& j=1;\\
		a_k
		\big(
		q^{\bar\nu_j-\la_{j+1}}-
		q^{\bar\nu_{j-1}-\la_{j}}
		\big),
		& 2\le j\le k-1;\\
		a_k
		\big(1-
		q^{\bar\nu_{k-1}-\la_k}
		\big),
		& j=k
	\end{cases}
\end{align*}
(see \eqref{mixed_h_qWhittaker} and \eqref{big_F_8}).
These jump rates, however, are not 
always nonnegative.
The dynamics $\bq^{(N)}_{\text{\textit{q-push}}}$
can be viewed as a $q$-Whittaker
deformation of the column insertion dynamics
$\bq^{(N)}_{\mathit{col}}$ (\S \ref{sub:rs_type_fundamental_dynamics_in_the_schur_case_and_deterministic_insertion_algorithms}).


\subsubsection{Remark: a nearest neighbor 
Markov process inspired by
Dynamics \ref{dyn:OConnell}} 
\label{ssub:remark_a_nearest_neighbor_markov_process_inspired_by_dynamics_dyn:oconnell}

Let us briefly mention one more 
nearest neighbor $q$-Whittaker-multivariate
`dynamics'
which can be 
``invented'' by looking at the 
O'Connell--Pei's insertion algorithm
(see 
\S \ref{sub:o_connell_pei_s_process_on_interlacing_arrays}).
The latter dynamics is RSK-type but not nearest neighbor
(see Definitions \ref{def:nn_dynamics} and 
\ref{def:RS_type}); 
we aim to turn it into a nearest neighbor one,
but the result will not be RSK-type.
In order to do that, one should look at
identity \eqref{Oconnell_514_identity}
(it is equivalent to \eqref{P2'_T_S})
which governs Dynamics \ref{dyn:OConnell}.
This identity has the form
\begin{align*}
	f_jT_j+\big(\cdots\big)=S_j,
\end{align*}
where $f_j$ is given by \eqref{f_small},
and dots mean all the remaining terms. 
Let us 
understand these remaining terms
$(\cdots)$
as 
$a_k^{-1}$ times 
the probability of an
independent jump of the $j$th particle.
In fact, these terms ``collapse'',
and their sum is simply 
$(1-F_j)F_{j+1}$, see \eqref{OConnell_proof_collapsing}.
In this way
we arrive at a new 
\emph{nearest neighbor}
dynamics 
governed by 
identity 
\eqref{OConnell_proof_collapsing},
which now should be understood 
according to 
the linear system \eqref{rlw_system_of_equations}.
To be more precise, 
the modified dynamics
can be described as follows:
\begin{dynamics}[Nearest neighbor 
modification of Dynamics \ref{dyn:OConnell}]
	\label{dyn:collapsed_OConnell}
	{\ }\begin{enumerate}[(1)]
		\item (independent jumps)
		Every particle $\la_j^{(k)}$ 
		independently jumps to the right at rate
		$\big(1-F_{j}(\la^{(k-1)},\la^{(k)})\big)
		F_{j+1}(\la^{(k-1)},\la^{(k)})$.
		\item (triggered moves)
		When a particle $\la^{(k-1)}_{j}$
		moves, it long-range pushes 
		its immediate upper right neighbor
		with probability
		$f_j(\nu^{(k-1)},\la^{(k)})$ (see \eqref{f_small}), 
		where 
		$\nu^{(k-1)}$
		differs from $\la^{(k-1)}$
		only by
		$\nu^{(k-1)}_j=\la^{(k-1)}_j+1$.
		With the complementary probability 
		$1-f_j(\nu^{(k-1)},\la^{(k)})$,
		the move of $\la^{(k-1)}_{j}$ does not 
		propagate from level $k-1$ to level $k$.
	\end{enumerate}
\end{dynamics}

In this process there is no need to 
donate independent jumps or triggered 
moves (cf. the donation
rule in \S \ref{donation}).
Indeed, if a particle $\la_j^{(k)}$
is blocked, (i.e., if
$\la_j^{(k)}=\la_{j-1}^{(k-1)}$),
then one automatically 
has $F_j=1$ and $f_j=0$.

It can be readily checked that the
jump rates and probabilities of triggered moves
in Dynamics \ref{dyn:collapsed_OConnell} are nonnegative.

For $q=0$, both 
Dynamics 
\ref{dyn:OConnell} and
\ref{dyn:collapsed_OConnell}
turn into the column insertion
Schur-multivariate dynamics 
$\bq^{(N)}_{\mathit{col}}$.



\subsection{TASEPs} 
\label{sub:taseps}

\subsubsection{Left-Markov and right-Markov multivariate dynamics} 
\label{ssub:left_markov_and_right_markov_multivariate_dynamics}

In this subsection we examine nearest neighbor
$q$-Whittaker-multivariate
dynamics ($0<q<1$)
on interlacing arrays $\boldsymbol\la=(\la^{(1)}\prec \ldots\prec\la^{(N)})$
(see Fig.~\ref{fig:la_interlacing})
with the property that the leftmost particles
$\{\la^{(k)}_{k}\colon k=1,\ldots,N\}$
or the rightmost 
particles
$\{\la^{(k)}_{1}\colon k=1,\ldots,N\}$
of the array evolve according to a \emph{Markov} 
process (in its own filtration).\footnote{Multivariate dynamics
of the whole interacting array are of course Markov.
The question is when 
these Markov processes remain Markov 
when projected to a subset of particles.}
We formulate natural sufficient conditions implying 
these \emph{left-Markov} and \emph{right-Markov} properties.

In this subsection we also
briefly discuss one-dimensional
Markov dynamics of interacting particles on 
$\Z$ arising in this way. These 
dynamics (coming from either left- or right- Markov property) 
can be viewed as versions of the 
\emph{TASEP} (=~\emph{totally asymmetric simple
exclusion process}). 

Left-Markov multivariate dynamics induce 
\emph{$q$-TASEP} 
\cite[\S3.3.2]{BorodinCorwin2011Macdonald}, 
\cite{BorodinCorwinSasamoto2012}
as the
Markov evolution of leftmost particles.
The consideration of right-Markov 
$q$-Whittaker-multivariate dynamics 
gives rise to
a new interacting particle system on $\Z$ which we call
\textit{$q$-PushTASEP}. 

For $q=0$, $q$-TASEP and $q$-PushTASEP 
degenerate into TASEP and PushTASEP, respectively. 
About these $q=0$ exclusion processes, 
e.g., see \cite{BorFerr2008DF} and references therein.



\subsubsection{When the leftmost particles in the array evolve in a Markovian way} 
\label{ssub:when_leftmost_particles_in_the_array_evolve_in_a_markovian_way}

Let us now consider nearest neighbor 
$q$-Whittaker-multivariate
`dynamics' $\bq^{(N)}$ in which
the leftmost particles
$\{\la^{(k)}_{k}\colon k=1,\ldots,N\}$
evolve in a Markovian way. 

Natural sufficient
conditions for the 
left-Markov property are the following
(we use the notation
$\la=\la^{(k)}$,
and $\bar\nu=\nu^{(k-1)}$
explained in 
\S \ref{sub:specialization_of_formulas_from_s_sub:sequential_update_dynamics_in_continuous_time}):
\begin{enumerate}[(L1)]
	\item For each $k$, the jump rate of the 
	particle $\la^{(k)}_{k}$, i.e., 
	$W_k(\la,\la+\de_k\,|\,\bar\nu)$,
	must not depend on the coordinates
	of the remaining particles $\{\la^{(m)}_{j}
	\colon m=1,\ldots,N,\;j=1,\ldots,k-1\}$.
	\item 
	For each $k$, the probability
	$V_k(\la,\la+\de_k\,|\,\bar\nu-\bar\de_{k-1},\bar\nu)$
	that $\la^{(k-1)}_{k-1}$
	pulls $\la^{(k)}_{k}$
	must be zero.
\end{enumerate}

Condition (L1) is rather obvious for 
the left-Markov property. 
As for (L2), note that 
the (long-range) pulling interaction
is \emph{weaker} than the short-range pushing
(cf. \S \ref{sub:when_a_jumping_particle_has_to_push_its_immediate_upper_right_neighbor}). Thus, 
when $\la^{(k-1)}_{k-1}=\la^{(k)}_{k-1}$, the 
moving particle $\la^{(k-1)}_{k-1}$ cannot pull $\la^{(k)}_{k}$.
Since in a left-Markov situation
the pulling must not 
depend on the coordinate 
of $\la^{(k)}_{k-1}$, 
we conclude
that the pulling probabilities 
in (L2)
must be zero.


\begin{proposition}\label{prop:qTASEP}
	Let $0<q<1$.
	Under any 
	nearest neighbor 
	$q$-Whittaker-multivariate 
	`dynamics' satisfying
	{\rm{}(L1)--(L2)}, 
	each of the
	leftmost particles $\la^{(k)}_{k}$,
	$k=1,\ldots,N$,
	jumps to the right by one 
	independently
	of others at rate 
	\begin{align}\label{qTASEP_rate}
		a_k\big(1-q^{\la^{(k-1)}_{k-1}-\la^{(k)}_{k}}
		\big)
	\end{align}
	(with the agreement that 
	this rate is equal to $a_1$ for $k=1$). 
	Note that when $\la^{(k)}_{k}$
	is blocked (i.e., if $\la^{(k)}_{k}=\la^{(k-1)}_{k-1}$),
	then the rate vanishes.
\end{proposition}
The evolution of the leftmost
particles described by the proposition 
it known as
\textit{$q$-TASEP}, see \cite[\S3.3.2]{BorodinCorwin2011Macdonald}
and \cite{BorodinCorwinSasamoto2012}.
\begin{proof}
	This is evident when looking 
	at the last 
	equation of \eqref{rlw_system_of_equations}
	corresponding to each slice $\GT_{(k-1;k)}$.
	Because there is no pulling, this equation 
	simply reads $w_k=S_k$ (where $w_j$
	is defined by \eqref{w_m_shorthand}). 
	Using \eqref{T_1S_1}, 
	we see that the jump rate 
	of each $\la^{(k)}_{k}$ is given by \eqref{qTASEP_rate}.
\end{proof}

\begin{remark}\label{rmk:Macdonald_bad}
	There is no hope of 
	obtaining an analogue of
	Proposition \ref{prop:qTASEP}
	in the general Macdonald setting
	(i.e., when $t\ne0$)
	because in this case the quantity
	$S_k(\la^{(k-1)},\la^{(k)})$ 
	depends
	on \emph{all} the particles
	at levels $k-1$ and $k$, see~\eqref{S_j}.
\end{remark}

Proposition \ref{prop:qTASEP}
suggests that 
at the $q$-Whittaker level 
(in fact, one can include the Schur degeneration,
so $0\le q<1$), 
the evolution
of the leftmost particles 
can be Markovian in a unique way which is
dictated by the $q$-TASEP.
However, this Markovian evolution
of the leftmost particles can be extended to
a multivariate dynamics of the 
whole interlacing array $\boldsymbol\la$
in many ways. In particular, 
the following $q$-Whittaker-multivariate `dynamics'
considered in the present paper
are left-Markov and lead to $q$-TASEP:
\begin{enumerate}[$\bullet$]
	\item The push-block dynamics
	(Dynamics \ref{dyn:q_DF});
	\item
	Dynamics driven by O'Connell--Pei's insertion algorithm, 
	and its modified version
	(Dynamics \ref{dyn:OConnell} and 
	\ref{dyn:collapsed_OConnell});
	\item  
	Any right-pushing `dynamics' (Dynamics \ref{dyn:Ri});
	\item 
	The ``column'' RSK-type `dynamics' 
	$\bq^{(N)}_{\RSD[\boldsymbol h]}$, i.e.,
	the one with 
	$\boldsymbol h=(1,2,\ldots,N)$;
	\item The $q$-Whittaker-multivariate
	`dynamics' with deterministic move
	propagation dictated by the column insertion
	algorithm
	(i.e., the `dynamics' of 
	Proposition \ref{prop:q_h_insertion}
	with $\boldsymbol h=(1,2,\ldots,N)$). 
\end{enumerate}
Using mixing of `dynamics' 
(\S \ref{sub:characterization_of_nearest_neighbor_dynamics_}),
it is possible to produce a variety
of other left-Markov `dynamics'.

To conclude the discussion of left-Markov 
multivariate `dynamics', we note that
it is possible to 
generalize the left-Markov
property to, say, two 
leftmost particles
at each level $k$.
That is, one could consider $q$-Whittaker-multivariate
`dynamics' under which the evolution
of $2N-1$ particles
$\la^{(1)}_1$ and
$\la^{(k)}_{k},\la^{(k)}_{k-1}$
(where $k=2,\ldots,N$)
is Markovian.
In this way one gets interacting
particle systems on $\Z\sqcup\Z$
(one copy of $\Z$ contains all the 
$\la^{(k)}_{k}$'s, and the remaining $N-1$
particles live on the other copy of $\Z$).
This ``two-diagonal'' 
setting
is not as rigid 
as in Proposition \ref{prop:qTASEP}:
one can construct \emph{many different}
Markov evolutions
of configurations of particles on $\Z\sqcup\Z$
which come from various multivariate
`dynamics' on the whole interlacing array $\boldsymbol\la$.
Fixed-time distributions
of all these Markov processes on 
configurations on $\Z\sqcup\Z$
are the same if the processes 
start from the same initial conditions.

\subsubsection{When the rightmost particles in the array evolve in a Markovian way} 
\label{ssub:when_rightmost_particles_in_the_array_evolve_in_a_markovian_way}

We now aim to discuss \emph{right-Markov}
nearest neighbor 
$q$-Whittaker-multivariate `dynamics', i.e.,
`dynamics' under which the rightmost
particles 
$\la^{(k)}_{1}$, $k=1,\ldots,N$
of the array $\boldsymbol\la$
evolve in a Markovian way.
Thus, the remaining particles
$\la^{(k)}_{j}$, $2\le j\le k$, must not
influence the rightmost ones. 
Since all the particles of the array $\boldsymbol\la$
jump to the right, the nature of this influence differs from 
the one in \S \ref{ssub:when_leftmost_particles_in_the_array_evolve_in_a_markovian_way}.

We propose the following natural sufficient
conditions for the right-Markov property:
\begin{enumerate}[(R1)]
	\item The jump rate of each rightmost 
	particle $\la^{(k)}_{1}$, 
	i.e., $W_k(\la,\la+\de_1\,|\,\bar\nu)$,
	must not depend on the coordinates
	of the remaining particles $\{\la^{(m)}_{j}
	\colon m=1,\ldots,N,\;j=2,\ldots,k\}$.
	\item The probability 
	$V_k(\la,\la+\de_1\,|\,\bar\nu-\bar\de_1,\bar\nu)$
	that 
	$\la^{(k-1)}_{1}$ long-range pushes
	$\la^{(k)}_{1}$ ($k=2,\ldots,N$)
	also must not depend on the 
	coordinates of the
	remaining
	particles.
	\item In the multivariate
	`dynamics' of the whole interlacing array
	$\boldsymbol\la$, there must be no donations of jumps 
	or pushes (cf. the rule of \S\ref{donation}). That is,
	for each $k=2,\ldots,N$ and $j=2,\ldots,k$, the 
	quantities 
	$W_k(\la,\la+\de_j\,|\,\bar\nu)$ and
	$V_k(\la,\la+\de_j\,|\,\bar\nu-\bar\de_j,\bar\nu)$
	must depend on $\bar\nu$ and $\la$ in such a way that they 
	automatically become zero if $\la_j=\bar\nu_{j-1}$ 
	(i.e., if the particle $\la_j=\la_j^{(k)}$
	is blocked).
\end{enumerate}
Here we also used the notation
$\la=\la^{(k)}$,
and $\bar\nu=\nu^{(k-1)}$
explained in 
\S \ref{sub:specialization_of_formulas_from_s_sub:sequential_update_dynamics_in_continuous_time}.

Let us briefly comment on condition (R3)
(the two other conditions are rather obvious).
Assume that (R3) does not hold for, say, the particle
$\la^{(k)}_{j}$. Consider the configuration in which
$\la^{(k)}_j=\la^{(k-1)}_{j-1}$, 
$\la^{(k)}_{j-1}=\la^{(k-1)}_{j-2}$, etc.,
and $\la^{{(k)}}_{2}=\la^{(k-1)}_{1}$.
Then, according to the rule of \S\ref{donation}, 
$\la^{(k)}_{j}$ must donate its move
(which could be an independent jump or a triggered move)
to $\la^{(k)}_{1}$, its first unblocked neighbor
at the same level $k$.
On the other hand, if, 
say,
$\la^{{(k)}}_{2}<\la^{(k-1)}_{1}$,
then the donated move from $\la^{(k)}_{j}$
goes to $\la^{{(k)}}_{2}$. 
We see (at least informally) that the dynamics
of the rightmost particles depends on, e.g., 
the position of $\la^{{(k)}}_{2}$.
Thus, condition (R3) is also quite natural.

\begin{proposition}
\label{prop:qpushTASEP}
	Let $0<q<1$.
	Under any 
	nearest neighbor 
	$q$-Whittaker-multivariate `dynamics'
	satisfying {\rm{}(R1)--(R3)}, 
	each rightmost particle $\la^{(k)}_{1}$
	has jump rate $a_k$. When any of the rightmost
	particles
	$\la^{(k-1)}_{1}$ moves (due to an independent jump
	or a push), it long-range pushes
	the next rightmost particle
	$\la^{(k)}_{1}$ with probability
	$q^{\la^{(k)}_{1}-\la^{(k-1)}_{1}}$
	(here $\la^{(k-1)}_{1}$ is the coordinate
	before the move). Note that 
	this interaction has the possibility to
	propagate to 
	all the higher levels $k+1,k+2,\ldots,N$.
\end{proposition}
The evolution of the 
rightmost particles described 
by the proposition is called
\textit{$q$-PushTASEP}.
\begin{proof}
	Let us 
	take any slice $\GT_{(k-1;k)}$, and 
	look at the first equation of the system
	\eqref{rlw_system_of_equations}. It can be rewritten in 
	the following form
	(we use \eqref{T_1S_1}):
	\begin{align*}
		w_1=
		S_1(\bar\nu,\la)-r_1 
		T_1(\bar\nu,\la)=
		\frac{
		1-r_1+q^{-\la_2}
		(r_1q^{\bar\nu_1}-q^{\la_1+1})}
		{1-q^{\la_1+1-\bar\nu_1}},
	\end{align*}
	where $w_1$ is $a_k^{-1}$
	times the rate of independent jump
	of the particle $\la^{(k)}_{1}$, and $r_1$
	is the probability with which 
	$\la^{(k-1)}_{1}$
	long-range pushes $\la^{(k)}_{1}$.
	Note that here we have used (R3),
	because otherwise the first equation 
	of \eqref{rlw_system_of_equations}
	could contain the term $r_j T_j$ for some $j\ge2$.

	By (R1)--(R2), both $w_1$ and $r_1$ must not 
	depend on the coordinates of $\la_2,\la_3,\ldots,\la_k$ and 
	$\bar\nu_2,\bar\nu_3,\ldots,\bar\nu_{k-1}$.
	Next, observe that if the factor
	$r_1q^{\bar\nu_1}-q^{\la_1+1}$
	is not identically zero, then $w_1$ 
	depends on $\la_2$. So, one should have 
	$r_1=q^{\la_1+1-\bar\nu_1}$, and hence $w_1=1$.
	This concludes the proof.
\end{proof}

Similarly to 
Remark \ref{rmk:Macdonald_bad},
there is little hope of finding 
right-Markov multivariate
`dynamics' with the general Macdonald parameters.

Proposition \ref{prop:qpushTASEP}
(in analogy with Proposition \ref{prop:qTASEP} above)
suggests a certain rigidity of the 
$q$-PushTASEP. 
One can, however, extend this Markov process
from the rightmost particles to 
the whole interlacing array $\boldsymbol\la$ 
in many ways:
there are a lot of multivariate
`dynamics' satisfying
conditions (R1)--(R3). 

However, in the subclass of RSK-type
`dynamics', such an extension of 
the $q$-PushTASEP is unique (this can be readily checked), 
namely, 
this is the dynamics
$\bq^{(N)}_{\text{\textit{q-row}}}$ (Dynamics \ref{dyn:q_row}).

Similarly to the discussion of 
\S \ref{ssub:when_leftmost_particles_in_the_array_evolve_in_a_markovian_way}, one can also consider 
multivariate `dynamics' on the whole array
which induce a Markov evolution
of the particles on, say, two rightmost
diagonals of the array.
This does not define an induced Markov evolution
uniquely.


\subsubsection{$q$-PushTASEP} 
\label{ssub:_q_pushtasep}

Let us now represent $q$-PushTASEP 
(process of Proposition \ref{prop:qpushTASEP})
as a
system of interacting particles on $\Z$
which do not collide. 
Set $x_n=\la^{(n)}_{1}+n$, $n=1,\ldots,N$,
and treat the integers $x_1<x_2<\ldots<x_N$ 
as particle location on $\Z$. The evolution 
of $q$-PushTASEP in continuous time is described
as follows. Each particle $x_n$ has an independent
exponential clock with rate $a_n$. When the 
clock rings (say, at some time $\tau$), 
$x_n$ jumps to the right by one, so 
$x_n(\tau+d\tau)=x_n(\tau)+1$.
Moreover, every particle $x_k$
that has just moved, i.e., for which
$x_k(\tau+d\tau)=x_k(\tau)+1$,
long-range pushes (=~forces to immediately
move to the right by one)
its right neighbor $x_{k+1}$
with probability 
\begin{align*}
 	q^{x_{k+1}(\tau)-x_k(\tau)-1}=
	q^{x_{k+1}(\tau+d\tau)-x_k(\tau+d\tau)}.
\end{align*} 
Note that when $x_k(\tau)=x_{k+1}(\tau)-1$
(i.e., the destination of $x_k$ is in fact occupied),
the probability of push is one, so in this case
$x_{k+1}$ is always pushed and thus frees the destination
for $x_k$: 
\begin{align*}
	x_k(\tau+d\tau)=x_k(\tau)+1,\
	x_{k+1}(\tau+d\tau)=x_{k+1}(\tau)+1 
	\ \ \mbox{if $x_k(\tau)=x_{k+1}(\tau)-1$}.
\end{align*}
If the particle $x_{k+1}$
was pushed, it can continue the pushing interaction and 
push $x_{k+2}$, and so on.

Let us write down the Markov generator 
$L^{\text{$q$-PushTASEP}}$
of this process. It acts on functions
$f(x_1,\ldots,x_N)$, where $x_1<\ldots <x_N$:
\begin{align*}&
	(L^{\text{$q$-PushTASEP}}f)(x_1,\ldots,x_N)
	\\&
	\hspace{40pt}=\sum_{i=1}^{N}
	a_i\sum_{j=i}^{N}
	q^{x_j-x_i-(j-i)}\big(
	1-q^{x_{j+1}-x_j-1}\big)\times\\&
	\hspace{90pt}\times
	\big(
	f(x_1,\ldots,x_{i-1},
	x_{i}+1,\ldots,x_j+1,x_{j+1},\ldots,x_N)
	-
	f(x_1,\ldots,x_N)
	\big).
\end{align*}
Here $q^{x_j-x_i-(j-i)}\big(
1-q^{x_{j+1}-x_j-1}\big)$ in the sum 
is the probability that the particle 
$x_i$, if it jumps, will push the particles
$x_{i+1},\ldots,x_j$, and not $x_{j+1}$.
The study of $q$-PushTASEP is 
to be continued in a subsequent work.

In \S \ref{sub:connections_with_o_connell_yor_semi_discrete_directed_polymer} below we explain connections of the $q$-PushTASEP
to the O'Connell--Yor semi-discrete directed polymer.
It turns out that $q$-PushTASEP is more
directly related to the 
time 
evolution of the polymer partition functions than 
the $q$-TASEP.

\medskip

The $q$-PushTASEP is a 
natural $q$-deformation of the 
PushTASEP 
considered in \cite{BorFerr2008DF}
and first introduced in
\cite{Spitzer1970} under the name of \emph{long-range TASEP}.
In PushTASEP, each particle $x_n$
jumps to the right by one 
independently of others
at rate $a_n$.
If the destination of $x_n$
is occupied by $x_{n+1}$ (i.e., $x_{n}(\tau)=x_{n+1}(\tau)-1$), 
then the jumping particle $x_n$
(short-range) pushes $x_{n+1}$ to the right by one.
More generally, $x_n$ finds the 
block of particles to the right of itself 
($x_n=x_{n+1}-1=x_{n+2}-2=\ldots=x_{n+j}-j$)
and 
pushes the whole block to the right by one (note that the 
block can be empty).
See also 
\cite{BorFerr08push}
for an interacting particle system
unifying TASEP and PushTASEP.



\subsection{Formal scaling 
limits as $q\nearrow1$. Diffusions
on Whittaker processes.
Connections to the O'Connell--Yor
semi-discrete directed polymer} 
\label{sub:limit_of_our_formulas_as_qto1_}

\subsubsection{Setup} 
\label{ssub:setup_qWhittaker}

In this subsection we discuss 
a scaling limit of our $q$-Whittaker-multivariate 
`dynamics' of 
\S \ref{sub:multivariate_dynamics_in_the_q_whittaker_case} as $q\nearrow1$. 
The limit transition we consider
takes $q$-Whittaker processes 
to \emph{Whittaker processes}.
See \cite[Ch. 4]{BorodinCorwin2011Macdonald}
for definition and detailed discussion of
Whittaker processes.

The scaling in question is defined
as follows 
\cite[Thm. 4.1.21]{BorodinCorwin2011Macdonald}:
\begin{align}
	\label{q_to_1_scaling}
	\begin{array}{ll}
		q=e^{-\varepsilon},
		\qquad
		\tau=\varepsilon^{-2}\cdot\tlim,
		\qquad
		a_k=e^{-\varepsilon \cdot\alim_k},
		&\  k=1,\ldots,N,
		\\
		\la^{(k)}_{j}=
		\tlim\cdot \varepsilon^{-2}
		-(k+1-2j)\varepsilon^{-1}\log\varepsilon
		+\lalim^{(k)}_{j}\varepsilon^{-1},
		&\  k=1,\ldots,N,\  
		j=1,\ldots,k.
	\end{array}
\end{align}
Here $\tlim>0$ is the scaled time,
and $(\alim_1,\ldots,\alim_N)\in\R^{N}$
are scaled values of the $a_j$'s.
The $q$-Whittaker process
$\M_{asc,t=0}(a_1,\dots,a_N;\rho_{\tau})$,
where $\rho_{\tau}$ is the Plancherel
specialization 
(defined as in \S \ref{sub:nonnegative_specializations}
with the second Macdonald parameter $t$ being zero),
leads (via \eqref{q_to_1_scaling}) 
to a probability measure on the 
real numbers $\lablim=\{\lalim^{(k)}_{j}\colon
k=1,\ldots,N,\; j=1,\ldots,k\}$.
As $\varepsilon\downarrow 0$, this probability
measure on $\R^{\frac{N(N+1)}2}$ 
converges to the Whittaker process.
Note that the quantities $\lalim^{(k)}_{j}$
are \emph{any} real numbers, i.e., they
do not have to satisfy interlacing
constraints.

In this section we aim to formally write systems of 
SDEs
for diffusions on $\R^{\frac{N(N+1)}2}$ 
which correspond to 
various $q$-Whittaker-multivariate
`dynamics'.
These diffusions should
act on Whittaker processes
in a way similar to 
Remark \ref{rmk:multivariate_action_on_Gibbs}: 
the application of the diffusion semigroup to 
a Whittaker process changes its parameter $\tlim$ 
(cf. \cite[Def. 4.1.16]{BorodinCorwin2011Macdonald}).


\subsubsection{Expansion of $T_i$ and $S_j$} 
\label{ssub:expansions_of_t_i_and_s_j}

Let us first consider the 
scaling \eqref{q_to_1_scaling}
of the quantities $T_i$ and $S_j$
given by \eqref{T_S_Whittaker}--\eqref{T_1S_1}.
They are defined for a particular slice
$\GT_{(k-1;k)}$, so we assume that $k=2,\ldots,N$
is fixed. Let $\la^{(k-1)}\in\GT_{k-1}$
and $\la^{(k)}\in\GT_k$ depend on 
$\lalim^{(k-1)}\in\R^{k-1}$ and 
$\lalim^{(k)}\in\R^{k}$, respectively,
as in \eqref{q_to_1_scaling}.
\begin{proposition}
	For $\la^{(k-1)}$ and $\la^{(k)}$
	as above, expansions of 
	$T_i(\la^{(k-1)},\la^{(k)})$ and 
	$S_j(\la^{(k-1)},\la^{(k)})$
	in
	$\varepsilon$ 
	(up to the first order)
	look as follows
	($i=1,\ldots,k-1$ and $j=2,\ldots,k-1$):
	\begin{align}\label{T_eps_exp}
		T_i(\la^{(k-1)},\la^{(k)})
		&=1-
		\varepsilon 
		e^{\lalim^{(k)}_{i+1}-\lalim^{(k-1)}_{i}}
		+
		\varepsilon 
		e^{\lalim^{(k-1)}_{i}-\lalim^{(k)}_{i}}
		+O(\varepsilon^{2});\\
		S_j(\la^{(k-1)},\la^{(k)})
		&=1-
		\varepsilon 
		e^{\lalim^{(k)}_{j}-\lalim^{(k-1)}_{j-1}}
		+
		\varepsilon 
		e^{\lalim^{(k-1)}_{j}-\lalim^{(k)}_{j}}
		+O(\varepsilon^{2}).
		\label{S_eps_exp}
	\end{align}
	The remaining cases
	are 
	\begin{align}
		\label{S1_eps_exp}
		S_1(\la^{(k-1)},\la^{(k)})&=
		1+\varepsilon e^{\lalim^{(k-1)}_{1}-
		\lalim^{(k)}_{1}};\\
		S_k(\la^{(k-1)},\la^{(k)})&=
		1-\varepsilon e^{\lalim^{(k)}_{k}-
		\lalim^{(k-1)}_{k-1}}.
		\label{Sk_eps_exp}
	\end{align}
\end{proposition}
\begin{proof}
	This is readily obtained using 
	definitions
	\eqref{T_S_Whittaker}--\eqref{T_1S_1}
	and 
	scaling \eqref{q_to_1_scaling}.
	Note that factors of the form
	$1-q^{\la^{(k)}_{j}-\la^{(k)}_{j+1}+1}$
	(where both $\la$'s have the same upper index)
	do not contribute to expansions  
	\eqref{T_eps_exp}--\eqref{Sk_eps_exp}
	because they have order $\varepsilon^{2}$.
\end{proof}
We also clearly have, using \eqref{big_F_8},
\begin{align}
	F_j(\la^{(k-1)},\la^{(k)})
	=O(\varepsilon^{2})+
	\begin{cases}
		0,&j=1;\\
		\varepsilon
		e^{\lalim^{(k)}_{j}
		-\lalim^{(k-1)}_{j-1}},&2\le j\le k;\\
		1,&j=k+1.
	\end{cases}
	\label{F_eps_exp}
\end{align}

To shorten formulas below,
let us introduce the following notation:
\begin{align}
	\label{F_Whittaker}
	\flim _{j}=\flim_{j}(\lalim^{(k-1)},\lalim^{(k)})&:=
	\begin{cases}
		0,&j=1;\\
		\exp\big(
		\lalim^{(k)}_{j}-
		\lalim^{(k-1)}_{j-1}
		\big)
		,& 2\le j\le k;\\
		0,&j=k+1,
	\end{cases}
	\\\rule{0pt}{20pt}
	\label{E_Whittaker}
	\elim _{j}=
	\elim_{j}(\lalim^{(k-1)},\lalim^{(k)})&:=
	\begin{cases}
		\exp\big(
		\lalim^{(k-1)}_{j}-
		\lalim^{(k)}_{j}
		\big),& 1\le j\le k-1;\\
		0,& j=k.
	\end{cases}
\end{align}
The $\flim_j$'s will appear in 
the SDEs in 
connection with 
(long-range) pushing (of upper \emph{right} neighbors), 
and the $\elim_j$'s will
be related to pulling
(of upper \emph{left} neighbors).

It can be readily checked that
(cf. \eqref{T_eps_exp}--\eqref{F_eps_exp})
\begin{align}
	\begin{array}{rcll}
		T_i&=&1-\varepsilon \flim_{i+1}+\varepsilon
		\elim_i+O(\varepsilon^{2}),& 
		\qquad i=1,\ldots,k-1;\\
		\rule{0pt}{12pt}
		S_j&=&1-\varepsilon \flim_{j}+\varepsilon
		\elim_j+O(\varepsilon^{2}),& 
		\qquad j=1,\ldots,k;\\
		\rule{0pt}{12pt}
		F_j&=&1_{j=k+1}+
		\varepsilon \flim_j+O(\varepsilon^{2}),
		& \qquad j=1,\ldots,k+1.
	\end{array}
	\label{E_F_expansions}
\end{align}


\subsubsection{Diffusions related to Whittaker processes} 
\label{ssub:whittaker_multivariate_diffusions}

Here we will write down systems of
SDEs
for diffusions in $\R^{\frac{N(N+1)}2}$
which correspond
to our $q$-Whittaker-multivariate `dynamics'
considered in 
\S \ref{sub:multivariate_dynamics_in_the_q_whittaker_case}.
This correspondence is seen 
with the help of expansions obtained in 
\S \ref{ssub:expansions_of_t_i_and_s_j}.

\begin{remark}
	It seems likely that 
	one can 
	obtain \emph{convergence} (like in
	\cite[Thm. 4.1.27]{BorodinCorwin2011Macdonald})
	of $q$-Whittaker-multivariate
	`dynamics' to the corresponding diffusions.
	That is, it should be possible to 
	prove 
	the actual convergence 
	(under the scaling \eqref{q_to_1_scaling},
	as $\varepsilon\downarrow0$) of 
	measures on trajectories,
	even if the pre-limit
	$q$-Whittaker-multivariate `dynamics'
	admits negative jump rates or probabilities of 
	triggered moves. However, we will not pursue this 
	direction in the present paper,
	and support 
	the correspondence between
	$q$-Whittaker-multivariate `dynamics'
	and systems of SDEs 
	only by informal computations 
	as presented in 
	\S \ref{ssub:informal_argument_for_convergence}
	below.
\end{remark}

First, note that in \emph{any} 
$q$-Whittaker-multivariate `dynamics', the 
bottommost particle $\la^{(1)}_{1}$
jumps to the right independently of other particles
at rate $a_1$. Thus,
the corresponding equation for $\lalim^{(1)}_{1}$
has the following form
(in particular, we have used the scaling
\eqref{q_to_1_scaling} of 
drifts, $a_1=e^{-\varepsilon\cdot \alim_1}$):
\begin{align}\label{G_11_always}
	d\lalim^{(1)}_{1}=dW^{(1)}_{1}-\alim_1 d\tlim,
\end{align}
where $W^{(1)}_1$ is the standard one-dimensional
Brownian motion.
In all systems of SDEs 
below we will assume that $k=2,\ldots,N$, 
and that 
\eqref{G_11_always}
is 
the equation corresponding
to $k=1$.

We will consider the usual four 
families of fundamental nearest neighbor 
`dynamics' (cf. \S \ref{sub:fundamental_dynamics_}),
and write down systems of 
SDEs corresponding to them.
Moreover, one can also consider various mixings
of the fundamental `dynamics' (as defined in 
\S \ref{sub:characterization_of_nearest_neighbor_dynamics_}),
and readily obtain a variety of other diffusions.

\medskip
\noindent
\textbf{(push-block dynamics)}
Scaling limit \eqref{q_to_1_scaling} 
of the push-block dynamics (Dynamics \ref{dyn:q_DF})
was considered in \cite{BorodinCorwin2011Macdonald}
(see Theorem 4.1.27). 
This leads to the following system of SDEs:
\begin{align}\label{push_block_SDE}
	d\lalim^{(k)}_{j}=dW^{(k)}_{j}+
	\big(-\alim_k+\elim_j(\lalim^{(k-1)},\lalim^{(k)})
	-\flim_j(\lalim^{(k-1)},\lalim^{(k)})\big)d\tlim,
\end{align}
where $k=2,\ldots,N$, $j=1,\ldots,k$.
Here and below $\{W^{(k)}_{j}\}_{1\le j\le k\le N}$ 
mean independent
standard one-dimensional Brownian
motions.
Diffusion \eqref{push_block_SDE} 
coincides with the ``symmetric
dynamics'' of \cite[\S9]{Oconnell2009_Toda},
and it degenerates in a certain limit to the Warren process
\cite{warren2005dyson}.

\medskip

\noindent
\textbf{(right-pushing `dynamics')}
Each right-pushing fundamental `dynamics' (Dynamics \ref{dyn:Ri})
depending on $\boldsymbol h
\in\{1\}\times
\{1,2\}\times \ldots\times\{1,2,\ldots,N-1\}$
is a ``minimal perturbation'' of the push-block
dynamics (cf. \S \ref{ssub:conclusion}). 
This leads to a change 
in the equation number $j=h^{(k)}$ 
for each $k=2,\ldots,N$. Other equations
stay the same as in \eqref{push_block_SDE}. 
Thus, we get the following
system corresponding to 
the right-pushing 
`dynamics':\footnote{Here and below we will omit the 
dependence of the $\elim_j$'s and the $\flim_j$'s on
$(\lalim^{(k-1)},\lalim^{(k)})$ which is the same as in  
\eqref{push_block_SDE}.}
\begin{align}\label{right_SDE}	
	d\lalim^{(k)}_{j}
	=
	\begin{cases}
		dW^{(k)}_{j}+
		\big(-\alim_k+\elim_j
		-
		\flim_j
		\big)d\tlim,
		& j\ne h^{(k-1)};\\\rule{0pt}{15pt}
		d\lalim^{(k-1)}_{j}+
		\big(\flim_{j+1}-\flim_j\big)d\tlim,&
		j=h^{(k-1)},
	\end{cases}
\end{align}
where $k=2,\ldots,N$, $j=1,\ldots,k$.
Note that the SDEs \eqref{right_SDE}
involve $1+\frac{N(N-1)}{2}$
independent Brownian motions $W^{(k)}_{j}$
in contrast with $\frac{N(N+1)}{2}$
for the push-block case \eqref{push_block_SDE}.

The right-pushing 
dynamics of Proposition \ref{prop:q_honest_Markov}
corresponds to the system of SDEs 
\eqref{right_SDE}
with $\boldsymbol h=(1,\ldots,1)$ ($N-1$ ones).

\medskip

\noindent
\textbf{(left-pulling `dynamics')}
The system corresponding to the left-pulling
fundamental
`dynamics' (Dynamics
\ref{dyn:Li}) depending on 
$\boldsymbol h
\in\{1\}\times
\{1,2\}\times \ldots\times\{1,2,\ldots,N-1\}$
is similar to \eqref{right_SDE}, and is as follows:
\begin{align}\label{left_SDE}	
	d\lalim^{(k)}_{j}
	=
	\begin{cases}
		dW^{(k)}_{j}+
		\big(-\alim_k+\elim_j
		-
		\flim_j
		\big)d\tlim,
		& j\ne h^{(k-1)}+1;\\\rule{0pt}{15pt}
		d\lalim^{(k-1)}_{j-1}+
		\big(\elim_{j}-\elim_{j-1}\big)d\tlim,&
		j=h^{(k-1)}+1,
	\end{cases}
\end{align}
$k=2,\ldots,N$, $j=1,\ldots,k$.

\medskip

\noindent
\textbf{(RSK-type `dynamics')}
The fundamental RSK-type `dynamics'
(Dynamics \ref{dyn:RS_type_fund})
depends on 
$\boldsymbol h\in\{1\}
\times\{1,2\}\times 
\ldots \times\{1,2,\ldots,N\}$, and corresponds 
to the following system of~SDEs:
\begin{align}\label{RS_SDE}
	d\lalim^{(k)}_{j}=
	\begin{cases}
		d\lalim^{(k-1)}_{j}+
		\big(\flim_{j+1}-\flim_j\big)d\tlim
		,& j<h^{(k)};\\\rule{0pt}{15pt}
		dW^{(k)}_{j}+
		\big(-\alim_k+\elim_j
		-
		\flim_j
		\big)d\tlim,& j=h^{(k)};\\\rule{0pt}{15pt}
		d\lalim^{(k-1)}_{j-1}+
		\big(\elim_{j}-\elim_{j-1}\big)d\tlim,& j>h^{(k)},
	\end{cases}
\end{align}
where $k=2,\ldots,N$, and $j=1,\ldots,k$.
Note that these SDEs use \emph{only} $N$ independent
Brownian motions, namely, $W^{(k)}_{h^{(k)}}$, 
$k=1,\ldots,N$.

For a fixed $\boldsymbol h$,
one can also see that 
the system of SDEs \eqref{RS_SDE}
corresponds to 
the 
$q$-Whittaker-multivariate
`dynamics' with 
deterministic move propagation
dictated by the
$\boldsymbol h$-insertion 
(see Proposition \ref{prop:q_h_insertion}).
Thus, 
for each $\boldsymbol h$
one has  
two $q$-Whittaker-multivariate 
`dynamics' 
resulting in the same 
SDEs~\eqref{RS_SDE}.

Moreover, in the particular case 
$\boldsymbol h=(1,2,\ldots,N)$, 
along with the fundamental
RSK-type `dynamics'
and the `dynamics'
$\bq^{(N)}_{\text{\textit{q-push}}}$
(which comes from Proposition \ref{prop:q_h_insertion}),
there is a third process which corresponds
to \eqref{RS_SDE}, namely,
Dynamics \ref{dyn:collapsed_OConnell}.

\begin{remark}
	Let $k=1,\ldots,N$.
	We note that under any of the systems of SDEs
	\eqref{push_block_SDE}--\eqref{RS_SDE},
	the coordinates
	$\lalim^{(k)}_{1},\ldots,\lalim^{(k)}_{k}$
	must
	evolve according to one and the same diffusion
	process in $\R^{k}$. This statement
	is parallel to the relation between multivariate
	and univariate
	dynamics on Macdonald processes
	(cf. \S \ref{sec:ascending_macdonald_processes} 
	and \S \ref{sec:multivariate_continuous_time_dynamics_on_interlacing_arrays_}).
	The univariate diffusions are related to the
	quantum Toda lattice Hamiltonian, cf. \cite{Oconnell2009_Toda}
	and \cite[\S 5.2]{BorodinCorwin2011Macdonald}.
\end{remark}


\subsubsection{Empiric argument for convergence} 
\label{ssub:informal_argument_for_convergence}

To illustrate how one can obtain SDEs 
from a $q$-Whittaker-multivariate `dynamics', 
let us consider one such `dynamics', 
namely, 
$\bq^{(N)}_{\text{\textit{q-row}}}$
(Dynamics~\ref{dyn:q_row}).
Take a small increment 
$d\tlim$
of the scaled time
$\tlim=\varepsilon^{2}\tau$ (see \eqref{q_to_1_scaling}). 
It corresponds
to a large time $\varepsilon^{-2}d\tlim$
spent by the 
dynamics at the
$q$-Whittaker level.
By \eqref{q_to_1_scaling},
the increments of $\boldsymbol\la$
and $\lablim$ 
during this time
are related as
\begin{align}\label{Delta_la_lalim}
	\underbrace{\la^{(k)}_{j}(\tau+\varepsilon^{-2}d\tlim)
	-\la^{(k)}_{j}(\tau)}_{\Delta \la^{(k)}_{j}}
	=\varepsilon^{-2}
	d\tlim + \varepsilon^{-1}
	\big(
	\underbrace{\lalim^{(k)}_{j}(\tlim+d\tlim)-
	\lalim^{(k)}_{j}(\tlim)}_{\Delta
	\lalim^{(k)}_{j}}
	\big).
\end{align}
On the other hand, we know 
the Markov evolution of $\la^{(k)}_{j}$.
If $j=1$,\footnote{That is, we speak about the 
rightmost particle. Note that these rightmost
particles under $\bq^{(N)}_{\text{\textit{q-row}}}$ 
evolve in a Markovian way as the $q$-PushTASEP, see
\S \ref{ssub:_q_pushtasep}.} 
then during the time $\varepsilon^{-2}d\tlim$
the independent jumps of $\la^{(k)}_{1}$,
will produce a Poisson increment with mean 
$a_k\varepsilon^{-2}d\tlim$;
under the scaling
\eqref{q_to_1_scaling}, 
this Poisson increment 
will turn into a Brownian motion with drift.

In addition to independent jumps, 
the particle
$\la^{(k)}_{1}$ is pushed by 
$\la^{(k-1)}_{1}$
with the following probability (see Dynamics \ref{dyn:q_row}):
\begin{align*}
	1-\frac{1-F_2}{T_1}=
	1- \frac{1-\varepsilon \flim_2}
	{1-\varepsilon \flim_2+\varepsilon \elim_1}
	+O(\varepsilon^{2})=\varepsilon \elim_1
	+O(\varepsilon^{2}).
\end{align*}
(here we used \eqref{E_F_expansions}).
One has to multiply 
this probability by the increment
of $\la^{(k-1)}_{1}$ which can be expressed using 
\eqref{Delta_la_lalim}:
\begin{align*}
	\big(\varepsilon \elim_1
	+O(\varepsilon^{2})
	\big)
	\big(
	\varepsilon^{-2}d\tlim+\varepsilon^{-1}
	\Delta \lalim^{(k-1)}_{1}
	\big)=
	\varepsilon^{-1}\elim_1 d\tlim
	+\elim_1 \cdot\Delta \lalim^{(k-1)}_{1}
	+O(\varepsilon).
\end{align*}
The only term which is relevant is the leading one,
$\varepsilon^{-1}\elim_1 d\tlim$.
Indeed, combining
the Poisson term with the above 
expression, 
we have
\begin{align*}
	\Delta\lalim^{(k)}_{1}
	=\frac{\Delta\la^{(k)}_{1}-\varepsilon^{-2}d\tlim}
	{\varepsilon^{-1}}=
	\frac{Z-\varepsilon^{-2}d\tlim}
	{\varepsilon^{-1}}
	+\elim_1 d\tlim+o(1),
\end{align*}
where $Z$ is the Poisson random variable
with mean $a_k\varepsilon^{-2}d\tlim$.
The first summand clearly gives the term
$dW^{(k)}_{1}-\alim_k d\tlim$, and the second one 
is equal to $\elim_1d\tlim=(\elim_1-\flim_1)d\tlim$
(cf. \eqref{F_Whittaker}).
In this way we get the first equation
of the desired system of SDEs
\eqref{RS_SDE}
with $\boldsymbol h=(1,\ldots,1)$.

All other equations
are obtained in a similar manner. Other particles
do not perform independent jumps, so 
only the equation for $d\lalim^{(k)}_{1}$ contains
the differential of a Brownian motion.


\subsubsection{Schematic pictures again. 
Geometric (tropical) RSK} 
\label{ssub:schematic_pictures_and_further_connections}

One can associate to any system
of SDEs \eqref{push_block_SDE}--\eqref{RS_SDE}
a schema\-tic
picture as on
Fig.~\ref{fig:big_rs}, \ref{fig:row_col}, 
and \ref{fig:pb_left_right}.
Namely, if at some level 
$k$, a particle $\la^{(k)}_{j}$
jumps independently (i.e., 
there are no 
dashed arrows ending at 
$\la^{(k)}_{j}$), then the corresponding
equation looks as 
$d\lalim^{(k)}_{j}=
dW^{(k)}_{j}+
\big(-\alim_k+\elim_j
-
\flim_j
\big)d\tlim$.
If $\la^{(k)}_{j}$
is at the end of a dashed arrow
pointing to the \emph{right}, 
then the equation has the form
$d\lalim^{(k)}_{j}=
d\lalim^{(k-1)}_{j}+
\big(\flim_{j+1}
-
\flim_j
\big)d\tlim$.
Finally, if an arrow points to the \emph{left} and 
ends at 
$\la^{(k)}_{j}$, then the equation is
$d\lalim^{(k)}_{j}=
d\lalim^{(k-1)}_{j-1}+
\big(\elim_{j}
-
\elim_{j-1}
\big)d\tlim$.

This observation
should be relevant to the
geometric (sometimes called tropical) 
Robinson--Schensted(--Knuth) correspondence
\cite{Kirillov2000_Tropical}, 
\cite{NoumiYamada2004}, 
\cite{Oconnell2009_Toda},
\cite{COSZ2011} (see, e.g., the beginning of
\S \ref{sub:_boldsymbol_h_robinson_schensted_correspondences}
for the explaination of RS/RSK terminology).
In particular, it should be 
possible to define 
$\boldsymbol h$-generalized
versions of the geometric
correspondence 
with the help of the SDEs \eqref{RS_SDE}.
This construction has to be in some sense parallel
(cf. \cite{NoumiYamada2004})
to what was done in 
\S \ref{sub:_boldsymbol_h_robinson_schensted_correspondences}
for the usual RS correspondences. 


\subsubsection{Involution} 
\label{ssub:symmetry}

Let us now discuss how the involution
described in 
\cite{Oconnell2009_Toda} and 
\cite[\S5.2.1]{BorodinCorwin2011Macdonald}
applies to our systems of SDEs
\eqref{push_block_SDE}--\eqref{RS_SDE}. 
This involution consists of two steps:
\begin{enumerate}[\bf{}1.]
	\item Change the sign of 
	the standard Brownian motions and of the drifts:
	\begin{align*}
		W^{(k)}_{j}\mapsto-W^{(k)}_{j},
		\qquad 
		\alim_k\mapsto -\alim_k,
		\qquad
		1\le j\le k\le N.
	\end{align*}
	\item Change the sign and 
	order of the $\lalim^{(k)}_j$'s 
	for fixed $k$:
	\begin{align*}
		\lalim^{(k)}_j\mapsto-
		\lalim^{(k)}_{k+1-j},
		\qquad
		1\le j\le k\le N.
	\end{align*}
\end{enumerate}
One can readily see (by definition
\eqref{F_Whittaker}--\eqref{E_Whittaker})
that this involution results in the following
swapping of the $\flim_j$'s and 
the $\elim_j$'s:
\begin{align*}
	\flim_j\mapsto \elim_{k+1-j},
	\qquad
	\elim_j\mapsto \flim_{k+1-j},
	\qquad
	1\le j\le k\le N.
\end{align*}

Under this involution, the system of SDEs
corresponding to the push-block multivariate
dynamics (``symmetric
dynamics'' of \cite[\S9]{Oconnell2009_Toda})
does not change.

The systems \eqref{right_SDE} and \eqref{left_SDE}
swap. More precisely, if 
\eqref{right_SDE} corresponds to the parameters
$\boldsymbol h=(h^{(1)},\ldots,h^{(N-1)})$, then
the involution takes this system to
\eqref{left_SDE} with another
$\boldsymbol h=(2-h^{(1)},3-h^{(2)},
\ldots,
k+1-h^{(k)},
,\ldots,
N-h^{(N-1)})$.

An RSK-type systems of SDEs 
\eqref{RS_SDE} stays RSK-type,
but changes its parameter 
$\boldsymbol h$
in the same way as above,
\begin{align*}
	(h^{(1)},\ldots,h^{(N)})
	\mapsto
	(2-h^{(1)},\ldots,k+1-h^{(k)},\ldots,N+1-h^{(N)}).
\end{align*}
In particular (as was implicitly observed 
in \cite[\S5.2]{BorodinCorwin2011Macdonald}), 
the ``row'' system with $\boldsymbol h=(1,1,\ldots,1)$
turns into the ``column'' one
with $\boldsymbol h=(1,2,\ldots,N)$.

In terms of schematic pictures 
(\S \ref{ssub:schematic_pictures_and_further_connections}), 
the involution
has a graphical interpretation: it
simply reflects the picture 
with respect to the vertical axis.


\subsubsection{Connection of $q$-PushTASEP to the O'Connell--Yor semi-discrete\\ directed polymer} 
\label{sub:connections_with_o_connell_yor_semi_discrete_directed_polymer}

Let us first briefly recall the definition of the 
O'Connell--Yor polymer partition 
function \cite{OConnellYor2001}, \cite{Oconnell2009_Toda}. 
Let $B_1,\ldots,B_N$ be $N$ independent
standard Brownian motions such that $B_i$ has drift $\blim_i$.
Let us take some $k=1,\ldots,N$, and
define for $0<s_1<\ldots<s_{k-1}<\tlim$:
\begin{align*}
	E_{s_1,\ldots,s_{k-1}}:=
	B_1(s_1)+
	\big(B_2(s_2)-B_2(s_1)\big)
	+\ldots+
	\big(B_k(\tlim)-B_k(s_{k-1})\big).
\end{align*}
This may be regarded as the energy of an
up-right path in $\R\times\Z$
from $(0,1)$ to $(\tlim,k)$ which either proceeds to the 
right or jumps up by one unit. Here $s_1,\ldots,s_{k-1}$
are moments of jumps.

The \emph{semi-discrete directed polymer partition function}
is given by
\begin{align*}
	\Zf^{(k)}(\tlim):=
	\int_{0<s_1<\ldots<s_{k-1}<\tlim} 
	e^{E_{s_1,\ldots,s_{k-1}}}ds_1 \ldots ds_{k-1},
\end{align*}
where the integral is taken over the $(k-1)$-dimensional
simplex $0<s_1<\ldots<s_{k-1}<\tlim$ with respect to 
the Lebesgue measure
$ds_1 \ldots ds_{k-1}$ on this simplex.
In fact, for each $k$, one gets its own
partition function $\Zf^{(k)}(\tlim)$,
so we have an hierarchy of $N$
partition functions. 
\begin{remark}\label{rmk:free_energy_big_hierarchy}
	One can 
	introduce even more partition functions 
	(indexed by $(k,j)$ with $1\le j\le k\le N$),
	cf.
	\cite{OConnellWarren2011}, 
	\cite{Oconnell2009_Toda}, and
	\cite[\S 5.2.1]{BorodinCorwin2011Macdonald}. 
	They correspond to taking several nonintersecting
	up-right paths.
\end{remark}

The \emph{free energies}
$\mathsf{F}^{(k)}(\tlim):=
\log(\Zf^{(k)}(\tlim))$, $k=1,2,\ldots,N$,
satisfy a certain system of SDEs. Let us explain how one can 
intuitively write down this system. The $\Zf^{(k)}$'s 
can be written in the following hierarchical form:
\begin{align*}
	\Zf^{(k)}(\tlim)=
	\int_{s_{k-2}}^{\tlim}ds_{k-1}\,
	e^{B_k(\tlim)-B_k(s_{k-1})}\,\Zf^{(k-1)}(s_{k-1}).
\end{align*}
Formally taking the $\tlim$ derivative, one gets
\begin{align*}
	d\Zf^{(k)}
	=\Zf^{(k-1)}d\tlim+
	\Zf^{(k)}d B_k=
	\Zf^{(k)}dW_k+
	\big(\blim_k\Zf^{(k)}+
	\Zf^{(k-1)}
	\big)d\tlim,
\end{align*}
where $W_k$ is the standard one-dimensional \emph{driftless}
Brownian motion.
Thus, the free energies satisfy the following
SDEs:
\begin{align}\label{Free_energy_SDEs}
	d\mathsf{F}^{(k)}=dW_k+
	\big(\blim_k+e^{\mathsf{F}^{(k-1)}-\mathsf{F}^{(k)}}\big)
	d\tlim,
	\qquad k=1,\ldots,N.
\end{align}

Clearly, these SDEs 
for the free energies
arise as parts 
corresponding to $\lalim^{(k)}_{1}$, $k=1,\ldots,N$,
of some of the systems
\eqref{push_block_SDE}--\eqref{RS_SDE}
(not all of them, in parallel to the fact that not all 
multivariate `dynamics' are left- or right-Markov, cf.
\S \ref{sub:taseps}).
Namely, one can find the SDEs \eqref{Free_energy_SDEs}
in the push-block system \eqref{push_block_SDE},
as well as in systems \eqref{right_SDE} and \eqref{RS_SDE}
for $\boldsymbol h$ consisting of all 1's (in both cases).
One should also set $\blim_k:=-\alim_k$.

The $q$-PushTASEP (\S \ref{ssub:_q_pushtasep}) 
is a proper discretization of 
the system \eqref{Free_energy_SDEs} in the sense that
under the scaling \eqref{q_to_1_scaling}, the evolution
of $q$-PushTASEP converges to 
the diffusions \eqref{Free_energy_SDEs}.
This may be seen from the
argument in \S \ref{ssub:informal_argument_for_convergence}
for the rightmost particles.

To observe connections with $q$-TASEP, one should
consider a larger hierarchy of free energies
(cf. Remark \ref{rmk:free_energy_big_hierarchy}).
We refer to \cite[\S 5.2.1]{BorodinCorwin2011Macdonald}
for this connection.




\appendix

\section{Macdonald polynomials and related objects} 
\label{sec:macdonald_processes}

In the appendix, we recall the definitions of symmetric 
functions, Macdonald polynomials, and other related objects. 
To make the presentation self-contained, we will list all the 
necessary facts and formulas along the way. Our exposition is 
based on \cite{Macdonald1995} (especially on Chapter VI); 
some parts of it closely 
follow \cite[\S2]{BorodinCorwin2011Macdonald}.

\subsection{Symmetric functions. Specializations} 
\label{sub:symmetric_functions}

Let $\Sym$ denote the 
\emph{algebra of symmetric functions}. 
The detailed definition and properties of 
$\Sym$ may be found in \cite[I.2]{Macdonald1995}. 
Here we will list facts that are 
important for the present paper. 

We understand $\Sym$ as a commutative algebra $\R[p_1,p_2,\ldots]$ which is generated by 1 and by the (algebraically independent) power sums
\begin{align*}
	p_k(x_1,x_2,\ldots)=\sum_{i=1}^{\infty}x_i^{k},\qquad
	k=1,2,\ldots.
\end{align*}
The products of power sums $p_\la:=p_{\la_1}\ldots p_{\la_\ell(\la)}$, where $\la$ runs over the set $\GT^+$ of all partitions (with the agreement $p_\varnothing=1$), form a linear basis in $\Sym$. The algebra $\Sym$ possesses a natural grading determined by setting $\deg p_k=k$, $k=1,2,\ldots$.

By a \emph{specialization} of the algebra $\Sym$ we mean an algebra homomorphism $\rho\colon\Sym\to\R$. Such a map is completely determined by its values $\rho(p_k)$ on the power sums. The \emph{trivial} specialization $\varnothing$ is defined as taking value 1 at the constant function $1\in\Sym$ and sending all the power sums $p_k$, $k\ge1$, to zero.

For two specializations $\rho_1$ and $\rho_2$ we define their \emph{union} $\rho=(\rho_1,\rho_2)$ as the specialization defined on power sums as
\begin{align*}
p_k(\rho_1,\rho_2)=p_k(\rho_1)+p_k(\rho_2), \qquad k\ge 1.
\end{align*}

Important examples of specializations are the so-called \emph{finite length specializations}. Fix $N\ge1$ and let $y_1,\ldots,y_N$ be real numbers (we may also treat $y_i$'s as formal variables). Set
\begin{align*}
	\rho_{y_1,\ldots,y_N}(p_k)
	=p_k(y_1,\ldots,y_N):=y_1^{k}+\ldots+y_N^{k}.
\end{align*}
This specialization turns $\Sym$ into the algebra of symmetric polynomials in $N$ variables $y_1,\ldots,y_N$. 

In fact, every symmetric function $f\in\Sym$ can be understood as a sequence of symmetric polynomials $f_N(y_1,\ldots,y_N)$, $N=1,2,\ldots$, in $N$ variables of bounded degree (i.e., $\sup_N\deg f_N<\infty$), which are compatible with the operation of setting the last variable to zero: $f_{N+1}(y_1,\ldots,y_N,0)=f_{N}(y_1,\ldots,y_N)$. We have $f_N(y_1,\ldots,y_N)=\rho_{y_1,\ldots,y_N}(f)$.

The finite length specializations suggest the notation: 
for a symmetric function $f\in\Sym$ and 
a specialization $\rho$ 
we will often write $f(\rho)$ instead of $\rho(f)$.


\subsection{Macdonald symmetric functions} 
\label{sub:macdonald_symmetric_functions}

Macdonald symmetric functions form a remarkable two-parameter family of symmetric functions depending on parameters $q,t\in[0,1)$ 
(there parameters can also be considered formal). They are indexed by all partitions $\la\in\GT^+$ and may be defined as follows \cite[VI.4]{Macdonald1995}. Define first the scalar product $\langle\cdot,\cdot\rangle_{q,t}$ on $\Sym$ by
\begin{align*}
	\langle p_\la,p_\mu\rangle_{q,t}=\delta_{\la\mu}
	z_\la(q,t),\qquad
	z_\la(q,t):=\bigg(\prod_{i\ge1}i^{m_i}(m_i)!\bigg)
	\cdot
	\bigg(\prod_{i=1}^{\ell(\la)}
	\frac{1-q^{\la_i}}{1-t^{\la_i}}\bigg),
\end{align*}
where $\la=(1^{m_1}2^{m_2}\ldots)$ means that $\la$ has $m_1$ parts equal to 1, $m_2$ parts equal to 2, etc.

\begin{definition}
	The \emph{Macdonald symmetric functions} 
	$P_\la(x;q,t)$, $\la\in\GT^+$, 
	form a unique family of 
	homogeneous symmetric functions such that:
	\begin{enumerate}[\bf{}1.]
		\item The functions are pairwise 
		orthogonal with respect to 
		the scalar product $\langle\cdot,\cdot\rangle_{q,t}$.
		\item For every $\la$, we have
		\begin{align*}
			P_\la(x;q,t)=
			x_1^{\la_1}\ldots x_{\ell(\la)}^{\la_{\ell(\la)}}{}+
			{}\mbox{lower monomials in lexicographic order}.
		\end{align*}
		The dependence on the parameters $(q,t)$ 
		is in coefficients of the 
		lexicographically 
		lower monomials.\footnote{Lexicographic order 
		means that, for example, $x_1^{2}$ is higher 
		than $\mathrm{const}\cdot x_1x_2$ 
		which is in turn higher 
		than $\mathrm{const}\cdot x_2^{2}$.}
	\end{enumerate}
\end{definition}
Define 
\begin{align*}
	Q_\la:=
	\frac{P_\la}{\langle P_\la,P_\la \rangle_{q,t}},
	\qquad \la\in\GT^+,
\end{align*}
so that the functions $P_\la$ and $Q_\mu$ are 
orthonormal. 
(We will sometimes omit 
the notation $(q,t)$, 
and simply write $P_\la(x)$ 
or $P_\la$ instead of $P_\la(x;q,t)$;  
similarly for $Q_\la$.)

The \emph{Macdonald polynomials} are finite length specializations of the Macdonald symmetric functions:
\begin{align}\label{Macd_poly}
	P_\la(x_1,\ldots,x_N)=\rho_{x_1,\ldots,x_N}(P_\la),\qquad
	\la\in\GT^+.
\end{align}
If $N<\ell(\la)$, then $P_\la(x_1,\ldots,x_N)=0$. 

\begin{remark}\label{rmk:Macdonald_negative_sign}
	The Macdonald polynomials possess the following \emph{index shift property}: 
	\begin{align*}
		(x_1\cdot\ldots\cdot x_N)\cdot 
		P_\la(x_1,\ldots,x_N)=
		P_{\la+1}(x_1,\ldots,x_N), 
	\end{align*}
	where $\ell(\la)=N$, and $\la+1$ is the partition $(\la_1+1,\ldots,\la_N+1)$. Using this property, we may define the Macdonald symmetric polynomials $P_\la$ in $N$ variables $x_1,\ldots,x_N$ for every $\la\in\GT_N$ (i.e., for not necessarily nonnegative signatures). If the signature $\la$ has negative parts, then $P_\la$ is a Laurent polynomial.
\end{remark}

\emph{The $q$-Whittaker symmetric functions} are simply the Macdonald symmetric functions with $t=0$. We will sometimes denote them by $P_\la(x;q,t=0)$. Their name comes from the fact that the Macdonald polynomials in $\ell+1$ variables with $t=0$ are the $q$-deformed $\mathfrak{gl}_{\ell+1}$ Whittaker functions \cite{GerasimovLebedevOblezin2011}. 

When $q=t$, the Macdonald symmetric 
functions turn into the 
Schur symmetric functions $s_\la$. 
This is also true for the $q$-Whittaker functions: 
when $q=0$, they become the Schur functions. 

Other remarkable special cases of the 
Macdonald symmetric functions include the 
Hall-Little\-wood symmetric functions ($q=0$) 
and the Jack symmetric functions ($t=q^{\theta}$ and $q\to1$). 

In the main part of the present paper we focus on the 
general Macdonald case ($0<q,t<1$) and on
its degenerations to the $q$-Whittaker case
($t=0$), and further to the Schur ($q=t=0$) case.

\begin{remark}\label{rmk:q_whitt}
	In the appendix 
	in \S\S \ref{sub:symmetric_functions}--\ref{sub:commuting_markov_operators}
	we will write formulas and give definitions for the general Macdonald case only, i.e., with parameters $(q,t)$. The corresponding definitions and properties in the $q$-Whittaker case are obtained by taking the limits as $t\to0$, which exist and are readily written out. See also \cite[\S3.1]{BorodinCorwin2011Macdonald} for more references and for properties which are specific to the $q$-Whittaker functions.

	In 
	\S \ref{sub:schur_appendix}
	we discuss Schur functions and related 
	objects.
\end{remark}


\subsection{Skew shapes and skew semistandard tableaux} 
\label{sub:skew_shapes_and_semistandard_tableaux}

For two Young diagrams $\la,\mu\in\GT^+$ such that $\mu\subseteq\la$, the \emph{skew shape} $\la/\mu$ is defined as the set difference $\la\setminus\mu$. If $\mu=\varnothing$, one has $\la/\varnothing=\la$. There are two particular cases of skew shapes which are of interest: 
\begin{enumerate}[\bf{}1.]
	\item \emph{Horizontal strip} is a skew shape having no more than one box in every column. The fact that $\la/\mu$ is a horizontal strip means precisely that $\mu$ and $\la$ interlace: $\mu\prec\la$.\footnote{Here, by agreement, we choose $N$ so large that $\mu\in\GT_{N-1}^{+}$ and $\la\in\GT_N^{+}$ (recall that we may append nonnegative signatures by zeroes, see \S \ref{sub:signatures_young_diagrams_and_interlacing_arrays}).}
	\item \emph{Vertical strip} is a skew shape having no more than one box in every row. A skew shape $\la/\mu$ is a vertical strip iff the transposed skew shape $\la'/\mu'$ is a horizontal strip.\footnote{The transposition $\la\mapsto\la'$ is defined for any Young diagram $\la\in\GT^+$; it interchanges rows and columns of this Young diagram.}
\end{enumerate}

Let us extend the definition of a 
semistandard Young tableau 
(\S \ref{sub:semistandard_young_tableaux}) 
to skew shapes. 
We will use the identification of semistandard
tableaux with interlacing integer arrays 
(Proposition \ref{prop:SSYT}). 
A \emph{semistandard Young tableau} 
of skew shape $\la/\mu$ over the alphabet 
$\{1,\ldots,k\}$, where $\mu,\la\in\GT^+$ and $k\ge1$, 
can be defined as a sequence 
of interlacing nonnegative signatures
\begin{align}\label{trapezoidal_GT_schemes}
	\mu=\nu^{(N-k)}\prec\nu^{(N-k+1)}\prec
	\ldots\prec\nu^{(N)}=\la,
\end{align}
where $N$ is so large that 
$\mu\in\GT_{N-k}^{+}$ and $\la\in\GT_N^{+}$, 
and $\nu^{(j)}\in\GT_j^{+}$. 
Clearly, a semistandard tableau of skew shape $\la/\mu$
can be also viewed as 
an interlacing integer array 
(as on Fig.~\ref{fig:GT_scheme}) 
of trapezoidal shape having depth $k$, 
top row $\la$, and bottom row~$\mu$. 
For $\mu=\varnothing$ we return to the situation 
described in \S \ref{sub:semistandard_young_tableaux}.


To every skew semistandard Young tableau $\Ptab$ as in \eqref{trapezoidal_GT_schemes} one can associate a monomial as follows:
\begin{align}\label{tableau_monomial}
	x^{\Ptab}:=x_1^{|\nu^{(N-k+1)}|-|\nu^{(N-k)}|}
	x_2^{|\nu^{(N-k+2)}|-|\nu^{(N-k+1)}|}
	\ldots
	x_k^{|\nu^{(N)}|-|\nu^{(N-1)}|}.
\end{align}

\begin{remark}\label{rmk:negative_skew_tableaux}
	The above definition of a semistandard tableau can be extended to the following two cases allowing negative parts in signatures:
	
	\textbf{1.} If $\mu\in\GT_{N-k}$ and $\la\in\GT_N$, then semistandard tableaux of depth~$k$ (i.e., over the alphabet $\{1,\ldots,k\}$) are well-defined as sequences of interlacing signatures \eqref{trapezoidal_GT_schemes} (now these signatures are not necessarily nonnegative).

	\textbf{2.} If $\mu,\la\in\GT_N^{+}$, then one readily sees that the skew shape (i.e., the set difference of Young diagrams) $\la/\mu$ is the same as $(\la+1)/(\mu+1)$ (cf. Remark~\ref{rmk:Macdonald_negative_sign}). This allows to define $\la/\mu$ for every $\la,\mu\in\GT_N$ such that $\mu\subseteq\la$ (i.e., $\mu_i\le\la_i$ for all $i=1,\ldots,N$). We can also define the corresponding semistandard tableaux of shape $\la/\mu$ via definition \eqref{trapezoidal_GT_schemes} that worked for nonnegative signatures.

	In both these cases it is also 
	clear how to assign a monomial \eqref{tableau_monomial}
	to each semistandard skew tableau.
\end{remark}


\subsection{Skew Macdonald functions and polynomials} 
\label{sub:skew_functions}

\begin{definition}
	A \emph{skew Macdonald symmetric function} $Q_{\la/\mu}$ indexed by $\mu,\la\in\GT^+$ is defined as the only symmetric function such that $\langle Q_{\la/\mu},P_\nu\rangle_{q,t}=\langle Q_{\la},P_\mu P_\nu\rangle_{q,t}$ for all $\nu\in\GT^+$. 

	The $P$ version is then defined through $Q_{\la/\mu}$ 
	as $P_{\la/\mu}:=
	\dfrac{\langle P_\la,P_\la\rangle_{q,t}}
	{\langle P_\mu,P_\mu\rangle_{q,t}}Q_{\la/\mu}$.
\end{definition}
The skew functions vanish unless $\mu\subseteq\la$, i.e., unless $\mu_j\le \la_j$ for all $j=1,\ldots,\ell(\la)$. One also has $P_{\la/\varnothing}=P_\la$ and $Q_{\la/\varnothing}=Q_\la$.

There are combinatorial formulas for the skew functions $P_{\la/\mu}$ and $Q_{\la/\mu}$ expressing them as sums over semistandard Young tableaux of skew shape~$\la/\mu$. Let us consider specializations into finitely many variables $x_1,\ldots,x_k$ (which completely determine the corresponding symmetric functions, cf. \S \ref{sub:symmetric_functions}). We have \cite[VI.7]{Macdonald1995}
\begin{align}\label{combinatorial_skew_Macdonald}
	P_{\la/\mu}(x_1,\ldots,x_k)=\sum_{\Ptab}\psi_\Ptab x^{\Ptab}
	,\qquad
	Q_{\la/\mu}(x_1,\ldots,x_k)=\sum_{\Ptab}\varphi_\Ptab 
	x^{\Ptab}.
\end{align}
Both sums are taken over all semistandard 
Young tableaux of shape $\la/\mu$ over 
the alphabet $\{1,\ldots,k\}$ \eqref{trapezoidal_GT_schemes}, 
and $x^{\Ptab}$ is defined in \eqref{tableau_monomial}. 
If there are no such tableaux, then 
the corresponding polynomials are zero. 
The coefficients $\psi_\Ptab$ and $\varphi_\Ptab$ are 
defined via the following two steps:

\textbf{1.} For interlacing partitions $\ka\prec\nu$, $\ka,\nu\in\GT^+$, we set 
\begin{align}\label{psi_horizontal_strip_Macd}
	\psi_{\nu/\ka}=\psi_{\nu/\ka}(q,t)&=\prod_{1\le i\le j\le \ell(\ka)}
	\frac{f(q^{\ka_i-\ka_j}t^{j-i})f(q^{\nu_i-\nu_{j+1}}t^{j-i})}
	{f(q^{\nu_i-\ka_j}t^{j-i})f(q^{\ka_i-\nu_{j+1}}t^{j-i})},\\
	\label{phi_horizontal_strip_Macd}
	\varphi_{\nu/\ka}=\varphi_{\nu/\ka}(q,t)&=\prod_{1\le i\le j\le \ell(\nu)}\frac{f(q^{\nu_i-\nu_j}t^{j-i})f(q^{\ka_i-\ka_{j+1}}t^{j-i})}
	{f(q^{\nu_i-\ka_j}t^{j-i})f(q^{\ka_i-\nu_{j+1}}t^{j-i})},
\end{align}
where $f(u):=(tu;q)_{\infty}/(qu;q)_{\infty}$, and the (infinite) $q$-Pochhammer symbol is defined as 
\begin{align*}
	(a;q)_{\infty}:=\prod_{i=0}^{\infty}(1-aq^{i})=(1-a)(1-aq)(1-aq^{2})\ldots.
\end{align*}
Recall that $0\le q<1$, so the infinite product converges.

\textbf{2.} For a semistandard Young tableau \eqref{trapezoidal_GT_schemes} of shape $\la/\mu$ we set
\begin{align*}
	\psi_\Ptab=
	\psi_{\nu^{(N-k+1)}/\nu^{(N-k)}}
	\psi_{\nu^{(N-k+2)}/\nu^{(N-k+1)}}
	\ldots
	\psi_{\nu^{(N)}/\nu^{(N-1)}},
\end{align*}
and similarly for $\varphi_\Ptab$.
\begin{remark}
	[cf. Remarks \ref{rmk:Macdonald_negative_sign} and \ref{rmk:negative_skew_tableaux}]
	\label{rmk:negative_skew_functions}
	\textbf{1.}	If, say, $\ka\in\GT_{m-1}$ and $\nu\in\GT_m$ (not necessarily nonnegative signatures), then the quantities $\psi_{\nu/\ka}$ and $\varphi_{\nu/\ka}$ are still well-defined by \eqref{psi_horizontal_strip_Macd}--\eqref{phi_horizontal_strip_Macd}; one should simply replace $\ell(\ka)$ by $m-1$ and $\ell(\la)$ by $m$. Moreover, $\psi_{\nu/\ka}$ and $\varphi_{\nu/\ka}$ are translation-invariant: they do not change if one replaces $\ka$ and $\nu$ by $\ka+1$ and $\nu+1$, respectively.

	Thus, if $\la\in\GT_N$ and $\mu\in\GT_{N-k}$, we may define the (in general, Laurent) polynomials $P_{\la/\mu}(x_1,\ldots,x_k)$ and $Q_{\la/\mu}(x_1,\ldots,x_k)$ by \eqref{combinatorial_skew_Macdonald}. (Note that Remark \ref{rmk:Macdonald_negative_sign} is a particular case of this definition when $\mu=\varnothing$.)

	\textbf{2.}
	If $\mu,\la\in\GT_N^{+}$, we have $P_{\la/\mu}=P_{\la+1/\mu+1}$, and same for $Q_{\la/\mu}$. Thus, we may define the skew (ordinary, not Laurent) polynomials $P_{\la/\mu}$ and $Q_{\la/\mu}$ in any number of variables for not necessarily nonnegative $\la,\mu\in\GT_N$. 
	They vanish unless $\mu_i\le \la_i$ for all $i$. 
	This implies that the symmetric 
	functions $P_{\la/\mu},Q_{\la/\mu}\in\Sym$ 
	are also well-defined in this case.
\end{remark}

In particular, for any $\mu\in\GT_{N-1}$ and $\la\in\GT_N$ one has
\begin{align}\label{P_one_variable}
	P_{\la/\mu}(x_1)=
	\begin{cases}
		\psi_{\la/\mu}x_1^{|\la|-|\mu|},&\mbox{$\la/\mu$ is a horizontal strip},\\
		0,&\mbox{otherwise},
	\end{cases}
\end{align}
and 
\begin{align}\label{Q_one_variable}
	Q_{\la/\mu}(x_1)=
	\begin{cases}
		\varphi_{\la/\mu}x_1^{|\la|-|\mu|},&\mbox{$\la/\mu$ is a horizontal strip},\\
		0,&\mbox{otherwise}.
	\end{cases}
\end{align}


\subsection{Macdonald-nonnegative specializations} 
\label{sub:nonnegative_specializations}

A specialization $\rho$ of the algebra of symmetric functions $\Sym$ (\S \ref{sub:symmetric_functions}) is said to be \emph{Macdonald nonnegative} if it takes nonnegative values on all skew Macdonald symmetric functions: $P_{\lambda/\mu}(\rho;q,t)\ge 0$ for any partitions $\la,\mu\in\GT^+$.

There is no known classification of Macdonald-non\-ne\-ga\-tive specializations.\footnote{The answer for the $q$-Whittaker case is also unknown. On the other hand, specializations taking nonnegative values on Jack (and, in particular, Schur) symmetric functions are completely described: see \cite{Kerov1998} and references therein.} A wide class of nonnegative specializations was considered by Kerov \cite[II.9]{Kerov-book}. He conjectured that they exhaust all possible nonnegative specializations. 

These specializations depend on nonnegative parameters $\{\al_i\}_{i\ge1}$, $\{\be_i\}_{i\ge1}$ and $\ga$ such that $\sum_{i=1}^{\infty}(\al_i+\be_i)<\infty$. They are defined on the power sums via the exponent of a generating function (in the formal variable $u$) as follows:
\begin{align}\label{Pi_nonneg_spec}
	\exp\bigg(
	\sum_{n=1}^{\infty}\frac{1}{n}
	\frac{1-t^{n}}{1-q^{n}}p_n(\rho)u^n
	\bigg)=
	\exp(\gamma u) \prod_{i\ge 1} \frac{(t\alpha_iu;q)_\infty}{(\alpha_i u;q)_\infty}\,(1+\beta_i u)=: \Pi(u;\rho).
\end{align}
In more detail, this means that
\begin{align*}
	p_1(\rho)&=\sum_{i\ge1}\al_i+\bigg(\ga+\sum_{i\ge1}\be_i\bigg)\frac{1-q}{1-t},\qquad \qquad
	p_k(\rho)&=\sum_{i\ge1}\al_i^{k}+
	(-1)^{k-1}\frac{1-q^{k}}{1-t^{k}}\sum_{i\ge1}\be_i^{k},
\end{align*}
where $k=2,3,\ldots$.
It can be verified that \eqref{Pi_nonneg_spec} defines a Macdonald-nonnegative specialization, cf. \cite[Prop. 2.2.2]{BorodinCorwin2011Macdonald}.

When $\ga=0$, all $\be_i=0$, and only finitely many of the $\al_i$'s are nonzero, then the specialization defined by \eqref{Pi_nonneg_spec} is reduced to a finite length specialization discussed in \S \ref{sub:symmetric_functions}.

We will refer to $\be_i$ as to \emph{dual variables}. We will often denote by $\hat \be_1$ the specialization with a singe nonzero dual variable $\be_1>0$ and with $\ga=0$, $\al_1=\al_2=\ldots=0$, $\be_2=\be_3=\ldots=0$.

The specialization with $\al_j=\be_j=0$ for all $j$ and $\ga\ge0$ will be called \emph{Plancherel} and denoted by $\rho_\ga$.


\subsection{Endomorphism $\omega_{q,t}$ and dual specializations} 
\label{sub:endomorphism_q,t_}

There is an endomorphism of the algebra $\Sym$ of symmetric functions which is defined on the power sums as \cite[VI]{Macdonald1995}
\begin{align*}
	\omega_{q,t}p_k:=(-1)^{k-1}\frac{1-q^{k}}{1-t^{k}}p_k,\qquad k=1,2,\ldots.
\end{align*}
We have $\omega_{q,t}\omega_{t,q}=\mathrm{id}$, and 
\begin{align*}
	\omega_{q,t}P_{\la/\mu}(x;q,t)
	=Q_{\la'/\mu'}(x;t,q),\qquad
	\omega_{q,t}Q_{\la/\mu}(x;q,t)=P_{\la'/\mu'}(x;t,q),
\end{align*}
where $\mu,\la\in\GT^+$ and $\mu'$ and $\la'$ are the transposed Young diagrams.

One readily sees that applying the endomorphism $\omega_{t,q}\colon \Sym\to\Sym$ and then a $(q,t)$-Macdonald-nonnegative specialization $\Sym\to\C$, one gets another specialization which is now $(t,q)$-Macdonald-nonnegative:
\begin{align*}
	\rho\{\al,\be;\ga\mid q,t\}
	\circ\omega_{t,q}=
	\rho\{\be,\al;\tfrac{1-q}{1-t}\ga\mid t,q\}.
\end{align*}
Here by $\rho\{\al,\be;\ga\mid q,t\}$ we have denoted the specialization defined by \eqref{Pi_nonneg_spec}. This observation implies that under the single beta specialization $\hat\be_1$ one has 
\begin{align}\label{P_one_dual_variable}
	P_{\la/\mu}(\hat\be_1)=
	\begin{cases}
		\varphi'_{\la/\mu}\be_1^{|\la|-|\mu|},&\mbox{$\la/\mu$ is a vertical strip},\\
		0,&\mbox{otherwise},
	\end{cases}
\end{align}
and 
\begin{align}\label{Q_one_dual_variable}
	Q_{\la/\mu}(\hat\be_1)=
	\begin{cases}
		\psi'_{\la/\mu}\be_1^{|\la|-|\mu|},&\mbox{$\la/\mu$ is a vertical strip},\\
		0,&\mbox{otherwise}.
	\end{cases}
\end{align}
Here $\varphi'_{\la/\mu}(q,t):=\varphi_{\la'/\mu'}(t,q)$ and $\psi'_{\la/\mu}(q,t):=\psi_{\la'/\mu'}(t,q)$, and the quantities $\psi_{\nu/\ka}$ and $\varphi_{\nu/\ka}$ are given in \eqref{psi_horizontal_strip_Macd} and \eqref{phi_horizontal_strip_Macd}, respectively.

Thus defined ``dual'' quantities $\varphi'_{\la/\mu}$ and $\psi'_{\la/\mu}$ make sense if $\la,\mu\in\GT^+$ are nonnegative signatures. In particular, for $\la/\mu$ a vertical strip,
\begin{align}\label{psi_prime_vertical}
	\psi'_{\lambda/\mu}=
	\prod_{\substack{i<j\\ 
	\lambda_i=\mu_i,\lambda_j=\mu_{j}+1}}
	\frac{(1-q^{\mu_i-\mu_j}t^{j-i-1})
	(1-q^{\lambda_i-\lambda_j}t^{j-i+1})}
	{(1-q^{\mu_i-\mu_j} t^{j-i})(1-q^{\lambda_i-\lambda_j}t^{j-i})}.
\end{align}
Moreover, as in Remark \ref{rmk:negative_skew_functions}.2, using the obvious translation invariance $\psi'_{\lambda/\mu}=\psi'_{\lambda+1/\mu+1}$ (and same for $\varphi'_{\la/\mu}$), one may define $\psi'_{\lambda/\mu}$ and $\varphi'_{\lambda/\mu}$ for $\la,\mu\in\GT_N$ (i.e., for not necessarily nonnegative signatures, which, however, must have the same length). 

Let us consider a special case when $\la\in\GT_N$ differs from $\mu\in\GT_N$ as
\begin{align}\label{add_one_box}
	\la_j=\mu_j+1\mbox{ for some $j=1,\ldots,N$, and $\la_i=\mu_i$ for $i\ne j$}.
\end{align}
We will denote this relation by $\la=\mu+\de_j$. If $\la,\mu\in\GT_N^+$, then $\la=\mu+\de_j$ means that the Young diagram $\la$ is obtained from $\mu$ by adding one box to the $j$th row.

For $\la,\mu\in\GT_N$ with $\la=\mu+\de_j$ one can check that 
\begin{align}\label{psi_prime_one_box}
	\psi'_{\la/\mu}=\frac{1-q}{1-t}
	\varphi_{\la/\mu}=
	\prod_{i=1}^{j-1}
	\frac{(1-q^{\mu_i-\mu_j}t^{j-i-1})
	(1-q^{\la_i-\la_j}t^{j-i+1})}
	{(1-q^{\mu_i-\mu_j}t^{j-i})
	(1-q^{\la_i-\la_j}t^{j-i})}.
\end{align}


\subsection{Identities} 
\label{sub:identities}

Here we collect a number of useful 
formulas concerning skew and ordinary 
Macdonald symmetric functions.

Let $\rho_1$ and $\rho_2$ be two 
Macdonald-nonnegative\footnote{Most formulas in 
this subsection work without assuming 
Macdonald nonnegativity: 
one could instead take finite-length 
specializations at arbitrary variables.} 
specializations of the algebra of 
symmetric functions given by \eqref{Pi_nonneg_spec}. 
We will assume that these specializations are such 
that all expressions of the form $\Pi(\cdot;\cdot)$ 
(defined in~\eqref{Cauchy}) in this subsection are finite.

\subsubsection{Cauchy identity} 
\label{ssub:cauchy_identity}

We have the following identity \cite[VI.2]{Macdonald1995}:
\begin{align}\label{Cauchy}
	\sum_{\la\in\GT^+}
	P_\la(\rho_1)Q_\la(\rho_2)=
	\exp\bigg(
	\sum_{n=1}^{\infty}
	\frac{1}{n}\frac{1-t^{n}}{1-q^{n}}
	p_n(\rho_1)p_n(\rho_2)
	\bigg):=\Pi(\rho_1;\rho_2).
\end{align}
If one of the specializations, say, $\rho_1$, is a finite length specialization at the variables $a_1,\ldots,a_N$ (cf. \S \ref{sub:symmetric_functions}), then one has
\begin{align}\label{Cauchy_finite_length}
	\sum_{\la\in\GT^+}
	P_\la(a_1,\ldots,a_N)Q_\la(\rho_2)
	=\Pi(a_1;\rho_2)\ldots\Pi(a_N;\rho_2).
\end{align}
Here $\Pi(u;\rho)$ is defined in \eqref{Pi_nonneg_spec}.

It is clear from \eqref{Cauchy} that $\Pi(\rho_1;\rho_2)=\Pi(\rho_2;\rho_1)$. Moreover, if $\rho_1$ is a union of two specializations $\rho_1'$ and $\rho_1''$ (see \S \ref{sub:symmetric_functions}), then one has
\begin{align*}
	\Pi(\rho_1',\rho_1'';\rho_2)=
	\Pi(\rho_1';\rho_2)
	\Pi(\rho_1'';\rho_2).
\end{align*}
Because of that, we will
always write 
$\Pi(a_1,\ldots,a_N;\rho)$ 
instead of the product 
$\Pi(a_1;\rho)\ldots\Pi(a_N;\rho)$.


\subsubsection{Skew Cauchy identity} 
\label{ssub:skew_cauchy_identity}

We have \cite[VI.7]{Macdonald1995} for any fixed $\la,\nu\in\GT^+$:
\begin{align}\label{skew_Cauchy}
	\sum_{\ka\in\GT^+}
	P_{\ka/\la}(\rho_1)
	Q_{\ka/\nu}(\rho_2)
	=\Pi(\rho_1;\rho_2)
	\sum_{\mu\in\GT^+}
	Q_{\la/\mu}(\rho_2)P_{\nu/\mu}(\rho_1).
\end{align}
Note that the sum in the left-hand side 
is infinite while in the right-hand side
we have a finite summation over diagrams $\mu$
which must be inside both $\la$ and $\nu$.
In particular, for $\la=\varnothing$ the 
sum over $\mu$ in the right-hand side 
consists of a single 
summand corresponding to $\mu=\varnothing$, and one has
\begin{align}\label{skew_Cauchy_particular}
	\sum_{\ka\in\GT^+}
	P_{\ka}(\rho_1)
	Q_{\ka/\nu}(\rho_2)
	=\Pi(\rho_1;\rho_2)
	P_{\nu}(\rho_1).
\end{align}
Similarly,
\begin{align}\label{skew_Cauchy_particular_2}
	\sum_{\ka\in\GT^+}
	P_{\ka/\la}(\rho_1)
	Q_{\ka}(\rho_2)
	=\Pi(\rho_1;\rho_2)
	Q_{\la}(\rho_1).
\end{align}


\subsubsection{``Recurrence'' of skew Macdonald functions} 
\label{ssub:_recurrence_of_skew_macdonald_functions}

One has the following ``recurrence'' properties of the skew Macdonald functions \cite[VI.7]{Macdonald1995}:
\begin{align}\label{recurrence_skew_general_P}
	P_{\nu/\mu}(\rho_1,\rho_2)&=
	\sum_{\la\in\GT^+}
	P_{\nu/\la}(\rho_1)P_{\la/\mu}(\rho_2),\\
	Q_{\nu/\mu}(\rho_1,\rho_2)&=
	\sum_{\la\in\GT^+}
	Q_{\nu/\la}(\rho_1)Q_{\la/\mu}(\rho_2).
	\label{recurrence_skew_general_Q}
\end{align}
Here $\mu,\nu\in\GT^+$ are fixed.

Moreover, according to Remark \ref{rmk:negative_skew_functions}.1, we have for any (not necessarily nonnegative) signatures $\ka\in\GT_k$ and $\nu\in\GT_N$:
\begin{align}\label{recurrence_skew_negative}
	P_{\nu/\ka}(a_{N},a_{N-1},\ldots,a_{k+1})
	=\sum_{\la\in\GT_m}
	P_{\nu/\la}
	(a_N,\ldots,a_{m+1})
	P_{\la/\mu}(a_m,\ldots,a_{k+1})
\end{align}
for any fixed intermediate $m$, $k<m<N$ (see also Fig.~\ref{fig:GT_scheme}). A similar identity holds for the $Q$-functions.



\subsection{Commuting Markov operators and Macdonald measures} 
\label{sub:commuting_markov_operators}

This part of the appendix recalls some definitions 
from \cite[\S2.3.1]{BorodinCorwin2011Macdonald}. 
Our aim here is to describe Markov operators preserving 
the class of Macdonald measures on partitions, 
and list certain commutation relations which they satisfy.

We will consider \emph{Macdonald measures} living on the set of nonnegative signatures of length $k$:
\begin{align}\label{MMeasure_appendix}
	\M\M(a_1,\ldots,a_k;\rho)(\la)= \frac{P_{\lambda}(a_1,\ldots,a_k) Q_{\lambda}(\rho)} {\Pi(a_1,\ldots,a_k;\rho)},
	\qquad\la\in\GT_k^{+}.
\end{align}
These measures depend on positive parameters $a_1,\ldots,a_k$ and on a Mac\-do\-nald-nonnegative specialization $\rho$. We also assume that the normalizing constant $\Pi(a_1,\ldots,a_k;\rho)<\infty$. 

The \emph{stochastic links} from $\GT_k$ to $\GT_{k-1}$ (also depending on the parameters $a_1,\ldots,a_k$) are defined as 
\begin{align}\label{Links_Macdonald_appendix}
	\La^{k}_{k-1}(\ka,\nu)
	:=\frac
	{P_{\nu}(a_1,\ldots,a_{k-1})}
	{P_{\ka}(a_1,\ldots,a_{k})}
	P_{\ka/\nu}(a_k),\qquad
	\ka\in\GT_k,\ \nu\in\GT_{k-1}.
\end{align}
From \eqref{skew_Cauchy_particular_2} one readily gets (in matrix notation)
\begin{align}\label{MM_commutes_La}
	\M\M(a_1,\ldots,a_k;\rho)\La^{k}_{k-1}
	=
	\M\M(a_1,\ldots,a_{k-1};\rho).
\end{align}
That is, using a stochastic link $\La^{k}_{k-1}$, one can turn a Macdonald measure on $\GT_k$ into the corresponding Macdonald measure on $\GT_{k-1}$.

For any Macdonald-nonnegative specialization $\si$ define the matrix (we follow the notation of \cite[\S2.3.1]{BorodinCorwin2011Macdonald})
\begin{align}\label{p_Markov_Macdonald_operator}
	p^{\uparrow}_{\la\mu}(a_1,\ldots,a_k;\si):=
	\frac{1}{\Pi(a_1,\ldots,a_k;\si)}
	\frac{P_\mu(a_1,\ldots,a_k)}{P_\la(a_1,\ldots,a_k)}
	Q_{\mu/\la}(\si)
\end{align}
indexed by $\la,\mu\in\GT_k$. We assume that here also $\Pi(a_1,\ldots,a_k;\si)<\infty$. According to Remarks \ref{rmk:Macdonald_negative_sign} and \ref{rmk:negative_skew_functions}, all the objects in \eqref{p_Markov_Macdonald_operator} are well-defined for not necessarily nonnegative signatures $\la,\mu$. Moreover, the matrix elements $p^{\uparrow}_{\la\mu}$ are translation invariant, i.e., they do not change if one replaces $\la$ and $\mu$ by $\la+1$ and $\mu+1$, respectively.

The next proposition summarizes the properties of $p^{\uparrow}_{\la\mu}$:
\begin{proposition}\label{prop:commuting_Markov}
	{\rm\bf1.} The matrix 
	\begin{align*}
		p^{\uparrow}(a_1,\ldots,a_k;\si)=[p^{\uparrow}_{\la\mu}(a_1,\ldots,a_k;\si)]_{\la,\mu\in\GT_k}
	\end{align*}
	defines a Markov operator on $\GT_k$, i.e., it has nonnegative entries and
	\begin{align}\label{p_uparrow_sum_to_one}
		\sum_{\mu\in\GT_k}p^{\uparrow}_{\la\mu}(a_1,\ldots,a_k;\si)=1
		\qquad\mbox{for any $\la\in\GT_k$}.
	\end{align}

	{\rm\bf2.}
	The action of the Markov operator $p^{\uparrow}(a_1,\ldots,a_k;\si)$ on Macdonald measures \eqref{MMeasure_appendix} is given by
	\begin{align*}
		\M\M(a_1,\ldots,a_k;\rho)p^{\uparrow}(a_1,\ldots,a_k;\si)
		=\M\M(a_1,\ldots,a_k;\rho,\si).
	\end{align*}
	(Here $(\rho,\si)$ is the union of specializations, cf.
	\S \ref{sub:symmetric_functions}.)

	{\rm\bf3.} The operators $p^{\uparrow}(a_1,\ldots,a_k;\si)$ commute for various specializations $\si$:
	\begin{align*}
		p^{\uparrow}(a_1,\ldots,a_k;\si_1)
		p^{\uparrow}(a_1,\ldots,a_k;\si_2)=
		p^{\uparrow}(a_1,\ldots,a_k;\si_2)
		p^{\uparrow}(a_1,\ldots,a_k;\si_1).
	\end{align*}
	In fact, both sides are equal to $p^{\uparrow}(a_1,\ldots,a_k;\si_1,\si_2)$.

	{\rm\bf4.} The operators $p^{\uparrow}$ commute with the stochastic links $\La^{k}_{k-1}$ in the following sense:
	\begin{align}\label{p_up_commutes_links}
		p^{\uparrow}(a_1,\ldots,a_k;\si)\La^{k}_{k-1}
		=
		\La^{k}_{k-1}p^{\uparrow}(a_1,\ldots,a_{k-1};\si).
	\end{align}
\end{proposition}





\subsection{Schur polynomials and related objects} 
\label{sub:schur_appendix}

\subsubsection{Schur polynomials and Schur-nonnegative specializations} 
\label{ssub:schur_polynomials_and_schur_nonnegative_specializations}

In the Schur case, i.e., for $q=t$, 
the $P$- and $Q$- Macdonald 
polynomials
\eqref{Macd_poly}
coincide and become 
the \emph{Schur polynomials}, 
which are given by the following determinantal
formula:
\begin{align*}
	s_\la(x_1,\ldots,x_N)=
	\frac{\det[x_i^{\la_j+N-j}]_{i,j=1}^{N}}
	{\det[x_i^{N-j}]_{i,j=1}^{N}},
	\qquad \la\in\GT_N.
\end{align*}
If 
$\la$ contains negative parts, 
$s_\la(x_1,\ldots,x_N)$ becomes a 
symmetric Laurent polynomial in the variables $x_1,\ldots,x_N$
(cf. Remark \ref{rmk:Macdonald_negative_sign}).
For $\la\in\GT^+$, Schur polynomials $s_\la$
in arbitrarily many variables define the 
\emph{Schur symmetric functions} (see 
\S \ref{sub:symmetric_functions}
and~\S \ref{sub:macdonald_symmetric_functions}). 

Schur-nonnegative specializations of the 
algebra of symmetric functions
(they are defined similarly to 
\S \ref{sub:nonnegative_specializations})
are completely described by the
Thoma's theorem \cite{Thoma1964},
see also \cite{Kerov1998} and references therein.
Namely, these specializations 
depend on  
nonnegative parameters 
$\{\al_i\}_{i\ge1}$, $\{\be_i\}_{i\ge1}$, and $\ga$, 
and are defined using the generating series
(cf. \eqref{Pi_nonneg_spec}):
\begin{align}\label{Pi_Schur_nonneg}
	\exp\bigg(
	\sum_{n=1}^{\infty}
	\frac{p_n(\rho)}{n}u^n
	\bigg)=
	\exp(\gamma u) \prod_{i\ge 1} 
	\frac{1+\be_iu}{1-\al_iu}=: \Pi_{q=t}(u;\rho).
\end{align}
Expanding $\Pi_{q=t}$ 
as a Taylor series, we have
\begin{align*}
	\Pi_{q=t}(u;\rho)
	=\sum_{n=0}^{\infty}
	h_n(\rho)u^{n},
\end{align*}
where $h_n=s_{(n)}$, $n=0,1,\ldots$, are
the one-row Schur symmetric functions 
(also called \emph{complete 
homogeneous symmetric functions}).
Any Schur symmetric function can be written as a 
determinant of the one-row functions
(the Jacobi-Trudi formula):
\begin{align*}
	s_\la=\det[h_{\la_i-i+j}]_{i,j=1}^{N},
	\qquad \la\in\GT^{+}_{N},
\end{align*}
and thus one can in principle compute $s_\la(\rho)$
for any Schur-nonnegative specialization $\rho$.
In particular, see
\cite[Ex. I.3.5]{Macdonald1995},
\begin{align}\label{Schur_Plancherel_spec}
	s_\la(\rho_\ga)=\frac{\dim\la}{|\la|!}
	\ga^{|\la|},\qquad \la\in\GT^{+},
\end{align}
where $\rho_\ga$ is the 
Plancherel
specialization
corresponding to
a single nonzero parameter $\ga$, and
$\dim\la$ is the number of standard
Young tableaux of shape~$\la$ 
(Definition \ref{def:SYT}).
See \cite[I]{Macdonald1995} 
for a comprehensive treatment 
of Schur symmetric functions.

The skew Schur polynomials (and symmetric functions)
are defined in parallel to \S \ref{sub:skew_functions},
but now all the nonzero 
coefficients $\psi$ and $\varphi$ are 
simply equal to one, see \S \ref{sub:skew_functions}. 
In particular, the skew Schur polynomials in one
variable are given by
(cf. \eqref{P_one_variable}--\eqref{Q_one_variable})
\begin{align}\label{S_one_variable}
	s_{\la/\mu}(x_1)=x_1^{|\la|-|\mu|}1_{\mu\prec\la},
	\qquad\mu\in\GT_{N-1},\quad
	\la\in\GT_{N}.
\end{align}


\subsubsection{Schur degeneration of stochastic links} 
\label{ssub:schur_degeneration_of_stochastic_links}

In the Schur case the stochastic links 
\eqref{Links_Macdonald_appendix}
become (we have used \eqref{S_one_variable})
\begin{align}\label{La_Schur_Gibbs}
	\La^{k}_{k-1}
	(\la,\bar\la)
	=\frac{s_{\bar\la}(a_1,\ldots,a_{k-1})}
	{s_{\la}(a_1,\ldots,a_{k})}
	a_k^{|\la|-|\bar\la|}\cdot
	1_{\bar\la\prec\la},
	\qquad\bar\la\in\GT_{k-1},
	\quad\la\in\GT_{k}.
\end{align}
Here $a_1,\ldots,a_N$ are our usual positive variables.
In particular, in the case $a_1=\ldots=a_N=1$
the links become
\begin{align}\label{La_uniform_Gibbs}
	\La^{k}_{k-1}
	(\la,\bar\la)
	=\frac{\Dim_{N-1}\bar\la}
	{\Dim_N\la}\cdot
	1_{\bar\la\prec\la},
\end{align}
where $\Dim_N\la$
is the number of semistandard Young tableaux
of shape $\la$ over the alphabet $\{1,\ldots,N\}$,
and similarly for $\Dim_{N-1}\bar\la$
(see \S \ref{sub:semistandard_young_tableaux}).
The fact that $\Dim_N\la=s_\la(1,\ldots,1)$ ($N$ ones)
follows from the combinatorial formula
for the Schur polynomials, cf.
\eqref{combinatorial_skew_Macdonald}.





\begin{thebibliography}{10}

\bibitem{baik1999distribution}
J.~Baik, P.~Deift, and K.~Johansson.
\newblock {On the distribution of the length of the longest increasing
  subsequence of random permutations}.
\newblock {\em Journal of the American Mathematical Society}, 12(4):1119--1178,
  1999.
\newblock arXiv:math/9810105 [math.CO].

\bibitem{betea2011elliptically}
D.~Betea.
\newblock Elliptically distributed lozenge tilings of a hexagon, 2011.
\newblock arXiv:1110.4176 [math-ph].

\bibitem{biane2001approximate}
P.~Biane.
\newblock Approximate factorization and concentration for characters of
  symmetric groups.
\newblock {\em International Mathematics Research Notices}, 2001(4):179--192,
  2001.
\newblock arXiv:math/0006111 [math.RT].

\bibitem{BBO2004}
P.~Biane, P.~Bougerol, and N.~O'Connell.
\newblock Littelmann paths and brownian paths.
\newblock {\em Duke Mathematical Journal}, 130(1):127--167, 2005.
\newblock arXiv:math/0403171 [math.RT].

\bibitem{Borodin2010Schur}
A.~Borodin.
\newblock {Schur dynamics of the Schur processes}.
\newblock {\em Advances in Mathematics}, 228(4):2268--2291, 2011.
\newblock arXiv:1001.3442 [math.CO].

\bibitem{BorodinBufetov2013}
A.~Borodin and Al. Bufetov.
\newblock {Plancherel representations of $U(\infty)$ and correlated Gaussian
  Free Fields}.
\newblock 2013.
\newblock arXiv:1301.0511 [math.RT].

\bibitem{BorodinCorwin2011Macdonald}
A.~Borodin and I.~Corwin.
\newblock Macdonald processes, 2011.
\newblock arXiv:1111.4408 [math.PR].

\bibitem{BCGS2013}
A.~Borodin, I.~Corwin, V.~Gorin, and S.~Shakirov.
\newblock In preparation.
\newblock 2013.

\bibitem{BorodinCorwinSasamoto2012}
A.~Borodin, I.~Corwin, and T.~Sasamoto.
\newblock {From duality to determinants for q-TASEP and ASEP}.
\newblock 2012.
\newblock arXiv:1207.5035 [math.PR].

\bibitem{BorFerr08push}
A.~Borodin and P.~Ferrari.
\newblock {Large time asymptotics of growth models on space-like paths I:
  PushASEP}.
\newblock {\em Electron. J. Probab.}, 13:1380--1418, 2008.
\newblock arXiv:0707.2813 [math-ph].

\bibitem{BorFerr2008DF}
A.~Borodin and P.L. Ferrari.
\newblock Anisotropic growth of random surfaces in 2+1 dimensions.
\newblock 2008.
\newblock arXiv:0804.3035 [math-ph].

\bibitem{BorodinGorin2008}
A.~Borodin and V.~Gorin.
\newblock Shuffling algorithm for boxed plane partitions.
\newblock {\em Advances in Mathematics}, 220(6):1739--1770, 2009.
\newblock arXiv:0804.3071 [math.CO].

\bibitem{BorodinGorinSPB12}
A.~Borodin and V.~Gorin.
\newblock Lectures on integrable probability.
\newblock 2012.
\newblock arXiv:1212.3351 [math.PR].

\bibitem{BorodinGorin2013beta}
A.~Borodin and V.~Gorin.
\newblock {General beta Jacobi corners process and the Gaussian Free Field}.
\newblock 2013.
\newblock arXiv:1305.3627 [math.PR].

\bibitem{borodin-gr2009q}
A.~Borodin, V.~Gorin, and E.~Rains.
\newblock {q-Distributions on boxed plane partitions}.
\newblock {\em Selecta Mathematica, New Series}, 16(4):731--789, 2010.
\newblock arXiv:0905.0679 [math-ph].

\bibitem{BorodinOlshanski2010GTs}
A.~Borodin and G.~Olshanski.
\newblock {Markov processes on the path space of the Gelfand-Tsetlin graph and
  on its boundary}.
\newblock {\em Journal of Functional Analysis}, 263(1):248--303, September
  2012.
\newblock arXiv:1009.2029 [math.PR].

\bibitem{Chhaibi2013}
R.~Chhaibi.
\newblock {\em {Littelmann path model for geometric crystals, Whittaker
  functions on Lie groups and Brownian motion}}.
\newblock PhD thesis, 2013.
\newblock arXiv:1302.0902 [math.PR].

\bibitem{COSZ2011}
I.~Corwin, N.~O'Connell, T.~Sepp{\"a}l{\"a}inen, and N.~Zygouras.
\newblock {Tropical Combinatorics and Whittaker functions}.
\newblock 2011.
\newblock arXiv:1110.3489 [math.PR].

\bibitem{CorwinPetrov2013}
I.~Corwin and L.~Petrov.
\newblock In preparation.
\newblock 2013.

\bibitem{DiaconisFill1990}
P.~Diaconis and J.A. Fill.
\newblock Strong stationary times via a new form of duality.
\newblock {\em Ann. Probab.}, 18:1483--1522, 1990.

\bibitem{dyson1962brownian}
F.J. Dyson.
\newblock {A Brownian motion model for the eigenvalues of a random matrix}.
\newblock {\em Journal of Mathematical Physics}, 3(6):1191--1198, 1962.

\bibitem{fomin1995schur}
S.~Fomin.
\newblock {Schur Operators and Knuth Correspondences}.
\newblock {\em Journal of combinatorial theory. Series A}, 72(2):277--292,
  1995.

\bibitem{ForresterRains2007}
P.J. Forrester and E.M. Rains.
\newblock {Symmetrized models of last passage percolation and non-intersecting
  lattice paths}.
\newblock {\em Journal of Statistical Physics}, 129(5-6):833--855, 2007.
\newblock arXiv:0705.3925 [math-ph].

\bibitem{fulton1997young}
W.~Fulton.
\newblock {\em {Young Tableaux with Applications to Representation Theory and
  Geometry}}.
\newblock Cambridge University Press, 1997.

\bibitem{GerasimovLebedevOblezin2011}
A.~Gerasimov, D.~Lebedev, and S.~Oblezin.
\newblock {\em {On a classical limit of q-deformed Whittaker functions}}.
\newblock 2011.
\newblock arXiv:1101.4567 [math.AG].

\bibitem{johansson2000shape}
K.~Johansson.
\newblock {Shape fluctuations and random matrices}.
\newblock {\em Communications in mathematical physics}, 209(2):437--476, 2000.
\newblock arXiv:math/9903134 [math.CO].

\bibitem{Johansson2005lectures}
K.~Johansson.
\newblock Random matrices and determinantal processes.
\newblock 2005.
\newblock arXiv:math-ph/0510038.

\bibitem{Kato1980}
T.~Kato.
\newblock {\em Perturbation theory of linear operators}.
\newblock Springer-Verlag, New York, 2nd edition, 1980.

\bibitem{Kerov-book}
S.~Kerov.
\newblock {\em Asymptotic Representation Theory of the Symmetric Group and its
  Applications in Analysis}, volume 219.
\newblock AMS, Translations of Mathematical Monographs, 2003.

\bibitem{Kerov1998}
S.~Kerov, A.~Okounkov, and G.~Olshanski.
\newblock {T}he boundary of {Y}oung graph with {J}ack edge multiplicities.
\newblock {\em Intern. Math. Research Notices}, 4:173--199, 1998.
\newblock arXiv:q-alg/9703037.

\bibitem{Kirillov2000_Tropical}
A.N. Kirillov.
\newblock {Introduction to tropical combinatorics}.
\newblock In A.N. Kirillov and N.~Liskova, editors, {\em Physics and
  Combinatorics, Proceedings of the Nagoya 2000 International Workshop}, pages
  82--150, Singapore, 2001. World Scientific.

\bibitem{SmirnovSchubert2012}
V.~Kiritchenko, E.~Smirnov, and V.~Timorin.
\newblock {Schubert calculus and Gelfand-Zetlin polytopes}.
\newblock {\em Russian Mathematical Surveys}, 67(4):685--719, 2012.
\newblock arXiv:1101.0278 [math.AG].

\bibitem{Knuth1973art3}
D.E. Knuth.
\newblock {\em {The Art of Computer Programming, Vol. 3: Sorting and
  Searching}}.
\newblock Addison--Wesley, London, 1973.

\bibitem{Liggett2010}
T.~Liggett.
\newblock {\em Continuous Time Markov Processes: An Introduction}, volume 113
  of {\em Graduate Studies in Mathematics}.
\newblock AMS, Providence, RI, 2010.

\bibitem{Macdonald1995}
I.G. Macdonald.
\newblock {\em Symmetric functions and {H}all polynomials}.
\newblock Oxford University Press, 2nd edition, 1995.

\bibitem{Meliot20091/2}
P.-L. Meliot.
\newblock {Kerov's central limit theorem for Schur-Weyl measures of parameter
  1/2}.
\newblock 2009.
\newblock arXiv:1009.4034 [math.RT].

\bibitem{NoumiYamada2004}
M.~Noumi and Y.~Yamada.
\newblock {Tropical Robinson-Schensted-Knuth correspondence and birational Weyl
  group actions}.
\newblock In {\em Repsentation theory of algebraic groups and quantum groups},
  volume~40 of {\em Adv. Stud. Pure Math.}, pages 371--442. Math. Soc. Japan,
  Tokyo, 2004.
\newblock arXiv:math-ph/0203030.

\bibitem{OConnell2003Trans}
N.~O'Connell.
\newblock {A path-transformation for random walks and the Robinson-Schensted
  correspondence}.
\newblock {\em Transactions of the American Mathematical Society},
  355(9):3669--3697, 2003.

\bibitem{OConnell2003}
N.~O'Connell.
\newblock {Conditioned random walks and the RSK correspondence}.
\newblock {\em J. Phys. A}, 36(12):3049--3066, 2003.

\bibitem{Oconnell2009_Toda}
N.~O'Connell.
\newblock {Directed polymers and the quantum Toda lattice}.
\newblock {\em Ann. Probab.}, 40(2):437--458, 2012.
\newblock arXiv:0910.0069 [math.PR].

\bibitem{OConnellPei2012}
N.~O'Connell and Y.~Pei.
\newblock {A q-weighted version of the Robinson-Schensted algorithm}.
\newblock 2012.
\newblock arXiv:1212.6716 [math.CO].

\bibitem{OSZ2012}
N.~O'Connell, T.~Sepp{\"a}l{\"a}inen, and N.~Zygouras.
\newblock {Geometric RSK correspondence, Whittaker functions and symmetrized
  random polymers}.
\newblock 2011.
\newblock arXiv:1110.3489 [math.PR].

\bibitem{OConnellWarren2011}
N.~O'Connell and J.~Warren.
\newblock A multi-layer extension of the stochastic heat equation.
\newblock 2011.
\newblock arXiv:1104.3509 [math.PR].

\bibitem{OConnellYor2001}
N.~O'Connell and M.~Yor.
\newblock {Brownian analogues of Burke's theorem}.
\newblock {\em Stochastic Processes and their Applications}, 96(2):285--304,
  2001.

\bibitem{okounkov2003correlation}
A.~Okounkov and N.~Reshetikhin.
\newblock {Correlation function of Schur process with application to local
  geometry of a random 3-dimensional Young diagram}.
\newblock {\em Journal of the American Mathematical Society}, 16(3):581--603,
  2003.
\newblock arXiv:math/0107056 [math.CO].

\bibitem{sagan2001symmetric}
B.E. Sagan.
\newblock {\em {The symmetric group: representations, combinatorial algorithms,
  and symmetric functions}}.
\newblock Springer Verlag, 2001.

\bibitem{Schensted1961}
C.~Schensted.
\newblock Longest increasing and decreasing subsequences.
\newblock {\em Canad. J. Math.}, 13:179--191, 1961.

\bibitem{Spitzer1970}
F.~Spitzer.
\newblock {Interaction of Markov processes}.
\newblock {\em Adv. Math.}, 5(2):246--290, 1970.

\bibitem{Stanley1999}
R.~Stanley.
\newblock {\em Enumerative {C}ombinatorics. {V}ol. 2}.
\newblock Cambridge University Press, Cambridge, 2001.
\newblock With a foreword by Gian-Carlo Rota and appendix 1 by Sergey Fomin.

\bibitem{Thoma1964}
E.~Thoma.
\newblock Die unzerlegbaren, positive-definiten {K}lassenfunktionen der
  abz\"ahlbar unendlichen, symmetrischen {G}ruppe.
\newblock {\em Math. Zeitschr}, 85:40--61, 1964.

\bibitem{Vershik1986}
A.~Vershik and S.~Kerov.
\newblock The characters of the infinite symmetric group and probabiliy
  properties of the {R}obinson-{S}hensted-{K}nuth algorithm.
\newblock {\em Sima J. Alg. Disc. Math.}, 7(1):116--124, 1986.

\bibitem{warren2005dyson}
J.~Warren.
\newblock {Dyson's Brownian motions, intertwining and interlacing}.
\newblock {\em Electron. J. Probab.}, 12(19):573--590, 2007.
\newblock arXiv:math/0509720 [math.PR].

\bibitem{Weyl1946}
H.~Weyl.
\newblock {\em {The Classical Groups. Their Invariants and Representations}}.
\newblock Princeton University Press, 1997.

\end{thebibliography}
\end{document}